\documentclass{article}
\usepackage[margin=1in]{geometry}
\usepackage{amsthm,amsmath,amssymb}
\usepackage{graphicx}
\usepackage[hidelinks,breaklinks]{hyperref}
\usepackage{listings}
\usepackage{textcomp}
\newtheorem{theorem}{Theorem}[]

\newtheorem{corollary}[theorem]{Corollary}

\newtheorem{remark1}[theorem]{Remark}

\DeclareMathOperator{\E}{\mathop{}\mathbb{E}}
\def\epsilon{\varepsilon}

\newlength{\vertsep}
\newlength{\imsize}
\newlength{\imsizes}
\newlength{\imsized}
\title{Metrics of calibration for probabilistic predictions}
\author{Imanol Arrieta-Ibarra, Paman Gujral, Jonathan Tannen, Mark Tygert,
and Cherie Xu}

\begin{document}

\maketitle

\begin{abstract}
Many predictions are probabilistic in nature; for example,
a prediction could be for precipitation tomorrow, but with only a 30\% chance.
Given such probabilistic predictions together with the actual outcomes,
``reliability diagrams'' (also known as ``calibration plots'') help detect
and diagnose statistically significant discrepancies
--- so-called ``miscalibration'' ---
between the predictions and the outcomes.
The canonical reliability diagrams are based on histogramming
the observed and expected values of the predictions;
replacing the hard histogram binning with soft kernel density estimation
using smooth convolutional kernels is another common practice.
But, which widths of bins or kernels are best?
Plots of the cumulative differences between the observed and expected values
largely avoid this question, by displaying miscalibration directly
as the slopes of secant lines for the graphs. Slope is easy to perceive
with quantitative precision, even when the constant offsets
of the secant lines are irrelevant; there is no need to bin
or perform kernel density estimation.

The existing standard metrics of miscalibration each summarize
a reliability diagram as a single scalar statistic.
The cumulative plots naturally lead to scalar metrics
for the deviation of the graph of cumulative differences away from zero;
good calibration corresponds to a horizontal, flat graph
which deviates little from zero.
The cumulative approach is currently unconventional,
yet offers many favorable statistical properties,
guaranteed via mathematical theory backed by rigorous proofs
and illustrative numerical examples.
In particular, metrics based on binning or kernel density estimation
unavoidably must trade-off statistical confidence
for the ability to resolve variations as a function
of the predicted probability or vice versa.
Widening the bins or kernels averages away random noise
while giving up some resolving power.
Narrowing the bins or kernels enhances resolving power
while not averaging away as much noise.
The cumulative methods do not impose such an explicit trade-off.
Considering these results, practitioners probably should adopt
the cumulative approach as a standard for best practices.
\end{abstract}

\begin{figure}
\begin{center}
\parbox{\imsizes}{\includegraphics[width=\imsizes]
       {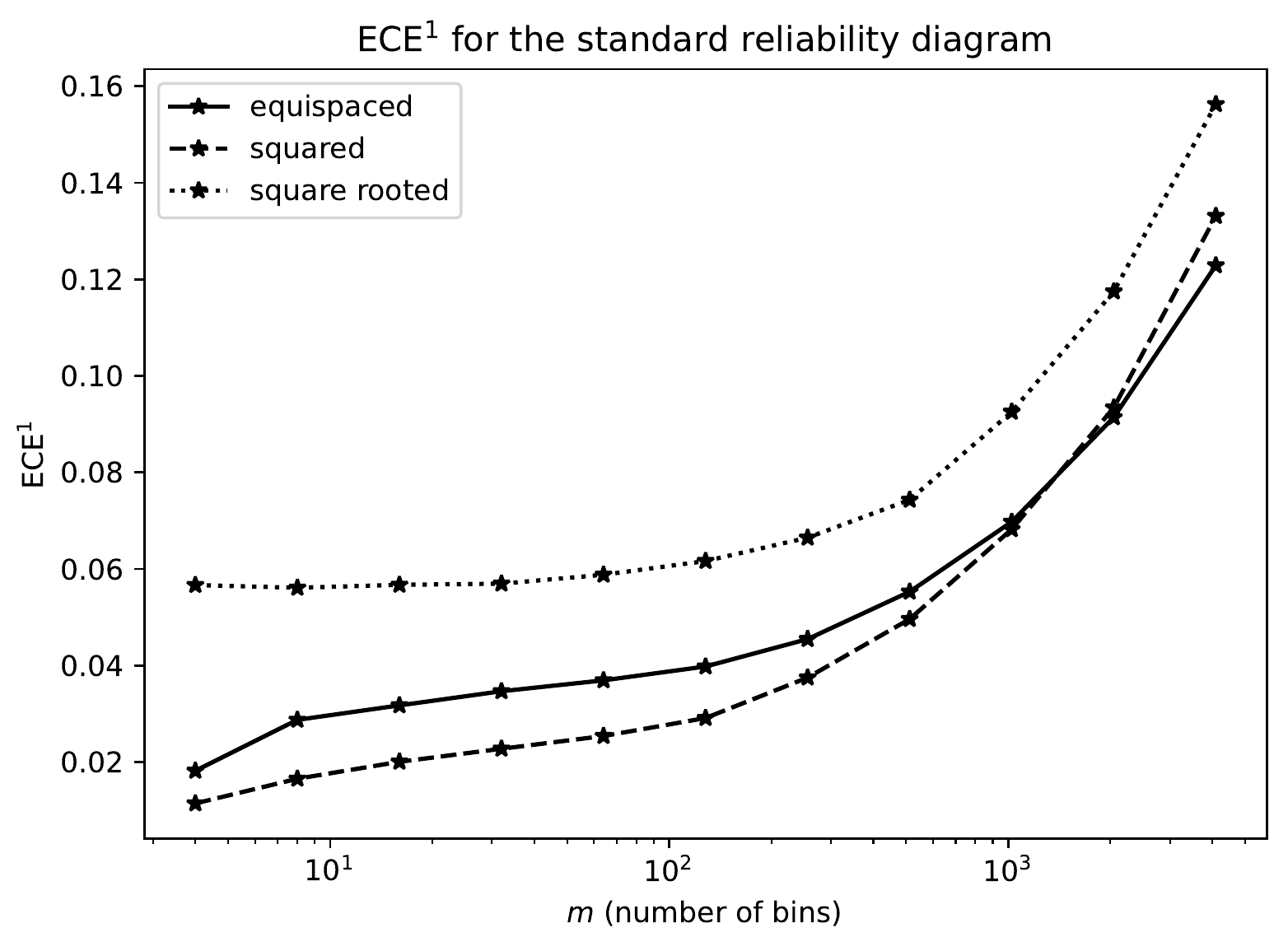}}
\hfil
\parbox{\imsizes}{\includegraphics[width=\imsizes]
       {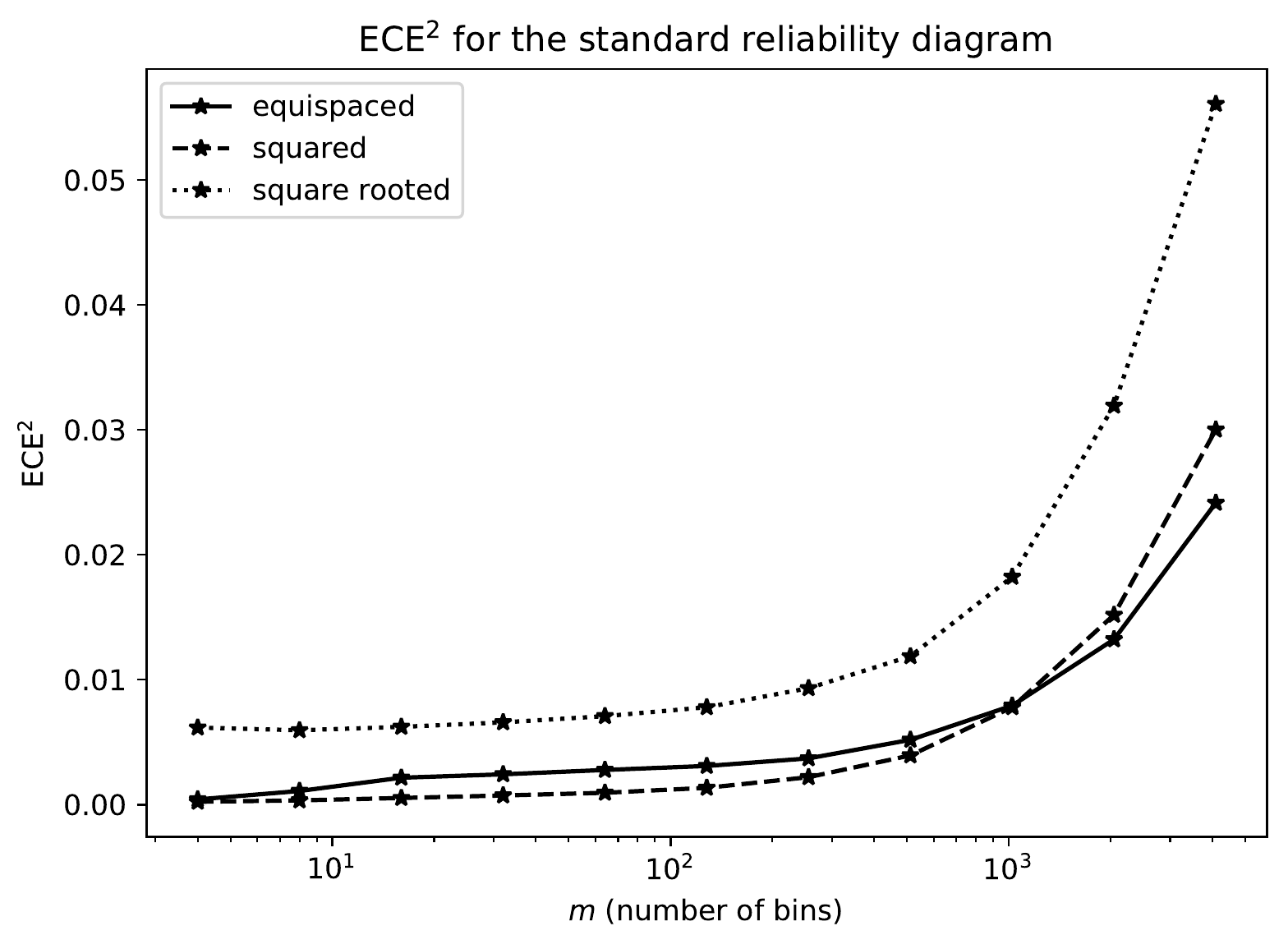}}

\parbox{\imsizes}{\includegraphics[width=\imsizes]
       {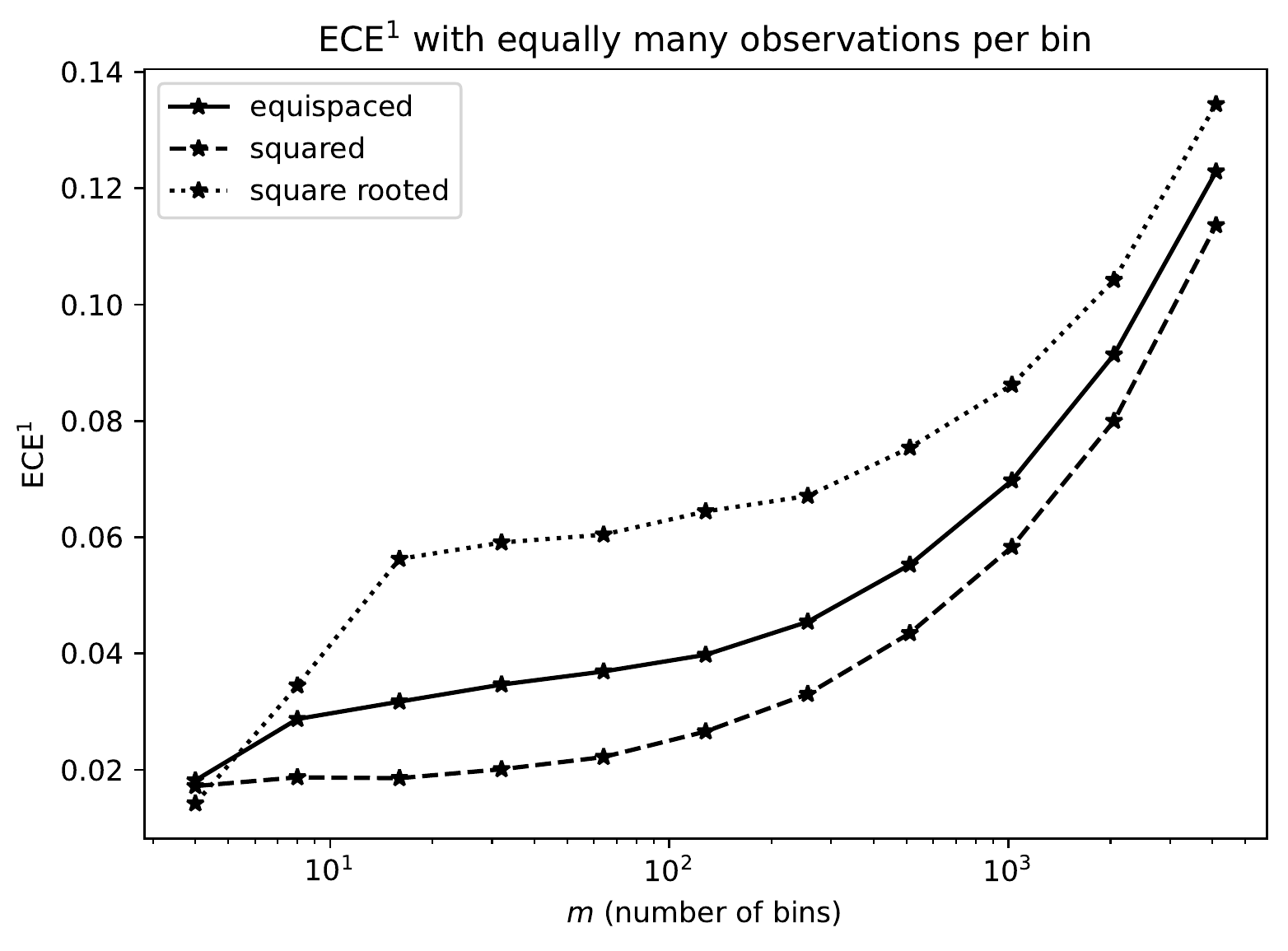}}
\hfil
\parbox{\imsizes}{\includegraphics[width=\imsizes]
       {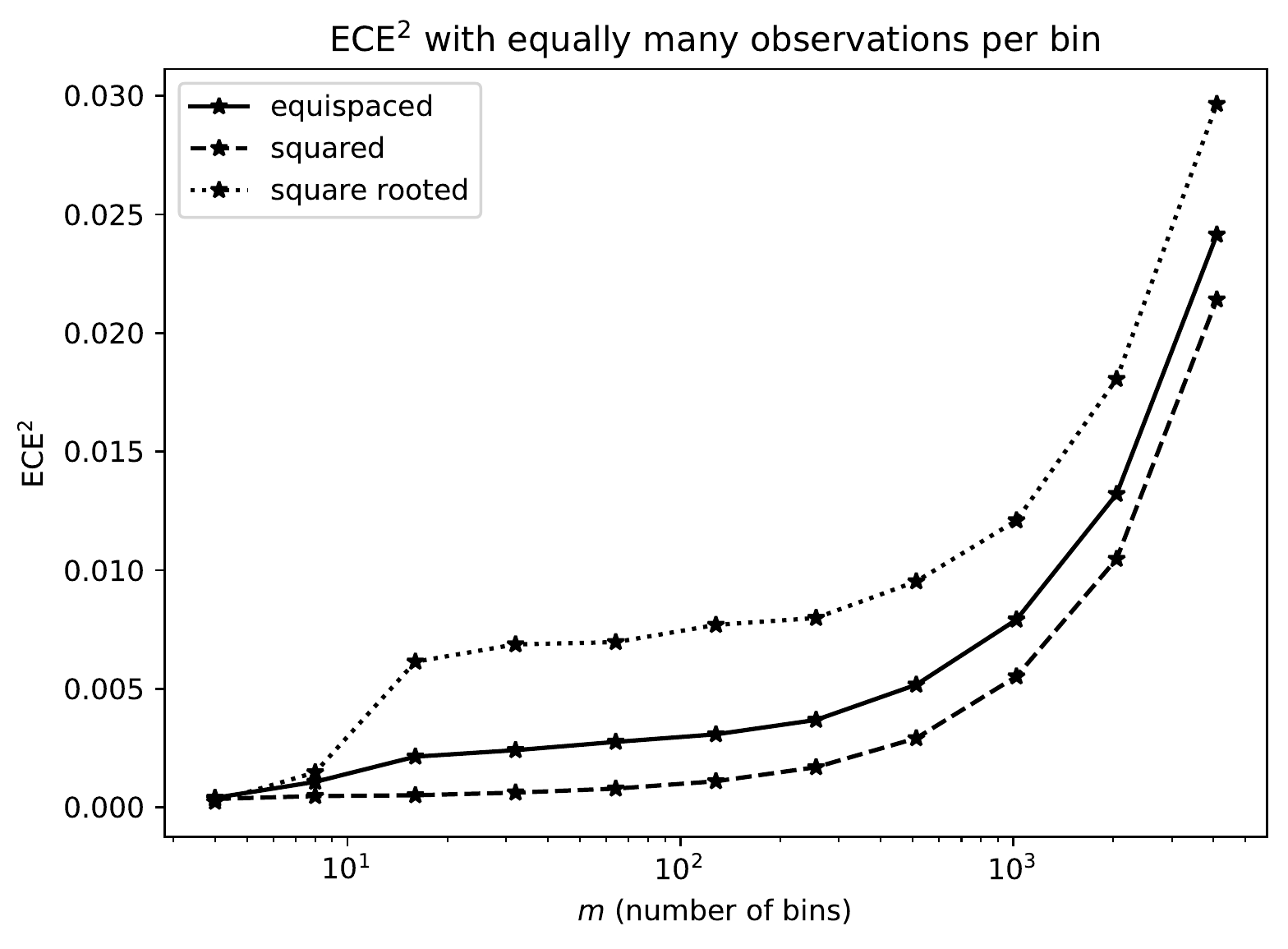}}
\end{center}
\caption{Empirical calibration errors for the data set
         of Subsection~\ref{synthetic} with sample size $n =$ 32,768;
         the scores are equispaced, squared after initially being equispaced,
         or square rooted after initially being equispaced,
         as indicated in the legends.}
\label{32768ece}
\end{figure}

\section{Introduction}

Given 100 independent observations of outcomes (``success'' or ``failure'')
of Bernoulli trials that are forecast to have an 80\% chance of success,
the forecasts are perfectly {\it calibrated} when 80 of the observations
report success. More generally, given some number, say $n$,
of independent observations of outcomes of Bernoulli trials
that are forecast to have a probability $S$ of success,
the predictions are perfectly calibrated when $nS$ of the observations
report success. Needless to say, the actual number of observations of success
is likely to vary around $nS$ randomly, so in practice we test not whether
$nS$ is exactly equal to the observed number of successes, but rather
whether the difference between $nS$ and the observed number of successes
is statistically significant. Such significance tests can be found
in any standard textbook on statistics.

The present paper considers the following more general setting:
Suppose we have $n$ observations $R_1$, $R_2$, \dots, $R_n$
of the outcomes of independent Bernoulli trials
with corresponding predicted probabilities
of success, say $S_1$, $S_2$, \dots, $S_n$.
For instance, each $S_k$ could be a classifier's probabilistic score
and the corresponding $R_k$ could be the indicator of correct classification,
with $R_k = 1$ when the classification is correct
and $R_k = 0$ when the classification is incorrect
(so $R_k$ could also be regarded as a class label,
where class 1 corresponds to ``the classifier succeeded''
and class 0 corresponds to ``the classifier erred'').
We would then want to test the hypothesis that the response $R_k$
is distributed as a Bernoulli variate with expected value $S_k$
for all $k = 1$, $2$, \dots, $n$, which is the same null hypothesis
considered in the previous paragraph when $S_1 = S_2 = \dots = S_n = S$.
The remainder of this paper simplifies the analysis by reordering the samples
(preserving the pairing of $R_k$ with $S_k$ for every $k$)
such that $S_1 \le S_2 \le \dots \le S_n$, with any ties ordered randomly,
perturbing so that $S_1 < S_2 < \dots < S_n$.

One method for measuring calibration is to histogram the responses
$R_1$, $R_2$, \dots, $R_n$ as a function of the scores
$S_1$, $S_2$, \dots, $S_n$; this involves partitioning the scores
into $m$ sets known as ``bins'' (or ``buckets'')
and calculating both the average score and the average response
for the scores and corresponding responses falling in each bin.
Summing over every bin either the absolute value of the difference
between the average response and the average score,
or else the square of the difference, each weighted by the width of the bin,
then estimates the area (or squared differences) between the observed responses
and ideal calibration. These sums are known as ``empirical calibration errors''
or ``estimated calibration errors,'' with the sum of the absolute values
being denoted ``ECE$^1$'' and the sum of the squares being denoted ``ECE$^2$''.

Another method for measuring calibration is to graph the cumulative differences
between the responses and the scores.
The expected slope of any secant line connecting two points on the graph
is equal to the average miscalibration over the scores between those points.
In the case of perfect calibration (for which each response
is equal to the corresponding score), the resulting graph
is a horizontal, flat line at zero.
Both the maximum deviation of the graph from zero and the range
(maximum minus minimum) of the graph measure the deviation
of the graph from the horizontal, flat ideal of perfect calibration.
We call these metrics ``empirical cumulative calibration errors''
or ``estimated cumulative calibration errors'' (ECCEs),
with the maximum absolute deviation being denoted ``ECCE-MAD''
and the range of deviations being denoted ``ECCE-R''.
Our earlier papers referred to ``ECCE-MAD''
as the ``Kolmogorov-Smirnov'' statistic
and to ``ECCE-R'' as the ``Kuiper'' statistic.

The present paper follows up and elaborates on problems highlighted earlier
by~\cite{gupta-rahimi-ajanthan-mensink-sminchisescu-hartley},
\cite{roelofs-cain-shlens-mozer}, \cite{tygert_full}, and~\cite{tygert_two};
they point out that the classical empirical calibration errors
based on binning vary significantly based on the choice of bins.
The choice of bins is fairly arbitrary and enables the analyst to fudge results
(whether purposefully or unintentionally).
Having to make such a critical yet arbitrary choice is especially fraught
when dealing with laws and regulators seeking a universal standard
for compliance.
Also concerning is bias observed by~\cite{roelofs-cain-shlens-mozer}
in the estimates given by empirical calibration errors based on binning.

The results of the present paper are all implicit in those
of~\cite{brocker}, \cite{gupta-rahimi-ajanthan-mensink-sminchisescu-hartley},
\cite{kumar-liang-ma}, \cite{nixon-dusenberry-zhang-jerfel-tran}, 
\cite{roelofs-cain-shlens-mozer}, \cite{simonoff-udina}, \cite{srihera-stute},
\cite{stute}, and~\cite{vaicenavicius-widmann-andersson-lindsten-roll-schoen}.
The purpose of the present paper is to provide a simple, rigorous exposition
of what might be viewed as a unifying thread throughout the other works.
In particular, this paper directly compares the ECEs to the ECCEs,
more extensively than prior work has.

The empirical calibration errors based on binning intrinsically
trade-off resolution for statistical confidence or vice versa.
Widening the bins averages away more noise in the estimates, while sacrificing
some of the power to resolve variations as a function of score.
Narrowing the bins resolves finer variations as a function of score,
while not averaging away as much noise in the estimates.
In contrast, the empirical {\it cumulative} calibration errors exhibit
no such explicit trade-off, with no parameters to adjust.
The empirical cumulative calibration errors are fully non-parametric
and uniquely, fully specified, statistically powerful and reliable.

The present paper directly treats unweighted samples.
To be sure, the results of the present paper generalize to the case
of weighted samples, in which each observation comes
with a positive real number that indicates how heavily
to weight the observation when combining it with the other observations
in expected values.
However, weighted sampling introduces additional complications that distract
from the comparison between the standard binned metrics
and the cumulative metrics, so the present paper focuses
on the case of unweighted sampling.
(Of course, unweighted sampling is equivalent to uniform weighting,
in which all weights are equal.)
Extensive graphical comparisons for the case of weighted sampling
are available from~\cite{tygert_full} and~\cite{tygert_two}.

The remainder of the present paper has the following structure:
the next section, Section~\ref{methods},
details the methodologies and proves theorems about their performance.
Then, Section~\ref{results} illustrates the methodologies and theory
of Section~\ref{methods} via several examples, using both synthetic
and measured data sets.\footnote{Permissively licensed open-source software
that automatically reproduces all figures and statistics reported below
is available at \url{https://github.com/facebookresearch/ecevecce}}
Finally, Section~\ref{conclusion} 
reviews the results, drawing conclusions, and the appendix supplements
the plots of Section~\ref{results} with a couple additional plots.

\section{Methods}
\label{methods}

This section gives theorems characterizing advantages of cumulative metrics
over the standard binned metrics. A full exposition requires
the detailed, rigorous proofs provided in this section; however,
the high-level strategy of all the derivations is quite simple
and straightforward, summarized as follows (please note that all sections
and subsections of the present paper use the same notation
that Subsection~\ref{defsub} introduces):

First, Subsection~\ref{defsub} defines both standard binned metrics
and cumulative metrics for assessing deviation from perfect calibration.
For a {\it perfectly} calibrated underlying distribution,
both the ideal, underlying calibration error
and the ideal, underlying cumulative calibration error are equal to 0,
whereas for an {\it imperfectly} calibrated underlying distribution,
both are greater than 0.

Next, Subsection~\ref{perfectsub} proves that,
for a {\it perfectly} calibrated underlying distribution,
as the sample size increases without bound
the cumulative metrics converge to 0
if the empirical cumulative distribution function of the scores
converges uniformly to a continuous cumulative distribution function,
while the expected values
of the standard binned metrics stay bounded away from 0
if the number of draws per bin remains bounded on some fixed range of scores,
as the maximum bin width becomes arbitrarily small.

Then, Subsection~\ref{imperfectsub} proves that,
for an {\it imperfectly} calibrated underlying distribution,
as the sample size becomes arbitrarily large
the expected values of the cumulative metrics stay bounded away from 0
if the empirical cumulative distribution function of the scores
converges uniformly to a cumulative distribution function,
and the expected values of the standard binned metrics stay bounded away from 0
if the maximum bin width becomes arbitrarily small.
This shows that the cumulative metrics can distinguish
any imperfectly calibrated underlying distribution from perfect calibration,
given enough observed scores and associated responses,
while in contrast there exist imperfectly calibrated distributions
which the standard binned metrics cannot distinguish
from perfect calibration as the maximum width of a bin approaches 0,
if the number of draws per bin remains bounded on some fixed range of scores.

Subsection~\ref{consequences} then derives the consequences, namely,
that the standard binned metrics are intrinsically subject
to an unavoidable trade-off, requiring infinitely denser observations
than the cumulative statistics in the limit required
for asymptotic consistency. 
Moreover, the trade-off is even more unwieldy when the number of observations
is limited: the standard binned metrics vary significantly
with the (rather arbitrary) choice of bins, whereas the cumulative metrics
work uniformly well without any tuning.
In fact, the cumulative metrics require no tuning at all ---
the cumulative metrics have no tuning parameters whatsoever;
fudging their results is simply impossible.

Finally, Subsection~\ref{graphical} summarizes the usual motivations
for the constructions of these particular metrics,
reviewing the associated graphical methods.

Most of the rest of the present section,
Subsections~\ref{defsub}--\ref{consequences},
now provides rigorous details of this strategy
(with Subsection~\ref{defsub} setting the notational conventions
used throughout the present paper).
Subsection~\ref{graphical} briefly reviews the principal motivations
for considering the particular metrics studied here.

\subsection{Definitions}
\label{defsub}

This subsection defines in detail the metrics analyzed below.
Subsubsection~\ref{binnedsub} defines the standard binned statistics,
while Subsubsection~\ref{cumsub} defines the cumulative statistics.
Graphical methods summarized in Subsection~\ref{graphical} motivate
all these definitions, though the present subsection omits discussion
of the motivation in order to keep the exposition concise and easily digestible
for those already familiar with the motivations.

\subsubsection{Binned}
\label{binnedsub}

The observations come as pairs of scores and responses.
Each score is a real number in the unit interval $[0, 1]$;
each response is a random variable whose value is either 0 or 1.
The positive integer $m$ will denote the number of bins;
bin $j$ contains the scores $S_j^k$ for $k = 1$, $2$, \dots, $n_j$,
for $j = 1$, $2$, \dots, $m$.
Each score $S_j^k$ comes with a response $R_j^k$;
the responses are independent and, under the null hypothesis
of perfect calibration, $R_j^k$ is a Bernoulli variate whose expected value
is $S_j^k$ (the variance is then $S_j^k \left( 1 - S_j^k \right)$).
We order the scores such that $S_i^k < S_j^\ell$ whenever $i < j$,
and $S_i^k < S_j^\ell$ when $i = j$ and $k < \ell$.
For notational convenience,
we define $S_0^1 = 0$ and $S_{m+1}^1 = 1$.
We denote by $n$ the total number of observations, that is,
$n = n_1 + n_2 + \dots + n_m$.
{\it We suppress the sequence index $n$ for $S_j^k$, $R_j^k$, $n_j$, and $m$;
in principle the full notation would be $S_j^k(n)$, $R_j^k(n)$, $n_j(n)$,
and $m(n)$, where $n = 1$, $2$, $3$, \dots};
below, ``convergence'' refers to convergence with respect
to the sample size $n$ increasing without bound.
The scores for each sample size $n$ are assumed to be distinct.
The average score in bin $j$ is
\begin{equation}
\label{avgscore}
\tilde{S}_j = \frac{1}{n_j} \sum_{k=1}^{n_j} S_j^k,
\end{equation}
while the average response in bin $j$ is
\begin{equation}
\label{avgresponse}
\tilde{R}_j = \frac{1}{n_j} \sum_{k=1}^{n_j} R_j^k.
\end{equation}
The mean-square empirical calibration error (ECE$^2$) is the Riemann sum
\begin{equation}
\label{l2}
\hbox{ECE}^2 = \sum_{j=1}^m \left( S_{j+1}^1 - S_j^1 \right)
\left( \tilde{R}_j - \tilde{S}_j \right)^2
= \sum_{j=1}^m \left(S_{j+1}^1 - S_j^1 \right)
\left( \sum_{k=1}^{n_j} \frac{R_j^k - S_j^k}{n_j} \right)^2;
\end{equation}
analogously, the $l^1$ empirical calibration error (ECE$^1$) is the Riemann sum
\begin{equation}
\label{l1}
\hbox{ECE}^1 = \sum_{j=1}^m \left( S_{j+1}^1 - S_j^1 \right)
\left| \tilde{R}_j - \tilde{S}_j \right|
= \sum_{j=1}^m \left(S_{j+1}^1 - S_j^1 \right)
\left| \sum_{k=1}^{n_j} \frac{R_j^k - S_j^k}{n_j} \right|.
\end{equation}

\subsubsection{Cumulative}
\label{cumsub}

We now define $S_1$, $S_2$, \dots, $S_n$ to be all $n = \sum_{j=1}^m n_j$
of the scores $S_j^k$, sorted in ascending order, so that
$S_1 < S_2 < \dots < S_n$; if $S_\ell$ is the score equal to $S_j^k$,
then we set $R_\ell$ to be equal to the corresponding response, $R_j^k$;
for notational convenience, we define $S_0 = 0$ and $S_{n+1} = 1$.
{\it As in the binned case, we suppress the sequence index $n$
for $S_\ell$ and $R_\ell$; in principle the full notation would be
$S_\ell(n)$ and $R_\ell(n)$, where $n = 1$, $2$, $3$, \dots};
as mentioned earlier, ``convergence'' will refer to convergence
with respect to the index $n$ increasing without bound. 
The scores for each sample size $n$ are assumed to be distinct.
The cumulative differences are
\begin{equation}
\label{cumulative}
C_k = \frac{1}{n} \sum_{j=1}^k (R_j - S_j)
\end{equation}
for $k = 1$, $2$, \dots, $n$.
The maximum absolute deviation of the empirical cumulative calibration error
(ECCE-MAD) is
\begin{equation}
\label{ks}
\hbox{ECCE-MAD} = \max_{1 \le k \le n} |C_k|
\end{equation}
and the range of the empirical cumulative calibration error (ECCE-R) is
\begin{equation}
\label{ku}
\hbox{ECCE-R} = \max_{0 \le k \le n} C_k - \min_{0 \le k \le n} C_k,
\end{equation}
where $C_k$ is defined in~(\ref{cumulative}) and $C_0 = 0$.
Another term for ``ECCE-MAD'' is the ``Kolmogorov-Smirnov metric,''
and another term for ``ECCE-R'' is the ``Kuiper metric'' ---
\cite{Kolmogorov}, \cite{Smirnov}, and~\cite{Kuiper} introduced
these statistics in order to solve a similar problem, namely,
determining the statistical significance of observed differences
in empirical probability distributions.

The absolute value of the total miscalibration
$\sum_{j \in I} (R_j - S_j)/n$
over any interval $I$ of indices is less than or equal to the ECCE-R;
indeed, the ECCE-R is the maximum of the absolute value
of the total miscalibration over any interval of indices:
$\hbox{ECCE-R} = \max_I \left| \sum_{j \in I} (R_j - S_j)/n \right|$.

\subsection{Perfectly calibrated responses}
\label{perfectsub}

This subsection analyzes the expected values of the metrics when the responses
arise from a perfectly calibrated distribution, that is,
under the assumption of the null hypothesis of perfect calibration.
The principal results of this subsection
are Corollaries~\ref{l2co} and~\ref{l1co} for the ECE
and Corollary~\ref{ecceco} for the ECCE.

\subsubsection{Binned}

The following theorem provides the lower limit of the ECE$^2$,
yielding the very useful Corollary~\ref{l2co}.

\begin{theorem}
\label{perfectly}
Assume the null hypothesis that the expected value of the response $R_j^k$ is
equal to the corresponding score $S_j^k$ for all $j = 1$, $2$, \dots, $m$;
$k = 1$, $2$, \dots, $n_j$.
Suppose also that $\max_{0 \le j \le m} (S_{j+1}^1 - S_j^1)$ converges to 0
as the sample size $n$ increases without bound,
and that $\nu = \nu_n$ is the step function
starting from $\nu(0) = \nu(S_0^1) = n_1$ and jumping to $\nu(S_j^k) = n_j$
for all $j = 1$, $2$, \dots, $m$; $k = 1$, $2$, \dots, $n_j$.
Then, as $n$ becomes arbitrarily large
the lower limit of the expected value of the ECE$^2$ defined in~(\ref{l2})
converges to
\begin{equation}
\label{lowerbound}
\liminf_{n \to \infty} \int_0^1 \frac{s (1 - s)}{\nu(s)} \, ds,
\end{equation}
where $\liminf$ denotes the lower limit and the sequence index $n$
of the function $\nu = \nu_n$ is suppressed in the notation.
\end{theorem}

\begin{proof}
Since the responses are all independent, the expected value of the ECE$^2$
from~(\ref{l2}) under the assumption of the null hypothesis is
\begin{equation}
\label{approx}
\sum_{j=1}^m \left( S_{j+1}^1 - S_j^1 \right)
\sum_{k=1}^{n_j} \frac{S_j^k \left( 1 - S_j^k \right)}{(n_j)^2}
= \sum_{j=1}^m \left( S_{j+1}^1 - S_j^1 \right)
\left( \frac{S_j^1 \left( 1 - S_j^1 \right)}{n_j}
+ \sum_{k=1}^{n_j}\frac{S_j^k \left( 1 - S_j^k \right)
                        - S_j^1 \left( 1 - S_j^1 \right)}{(n_j)^2} \right),
\end{equation}
which is a Riemann sum for which the second sum in the right-hand side
of~(\ref{approx}) converges uniformly to 0
as $\max_{0 \le j \le m} (S_{j+1}^1 - S_j^1)$ tends to 0
(uniformly over $j$ and independent of the values for $n_j$).
The lower limit of the right-hand side of~(\ref{approx})
converges to~(\ref{lowerbound}).
\end{proof}

The main consequence of this theorem is the following,
stating that the ECE$^2$ hits a ``noise floor.''

\begin{corollary}
\label{l2co}
If $\nu = \nu_n$ is bounded from above on some interval,
independent of the sample size $n$,
and is the step function starting from $\nu(0) = \nu(S_0^1) = n_1$
and jumping to $\nu(S_j^k) = n_j$
for all $j = 1$, $2$, \dots, $m$; $k = 1$, $2$, \dots, $n_j$,
then the expected value of the ECE$^2$ defined in~(\ref{l2}) is greater
than a fixed strictly positive real number for all sufficiently large $n$,
assuming the null hypothesis of perfect calibration
and that the bin width $(S_{j+1}^1 - S_j^1)$ converges uniformly to 0
(uniformly over $j = 0$, $1$, \dots, $m(n)$).
\end{corollary}

A similar corollary holds for the ECE$^1$ defined in~(\ref{l1}),
due to the following theorem.

\begin{theorem}
\label{dominance}
The ECE$^1$ defined in~(\ref{l1}) is greater than or equal to the ECE$^2$
defined in~(\ref{l2}).
\end{theorem}

\begin{proof}
It follows from the fact that both the scores and the responses fall
on the unit interval $[0, 1]$ that
\begin{equation}
\left| \sum_{k=1}^{n_j} \frac{R_j^k - S_j^k}{n_j} \right|
\le \sum_{k=1}^{n_j} \frac{|R_j^k - S_j^k|}{n_j}
\le \sum_{k=1}^{n_j} \frac{1}{n_j} = 1,
\end{equation}
so
\begin{equation}
\label{squared}
\left| \sum_{k=1}^{n_j} \frac{R_j^k - S_j^k}{n_j} \right|^2
\le \left| \sum_{k=1}^{n_j} \frac{R_j^k - S_j^k}{n_j} \right|.
\end{equation}
It follows from~(\ref{squared}) that
\begin{equation}
\sum_{j=1}^m (S_{j+1}^1 - S_j^1)
\left| \sum_{k=1}^{n_j} \frac{R_j^k - S_j^k}{n_j} \right|
\ge \sum_{j=1}^m (S_{j+1}^1 - S_j^1)
\left( \sum_{k=1}^{n_j} \frac{R_j^k - S_j^k}{n_j} \right)^2,
\end{equation}
which is the statement of the theorem expressed in terms
of the definitions in~(\ref{l2}) and~(\ref{l1}).
\end{proof}

Combining Theorem~\ref{dominance} and Corollary~\ref{l2co}
yields the following, stating that the ECE$^1$ hits a ``noise floor.''

\begin{corollary}
\label{l1co}
If $\nu = \nu_n$ is bounded from above on some interval,
independent of the sample size $n$,
and is the step function starting from $\nu(0) = \nu(S_0^1) = n_1$
and jumping to $\nu(S_j^k) = n_j$
for all $j = 1$, $2$, \dots, $m$; $k = 1$, $2$, \dots, $n_j$,
then the expected value of the ECE$^1$ defined in~(\ref{l1}) is greater
than a fixed strictly positive real number for all sufficiently large $n$,
assuming the null hypothesis of perfect calibration
and that the bin width $(S_{j+1}^1 - S_j^1)$ converges uniformly to 0
(uniformly over $j = 0$, $1$, \dots, $m(n)$).
\end{corollary}

\subsubsection{Cumulative}

The independence of the responses yields the following theorem,
which yields Corollary~\ref{ecceco} when combined with Theorem~\ref{Brownian}.

\begin{theorem}
\label{varthm}
Assume the null hypothesis that the expected value of the response $R_k$ is
equal to the corresponding score $S_k$ for all $k = 1$, $2$, \dots, $n$.
Then, the variance of $C_k$ defined in~(\ref{cumulative}) is
\begin{equation}
\label{var}
(\sigma_k)^2 = \frac{1}{n^2} \sum_{j=1}^k S_j (1 - S_j) \le \frac{k}{4n^2}
\le \frac{1}{4n}
\end{equation}
for $k = 1$, $2$, \dots, $n$.
\end{theorem}

The following theorem summarizes Sections~2.3 and~3 of~\cite{tygert_pvals}.

\begin{theorem}
\label{Brownian}
Assume the null hypothesis that the expected value of the response $R_k$ is
equal to the corresponding score $S_k$ for all $k = 1$, $2$, \dots, $n$.
Suppose again that the scores $S_1$, $S_2$, \dots, $S_n$ are all distinct
for each sample size $n$, and also that
$\max_{1 \le k \le n} S_k (1 - S_k) / \sum_{j=1}^n S_j (1 - S_j)$
converges to 0 as $n$ increases without bound.
Then, as $n$ becomes arbitrarily large the ECCE-MAD divided by $\sigma_n$
converges in distribution to the maximum of the absolute value
of the standard Brownian motion over the unit interval $[0, 1]$.
The ECCE-MAD is defined in~(\ref{ks}) and $\sigma_n$ is defined in~(\ref{var}).
Moreover, as $n$ increases without bound the ECCE-R divided by $\sigma_n$
converges in distribution to the range of the standard Brownian motion
over the unit interval $[0, 1]$.
The ECCE-R is defined in~(\ref{ku}) and $\sigma_n$ is defined in~(\ref{var}).
The expected value of the maximum of the absolute value
of the standard Brownian motion over $[0, 1]$ is $\sqrt{\pi/2} \approx 1.25$,
and the expected value of the range of the standard Brownian motion
over $[0, 1]$ is $2 \sqrt{2 / \pi} \approx 1.60$.
\end{theorem}

Combining Theorems~\ref{varthm} and~\ref{Brownian} yields the following.

\begin{corollary}
\label{ecceco}
As $n$ becomes arbitrarily large both the ECCE-MAD defined in~(\ref{ks})
and the ECCE-R defined in~(\ref{ku}) converge to 0,
assuming both the null hypothesis of perfect calibration
and that the scores $S_1$, $S_2$, \dots, $S_n$ are all distinct
for each sample size $n$, as well as that
$\max_{1 \le k \le n} S_k (1 - S_k) / \sum_{j=1}^n S_j (1 - S_j)$
converges to 0 as $n$ increases without bound.
\end{corollary}

\subsection{Imperfectly calibrated responses}
\label{imperfectsub}

This subsection analyzes the expected values of the metrics when the responses
arise from an imperfectly calibrated distribution, that is,
under the assumption of an ``alternative'' hypothesis that differs
nontrivially from the null hypothesis of perfect calibration.
The principal results of this subsection
are Corollaries~\ref{altl2co} and~\ref{altl1co} for the ECE
and Theorem~\ref{eccethm} for the ECCE.

\subsubsection{Binned}

The following theorem generalizes Theorem~\ref{perfectly}
beyond the case where the function $r$ in the theorem satisfies $r(s) = s$.

\begin{theorem}
\label{imperfectly}
Suppose that $r : [0, 1] \to [0, 1]$ is piecewise continuous
and independent of the sample size $n$.
Suppose also that the response $R_j^k$ is drawn from the Bernoulli distribution
whose expected value is $r(S_j^k)$ for all $j = 1$, $2$, \dots, $m$;
$k = 1$, $2$, \dots, $n_j$.
Suppose finally that $\max_{0 \le j \le m} (S_{j+1}^1 - S_j^1)$ converges to 0
as $n$ increases without bound,
and that $\nu = \nu_n$ is the step function
starting from $\nu(0) = \nu(S_0^1) = n_1$ and jumping to $\nu(S_j^k) = n_j$
for all $j = 1$, $2$, \dots, $m$; $k = 1$, $2$, \dots, $n_j$.
Then, as $n$ becomes arbitrarily large
the lower limit of the expected value of the ECE$^2$ defined in~(\ref{l2})
converges to
\begin{equation}
\label{lowerbound2}
\liminf_{n \to \infty} \int_0^1
\left( (r(s) - s)^2 + \frac{r(s) \, (1 - r(s))}{\nu(s)} \right) \, ds,
\end{equation}
where $\liminf$ denotes the lower limit and the sequence index $n$
of the function $\nu = \nu_n$ is suppressed in the notation.
\end{theorem}

\begin{proof}
Since the responses are all independent, the expected value of the ECE$^2$
from~(\ref{l2}) is
\begin{multline}
\label{bias-variance}
\sum_{j=1}^m \left( S_{j+1}^1 - S_j^1 \right)
\E\left[ \left( \sum_{k=1}^{n_j}
\frac{R_j^k - r(S_j^k) + r(S_j^k) - S_j^k}{n_j} \right)^2 \right] \\
= \sum_{j=1}^m \left( S_{j+1}^1 - S_j^1 \right)
\left( \left(\tilde{r}_j - \tilde{S}_j\right)^2 + \sum_{k=1}^{n_j}
\frac{\E\left[\left(R_j^k - r(S_j^k)\right)^2\right]}{(n_j)^2} \right),
\end{multline}
where $\tilde{r}_j$ and $\tilde{S}_j$ denote the averages
\begin{equation}
\tilde{r}_j = \frac{1}{n_j} \sum_{k=1}^{n_j} r(S_j^k)
\end{equation}
and
\begin{equation}
\tilde{S}_j = \frac{1}{n_j} \sum_{k=1}^{n_j} S_j^k
\end{equation}
for $j = 1$, $2$, \dots, $m$.
The fact that the variance of a Bernoulli distribution whose expected value
is $r(S_j^k)$ is $r(S_j^k) \left( 1 - r(S_j^k) \right)$ yields
\begin{multline}
\label{Bernoulli_variance}
\sum_{j=1}^m \left( S_{j+1}^1 - S_j^1 \right)
\left( \left(\tilde{r}_j - \tilde{S}_j\right)^2 + \sum_{k=1}^{n_j}
\frac{\E\left[\left(R_j^k - r(S_j^k)\right)^2\right]}{(n_j)^2} \right) \\
= \sum_{j=1}^m \left( S_{j+1}^1 - S_j^1 \right)
\left( \left(\tilde{r}_j - \tilde{S}_j\right)^2 + \sum_{k=1}^{n_j}
\frac{r(S_j^k) \left( 1 - r(S_j^k) \right)}{(n_j)^2} \right).
\end{multline}
Referencing terms in each bin to the same score yields
\begin{multline}
\label{bigsum}
\sum_{j=1}^m \left( S_{j+1}^1 - S_j^1 \right)
\left( \left(\tilde{r}_j - \tilde{S}_j\right)^2
+ \sum_{k=1}^{n_j} \frac{r(S_j^k) \left( 1 - r(S_j^k) \right)}{(n_j)^2} \right)
\\ = \sum_{j=1}^m \left( S_{j+1}^1 - S_j^1 \right) \left(
\left(r(S_j^1) - S_j^1\right)^2
+ \frac{r(S_j^1) \left( 1 - r(S_j^1) \right)}{n_j} \right) \\
+ \sum_{j=1}^m \left( S_{j+1}^1 - S_j^1 \right) \left(
\left(\tilde{r}_j - \tilde{S}_j\right)^2
- \left(r(S_j^1) - S_j^1\right)^2
+ \sum_{k=1}^{n_j} \frac{r(S_j^k) \left( 1 - r(S_j^k) \right)
                        -r(S_j^1) \left( 1 - r(S_j^1) \right)}{(n_j)^2}\right),
\end{multline}
which is the sum of two Riemann sums, the latter of which converges to 0
as $\max_{0 \le j \le m} (S_{j+1}^1 - S_j^1)$ tends to 0.
The lower limit of the right-hand side of~(\ref{bigsum})
converges to~(\ref{lowerbound2}),
so combining~(\ref{bias-variance})--(\ref{bigsum}) completes the proof.
\end{proof}

As with Theorem~\ref{perfectly} and Corollary~\ref{l2co}, the main consequence
of Theorem~\ref{imperfectly} is the following.

\begin{corollary}
\label{altl2co}
Suppose that $r : [0, 1] \to [0, 1]$ is piecewise continuous
and independent of the sample size $n$.
Suppose also that the response $R_j^k$ is drawn from the Bernoulli distribution
whose expected value is $r(S_j^k)$ for all $j = 1$, $2$, \dots, $m$;
$k = 1$, $2$, \dots, $n_j$.
If $r$ is also imperfectly calibrated,
that is, $\int_0^1 (r(s) - s)^2 \, ds > 0$,
then the expected value of the ECE$^2$ defined in~(\ref{l2}) is greater
than a fixed strictly positive real number for all sufficiently large $n$,
assuming that the bin width $(S_{j+1}^1 - S_j^1)$ converges uniformly to 0
(uniformly over $j = 0$, $1$, \dots, $m(n)$).
\end{corollary}

Combining Theorem~\ref{dominance} and Corollary~\ref{altl2co}
yields the following similar result for the ECE$^1$.

\begin{corollary}
\label{altl1co}
Suppose that $r : [0, 1] \to [0, 1]$ is piecewise continuous
and independent of the sample size $n$.
Suppose also that the response $R_j^k$ is drawn from the Bernoulli distribution
whose expected value is $r(S_j^k)$ for all $j = 1$, $2$, \dots, $m$;
$k = 1$, $2$, \dots, $n_j$.
If $r$ is also imperfectly calibrated,
that is, $\int_0^1 (r(s) - s)^2 \, ds > 0$,
then the expected value of the ECE$^1$ defined in~(\ref{l1}) is greater
than a fixed strictly positive real number for all sufficiently large $n$,
assuming that the bin width $(S_{j+1}^1 - S_j^1)$ converges uniformly to 0
(uniformly over $j = 0$, $1$, \dots, $m(n)$).
\end{corollary}

\subsubsection{Cumulative}

The following theorem is an analogue
for the ECCE of Corollaries~\ref{altl2co} and~\ref{altl1co}.

\begin{theorem}
\label{eccethm}
Suppose that the empirical cumulative distribution function
of the scores $S_1$, $S_2$, \dots, $S_n$ converges uniformly
to some cumulative distribution function $F$ as the sample size $n$
increases without bound.
Suppose further that $r : [0, 1] \to [0, 1]$ is independent
of the sample size $n$, is Riemann-Stieltjes integrable with respect to $F$,
and is imperfectly calibrated, that is,
$\int_0^1 |r(s) - s| \, dF(s) > 0$.
Suppose also that the response $R_k$ is drawn from the Bernoulli distribution
whose expected value is $r(S_k)$ for all $k = 1$, $2$, \dots, $n$, and that
$r(S_0) = S_0$ and $r(S_{n+1}) = S_{n+1}$ (where $S_0 = 0$ and $S_{n+1} = 1$).
Then, the expected values of the ECCE-MAD defined in~(\ref{ks})
and of the ECCE-R defined in~(\ref{ku}) stay bounded away from 0
for all sufficiently large $n$.
\end{theorem}

\begin{proof}
Applying the Chernoff or Hoeffding bound
for averages of independent Bernoulli variates
to deviate from their expected values by more than $n^{-1/4}$
(or for the unnormalized sums to deviate from their expected values
by more than $n^{3/4}$),
then union bounding across the averages for $k = 1$, $2$, \dots, $n$,
and finally applying the Borel-Cantelli Lemma over the sample size $n$
yields that as $n$ becomes arbitrarily large the ECCE-MAD defined in~(\ref{ks})
converges almost surely to
\begin{equation}
\label{rsmad}
\lim_{n \to \infty} \max_{1 \le k \le n}
\left| \frac{1}{n} \sum_{j=1}^k \left( r(S_j) - S_j \right) \right|
= \max_{0 \le t \le 1} \left| \int_0^t (r(s) - s) \, dF(s) \right|
\end{equation}
and that the ECCE-R defined in~(\ref{ku}) converges almost surely to
\begin{multline}
\label{rsr}
\lim_{n \to \infty} \left(
\max_{0 \le k \le n} \frac{1}{n} \sum_{j=0}^k \left( r(S_j) - S_j \right)
- \min_{0 \le k \le n} \frac{1}{n} \sum_{j=0}^k \left( r(S_j) - S_j \right)
\right) \\
= \max_{0 \le t \le 1} \int_0^t (r(s) - s) \, dF(s)
- \min_{0 \le t \le 1} \int_0^t (r(s) - s) \, dF(s);
\end{multline}
a less elementary means of proving convergence to~(\ref{rsmad}) and~(\ref{rsr})
is to use the Glivenko-Cantelli Theorem for Glivenko-Cantelli classes
(or other uniform strong laws of large numbers).
The dominated convergence theorem then yields convergence
of the expected values of the ECCE-MAD and the ECCE-R to the same values
in~(\ref{rsmad}) and~(\ref{rsr}), courtesy of the domination
\begin{equation}
\left| \frac{1}{n} \sum_{j=1}^k \left( R_j - S_j \right) \right|
\le \frac{1}{n} \sum_{j=1}^n \left| R_j - S_j \right| \le 1
\end{equation}
for $k = 1$, $2$, \dots, $n$ (recall that both $R_j$ and $S_j$ lie
in the unit interval $[0, 1]$, so $\left| R_j - S_j \right| \le 1$).

Now, if~(\ref{rsmad}) were 0, then
\begin{equation}
\int_0^t (r(s) - s) \, dF(s) = 0
\end{equation}
for all $t$ in the unit interval $[0, 1]$; differentiating with respect to $t$
would then show that $r(s) = s$ except on a set of measure 0 relative to $dF$,
making $\int_0^1 |r(s) - s| \, dF(s)$ be 0, too.
The assumption that $\int_0^1 |r(s) - s| \, dF(s) > 0$ thus implies that
the limit~(\ref{rsmad}) of the expected value of the ECCE-MAD
must be strictly positive, completing the proof for the ECCE-MAD.

Similarly, if~(\ref{rsr}) were 0, then
\begin{equation}
\label{constant}
\int_0^t (r(s) - s) \, dF(s) = c
\end{equation}
for all $t$ in the unit interval $[0, 1]$, for some real number $c$
(after all, the maximum and minimum of a function are the same
only if the function is equal to some constant $c$); as before,
differentiating both sides of~(\ref{constant}) with respect to $t$
would then show that $r(s) = s$ except on a set of measure 0 relative to $dF$,
making $\int_0^1 |r(s) - s| \, dF(s)$ be 0, too.
The assumption that $\int_0^1 |r(s) - s| \, dF(s) > 0$ thus implies that
the limit~(\ref{rsr}) of the expected value of the ECCE-R
must be strictly positive, completing the proof for the ECCE-R.
\end{proof}

\subsection{Consequences}
\label{consequences}

This subsection puts together the main results of this section.

For a {\it perfectly} calibrated underlying distribution,
both the ideal, underlying calibration error
and the ideal, underlying cumulative calibration error are equal to 0,
whereas for an {\it imperfectly} calibrated underlying distribution,
both are greater than 0.
For a {\it perfectly} calibrated underlying distribution,
Corollaries~\ref{l2co} and~\ref{l1co} show that the expected values
of the ECE$^1$ and of the ECE$^2$ stay bounded away from 0
if the number of draws per bin remains bounded on some fixed range of scores,
as the maximum bin width becomes arbitrarily small,
while Corollary~\ref{ecceco} shows that as the sample size $n$
increases without bound both the ECCE-MAD and the ECCE-R converge to 0
if the empirical cumulative distribution function of the scores
converges uniformly to a continuous cumulative distribution function.
For an {\it imperfectly} calibrated underlying distribution,
Corollaries~\ref{altl2co} and~\ref{altl1co} show that the expected values
of both the ECE$^1$ and the ECE$^2$ stay bounded away from 0
as the maximum bin width becomes arbitrarily small,
and Theorem~\ref{eccethm} shows that the expected values
of both the ECCE-MAD and the ECCE-R stay bounded away from 0
for all sufficiently large $n$
if the empirical cumulative distribution function of the scores
converges uniformly to a cumulative distribution function.
Thus, both the ECCE-MAD and the ECCE-R can distinguish
every imperfectly calibrated underlying distribution from perfect calibration,
given enough observed scores and corresponding responses,
while in contrast there are imperfectly calibrated distributions which
neither the ECE$^1$ nor the ECE$^2$ can distinguish from perfect calibration
as the maximum width of a bin approaches 0, if the number of draws per bin
remains bounded on some fixed range of scores.

Hence, any significance test based on the ECE$^1$ or the ECE$^2$
with a bounded number of draws per bin on some fixed range of scores
has no power asymptotically for some alternatives
or is asymptotically inconsistent.
This exposes a fundamental trade-off inherent in the ECE$^1$ and the ECE$^2$:
in order to attain nontrivial power and asymptotic consistency,
the number of draws per bin must not stay bounded on any fixed range of scores,
thus necessarily squandering observations that otherwise could have
contributed additional power to the significance test
(whereas the ECCE-MAD and the ECCE-R have no such explicit trade-off).
The trade-off becomes especially hard to handle when the number of observations
is limited; the ECCE-MAD and the ECCE-R work well without requiring
any hard decisions, whereas the ECE$^1$ and the ECE$^2$ depend on the choice
of bins, and that choice makes a big difference even asymptotically,
in the limit of large numbers of observations (with no obvious best setting
for finitely many observations).

To emphasize: to attain asymptotic consistency together
with nontrivial asymptotic power against the fixed alternative distributions
discussed above, both the ECE$^1$ and the ECE$^2$ require infinitely many draws 
per bin for almost all bins (where ``almost all'' refers
to ``almost everywhere'' on the unit interval $[0, 1]$)
--- denser almost everywhere by an unbounded factor more than for the ECCE.

\subsection{Graphical methods}
\label{graphical}

This subsection reviews the primary motivations for the definitions
of the ECEs and the ECCEs --- the ECEs and the ECCEs
are scalar summary statistics for certain graphical methods
of assessing calibration reviewed here.

\subsubsection{Motivation for the empirical calibration error}
\label{calibrationplot}

Formulae~(\ref{avgscore}) and~(\ref{avgresponse}) above
express the average score and average response in bin $j$ as
\begin{equation}
\label{avgscoredup}
\tilde{S}_j = \frac{1}{n_j} \sum_{k=1}^{n_j} S_j^k,
\end{equation}
and
\begin{equation}
\label{avgresponsedup}
\tilde{R}_j = \frac{1}{n_j} \sum_{k=1}^{n_j} R_j^k
\end{equation}
for $j = 1$, $2$, \dots, $m$, respectively.
Due to the central limit theorem, as $n_j$ becomes large,
$\tilde{R}_j$ concentrates around its expected value,
\begin{equation}
\label{eresponse}
\E[\tilde{R}_j] = \frac{1}{n_j} \sum_{k=1}^{n_j} \E[R_j^k],
\end{equation}
which is the average of the expected values of the responses in the bin.
For a perfectly calibrated underlying distribution, $\E[R_j^k] = S_j^k$,
and so~(\ref{eresponse}) would be equal to~(\ref{avgscoredup}).
The difference from perfect calibration in bin $j$ is therefore
the difference between $\tilde{R}_j$ from~(\ref{avgresponsedup})
and $\tilde{S}_j$ from~(\ref{avgscoredup}), in the limit that $n_j$ is large.
A so-called ``calibration plot'' or ``reliability diagram''
plots $\tilde{R}_j$ versus $\tilde{S}_j$ together
with the (diagonal) line through a graph of $\tilde{S}_j$ versus $\tilde{S}_j$
for $j = 1$, $2$, \dots, $m$, so that the difference from perfect calibration
is the vertical distance between the two graphs.
Section~\ref{results} presents many examples of such reliability diagrams;
see, for example, Figures~\ref{32768m16prob}, \ref{32768m16samp},
\ref{32768m64prob}, \ref{32768m64samp},
\ref{night-snakeprob}, \ref{night-snakesamp},
\ref{sidewinderprob}, \ref{sidewindersamp},
\ref{eskimo-dogprob}, \ref{eskimo-dogsamp},
\ref{wild-boarprob}, \ref{wild-boarsamp},
\ref{sunglassesprob}, \ref{sunglassessamp},
\ref{imagenetprob}, and~\ref{imagenetsamp}.

The ECE$^2$ from~(\ref{l2}) is the sum from $j = 1$ to $j = m$
of the bin width $(S_{j+1}^1 - S_j^1)$ times the square
of the difference between $\tilde{R}_j$ and $\tilde{S}_j$,
where the latter difference approaches the expected amount of miscalibration
in bin $j$ in the limit that $n_j$ is large; similarly,
the ECE$^1$ from~(\ref{l1}) is the sum from $j = 1$ to $j = m$
of the bin width $(S_{j+1}^1 - S_j^1)$ times the absolute value
of the difference between $\tilde{R}_j$ and $\tilde{S}_j$.
In the limit that the bin width $(S_{j+1}^1 - S_j^1)$ becomes small uniformly
over $j = 0$, $1$, \dots, $m$, and $n_j$ becomes large uniformly,
the ECE$^1$ becomes the total area between the graph
of $\tilde{R}_j$ versus $\tilde{S}_j$ and the graph
of $\tilde{S}_j$ versus itself, assuming that ``area'' is well-defined
in terms of the Riemann sum~(\ref{l1}), that is, that the ECE$^1$ converges
to a unique limit. This is the case when there exists
a fixed Riemann integrable function $r$ such that $\E[R_j^k(n)] = r(S_j^k(n))$
for all $n = 1$, $2$, $3$, \dots; $j = 1$, $2$, \dots, $m(n)$;
$k = 1$, $2$, \dots, $n_j(n)$;
again assuming that the bin width $(S_{j+1}^1 - S_j^1)$ becomes small uniformly
over $j = 1$, $2$, \dots, $m$, and $n_j$ becomes large uniformly over $j$.

\subsubsection{Motivation for the empirical cumulative calibration error}
\label{cumulativeplot}

Formula~(\ref{cumulative}) defines the cumulative difference $C_k$ as
\begin{equation}
\label{cumulativedup}
C_k = \frac{1}{n} \sum_{j=1}^k (R_j - S_j)
\end{equation}
for $k = 1$, $2$, \dots, $n$.
Plotting the cumulative sum $C_k$ versus $k/n$
results in a graph whose expected slope is the amount of miscalibration;
indeed, the expected slope from $j = k - 1$ to $j = k$ is
\begin{equation}
\frac{\E[C_k - C_{k-1}]}{k/n - (k-1)/n} = \E[R_k] - S_k
\end{equation}
for $k = 1$, $2$, \dots, $n$ --- and $\E[R_k] - S_k$ is precisely
the deviation from perfect calibration.
Thus, the slope of a secant line connecting two points on the graph
becomes the average miscalibration over the long range of $k$
between the two points.
Lack of miscalibration results in a fairly flat graph.
So, good calibration results in a flat, mainly horizontal graph
that deviates little from zero.
This motivates the definitions of the ECCE-MAD in~(\ref{ks})
and of the ECCE-R in~(\ref{ku}) --- they measure the deviation from zero,
the deviation from a perfectly flat, horizontal graph of perfect calibration.
The ECCE-MAD is simply the maximum absolute value of $C_k$,
while the ECCE-R is simply the range of $C_k$, where the range
is the maximum minus the minimum.
As mentioned in Subsubsection~\ref{cumsub} above,
the absolute value of the total miscalibration
$\sum_{j \in I} (R_j - S_j)/n$
over any interval $I$ of indices is less than or equal to the ECCE-R ---
the ECCE-R is the maximum of the absolute value
of the total miscalibration over any interval of indices:
$\hbox{ECCE-R} = \max_I \left| \sum_{j \in I} (R_j - S_j)/n \right|$.

Note that slope is easy to perceive independently
of any irrelevant constant vertical offset, and that slope
in the plot of $C_k$ versus $k/n$ is equal to the amount of miscalibration
(with the slope of a secant line becoming the average miscalibration
over the scores between two distant points
where the secant line intersects the graph).
Section~\ref{results} presents many examples of such graphs
of cumulative differences; see, for example,
Figures~\ref{32768cum}, \ref{night-snakecum}, \ref{sidewindercum},
\ref{eskimo-dogcum}, \ref{wild-boarcum}, \ref{sunglassescum},
and~\ref{imagenetcum}.

\section{Results and discussion}
\label{results}

This section illustrates the methods of the previous section
via analysis of both synthetic
and measured data sets.\footnote{Permissively licensed open-source software
that automatically reproduces all figures and statistics reported below
is available at \url{https://github.com/facebookresearch/ecevecce}}
The synthetic examples include ground-truths known by construction.
The synthetic examples first highlight practical problems with the ECEs,
then validate the theory of the previous section directly and explicitly.
The examples on measured data display even more extreme practical problems
with the ECEs, especially in comparison with the ECCEs.
Subsection~\ref{synthetic} presents the synthetic examples,
while Subsection~\ref{imagenet} analyzes in detail one of the most popular
data sets from computer vision, ImageNet.

The following are minor details common to both Subsection~\ref{synthetic}
and Subsection~\ref{imagenet}:
\begin{enumerate}
\item ``P-values'' (also known as ``attained significance levels'')
which are exact in the asymptotic limit that the sample size $n$ is large
accompany every value for the ECCE-MAD and for the ECCE-R
reported in the captions of the figures; efficient methods
for computing such P-values are detailed by~\cite{tygert_pvals}.
\item All reliability diagrams displayed in the present paper include
``error bars'' (actually lines, not bars) plotted in light gray.
Each diagram includes 20 such light-gray graphs,
obtained via bootstrap resampling, corresponding to a confidence level
of around 95\%. Details on their computation are available
in the appendix of~\cite{tygert_full}.
\item When the bins are equispaced with respect to the scores
(the scores are the predicted probabilities), we replace the bin width
$(S_{j+1}^1 - S_j^1)$ in~(\ref{l2}) and~(\ref{l1}) with $1/m$,
where $m$ is the number of bins;
both~(\ref{l2}) and~(\ref{l1}) are still Riemann sums with this change,
and so all the analysis given above remains valid without modification.
Such replacement is canonical in the classical reliability diagrams and ECEs
when the bins are equispaced.
\end{enumerate}

\subsection{Synthetic examples}
\label{synthetic}

This subsection presents the results of numerical experiments
on a toy example, generated synthetically so that the complete ground-truth
is known exactly. Knowing the ground-truth facilitates a fully rigorous
evaluation and validation of the methods of the previous section.
Figures~\ref{ecen}--\ref{cumnnorm} illustrate the theorems and corollaries
of Section~\ref{methods} as explicitly as possible, as detailed
in the penultimate paragraph of the present subsection.
Figures~\ref{32768ece}--\ref{32768cum} set the stage,
introducing the synthetic examples and some problems with binning encountered
in practice.

Figure~\ref{32768ece} displays the four kinds of empirical calibration errors,
as a function of $m$, the number of bins.
Whether any number $m$ of bins is optimal, much less ideal and representative,
is entirely unclear. The values of the empirical calibration errors
vary widely as the number $m$ of bins varies.
Figure~\ref{32768ece} corresponds to the sample size $n =$ 32,768
used in the present subsection; the appendix contains analogous plots
for the samples sizes $n =$ 8,192 and $n =$ 131,072.

Figure~\ref{32768exact} plots the probabilities of success
for the Bernoulli distributions from which the synthetic data set
draws responses at the specified scores, where the scores
are equispaced in the upper plot of the figure,
then squared in the middle plot, and finally square rooted
from the original equispaced values in the lower plot.
The sample size is $n =$ 32,768, which is the number of scores
(each paired with a response drawn from the Bernoulli distribution
whose probability of success is graphed) for each plot.

Figures~\ref{32768m16prob}--\ref{32768m64samp} display the reliability diagrams
from Subsubsection~\ref{calibrationplot}
(both with bins that are roughly equispaced with respect to the scores
and with each bin containing the same number of observations),
for $m = 16$ bins and $m = 64$.
Figure~\ref{32768cum} displays the cumulative plot
from Subsubsection~\ref{cumulativeplot},
along with the ground-truth ideal, constructing the ideal graph
using the exact expected values of the Bernoulli distributions
from which the observed responses are drawn; the empirical plot
closely resembles the ground-truth ideal.

Figures~\ref{ecen}--\ref{cumnnorm} illustrate explicitly the theory
of Section~\ref{methods}.
The upper plots of Figure~\ref{ecen}
correspond to Corollaries~\ref{l2co} and~\ref{l1co},
while the upper plots of Figure~\ref{cumn}
correspond to Corollary~\ref{ecceco};
the lower plots of Figure~\ref{ecen}
correspond to Corollaries~\ref{altl2co} and~\ref{altl1co},
while the lower plots of Figure~\ref{cumn}
correspond to Theorem~\ref{eccethm}.
Figure~\ref{cumnnorm} displays the ECCE-MAD and the ECCE-R
both normalized by $\sigma_n$ from~(\ref{var}); the perfectly calibrated data
of the upper plots in Figure~\ref{cumnnorm} results
in the ECCE-MAD / $\sigma_n$ hovering around its asymptotic expected value
$\sqrt{\pi/2} \approx 1.2533$
(asymptotic in the limit of large sample size $n$)
and in the ECCE-R / $\sigma_n$ hovering around its asymptotic expected value
$2 \sqrt{2/\pi} \approx 1.5958$.
Derivations of these expected values in the limit of large sample size $n$
are available in Section~3 of~\cite{tygert_pvals}.
Figure~\ref{ecen} displays graphically how the ECEs hit a noise floor,
staying around the same value for both the perfectly
and imperfectly calibrated distributions,
irrespective of the number of observations.
In contrast, Figure~\ref{cumn} illustrates how the ECCEs approach 0 rapidly
for the perfectly calibrated distribution as the sample size $n$ increases,
while staying well away from 0 for the imperfectly calibrated distribution.
Thus, the ECEs have trouble telling apart
the perfectly and imperfectly calibrated distributions,
whereas the power of the ECCEs increases indefinitely as the sample size grows.

The captions of the figures discuss the results and their consequences.

\subsection{ImageNet}
\label{imagenet}

This subsection presents the results of numerical experiments
on the standard training data set ``ImageNet-1000'' of~\cite{imagenet},
which is among the most popular data sets in computer vision.

The standard training data set ``ImageNet-1000''
contains a thousand labeled classes, each containing about 1,300 images
corresponding to a particular noun (often an animal such as a ``night snake,''
a ``sidewinder or horned rattlesnake,'' or an ``Eskimo dog or husky'');
the total number of images in the data set is $n =$ 1,281,167.
To generate the corresponding plots,
we calculate the scores $S_1$, $S_2$, \dots, $S_n$
using the pretrained ResNet18 classifier of~\cite{he-zhang-ren-sun}
from the computer-vision module, ``torchvision,''
in the PyTorch software library of~\cite{pytorch};
the score for an image is the probability assigned by the classifier
to the class predicted to be most likely,
with the scores randomly perturbed by about one part in $10^8$ to guarantee
their uniqueness.
For $k = 1$, $2$, \dots, $n$, the response $R_k$ corresponding to a score $S_k$
is $R_k = 1$ when the class predicted to be most likely is the correct class,
and $R_k = 0$ otherwise.

The figures presented below consider both the subsets of the full data set
corresponding to individual classes as well as the full data set
with all classes simultaneously.
Figures~\ref{night-snakeece}--\ref{sunglassescum} pertain
to individual classes, while Figures~\ref{imagenetece}--\ref{imagenetcum}
pertain to all classes together.

Figures~\ref{night-snakeece}--\ref{night-snakecum} consider the class
corresponding to the night snake,
Figures~\ref{sidewinderece}--\ref{sidewindercum} consider
the sidewinder or horned rattlesnake,
Figures~\ref{eskimo-dogece}--\ref{eskimo-dogcum} consider
the Eskimo dog or husky,
Figures~\ref{wild-boarece}--\ref{wild-boarcum} consider the wild boar,
and Figures~\ref{sunglassesece}--\ref{sunglassescum} consider sunglasses.
In each of these sets of four figures,
the first figure displays the four kinds of empirical calibration errors ---
the ECE$^1$ and the ECE$^2$ for when the bins are equispaced along the scores,
and the ECE$^1$ and the ECE$^2$ for when each bin (aside from the last)
contains the same number of observations.
The second figure in the set of four provides examples
of the reliability diagrams from Subsubsection~\ref{calibrationplot},
with the bins chosen such that $S_1^1$, $S_2^1$, \dots, $S_m^1$
are roughly equispaced on the unit interval $[0, 1]$
(with $m = 8$ in the upper plot and $m = 32$ in the lower plot).
The third figure in each set of four gives examples
of the reliability diagrams from Subsubsection~\ref{calibrationplot},
now with the bins chosen such that $n_1 = n_2 = \dots = n_{m-1} \approx n_m$
(again with $m = 8$ in the upper plot and $m = 32$ in the lower plot).
The fourth figure in the set of four provides an example
of the cumulative plot from Subsubsection~\ref{cumulativeplot}.
For each of Figures~\ref{night-snakeece}--\ref{sunglassescum},
the sample size is $n =$ 1,300.

Figure~\ref{imagenetece} displays the empirical calibration errors
for all classes analyzed simultaneously,
so that the sample size $n =$ 1,281,167.
Figure~\ref{imagenetprob} gives examples of the reliability diagrams
from Subsubsection~\ref{calibrationplot} for all 1,000 classes together,
with the bins chosen such that $S_1^1$, $S_2^1$, \dots, $S_m^1$
are roughly equispaced on the unit interval $[0, 1]$
(with $m = 128$ in the upper plot and $m = 1,024$ in the lower plot).
Figure~\ref{imagenetsamp} provides examples of the reliability diagrams
from Subsubsection~\ref{calibrationplot} for all 1,000 classes simultaneously,
with the bins chosen such that $n_1 = n_2 = \dots = n_{m-1} \approx n_m$
(again with $m = 128$ in the upper plot and $m = 1,024$ in the lower plot).
Figure~\ref{imagenetcum} provides an example of the cumulative plot
from Subsubsection~\ref{cumulativeplot} for all classes together.
Only the cumulative plot (Figure~\ref{imagenetcum}) conveniently reveals
that a third of all observations (specifically, those with probabilities
of at least 0.97) are well-calibrated.
The ranges of the graphs in Figure~\ref{imagenetece} are relatively narrow
as $m$, the number of bins, varies through the values 8, 16, 32, \dots, 1,024;
the empirical calibration errors could plausibly constitute decent metrics
in this setting, on account of their relatively stable values
as $m$ varies through the values 8, 16, 32, \dots, 1,024.

In contrast, all the empirical calibration errors vary enormously as a function
of $m$, the number of bins, for every individual class
from ImageNet investigated here
--- the range of the graphs in every one of Figures~\ref{night-snakeece},
\ref{sidewinderece}, \ref{eskimo-dogece}, \ref{wild-boarece},
and~\ref{sunglassesece} is wide even with merely modest variations in $m$.
Which choice of $m$, the number of bins, and binning strategy is best
--- if any --- is entirely unclear from these plots.
Whether any choice of $m$ or binning strategy yields a good metric
must be seriously questionable when the choice makes such a big difference
in the value of the metric, and which choices are better is unclear.

Thus, the empirical calibration errors may be most meaningful
when the probabilities of success for the Bernoulli distributions
underlying the observed data are smooth as a function of $m$,
the number of bins, and the sample size $n$ greatly exceeds
the minimum required to assess statistical significance reliably
by averaging away noise from sampling.
If the probabilities of success display multiscale behavior as a function
of $m$, with interesting variations present at finer and finer scales,
then any choice of $m$ will necessarily miss interesting variations
or fail to perform enough averaging to distinguish signal from noise.
In accord with the theory of Section~\ref{methods},
obtaining meaningful empirical calibration errors apparently requires
the sample size $n$ to be much larger than the ideal
attained by the empirical {\it cumulative} calibration errors.

More detailed discussion is available in the captions of the figures.

\begin{figure}
\begin{center}
\parbox{\imsized}{\includegraphics[width=\imsized]
{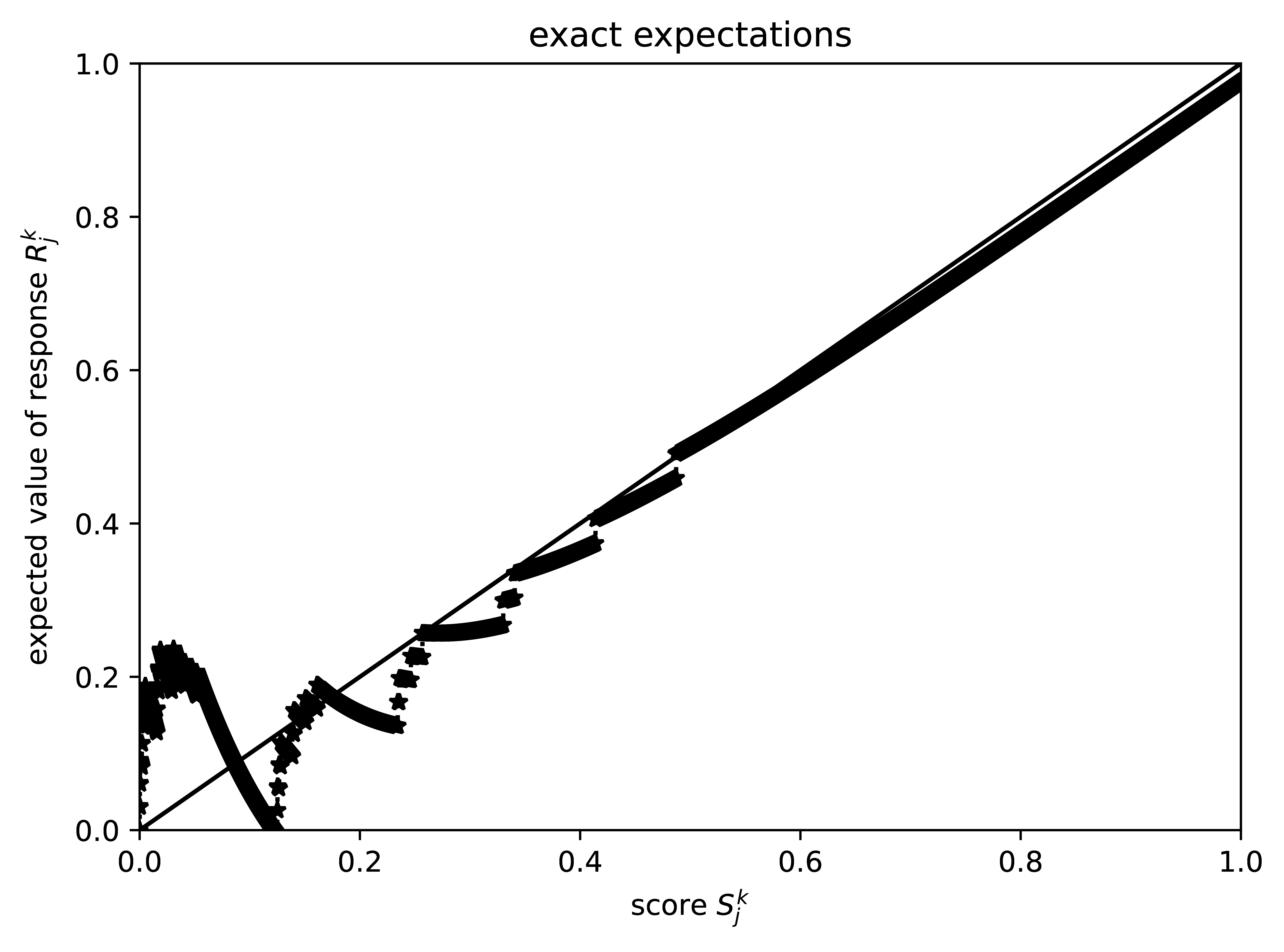}}

\parbox{\imsized}{\includegraphics[width=\imsized]
{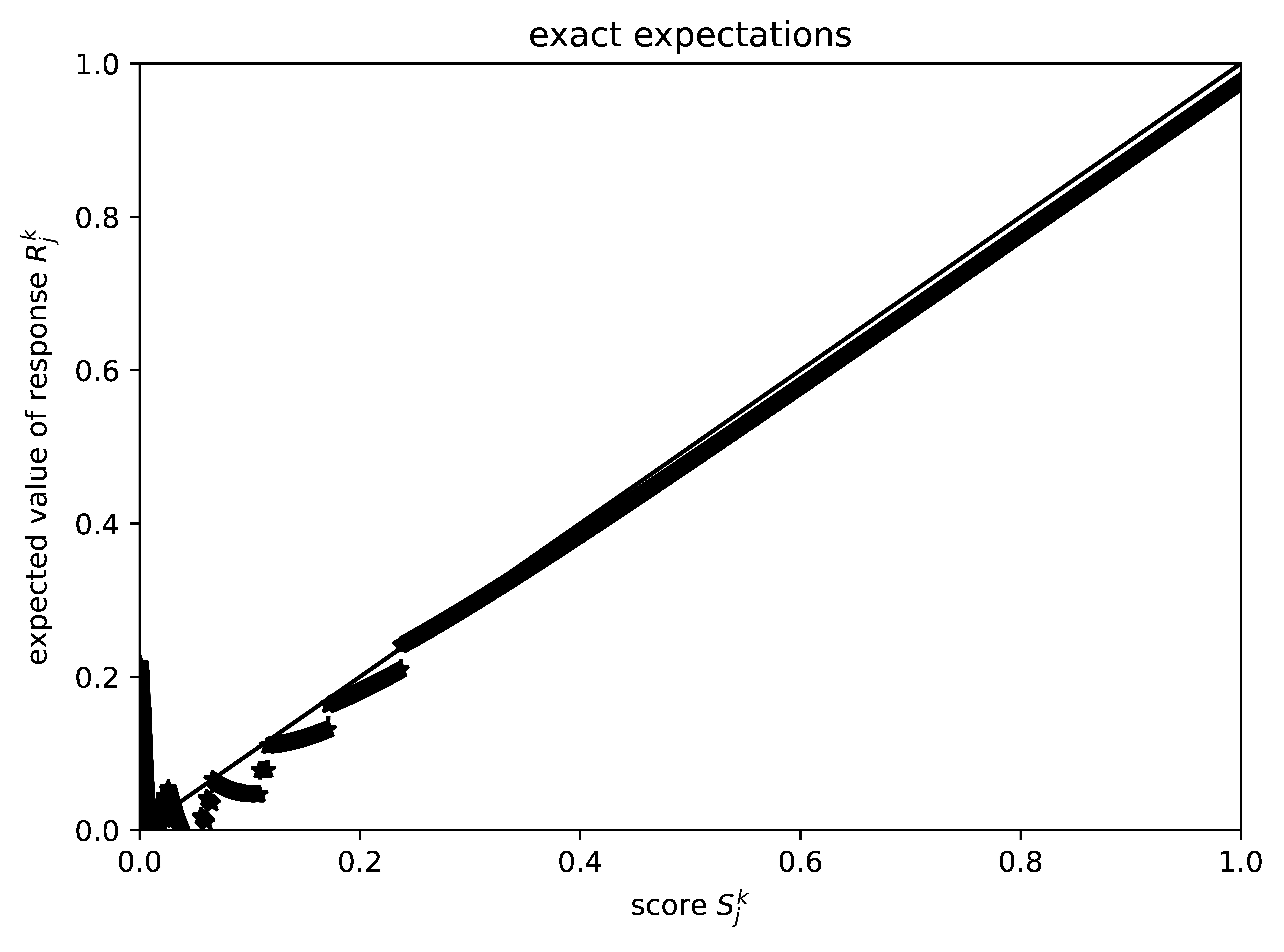}}

\parbox{\imsized}{\includegraphics[width=\imsized]
{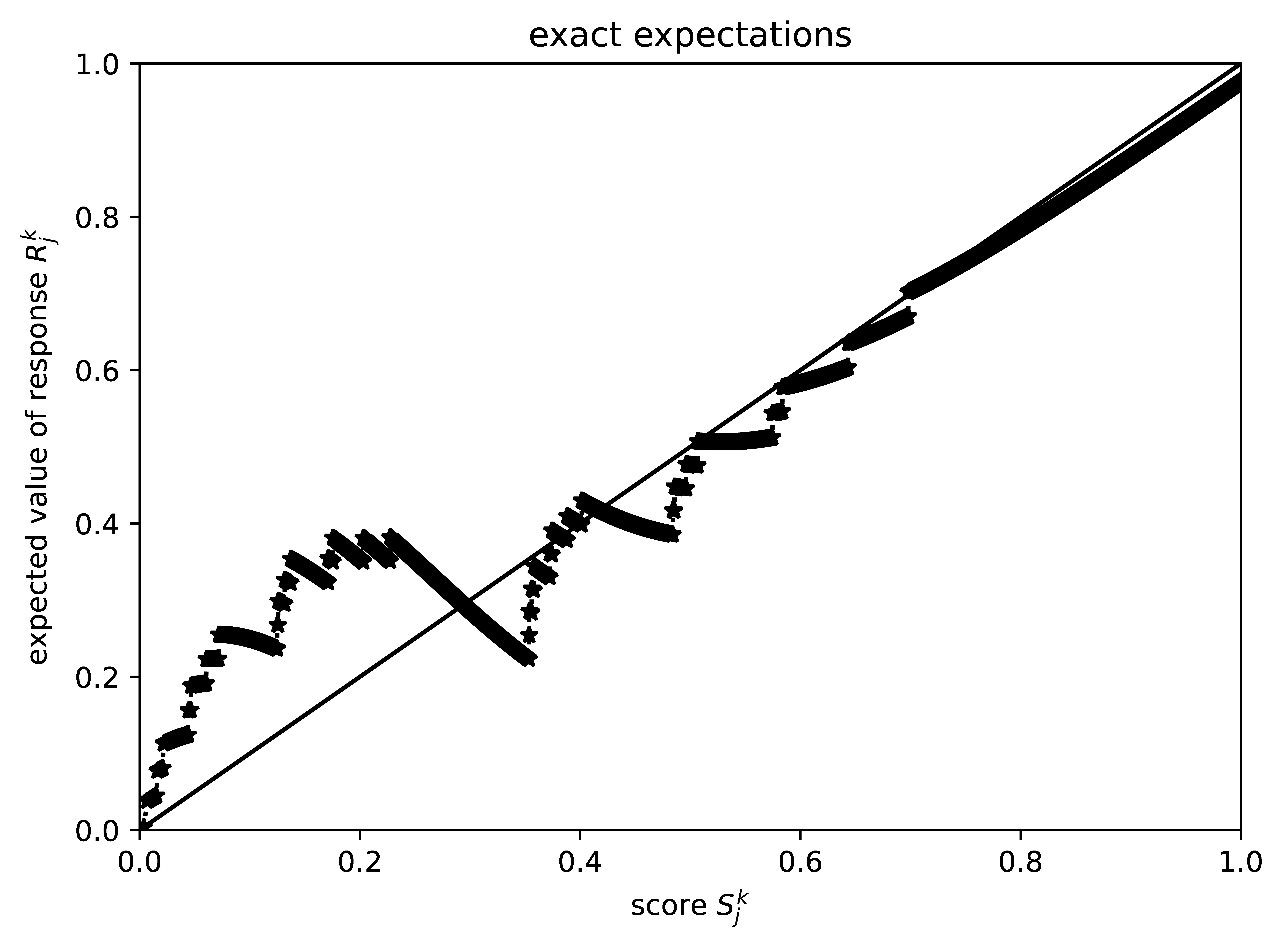}}
\end{center}
\caption{Probabilities of success for the Bernoulli distributions
         underlying the synthetic data set (which takes independent draws
         from these distributions to obtain the observed responses).
         The scores are equispaced in the top plot,
         squared in the middle plot, and square rooted in the bottom plot,
         with sample size $n =$ 32,768.}
\label{32768exact}
\end{figure}

\begin{figure}
\begin{center}
\parbox{\imsized}{\includegraphics[width=\imsized]
{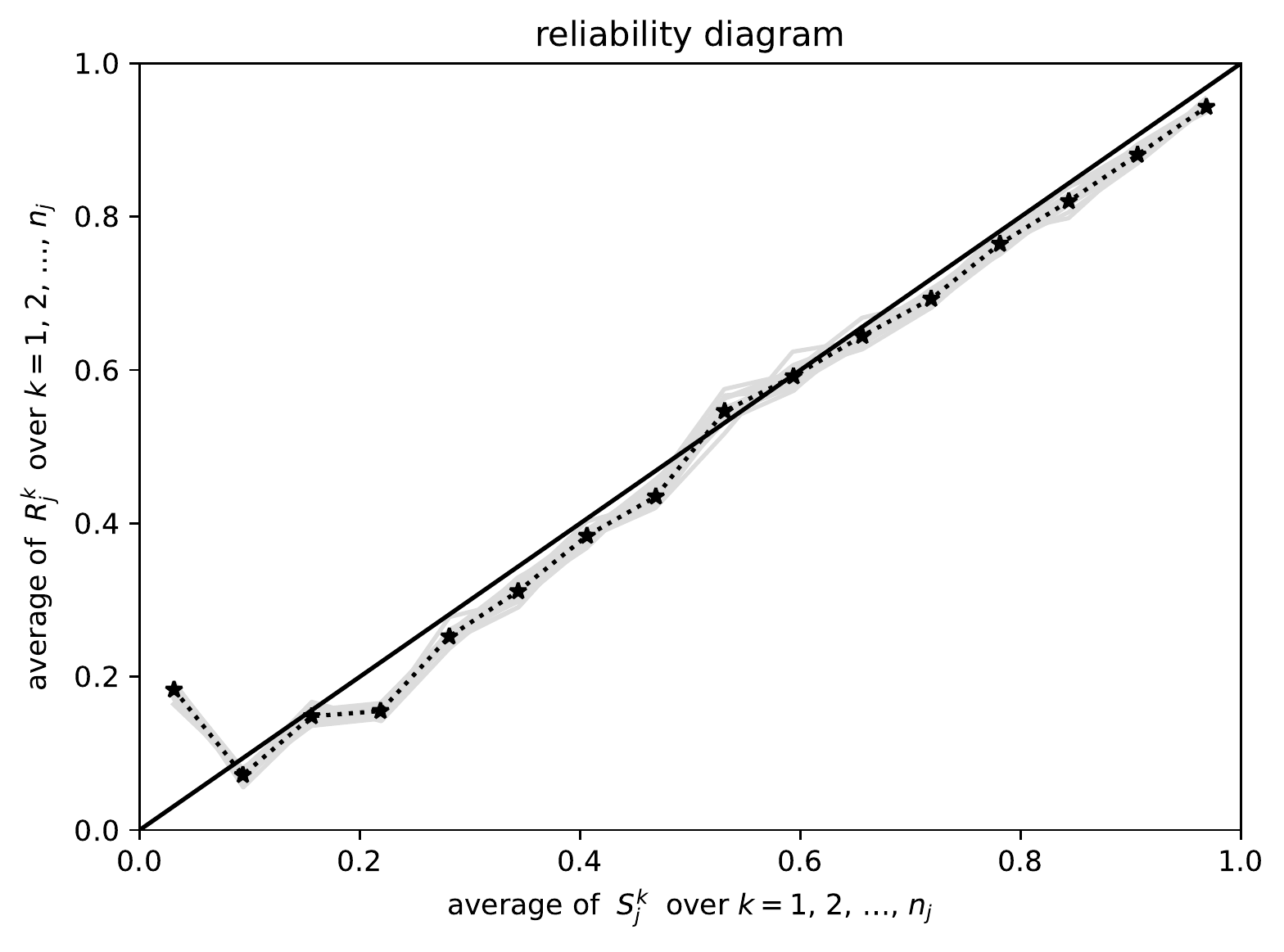}}

\parbox{\imsized}{\includegraphics[width=\imsized]
{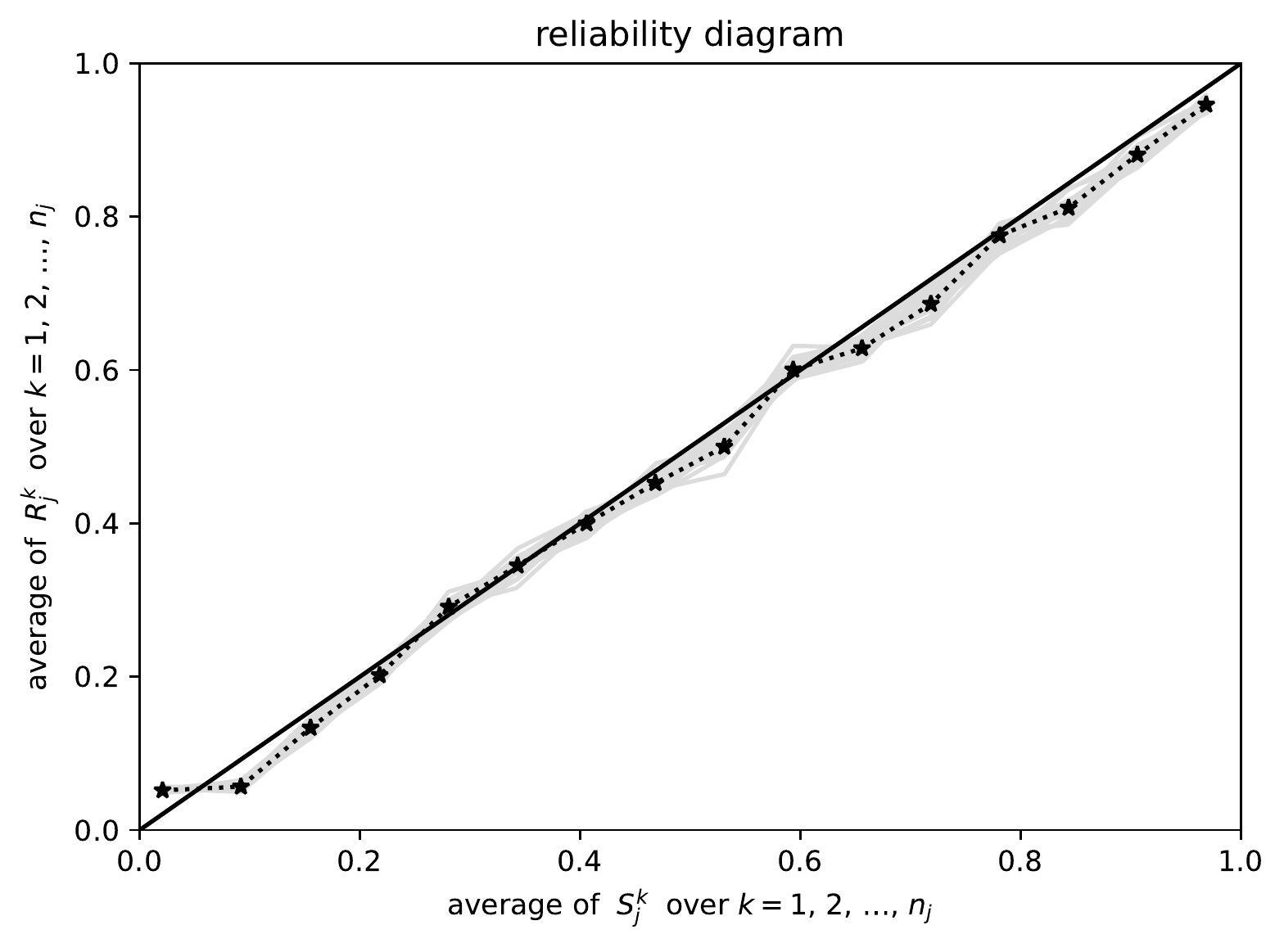}}

\parbox{\imsized}{\includegraphics[width=\imsized]
{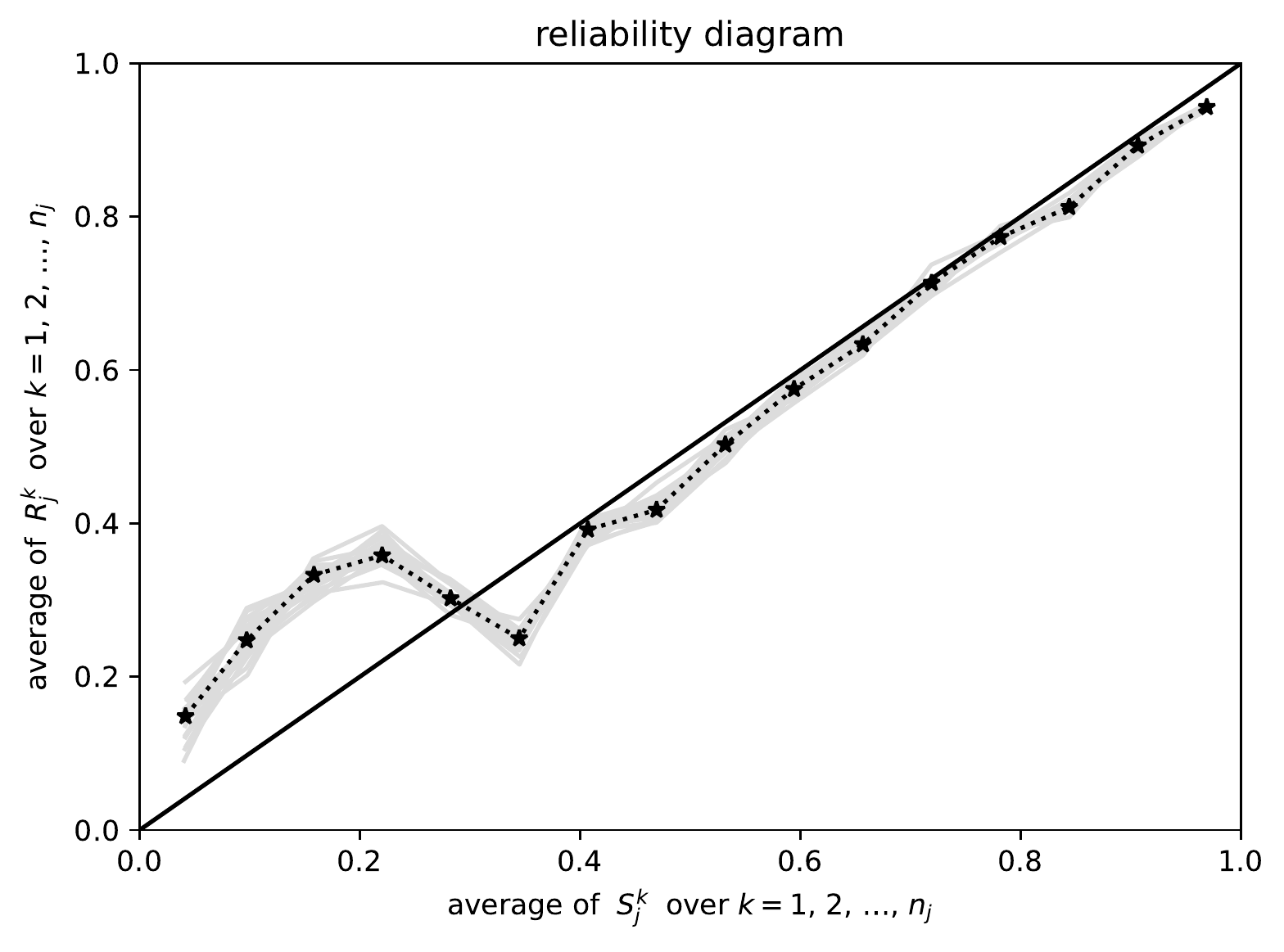}}
\end{center}
\caption{Reliability diagrams for the synthetic data set,
         with the bins roughly equispaced.
         The scores are equispaced in the top plot,
         squared in the middle plot, and square rooted in the bottom plot,
         with $m = 16$ bins and sample size $n =$ 32,768.}
\label{32768m16prob}
\end{figure}

\begin{figure}
\begin{center}
\parbox{\imsized}{\includegraphics[width=\imsized]
{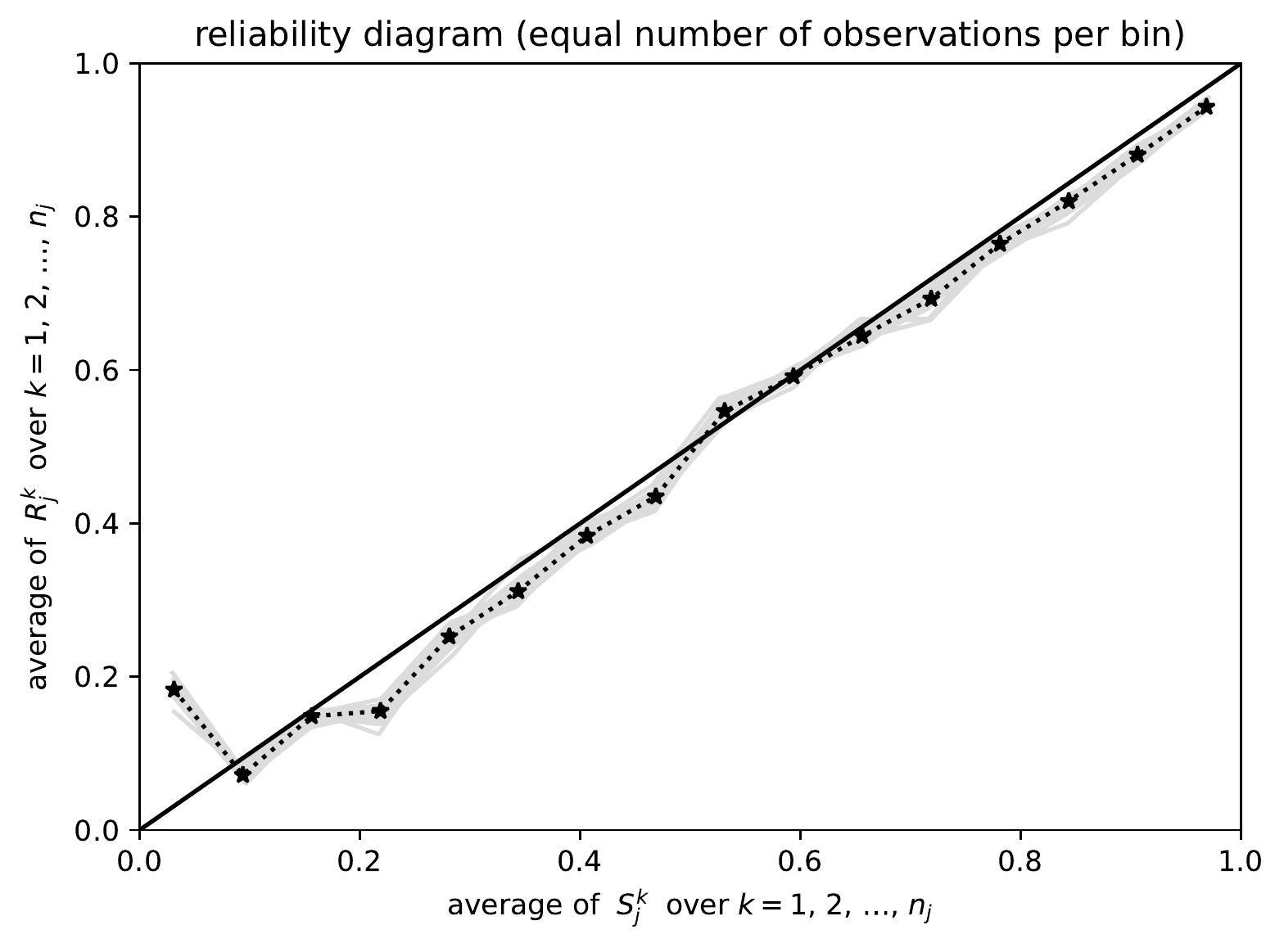}}

\parbox{\imsized}{\includegraphics[width=\imsized]
{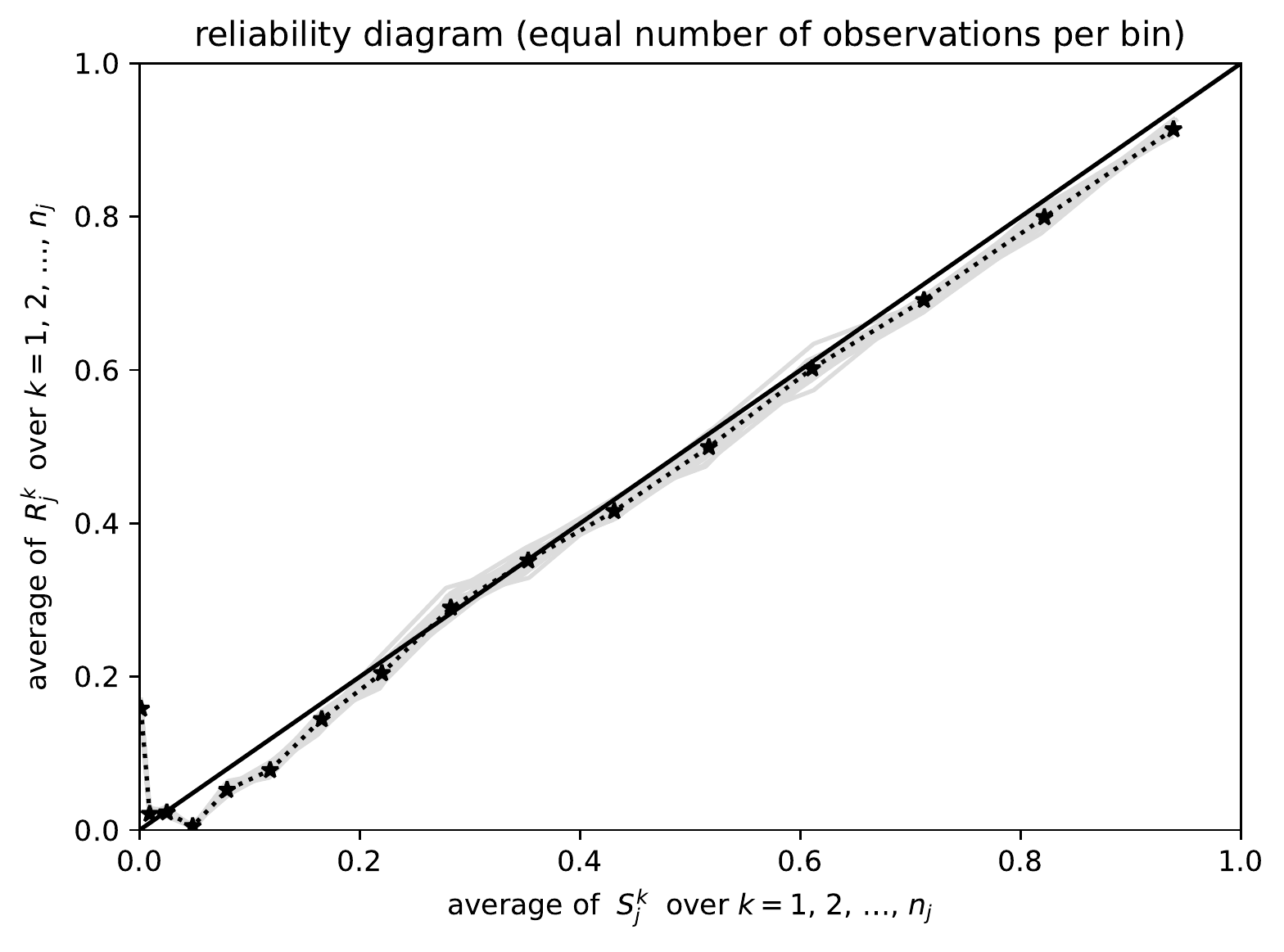}}

\parbox{\imsized}{\includegraphics[width=\imsized]
{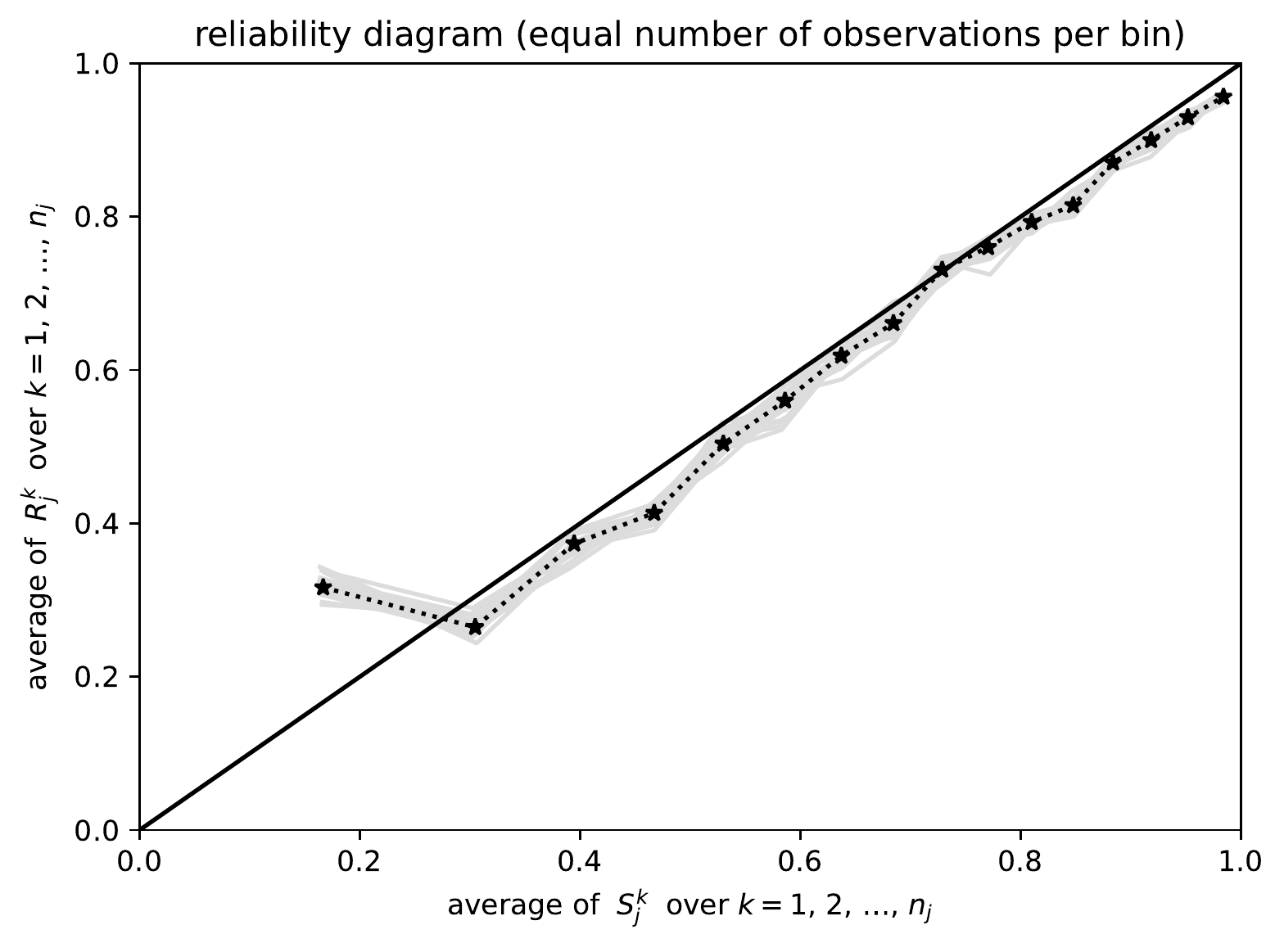}}
\end{center}
\caption{Reliability diagrams for the synthetic data set,
         with an equal number of observations per bin.
         The scores are equispaced in the top plot,
         squared in the middle plot, and square rooted in the bottom plot,
         with $m = 16$ bins and sample size $n =$ 32,768.}
\label{32768m16samp}
\end{figure}

\begin{figure}
\begin{center}
\parbox{\imsized}{\includegraphics[width=\imsized]
{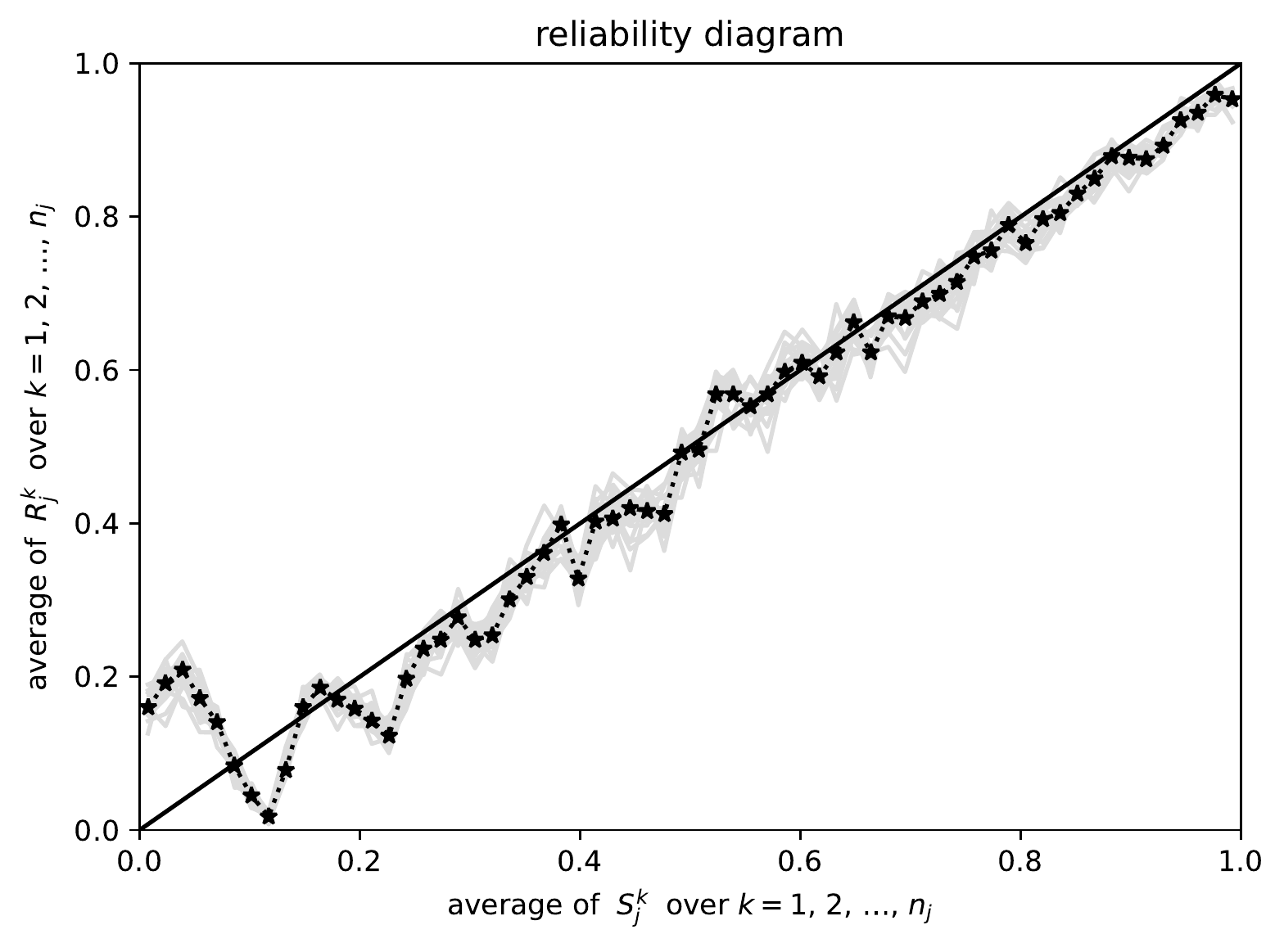}}

\parbox{\imsized}{\includegraphics[width=\imsized]
{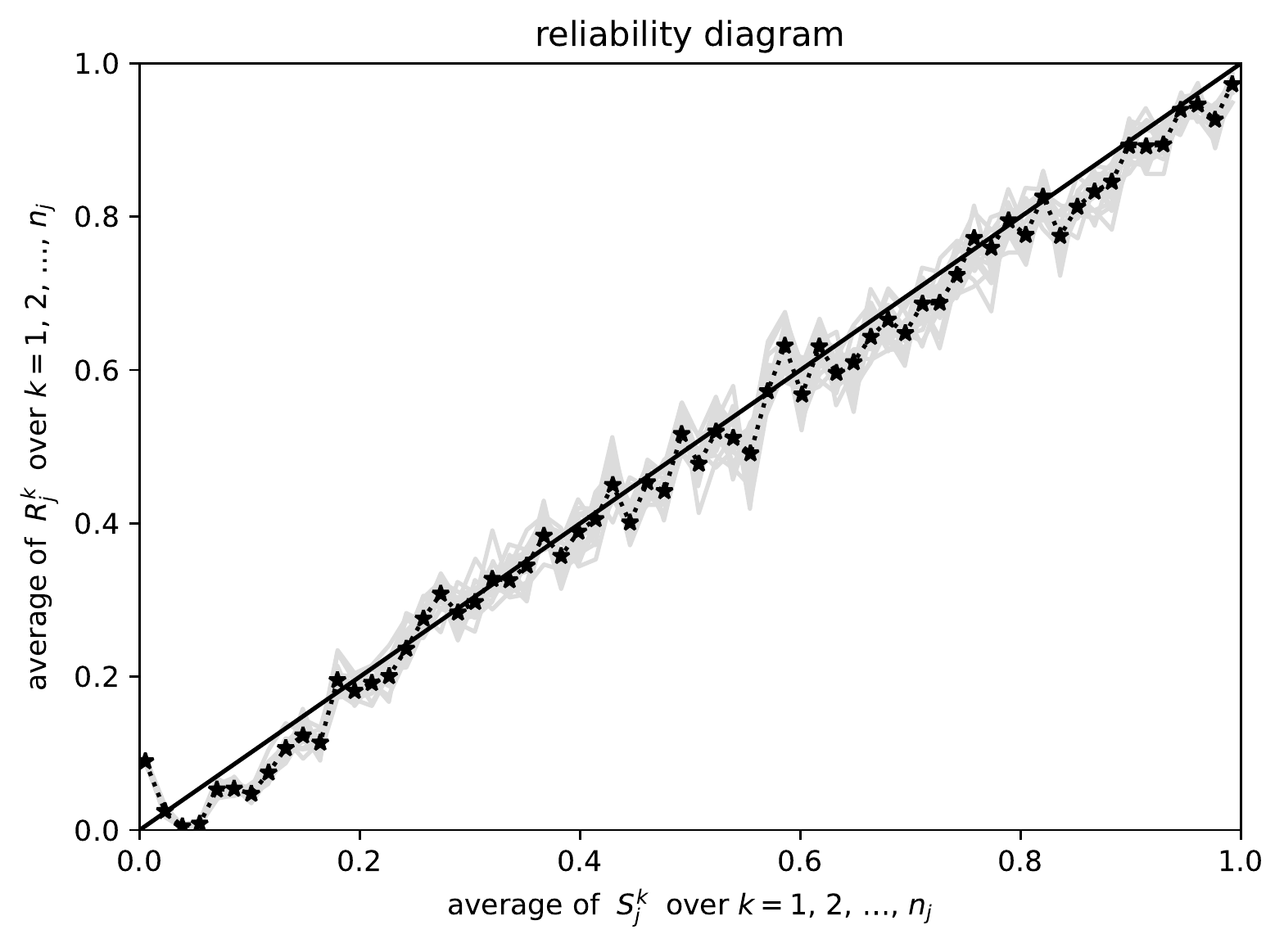}}

\parbox{\imsized}{\includegraphics[width=\imsized]
{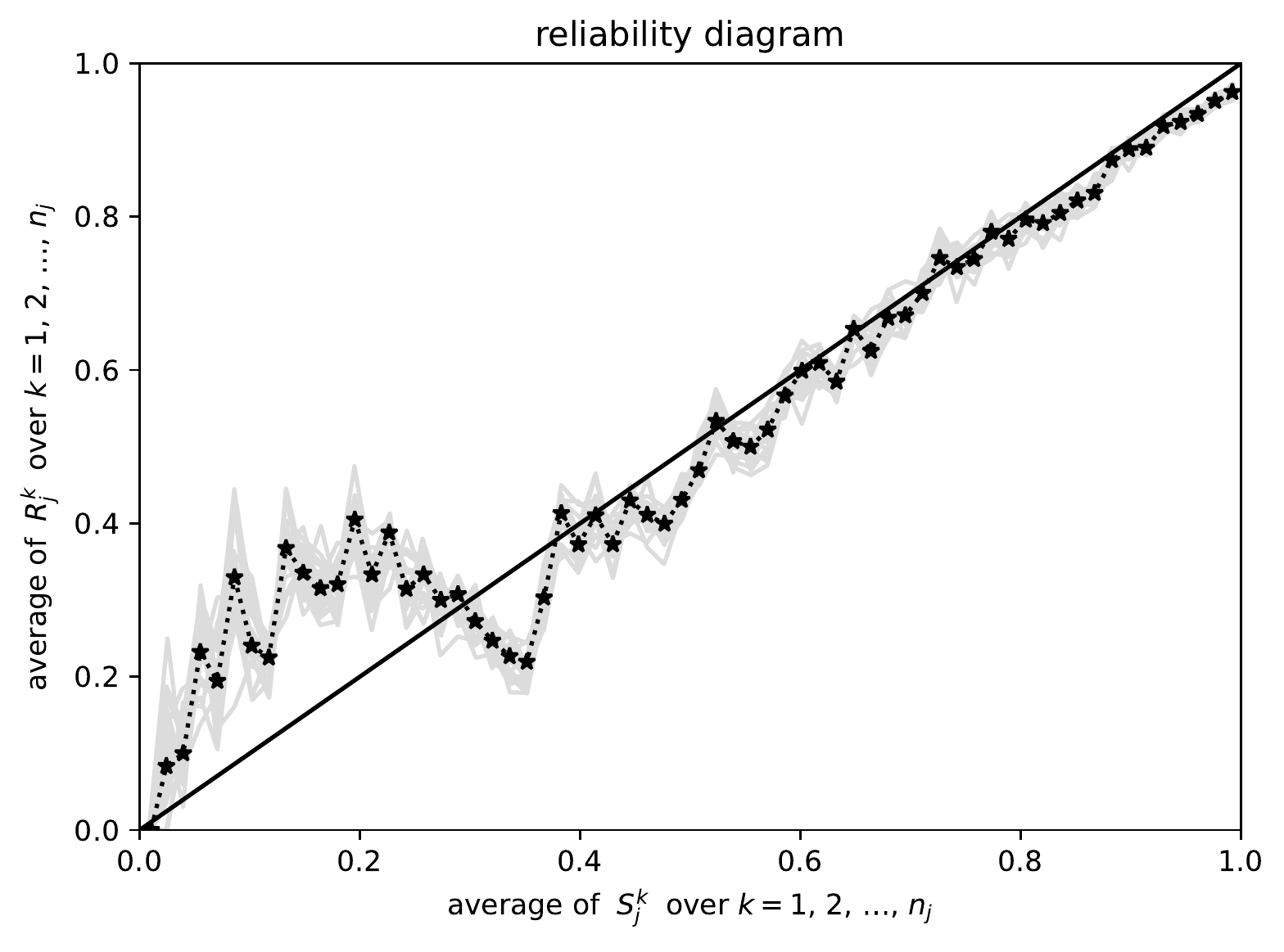}}
\end{center}
\caption{Reliability diagrams for the synthetic data set,
         with the bins roughly equispaced.
         The scores are equispaced in the top plot,
         squared in the middle plot, and square rooted in the bottom plot,
         with $m = 64$ bins and sample size $n =$ 32,768.}
\label{32768m64prob}
\end{figure}

\begin{figure}
\begin{center}
\parbox{\imsized}{\includegraphics[width=\imsized]
{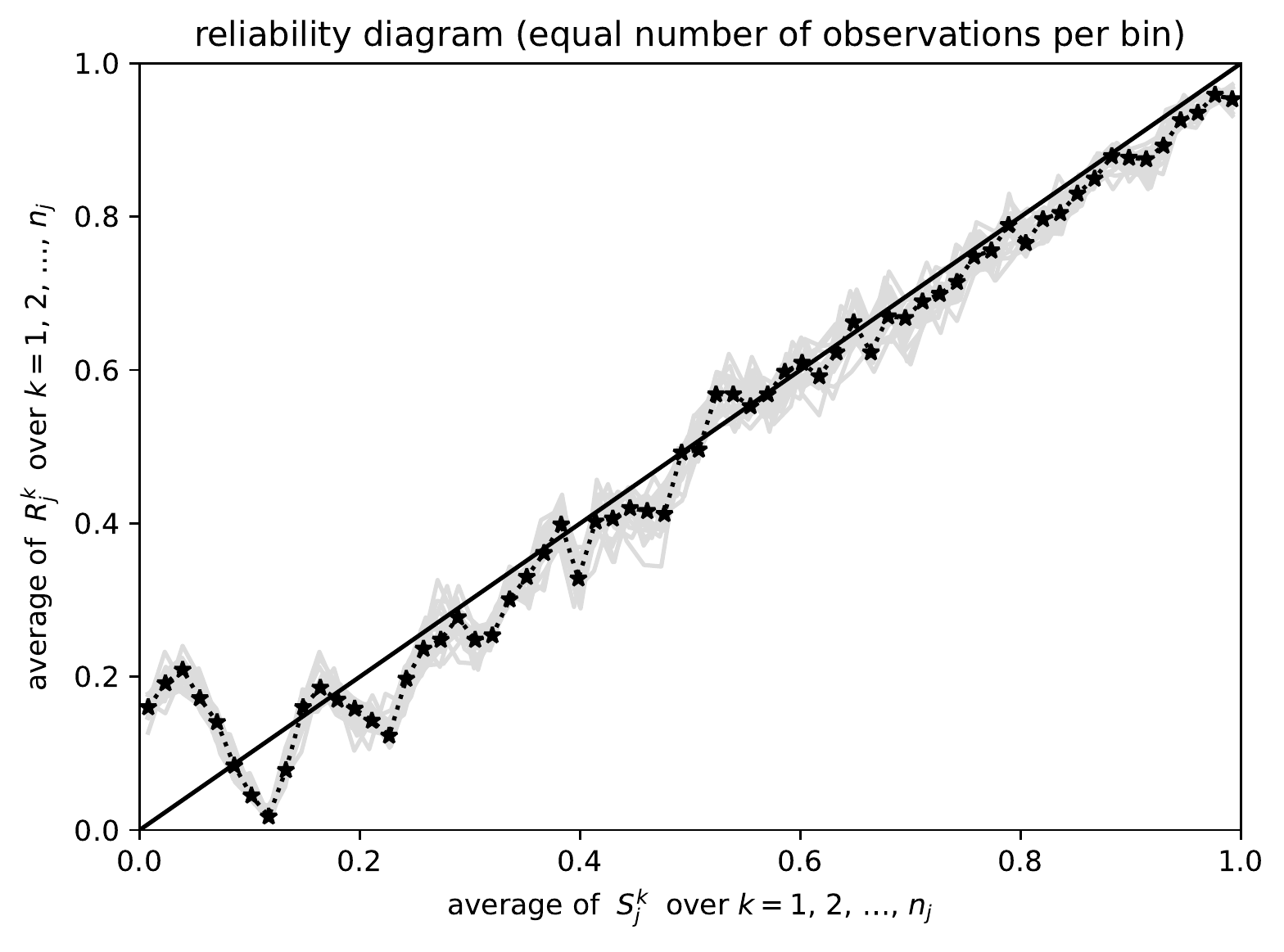}}

\parbox{\imsized}{\includegraphics[width=\imsized]
{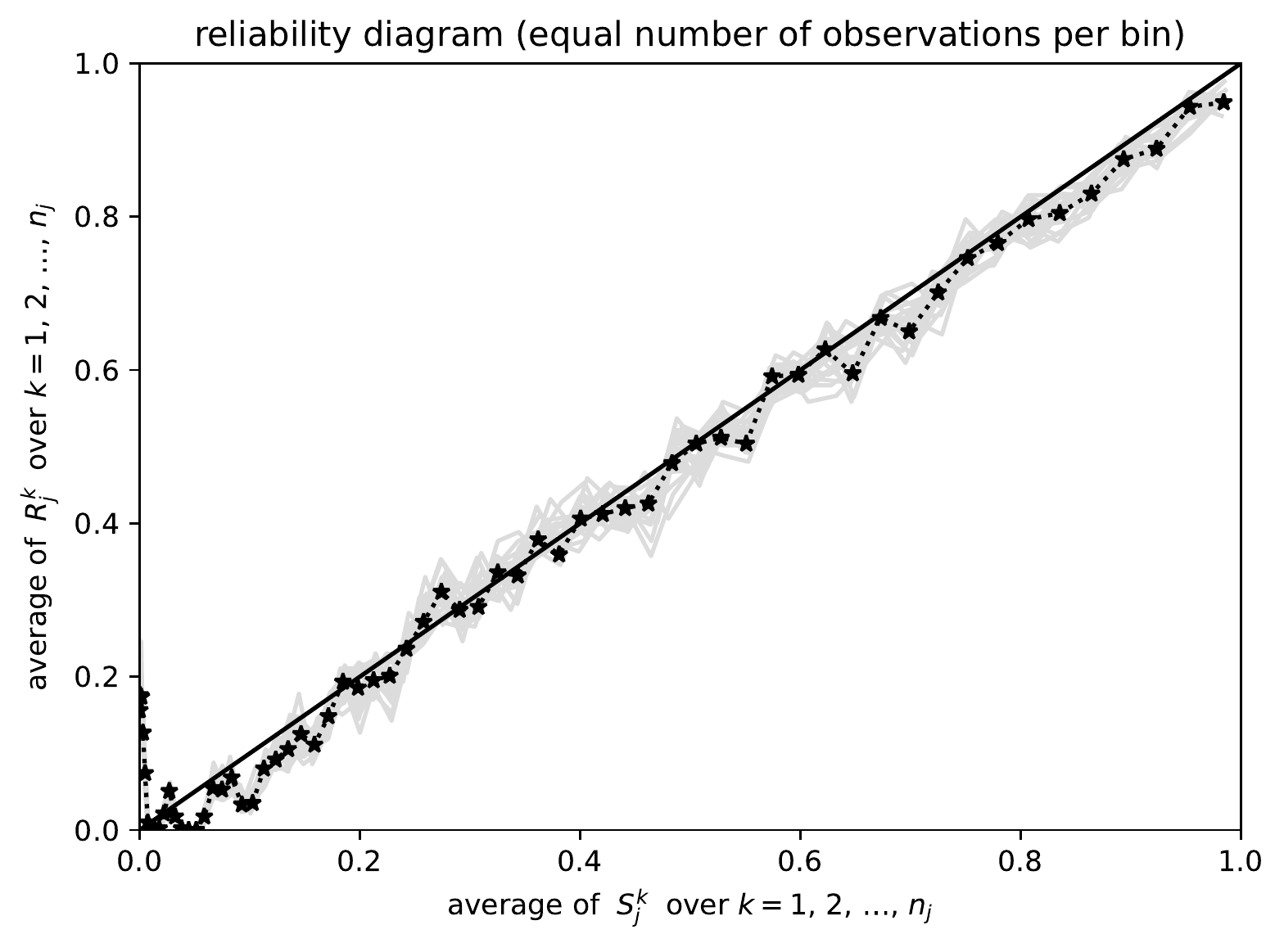}}

\parbox{\imsized}{\includegraphics[width=\imsized]
{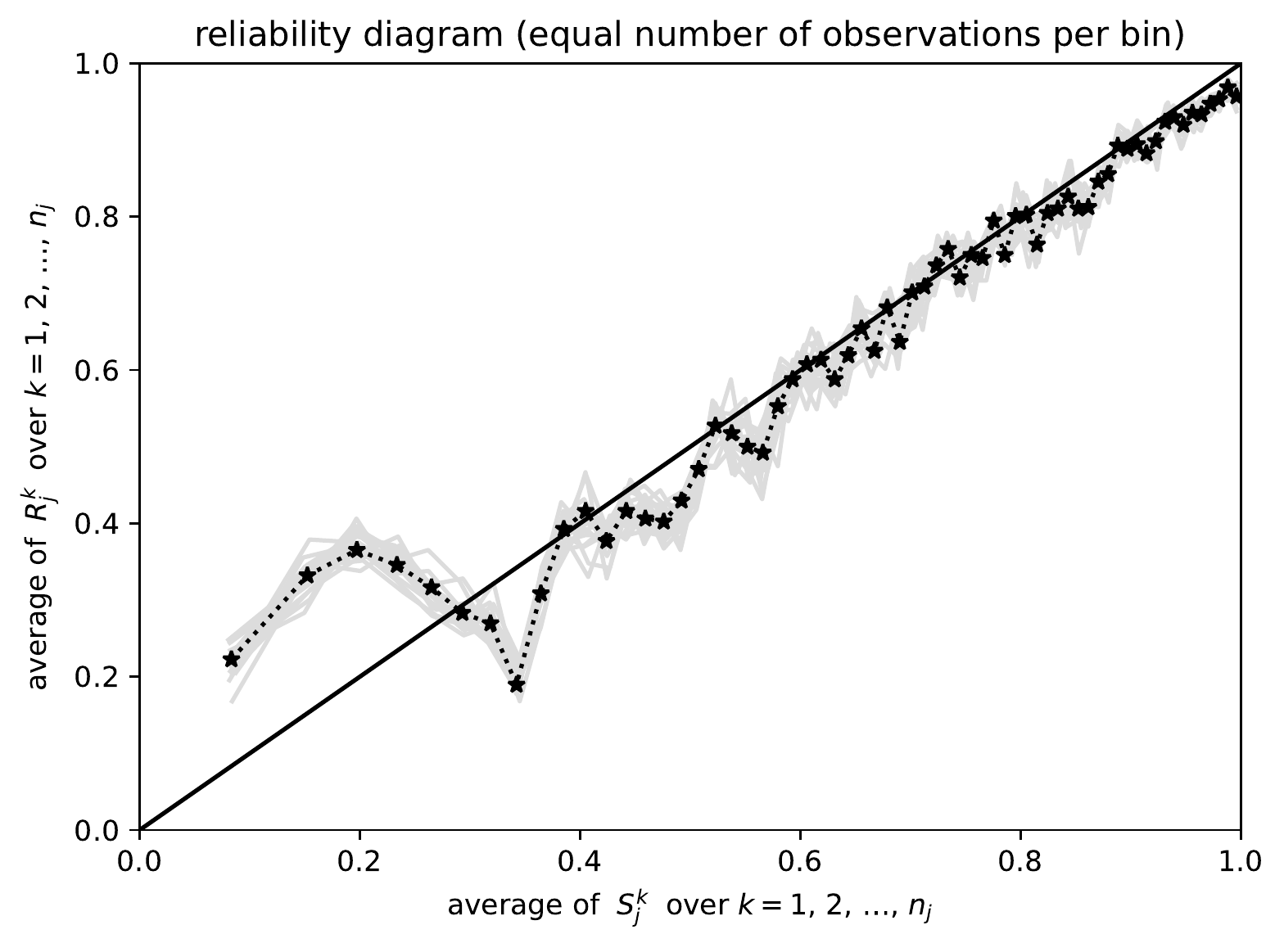}}
\end{center}
\caption{Reliability diagrams for the synthetic data set,
         with an equal number of observations per bin.
         The scores are equispaced in the top plot,
         squared in the middle plot, and square rooted in the bottom plot,
         with $m = 64$ bins and sample size $n =$ 32,768}
\label{32768m64samp}
\end{figure}

\begin{figure}
\begin{center}
\parbox{\imsize}{\includegraphics[width=\imsize]
{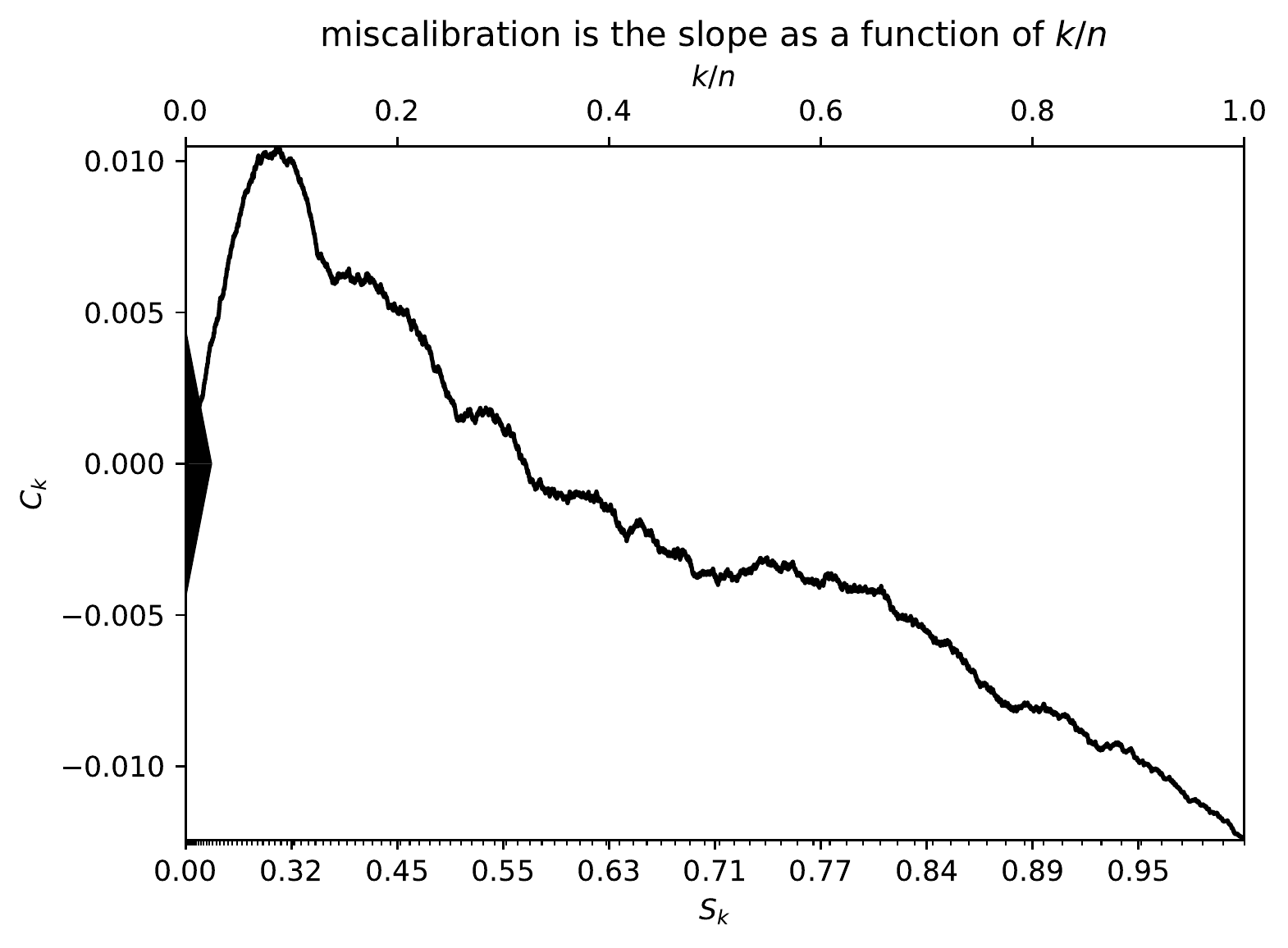}}

\parbox{\imsize}{\includegraphics[width=\imsize]
{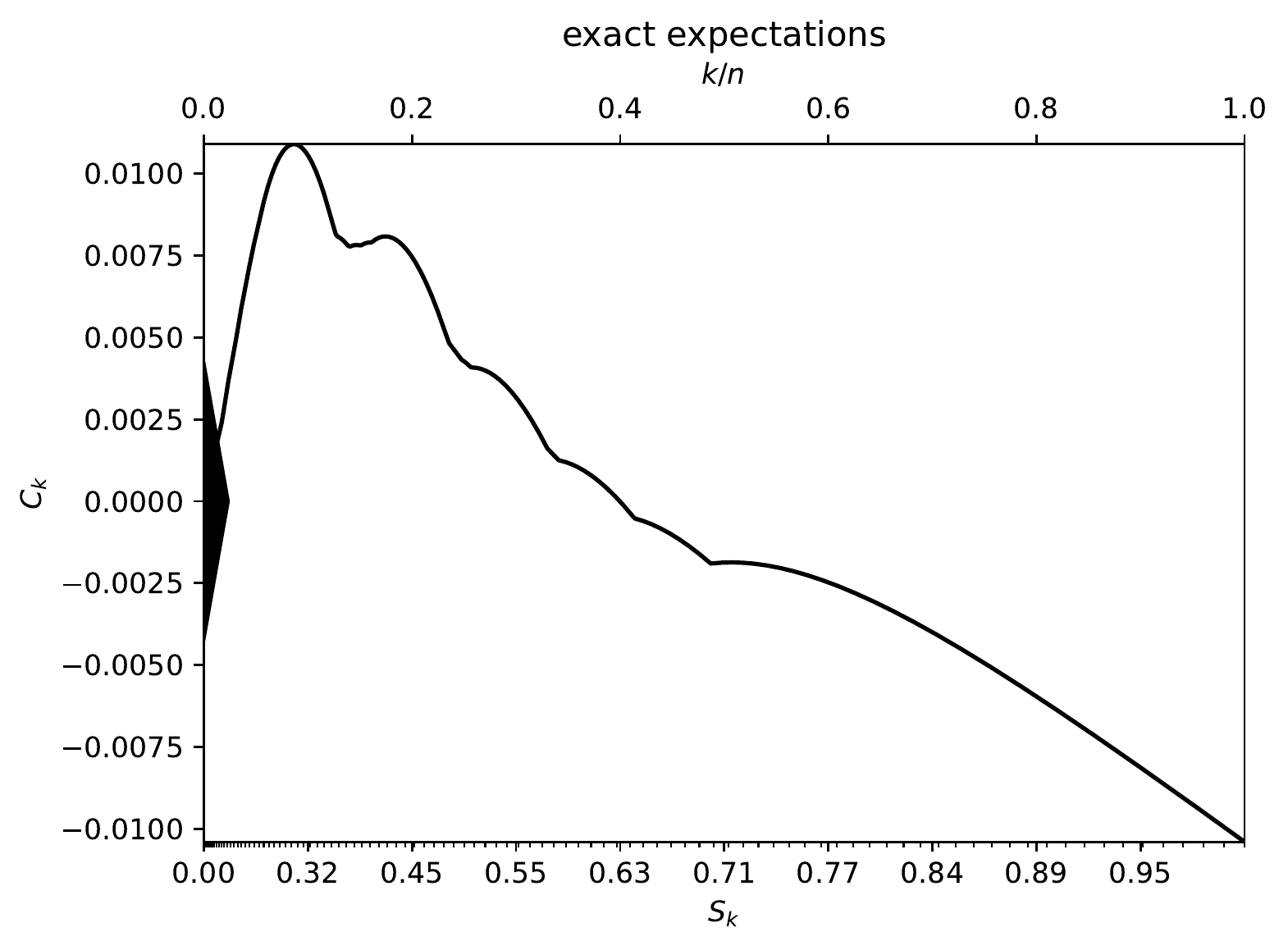}}
\end{center}
\caption{Cumulative plot for the synthetic data set with square-rooted scores
         and sample size $n =$ 32,768.
         The ECCE-MAD is $0.01243 / \sigma_n = 5.512$,
         and the ECCE-R is $0.02291 / \sigma_n = 10.16$;
         the associated asymptotic P-values are 7.1E--08
         and zero to double-precision accuracy, respectively.
         The upper plot is based on the empirical observations,
         while the lower plot is the ideal, based on full knowledge
         of the exact probabilities of success for the Bernoulli distributions
         from which the empirical observations were drawn.
         The slopes of secant lines in the upper plot appear to match
         the slopes of the corresponding secants in the lower plot
         reasonably well, aside from the expected statistical fluctuations
         (whose expected standard deviation is a quarter of the height
         of the triangle at the origin).
}
\label{32768cum}
\end{figure}

\begin{figure}
\begin{center}
\parbox{\imsizes}{\includegraphics[width=\imsizes]
{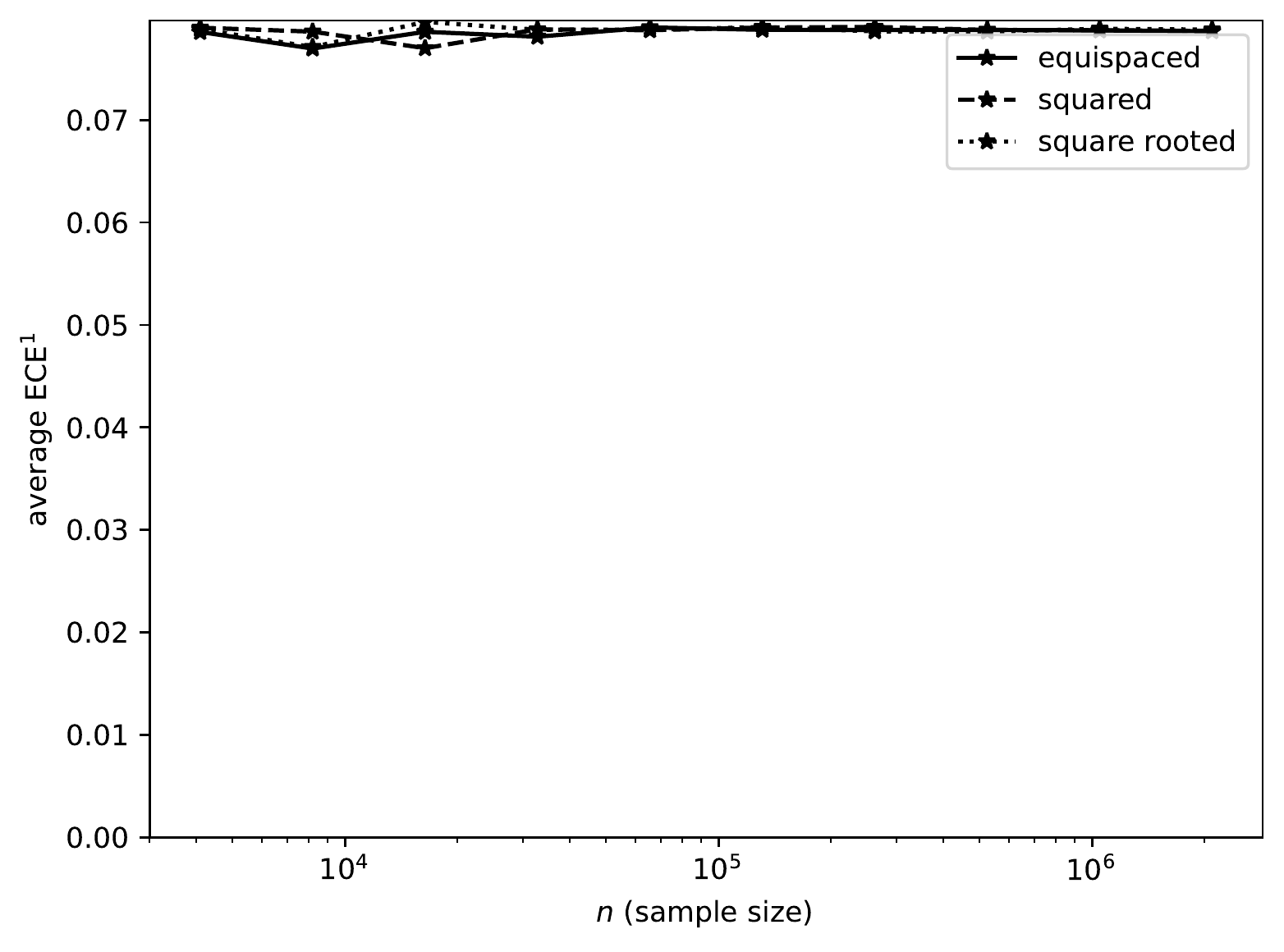}}
\hfil
\parbox{\imsizes}{\includegraphics[width=\imsizes]
{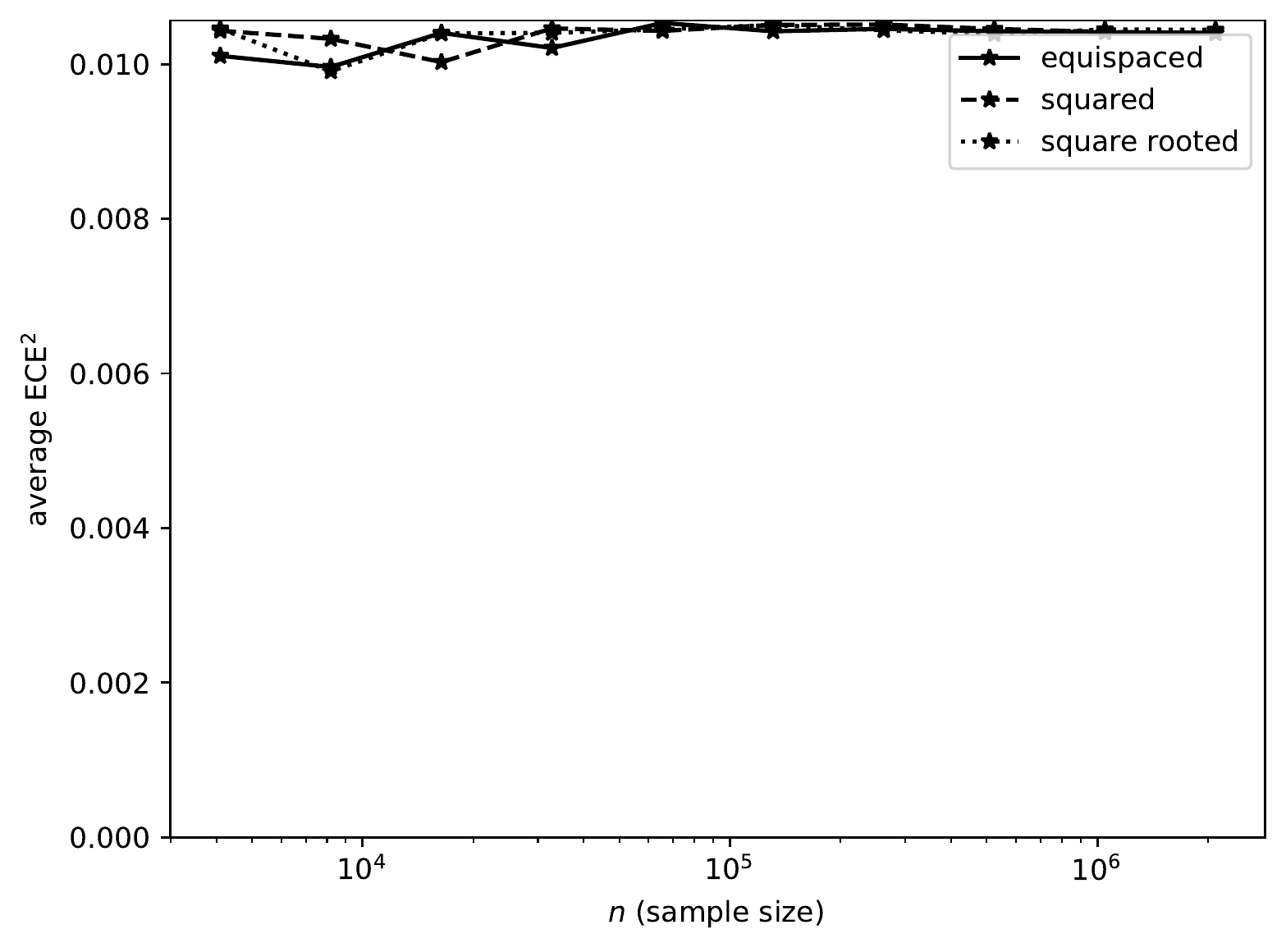}}

\parbox{\imsizes}{\includegraphics[width=\imsizes]
{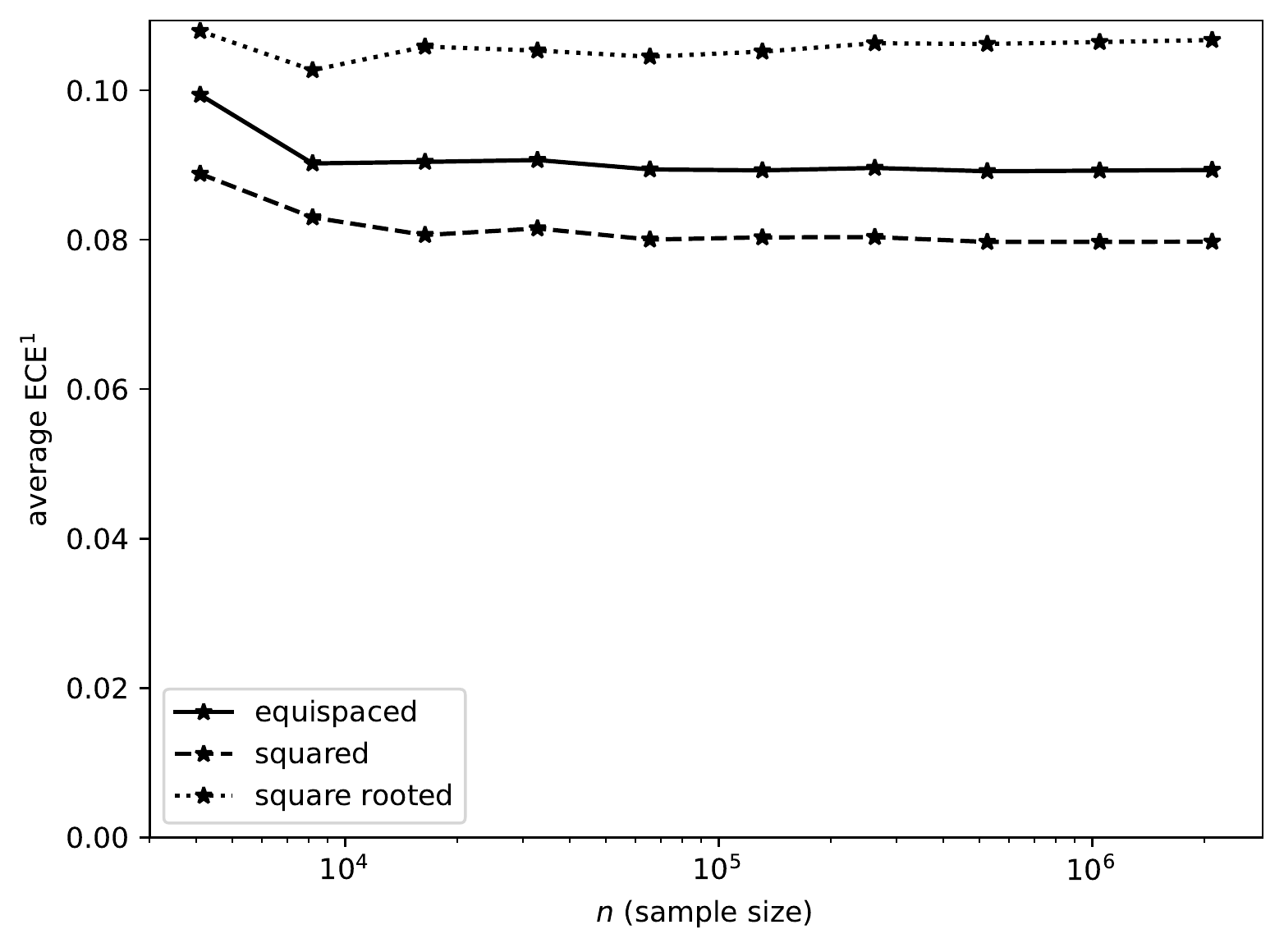}}
\hfil
\parbox{\imsizes}{\includegraphics[width=\imsizes]
{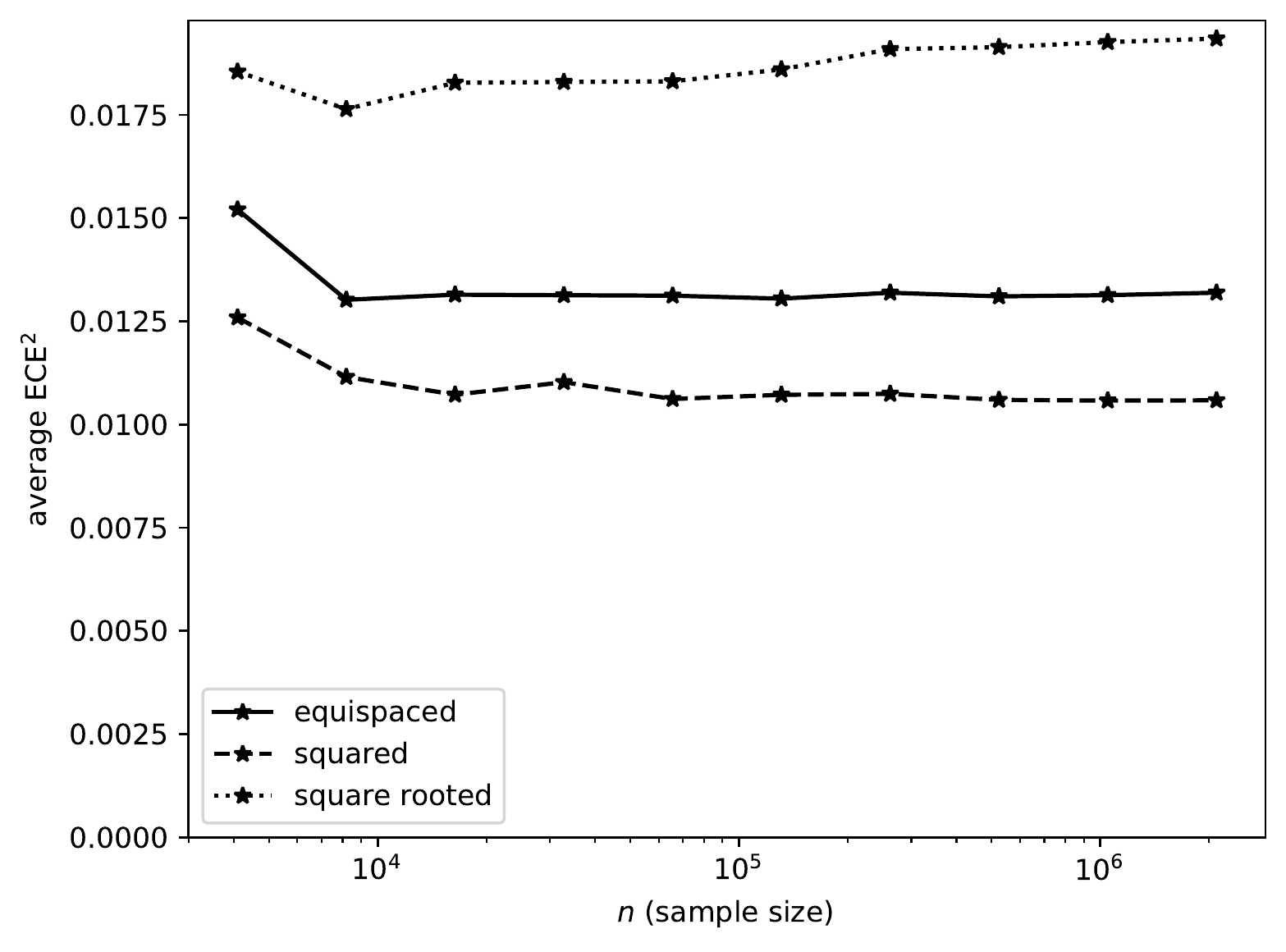}}
\end{center}
\caption{The upper plots display the ECE$^1$ and the ECE$^2$
         averaged over 9 synthetic data sets (reducing random variations
         by about $\sqrt{9} = 3$), each of which is perfectly calibrated.
         The lower plots display the ECE$^1$ and the ECE$^2$
         averaged over 9 synthetic data sets,
         each drawn from the distribution
         depicted in Figure~\ref{32768exact} for the sample size $n =$ 32,768.
         In all cases, each bin contains the same number of observations,
         namely 16; so $m$, the number of bins, is $n / 16$.
         The scores are equispaced, equispaced then squared,
         or equispaced and then square rooted, as indicated in the legends
         for the plots; the underlying alternative distributions of responses
         used for the lower plots here which correspond
         to these different distributions of scores
         are the upper, middle, and lower plots of Figure~\ref{32768exact}
         for $n =$ 32,768, respectively.
         Notice how all values for the ECE$^1$ are quite similar,
         as are all values for the ECE$^2$;
         distinguishing the perfectly calibrated data sets
         from the alternative distribution of Figure~\ref{32768exact}
         is very hard.}
\label{ecen}
\end{figure}

\begin{figure}
\begin{center}
\parbox{\imsizes}{\includegraphics[width=\imsizes]
{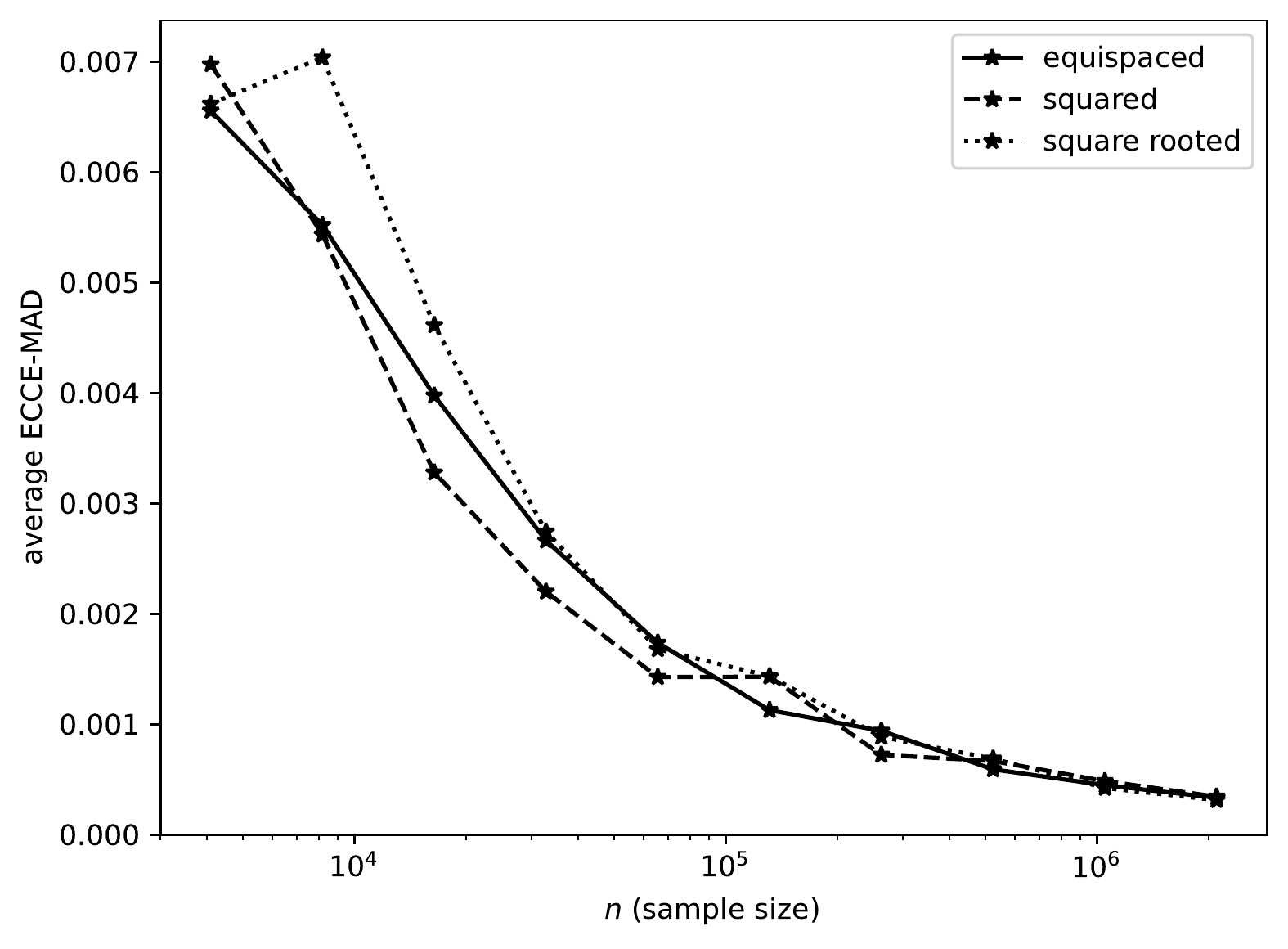}}
\hfil
\parbox{\imsizes}{\includegraphics[width=\imsizes]
{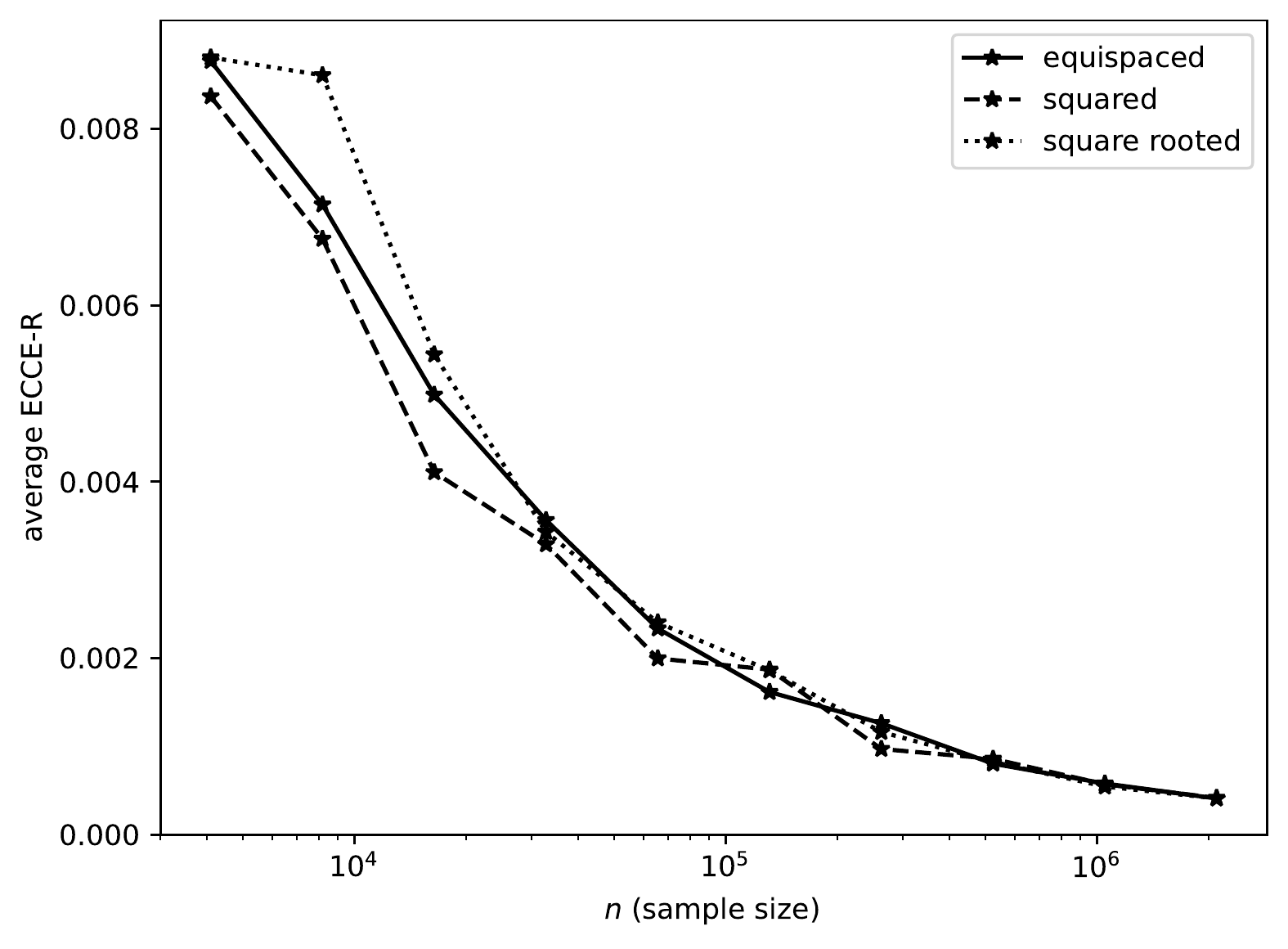}}

\parbox{\imsizes}{\includegraphics[width=\imsizes]
{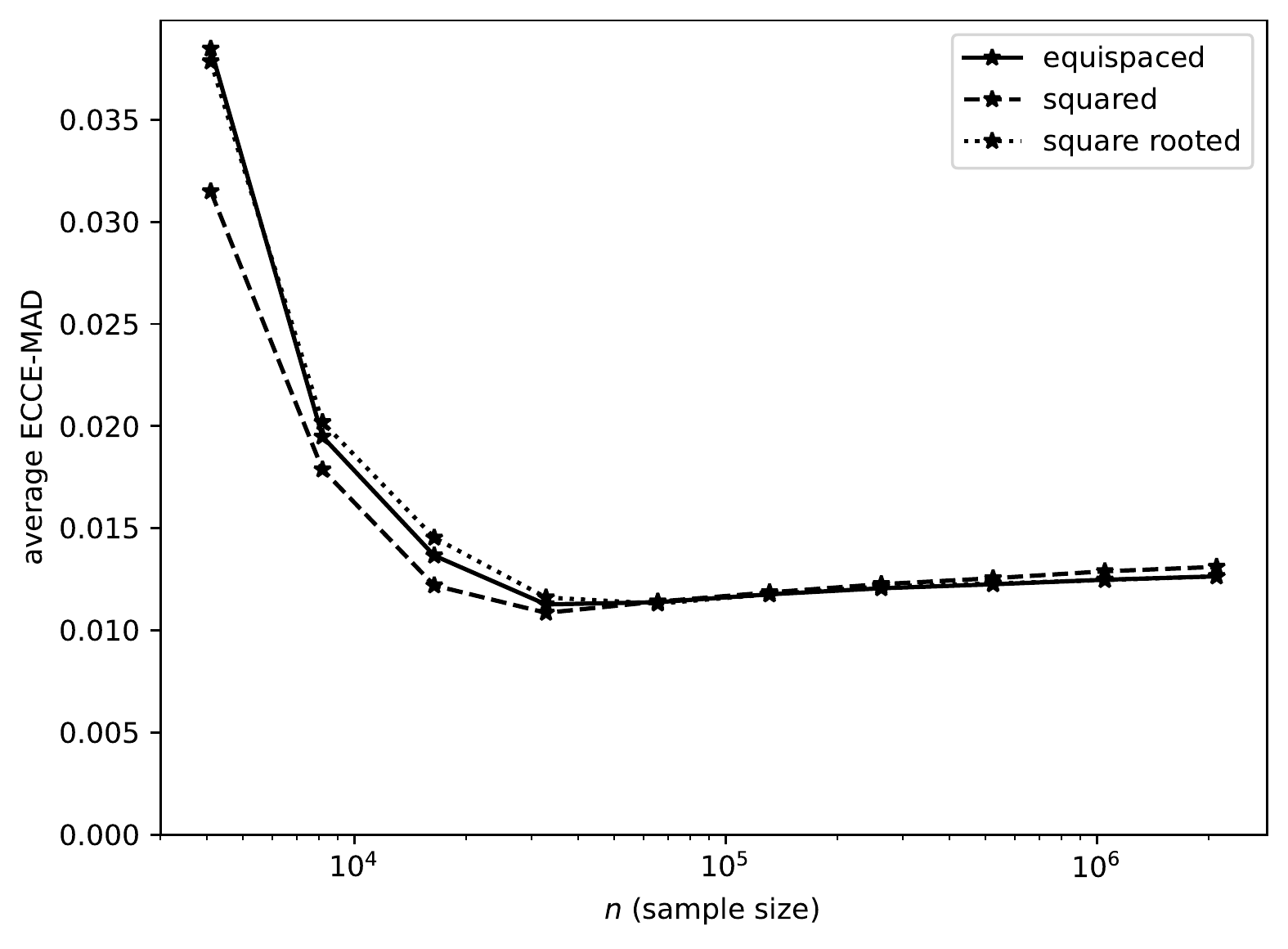}}
\hfil
\parbox{\imsizes}{\includegraphics[width=\imsizes]
{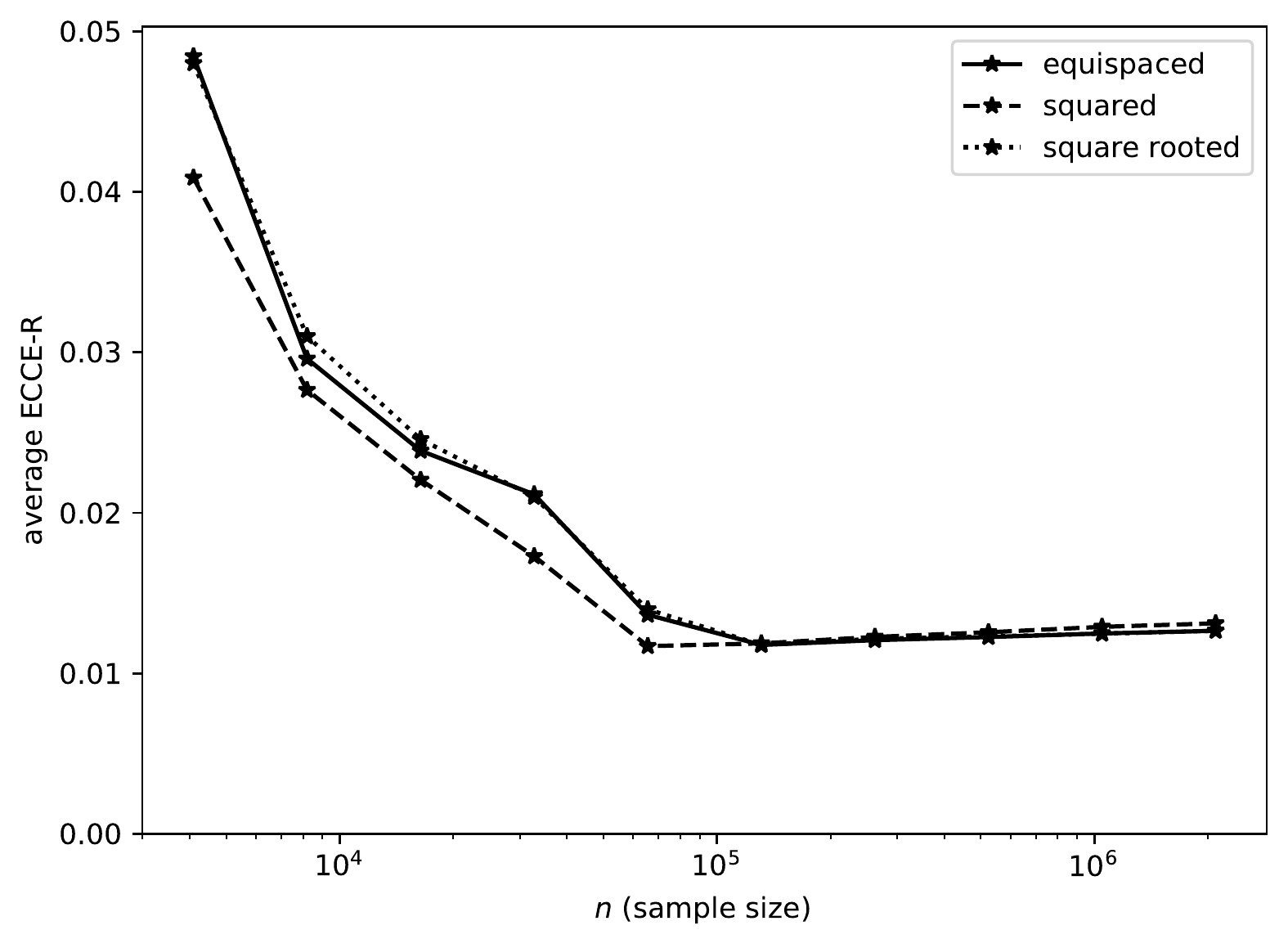}}
\end{center}
\caption{The upper plots display the ECCE-MAD and the ECCE-R
         averaged over 9 synthetic data sets (hence reducing random variations
         by a factor of about $\sqrt{9} = 3$),
         each of which is perfectly calibrated.
         The lower plots display the ECCE-MAD and the ECCE-R
         averaged over 9 synthetic data sets,
         each drawn from the distribution
         depicted in Figure~\ref{32768exact} for the sample size $n =$ 32,768.
         The scores are equispaced, equispaced then squared,
         or equispaced and then square rooted, as indicated in the legends
         for the plots; the underlying alternative distributions of responses
         used in the lower plots here which correspond
         to these different distributions of scores
         are the upper, middle, and lower plots of Figure~\ref{32768exact}
         for $n =$ 32,768, respectively.
         Notice how the values for the ECCE-MAD get much, much lower
         in the upper plot than in the lower plot, and, similarly,
         how the values for the ECCE-R get much, much lower in the upper plot
         than in the lower plot;
         distinguishing the perfectly calibrated data sets
         from the alternative distribution of Figure~\ref{32768exact}
         is easy with the ECCE-MAD or the ECCE-R,
         with high statistical confidence that increases
         as the sample size $n$ becomes large.
         The graphs in the lower plots stay flat as $n$ becomes large,
         while the graphs in the upper plots decay rapidly.}
\label{cumn}
\end{figure}

\begin{figure}
\begin{center}
\parbox{\imsizes}{\includegraphics[width=\imsizes]
{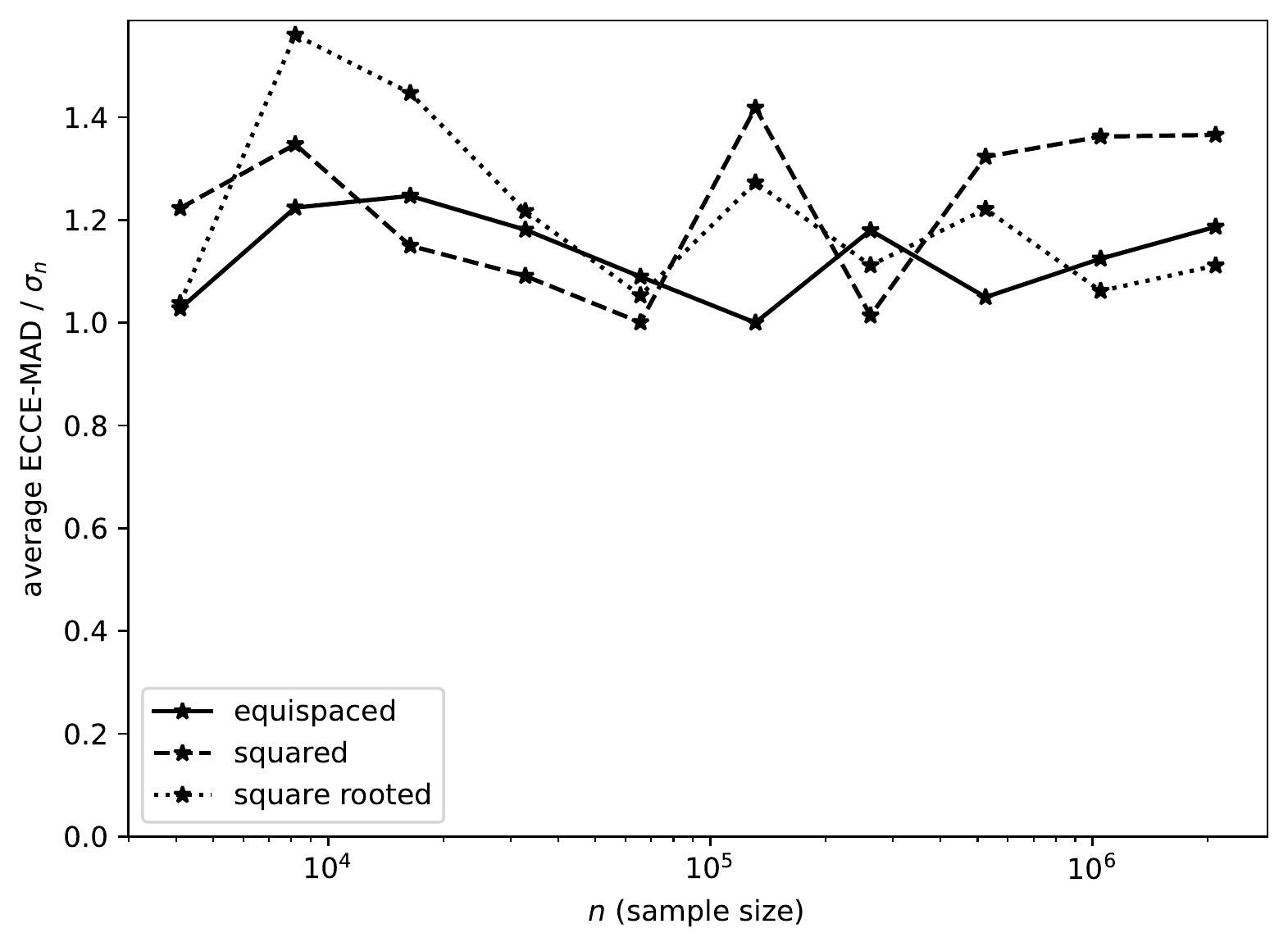}}
\hfil
\parbox{\imsizes}{\includegraphics[width=\imsizes]
{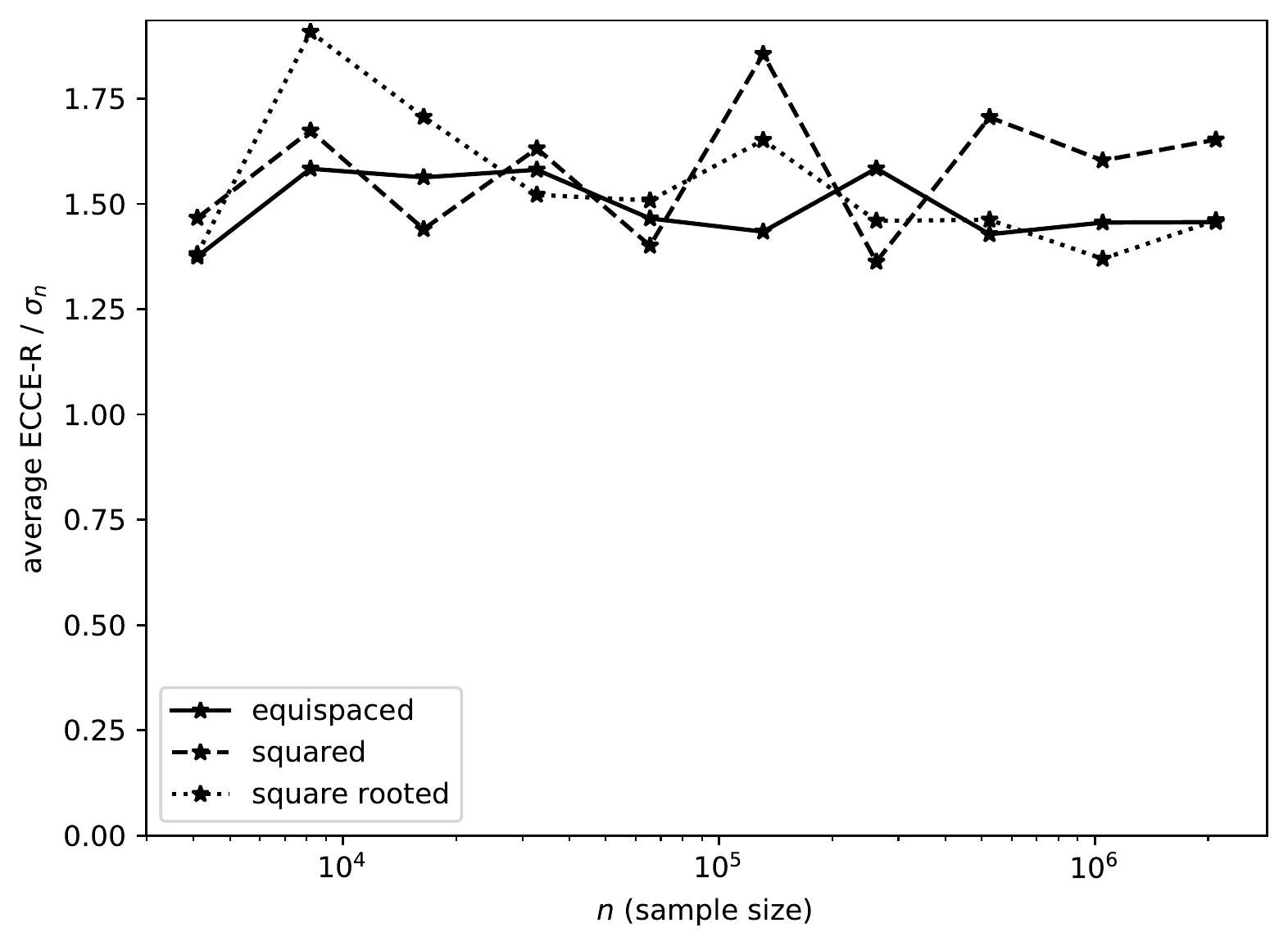}}

\parbox{\imsizes}{\includegraphics[width=\imsizes]
{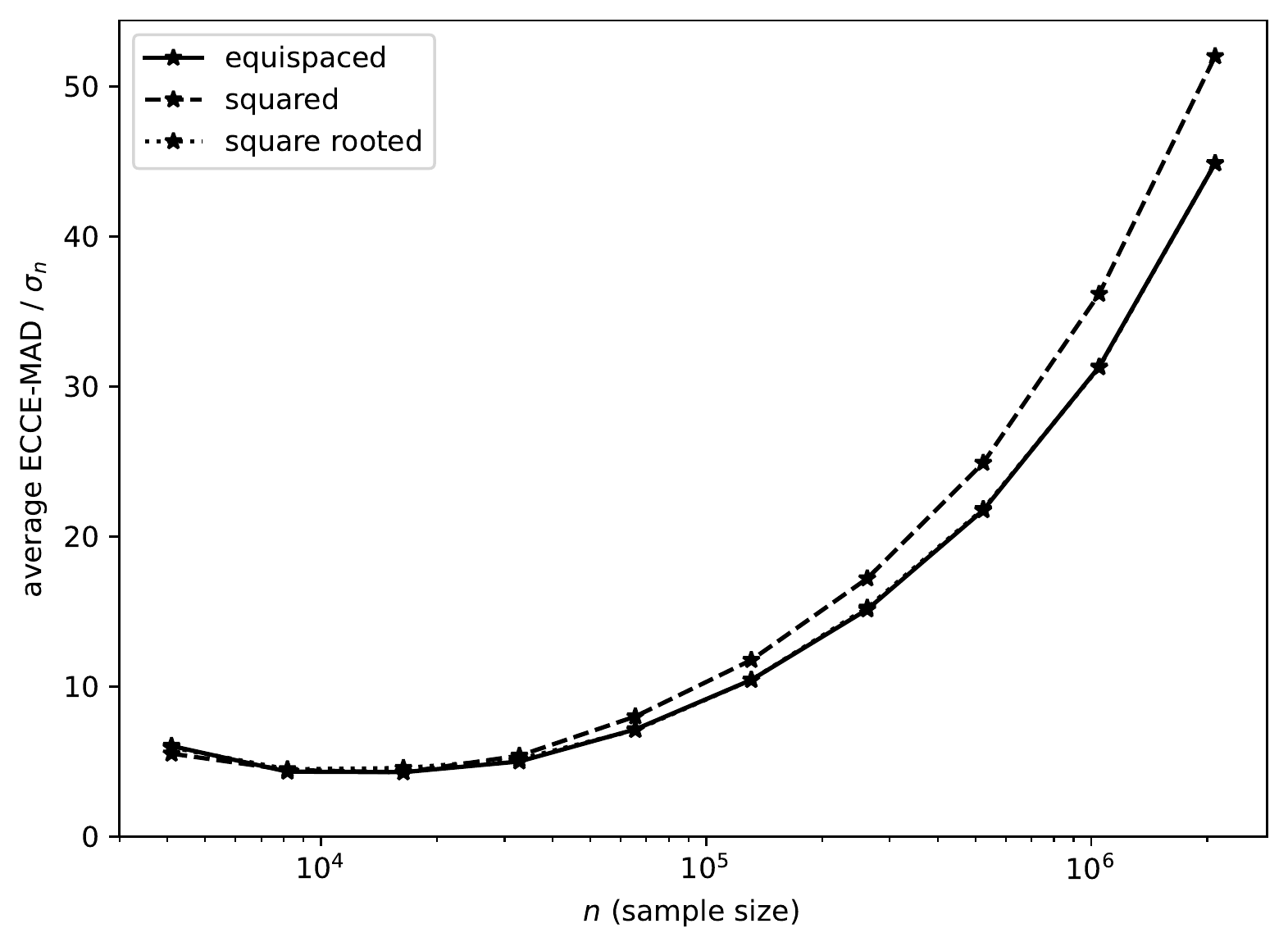}}
\hfil
\parbox{\imsizes}{\includegraphics[width=\imsizes]
{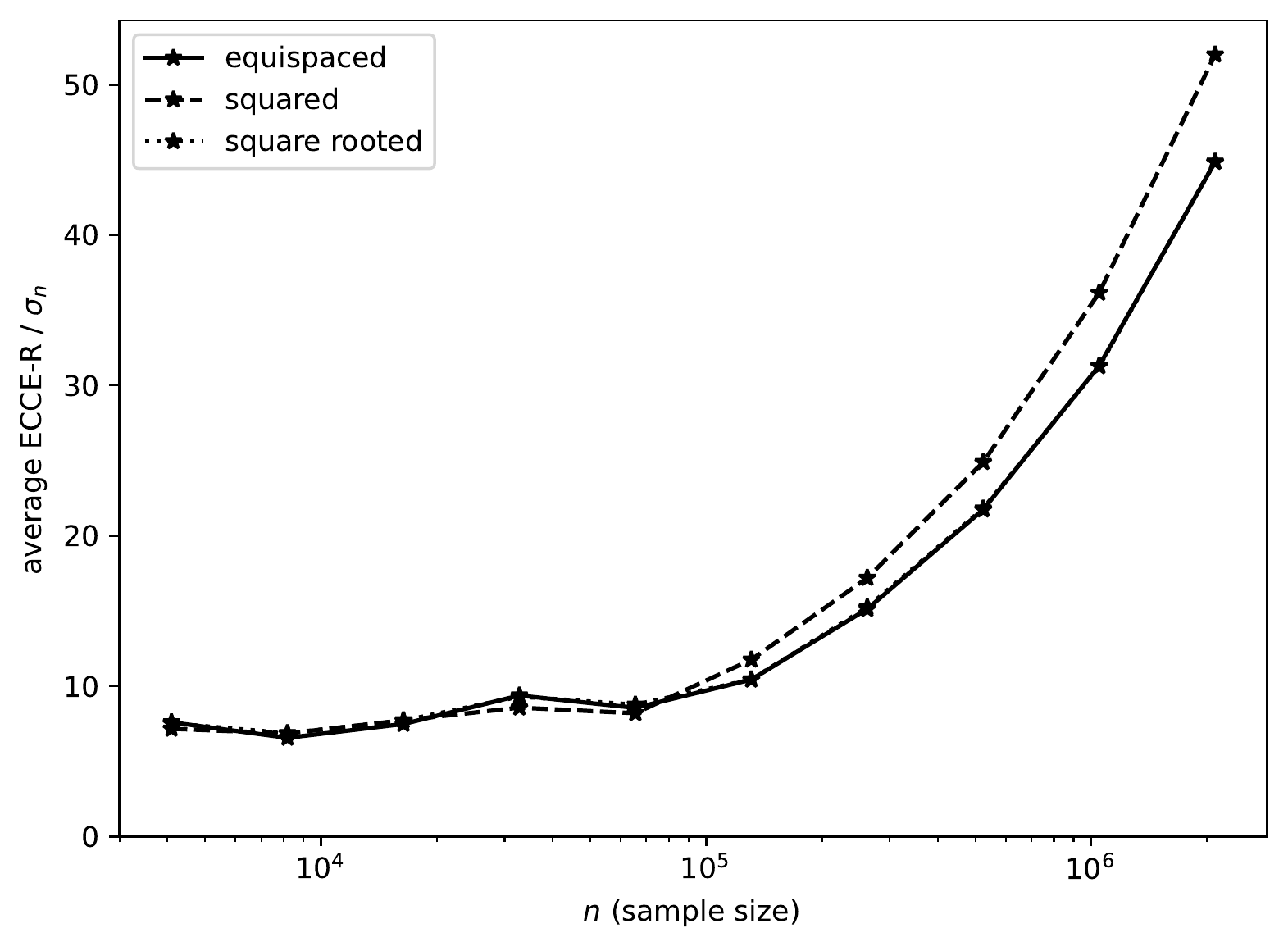}}
\end{center}
\caption{The upper plots display the normalized ECCE-MAD
         and the normalized ECCE-R
         averaged over 9 synthetic data sets (which reduces random variations
         by a factor of around $\sqrt{9} = 3$),
         each of which is perfectly calibrated.
         The lower plots display the normalized ECCE-MAD
         and the normalized ECCE-R averaged over 9 synthetic data sets,
         each drawn from the distribution
         depicted in Figure~\ref{32768exact} for the sample size $n =$ 32,768.
         The scores are equispaced, equispaced then squared,
         or equispaced and then square rooted, as indicated in the legends
         for the plots; the underlying alternative distributions of responses
         used for the lower plots here which correspond
         to these different distributions of scores
         are the upper, middle, and lower plots of Figure~\ref{32768exact}
         for $n =$ 32,768, respectively. The normalization factor $\sigma_n$
         is defined in~(\ref{var}). The average across the 9 realizations
         is over the quotient of the ECCE by $\sigma_n$, with a different value
         of $\sigma_n$ for every realization. Notice how the values
         for the ECCE-MAD / $\sigma_n$ get much, much higher
         in the lower plot than in the upper plot, and, similarly,
         how the values for the ECCE-R / $\sigma_n$ get much, much higher
         in the lower plot than in the upper plot;
         distinguishing the perfectly calibrated data sets
         from the alternative distribution of Figure~\ref{32768exact}
         is easy with the normalized ECCE-MAD or the normalized ECCE-R,
         with high statistical confidence that increases
         as the sample size $n$ becomes large.
         The graphs in the upper plots stay flat as $n$ becomes large,
         while the graphs in the lower plots increase explosively.}
\label{cumnnorm}
\end{figure}

\begin{figure}
\begin{center}
\parbox{\imsizes}{\includegraphics[width=\imsizes]
       {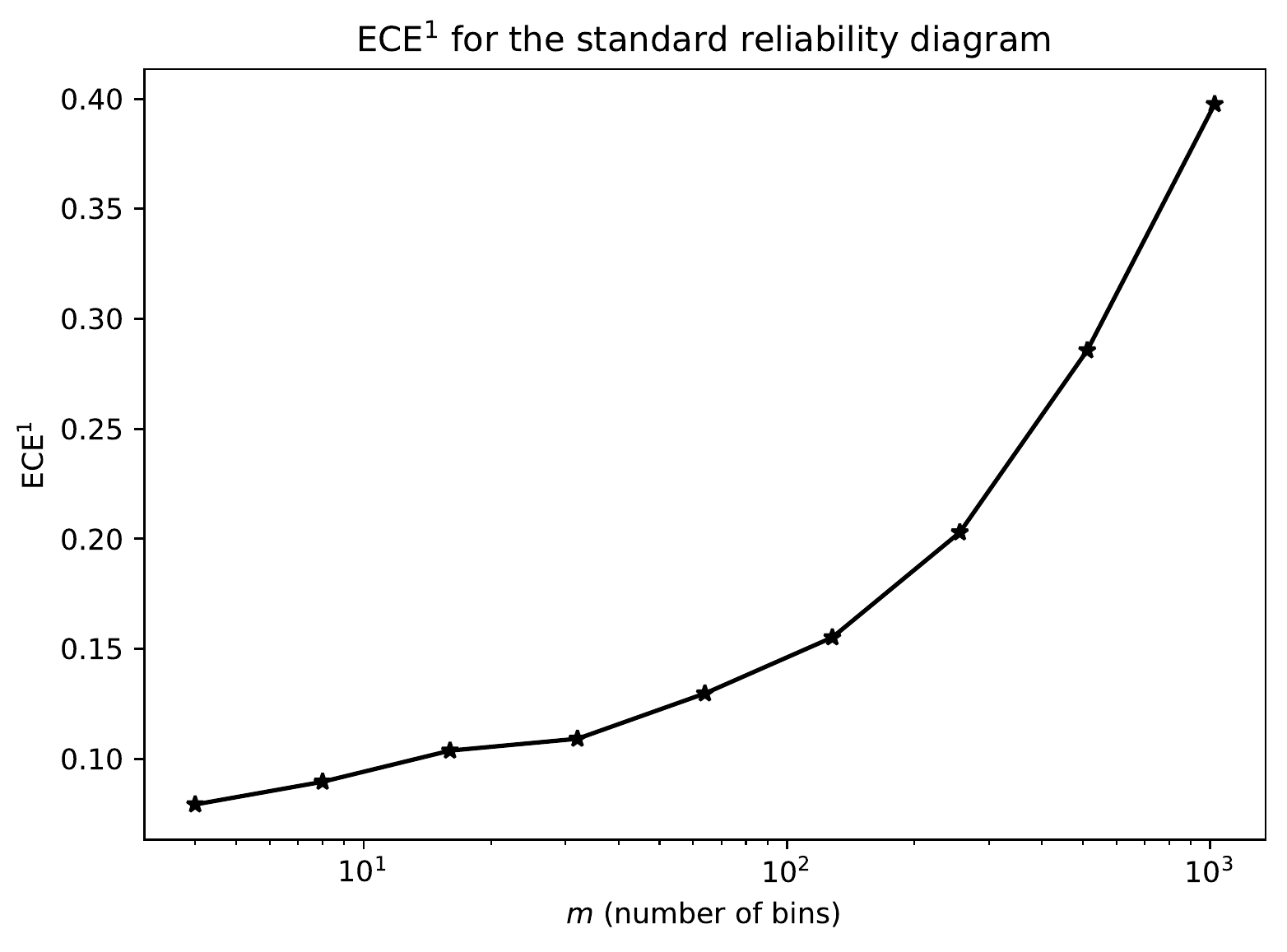}}
\hfil
\parbox{\imsizes}{\includegraphics[width=\imsizes]
       {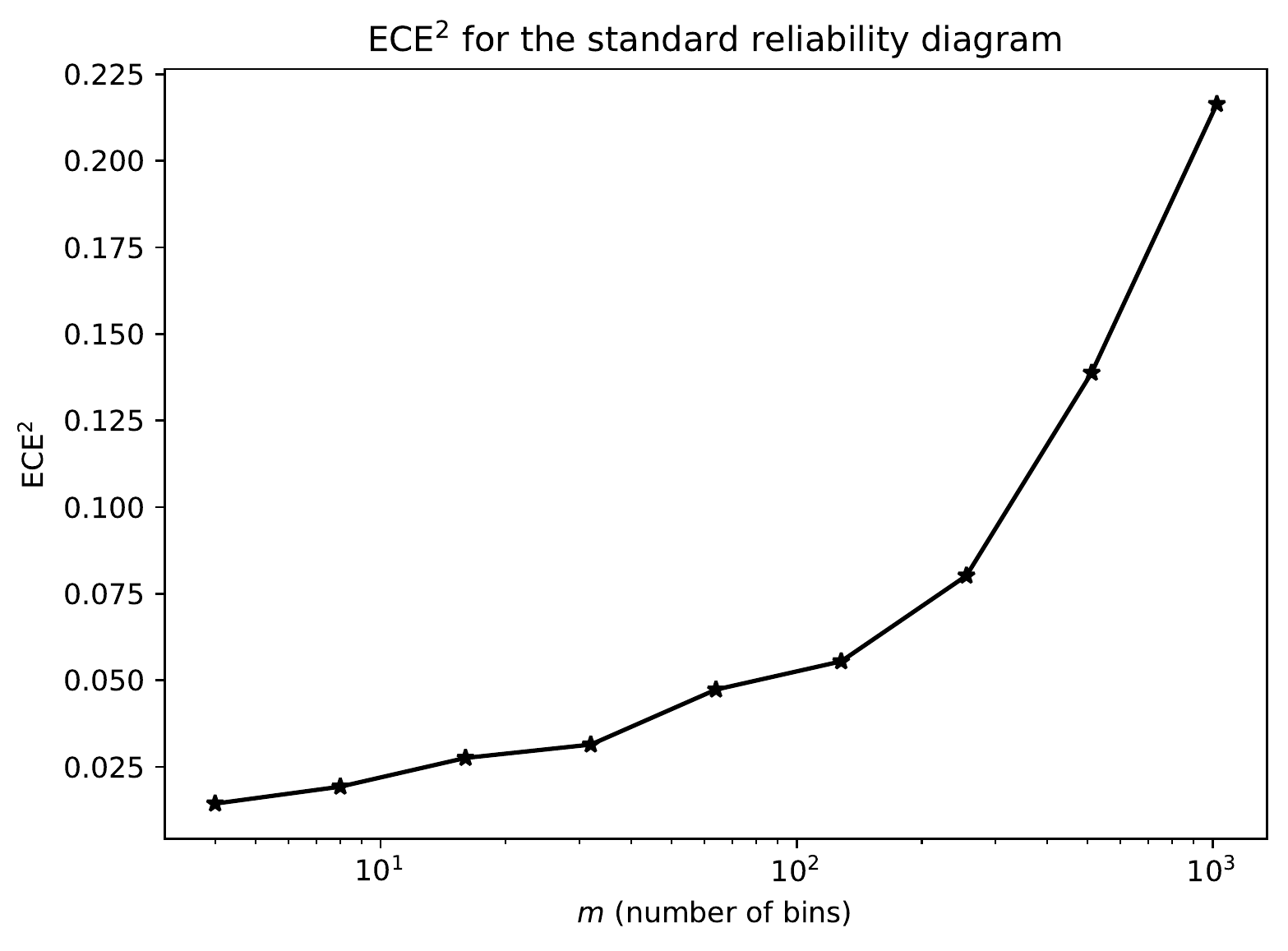}}

\parbox{\imsizes}{\includegraphics[width=\imsizes]
       {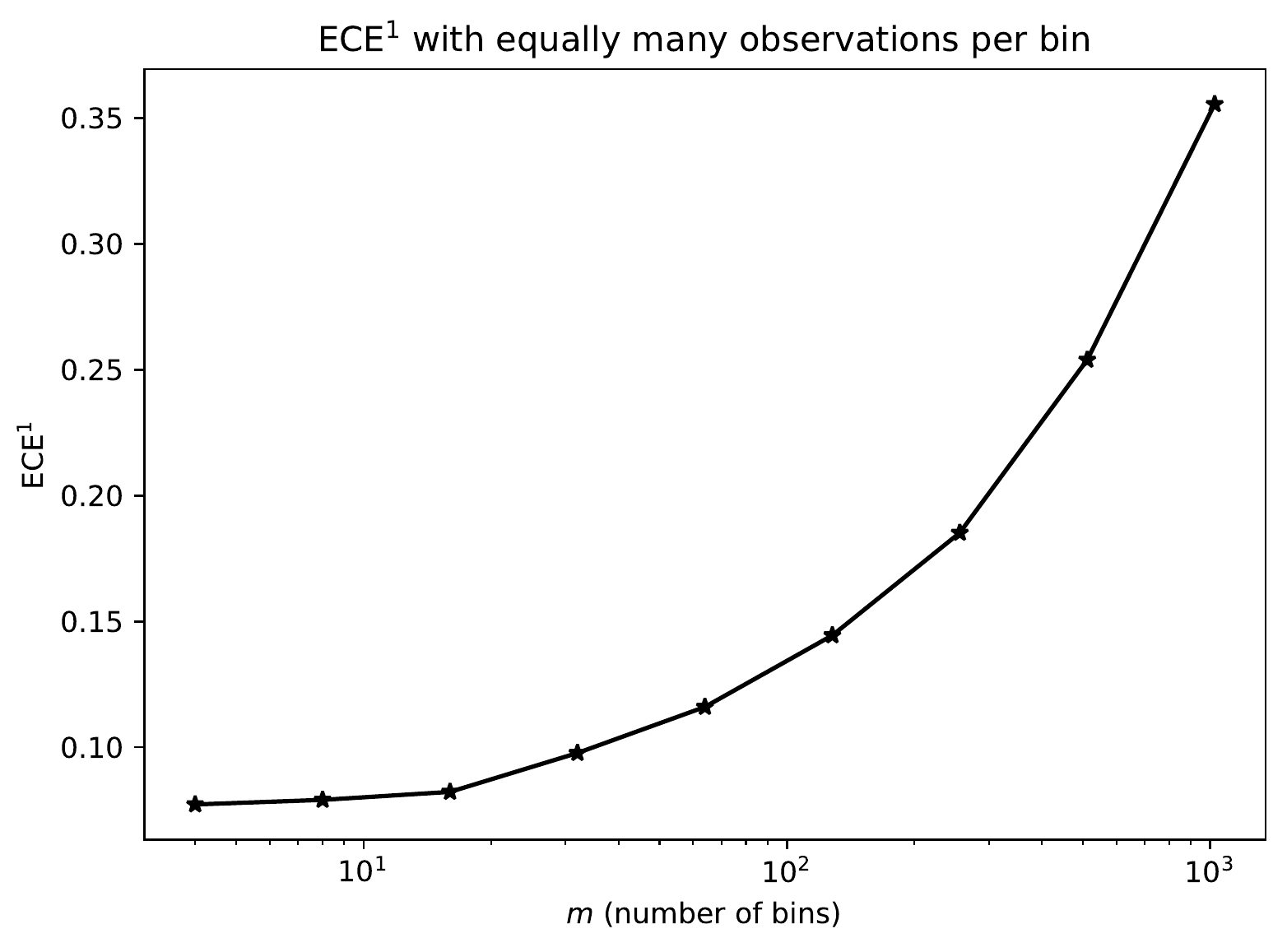}}
\hfil
\parbox{\imsizes}{\includegraphics[width=\imsizes]
       {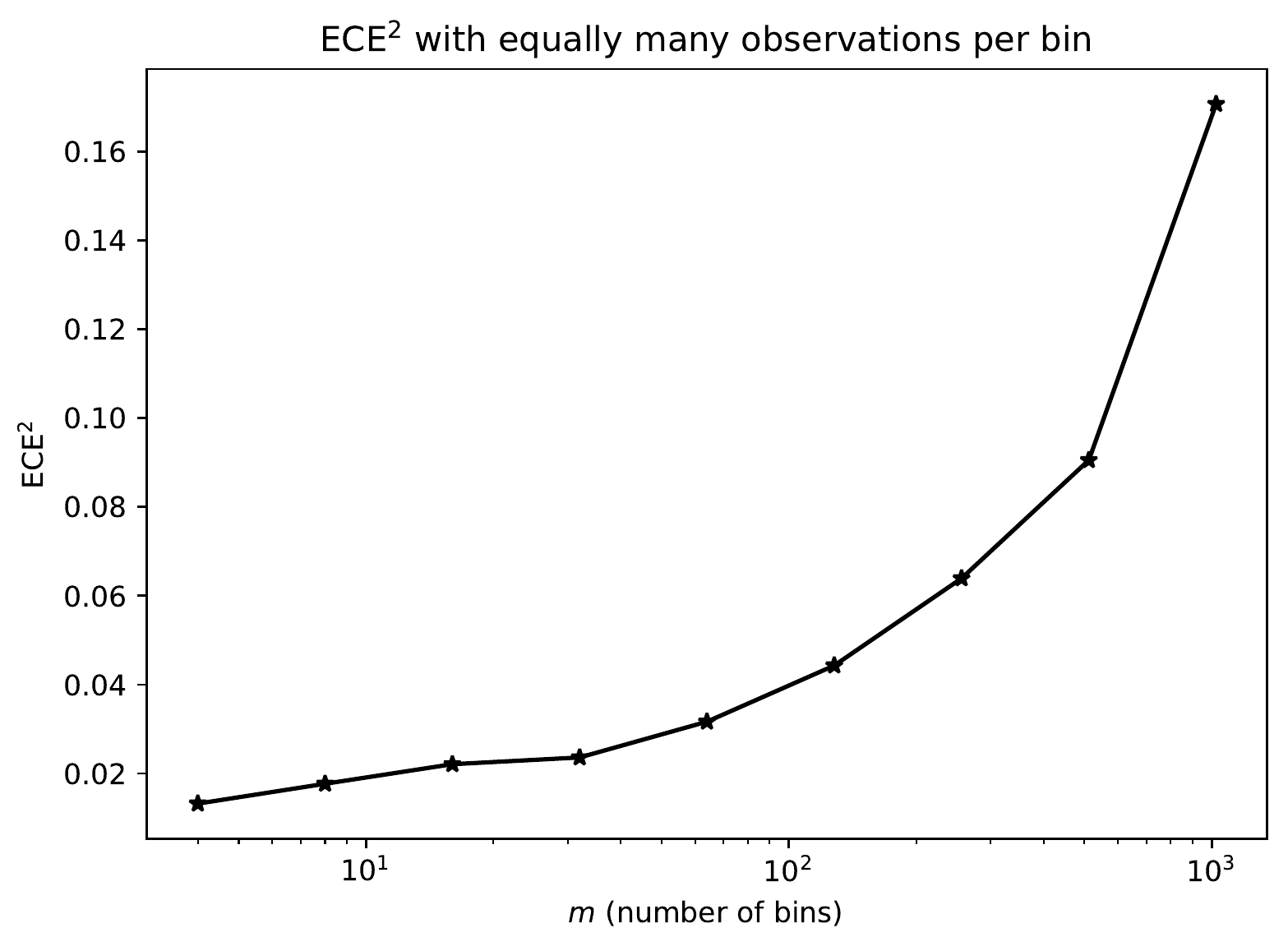}}
\end{center}
\caption{Empirical calibration errors for the night snake
         ({\it Hypsiglena torquata}), with sample size $n =$ 1,300.}
\label{night-snakeece}
\end{figure}

\begin{figure}
\begin{center}
\parbox{\imsize}{\includegraphics[width=\imsize]
       {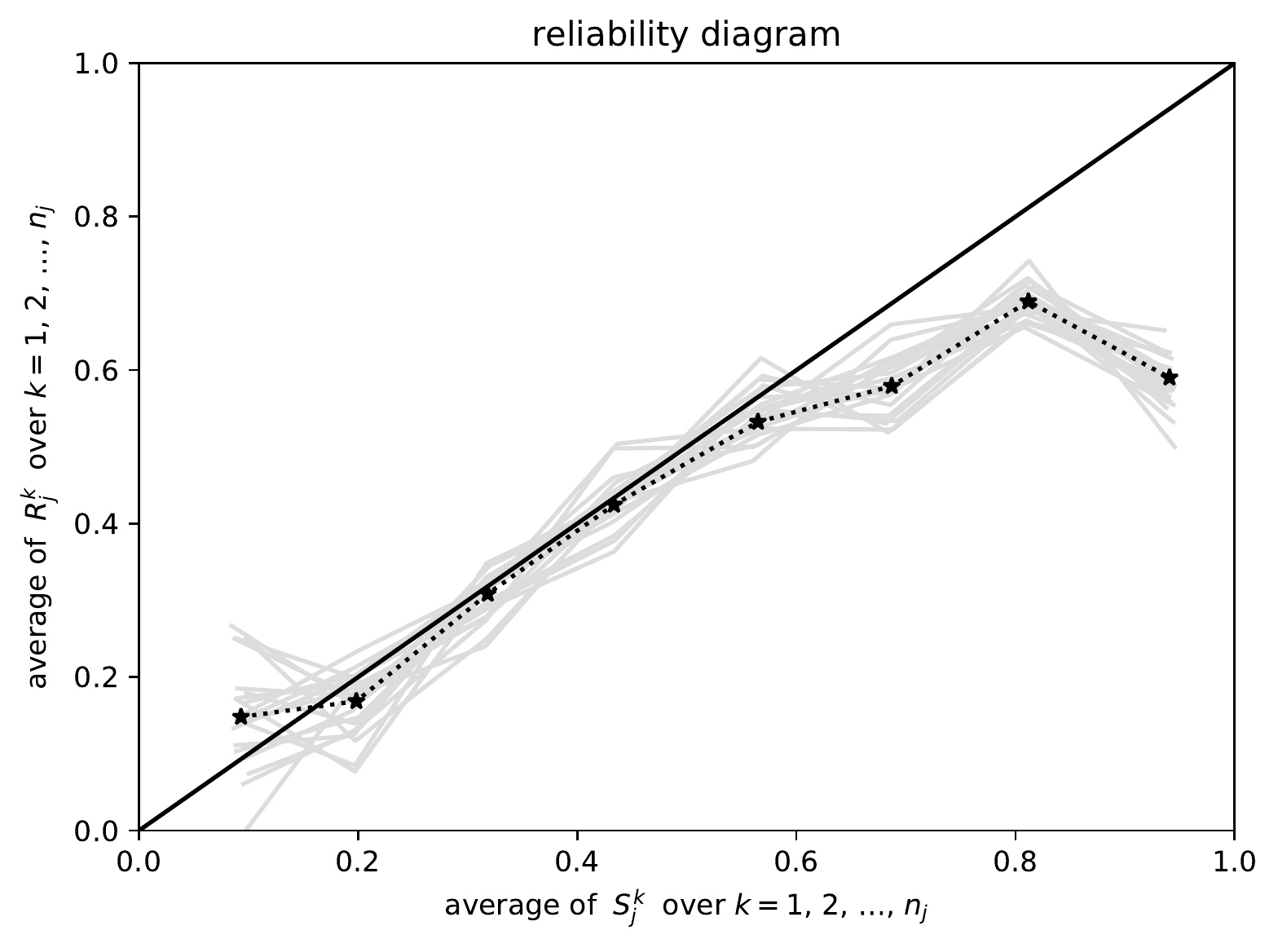}}

\parbox{\imsize}{\includegraphics[width=\imsize]
       {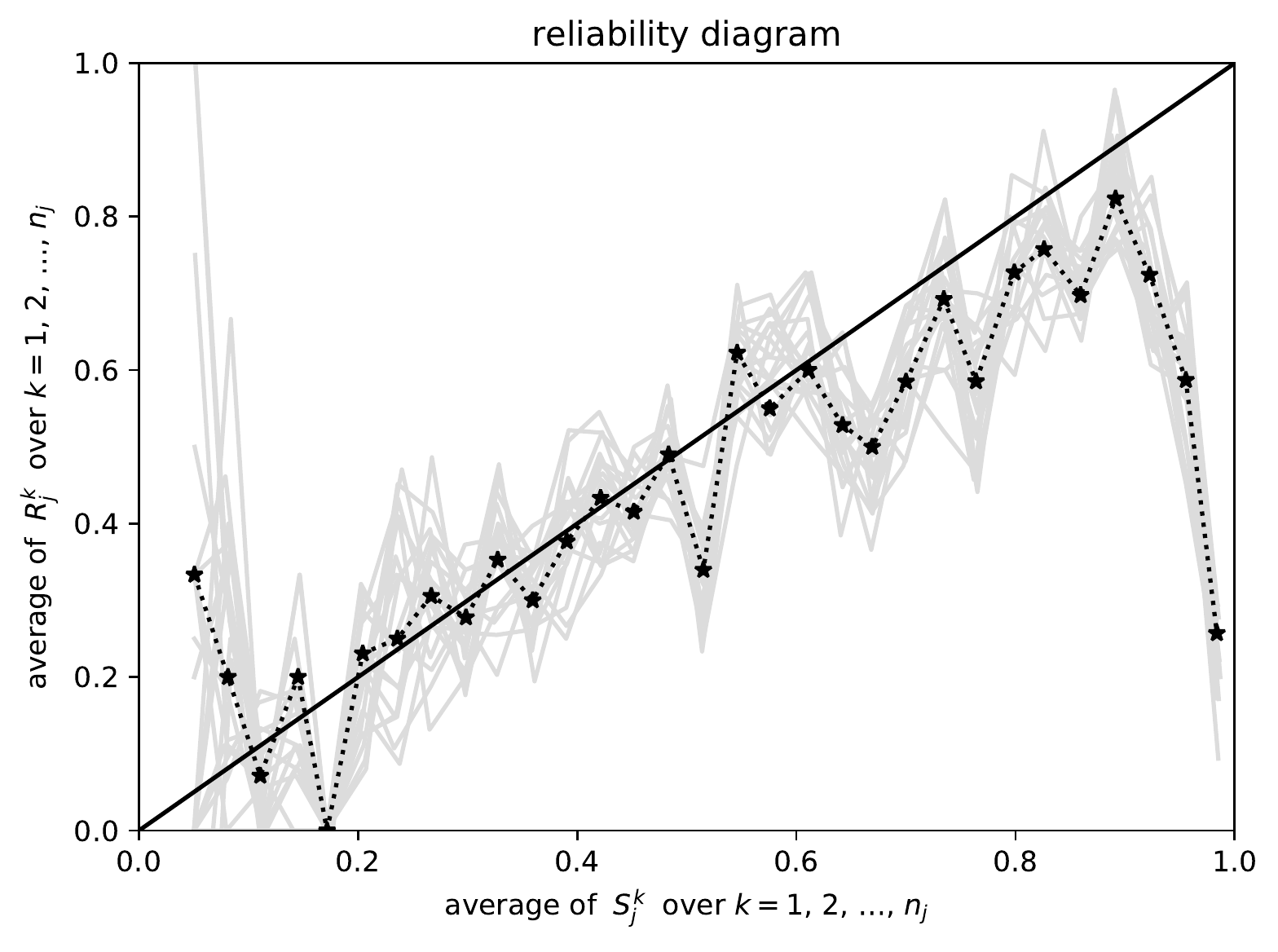}}
\end{center}
\caption{Reliability diagrams for the night snake ({\it Hypsiglena torquata}),
         with the bins roughly equispaced.
         There are $m = 8$ bins in the upper plot
         and $m = 32$ in the lower plot.}
\label{night-snakeprob}
\end{figure}

\begin{figure}
\begin{center}
\parbox{\imsize}{\includegraphics[width=\imsize]
       {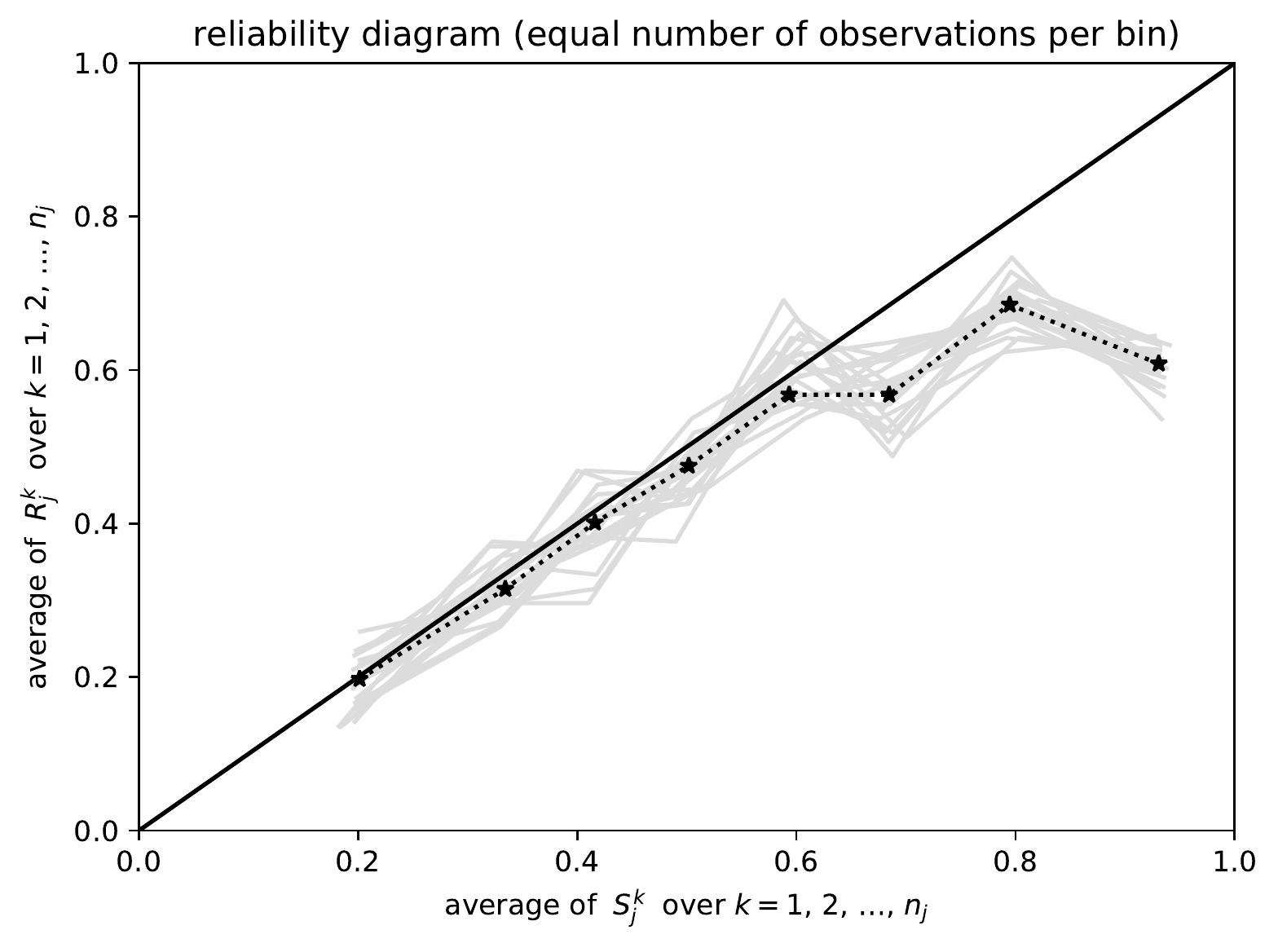}}

\parbox{\imsize}{\includegraphics[width=\imsize]
       {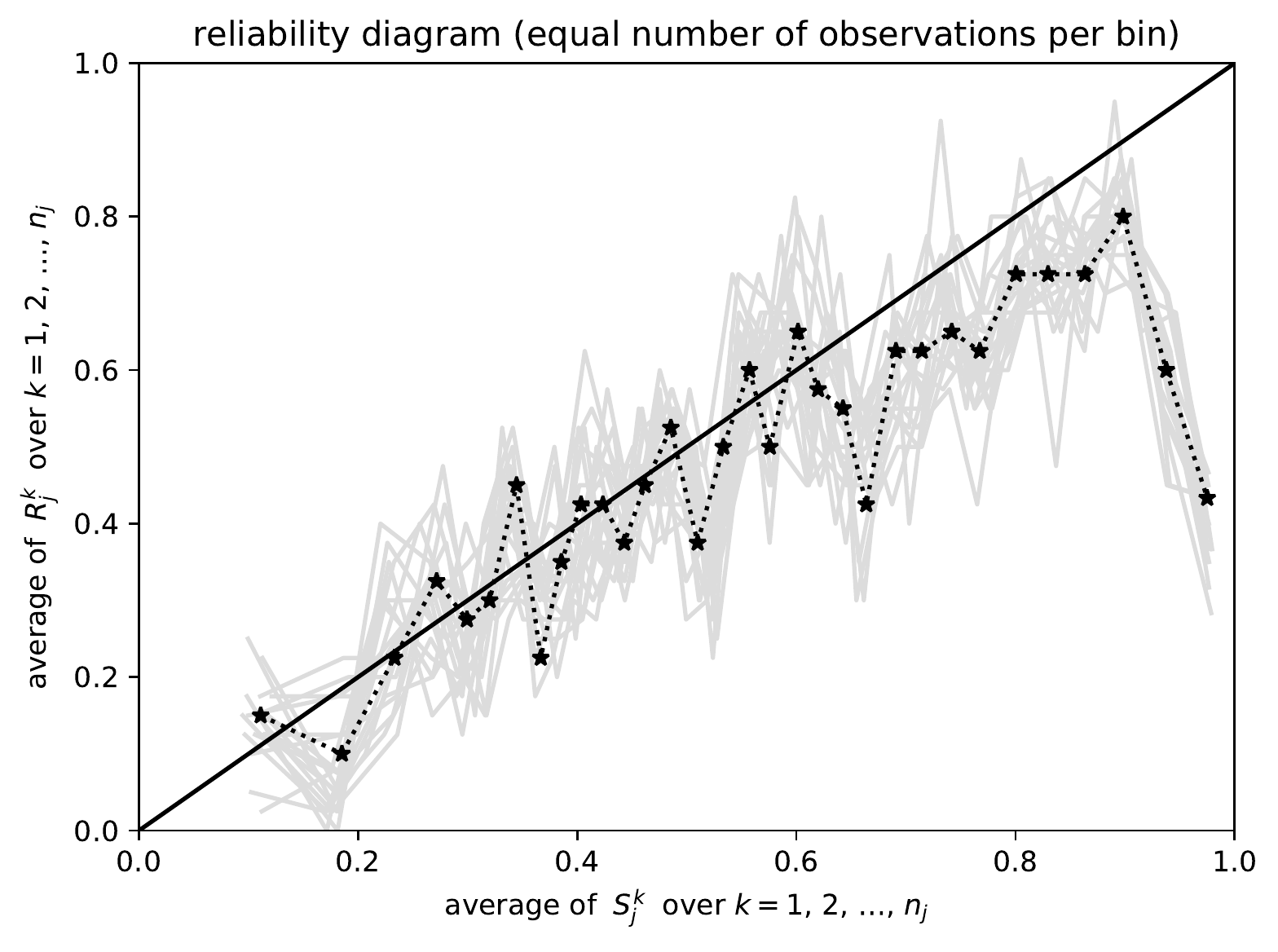}}
\end{center}
\caption{Reliability diagrams for the night snake ({\it Hypsiglena torquata}),
         with an equal number of observations per bin.
         There are $m = 8$ bins in the upper plot
         and $m = 32$ in the lower plot.}
\label{night-snakesamp}
\end{figure}

\begin{figure}
\begin{center}
\parbox{\imsize}{\includegraphics[width=\imsize]
       {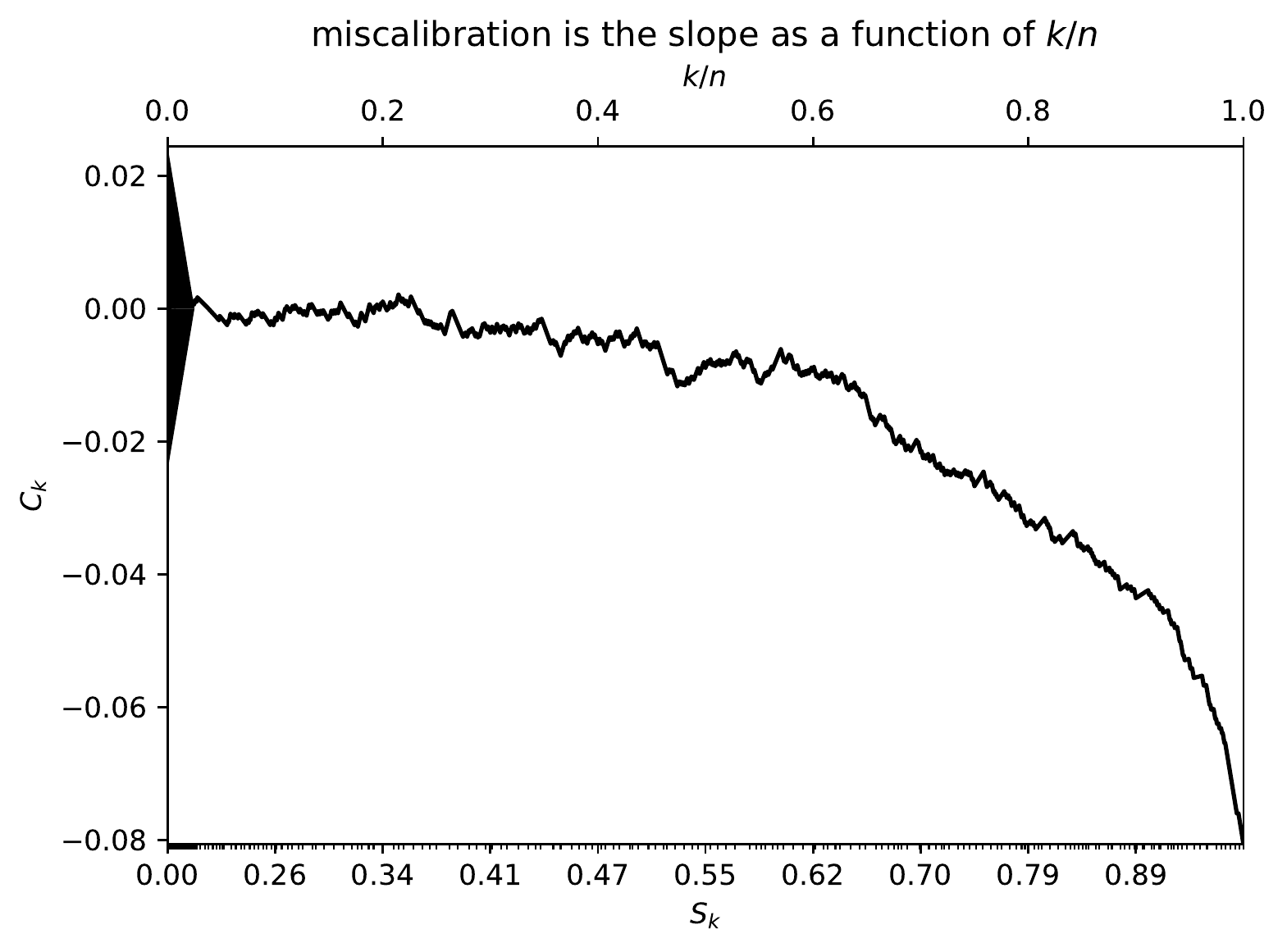}}
\end{center}
\caption{Cumulative plot for the night snake ({\it Hypsiglena torquata}),
         with sample size $n =$ 1,300.
         The ECCE-MAD is $0.08059 / \sigma_n = 6.607$,
         and the ECCE-R is $0.08270 / \sigma_n = 6.780$;
         the associated asymptotic P-values are 7.8E--11 and 4.8E--11,
         respectively.
}
\label{night-snakecum}
\end{figure}

\begin{figure}
\begin{center}
\parbox{\imsizes}{\includegraphics[width=\imsizes]
{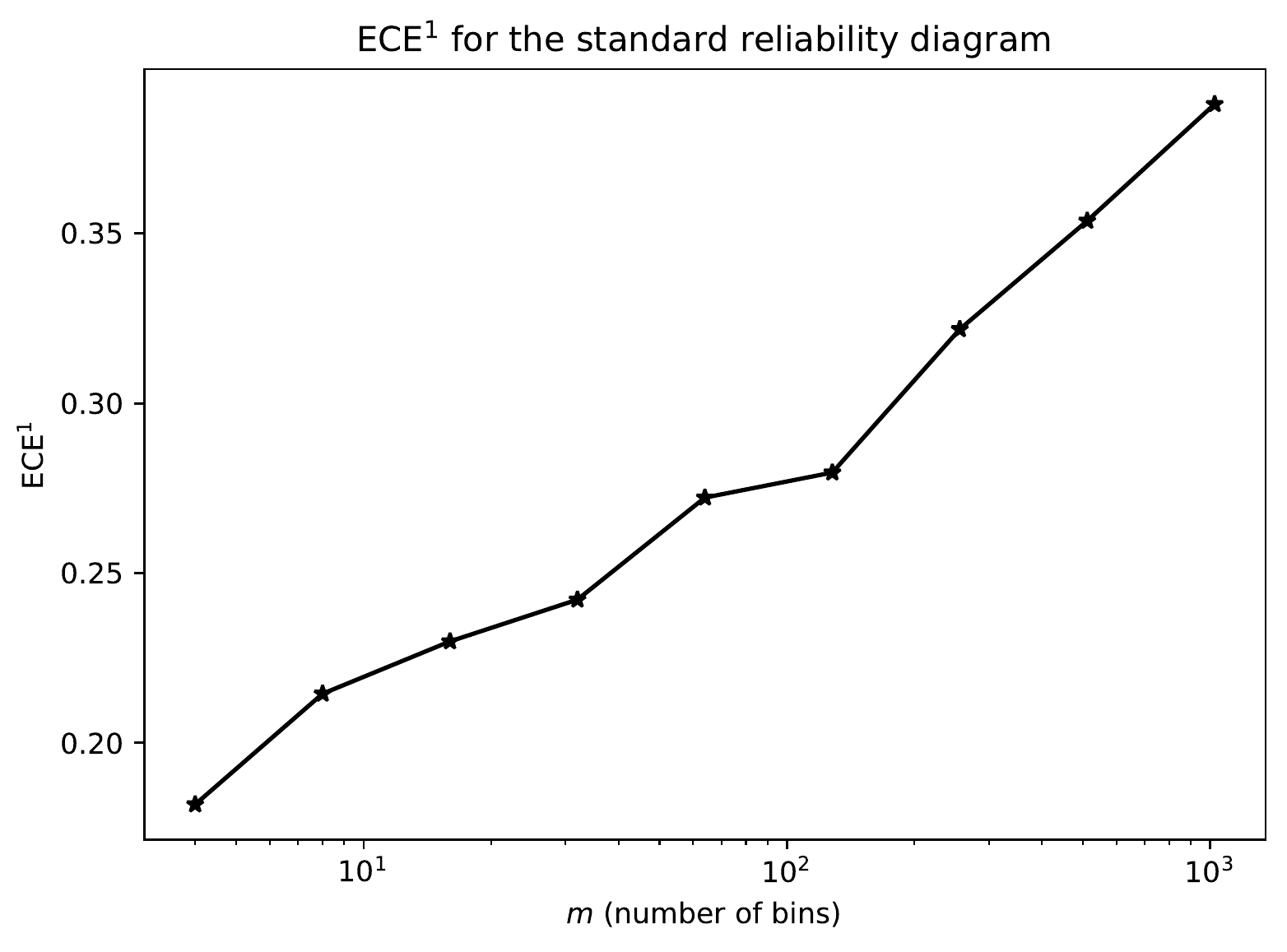}}
\hfil
\parbox{\imsizes}{\includegraphics[width=\imsizes]
{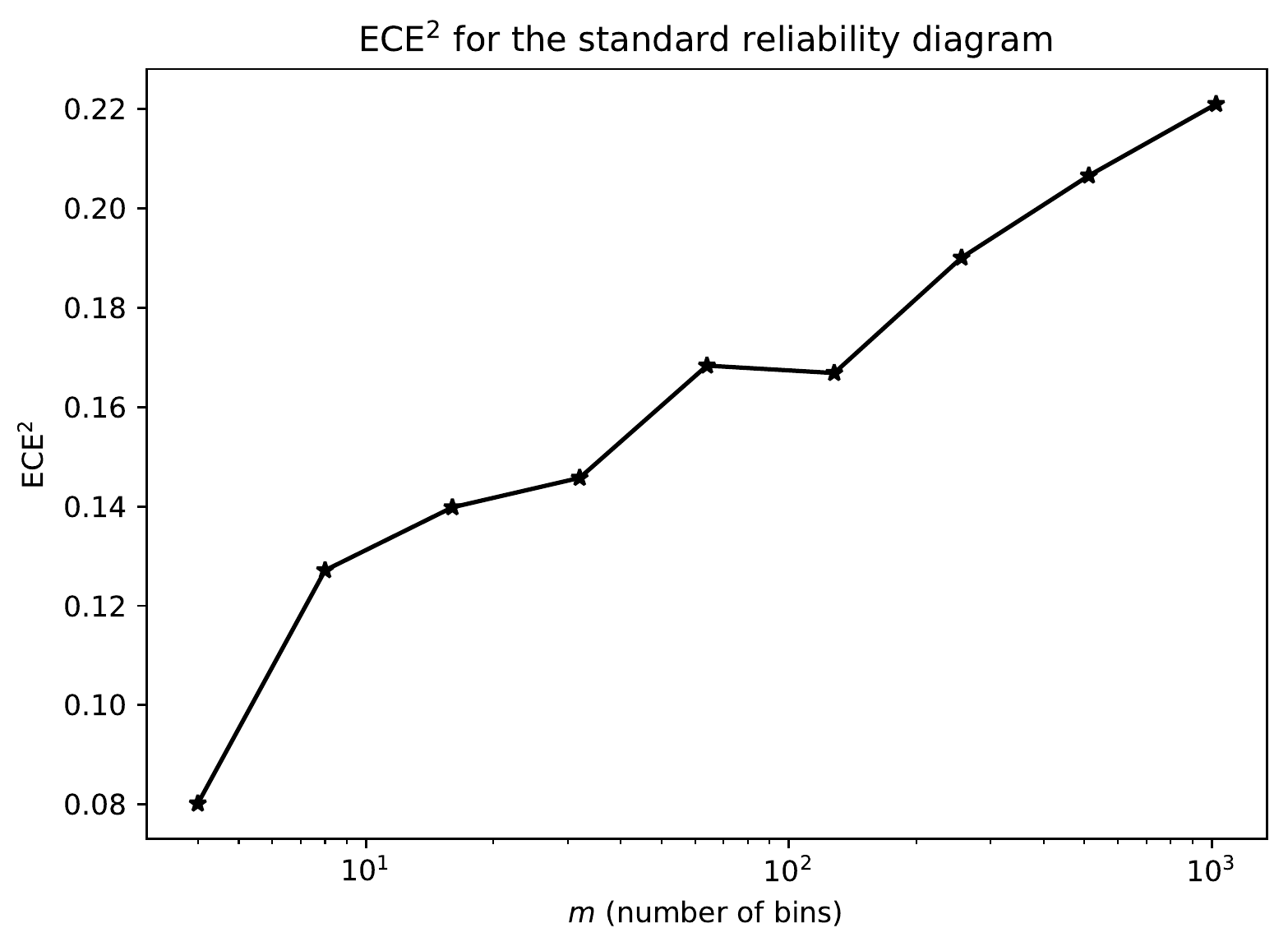}}

\parbox{\imsizes}{\includegraphics[width=\imsizes]
{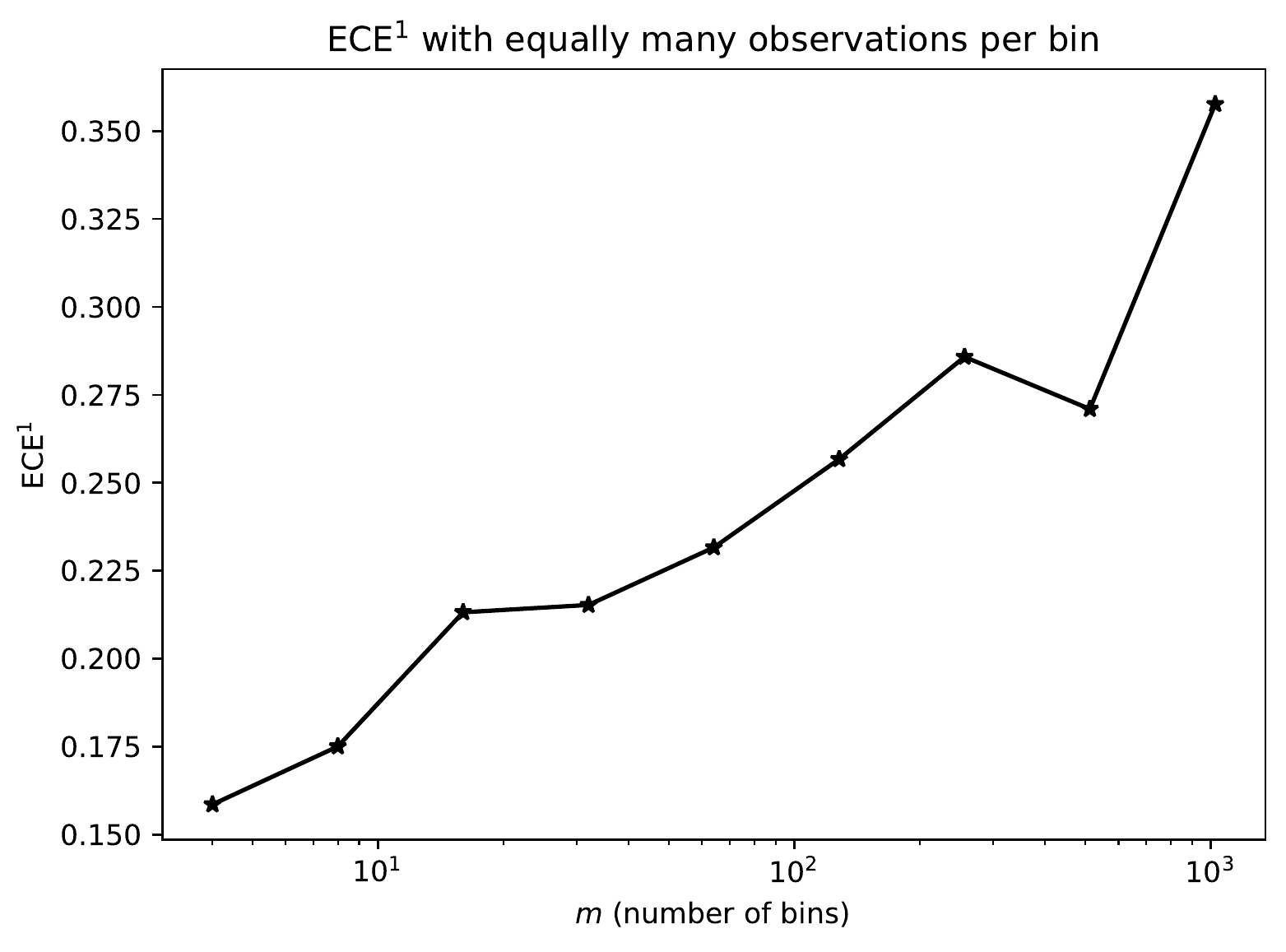}}
\hfil
\parbox{\imsizes}{\includegraphics[width=\imsizes]
{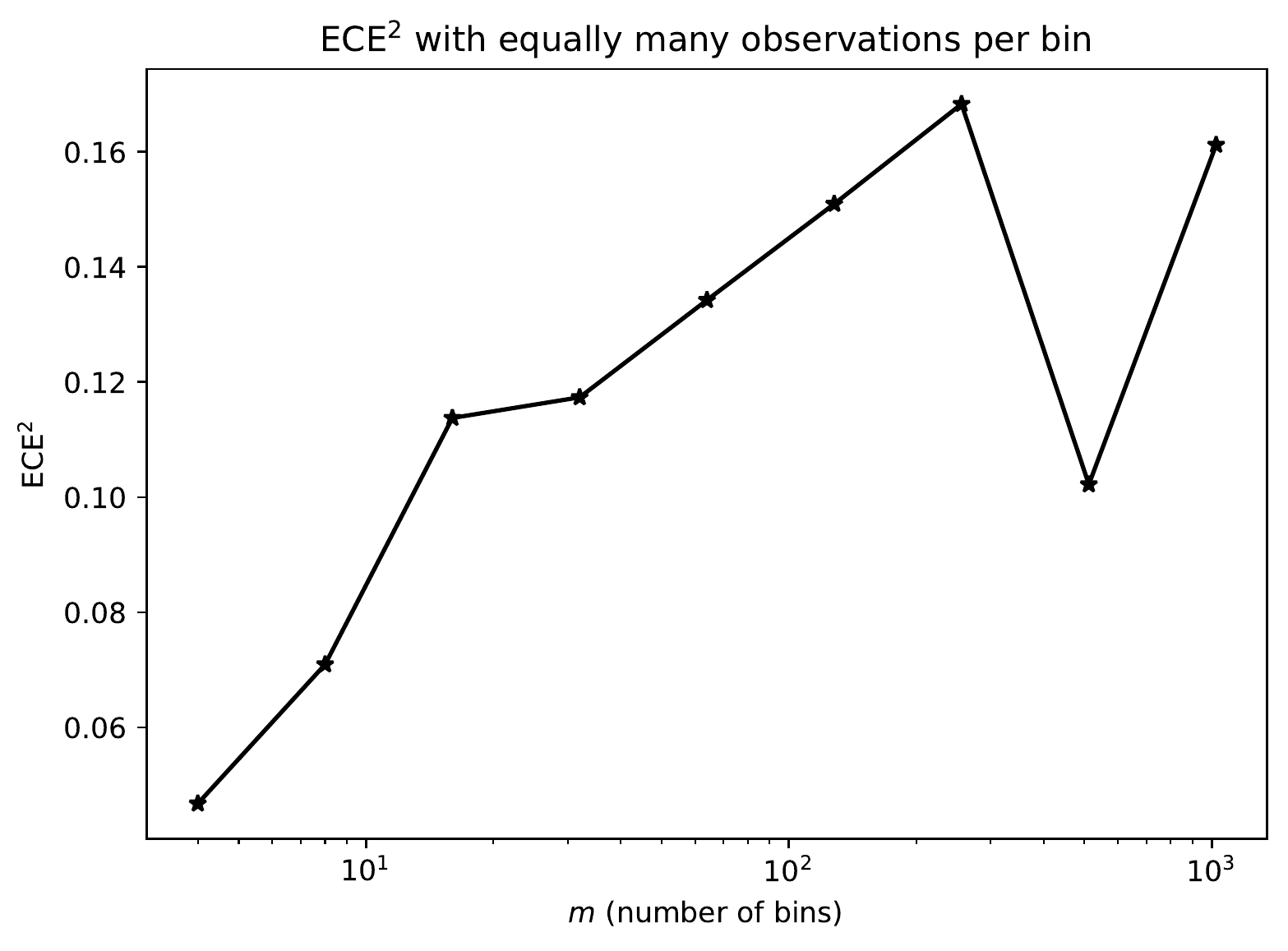}}
\end{center}
\caption{Empirical calibration errors for the sidewinder or horned rattlesnake
         ({\it Crotalus cerastes}), with sample size $n =$ 1,300.}
\label{sidewinderece}
\end{figure}

\begin{figure}
\begin{center}
\parbox{\imsize}{\includegraphics[width=\imsize]
{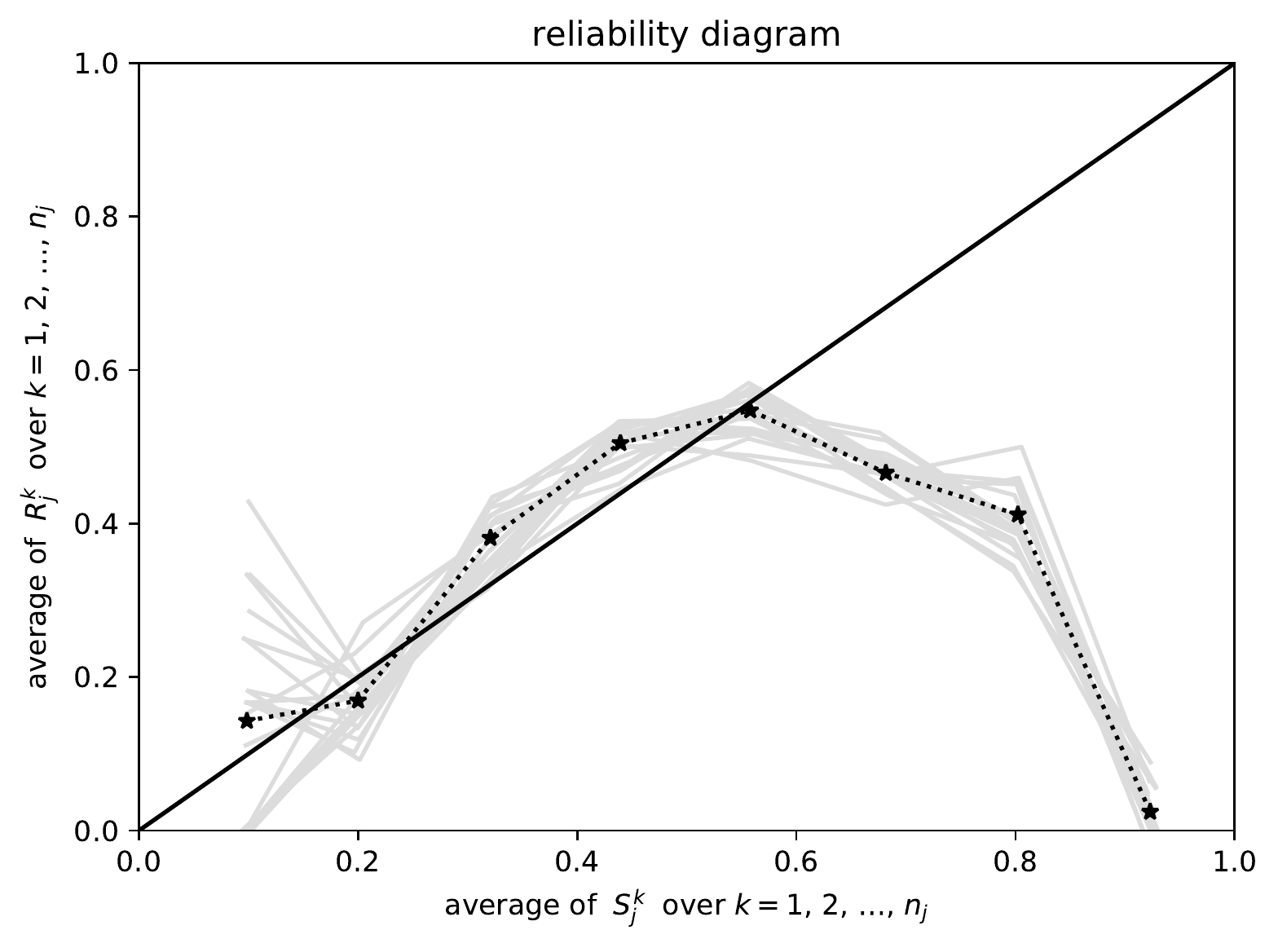}}

\parbox{\imsize}{\includegraphics[width=\imsize]
{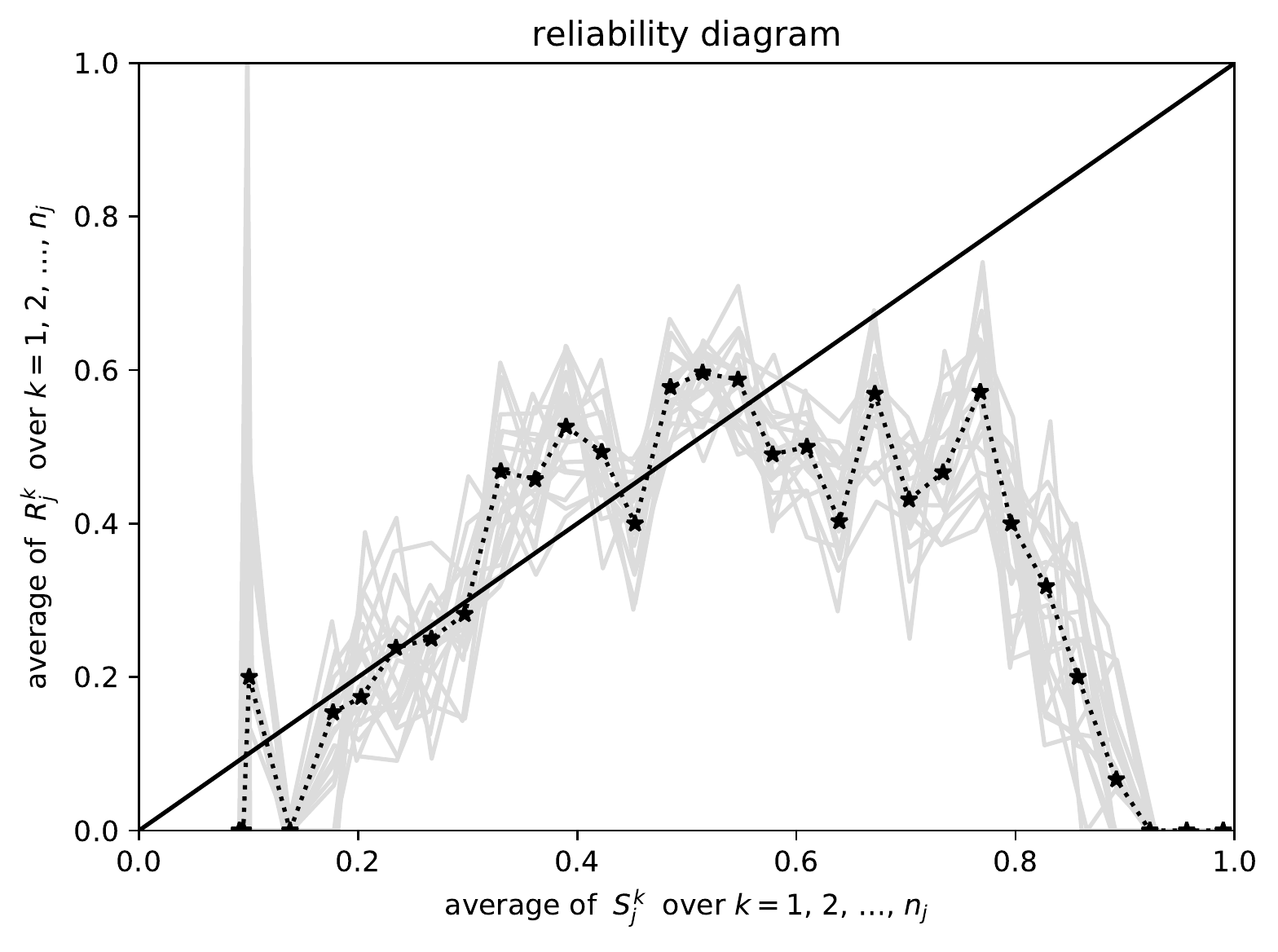}}
\end{center}
\caption{Reliability diagrams for the sidewinder or horned rattlesnake
         ({\it Crotalus cerastes}), with the bins roughly equispaced.
         There are $m = 8$ bins in the upper plot
         and $m = 32$ in the lower plot.}
\label{sidewinderprob}
\end{figure}

\begin{figure}
\begin{center}
\parbox{\imsize}{\includegraphics[width=\imsize]
{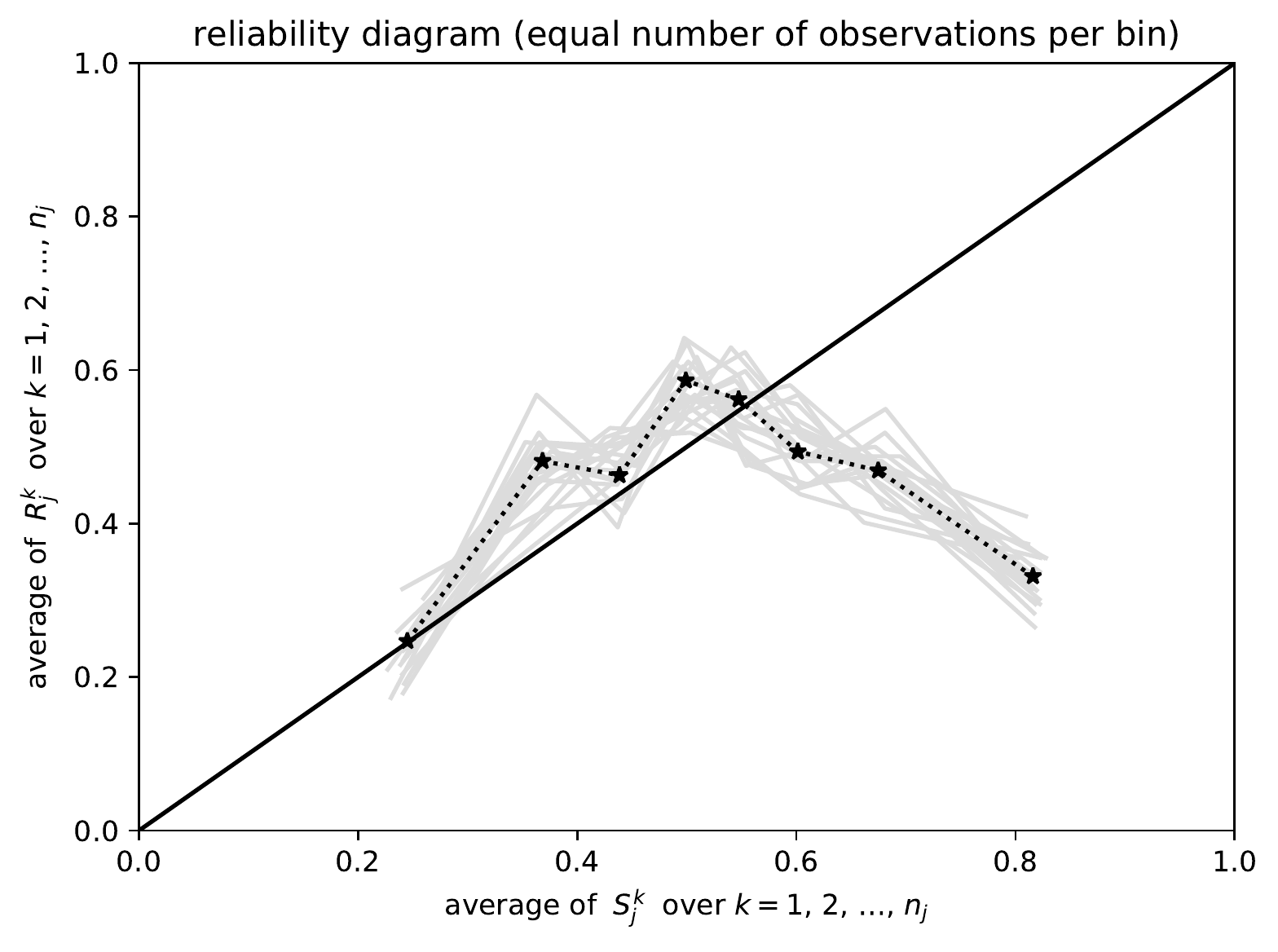}}

\parbox{\imsize}{\includegraphics[width=\imsize]
{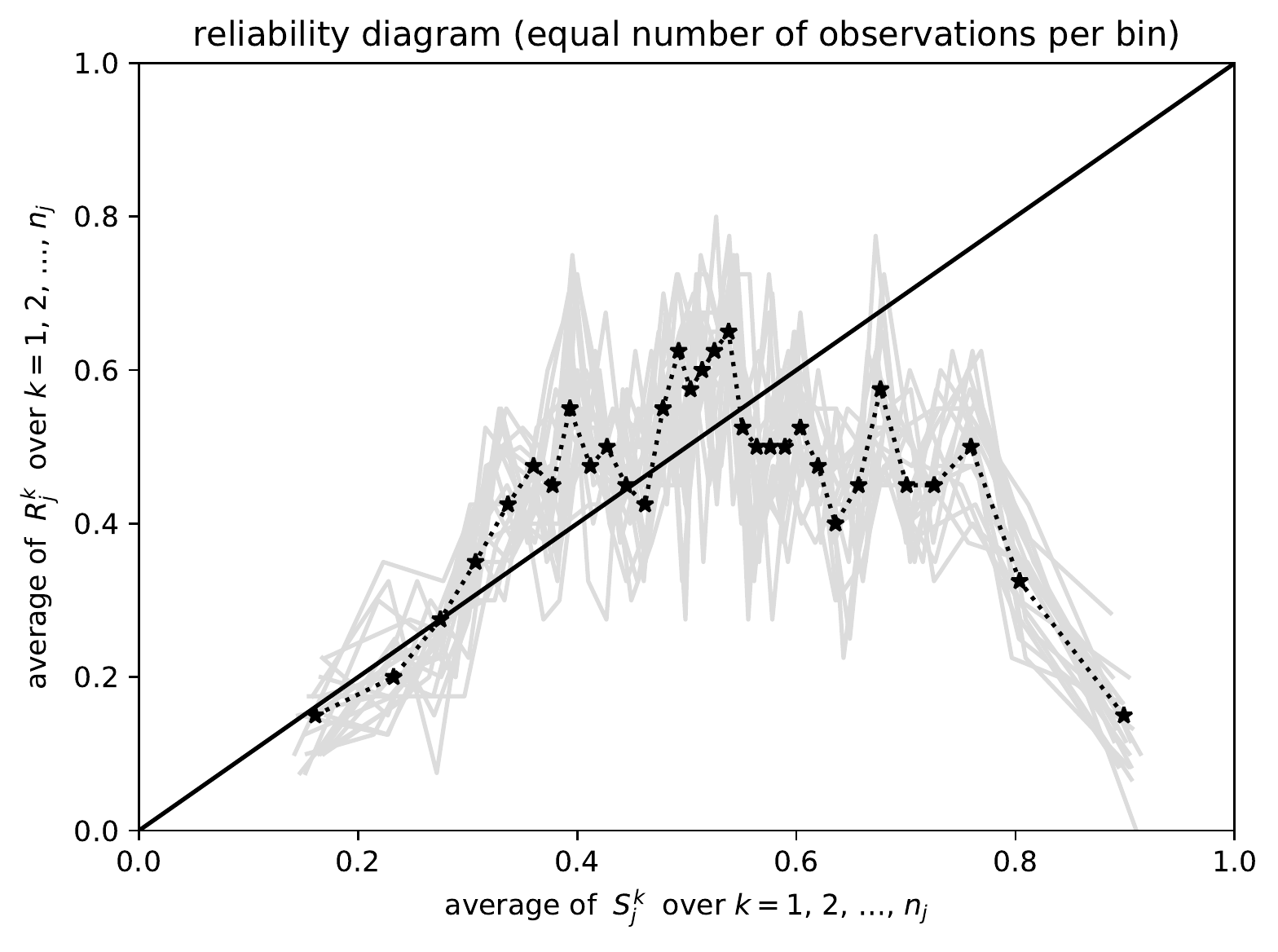}}
\end{center}
\caption{Reliability diagrams for the sidewinder or horned rattlesnake
         ({\it Crotalus cerastes}),
         with an equal number of observations per bin.
         There are $m = 8$ bins in the upper plot
         and $m = 32$ in the lower plot.}
\label{sidewindersamp}
\end{figure}

\begin{figure}
\begin{center}
\parbox{\imsize}{\includegraphics[width=\imsize]
{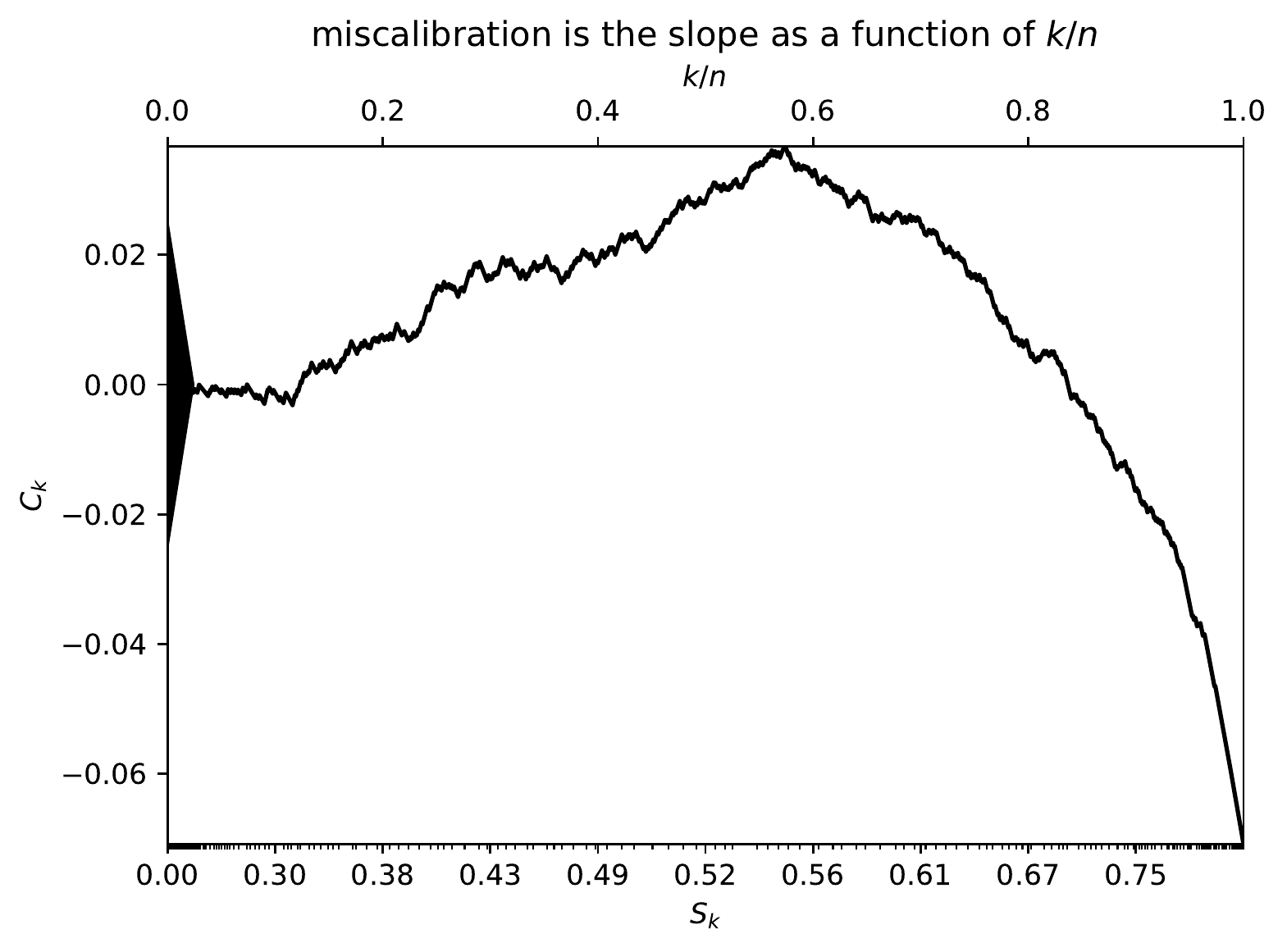}}
\end{center}
\caption{Cumulative plot for the sidewinder or horned rattlesnake
         ({\it Crotalus cerastes}), with sample size $n =$ 1,300.
         The ECCE-MAD is $0.07081 / \sigma_n = 5.446$,
         and the ECCE-R is $0.1075 / \sigma_n = 8.267$;
         the associated asymptotic P-values are 1.0E--7 for the ECCE-MAD
         and zero to double-precision accuracy for the ECCE-R.
}
\label{sidewindercum}
\end{figure}

\begin{figure}
\begin{center}
\parbox{\imsizes}{\includegraphics[width=\imsizes]
{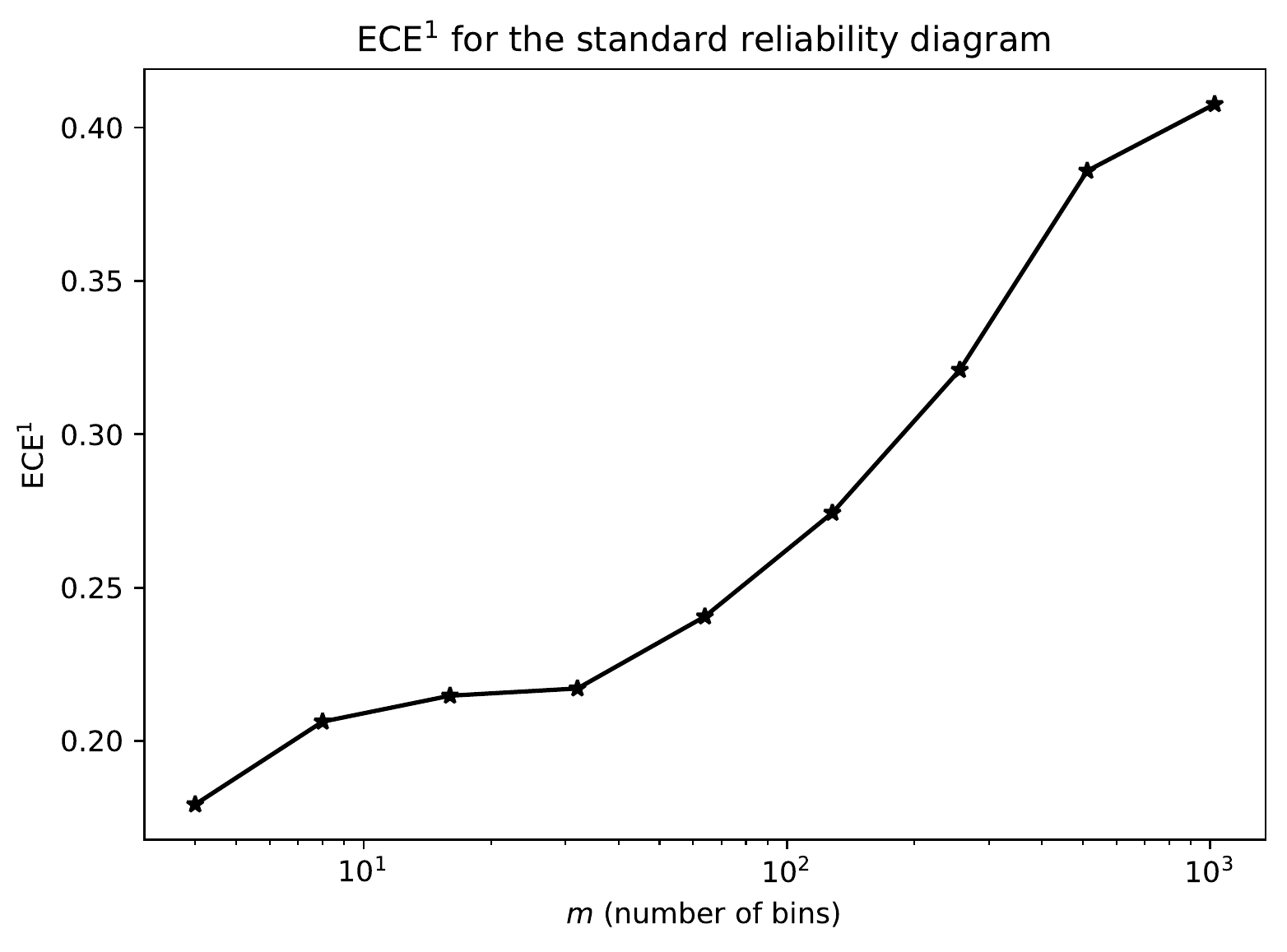}}
\hfil
\parbox{\imsizes}{\includegraphics[width=\imsizes]
{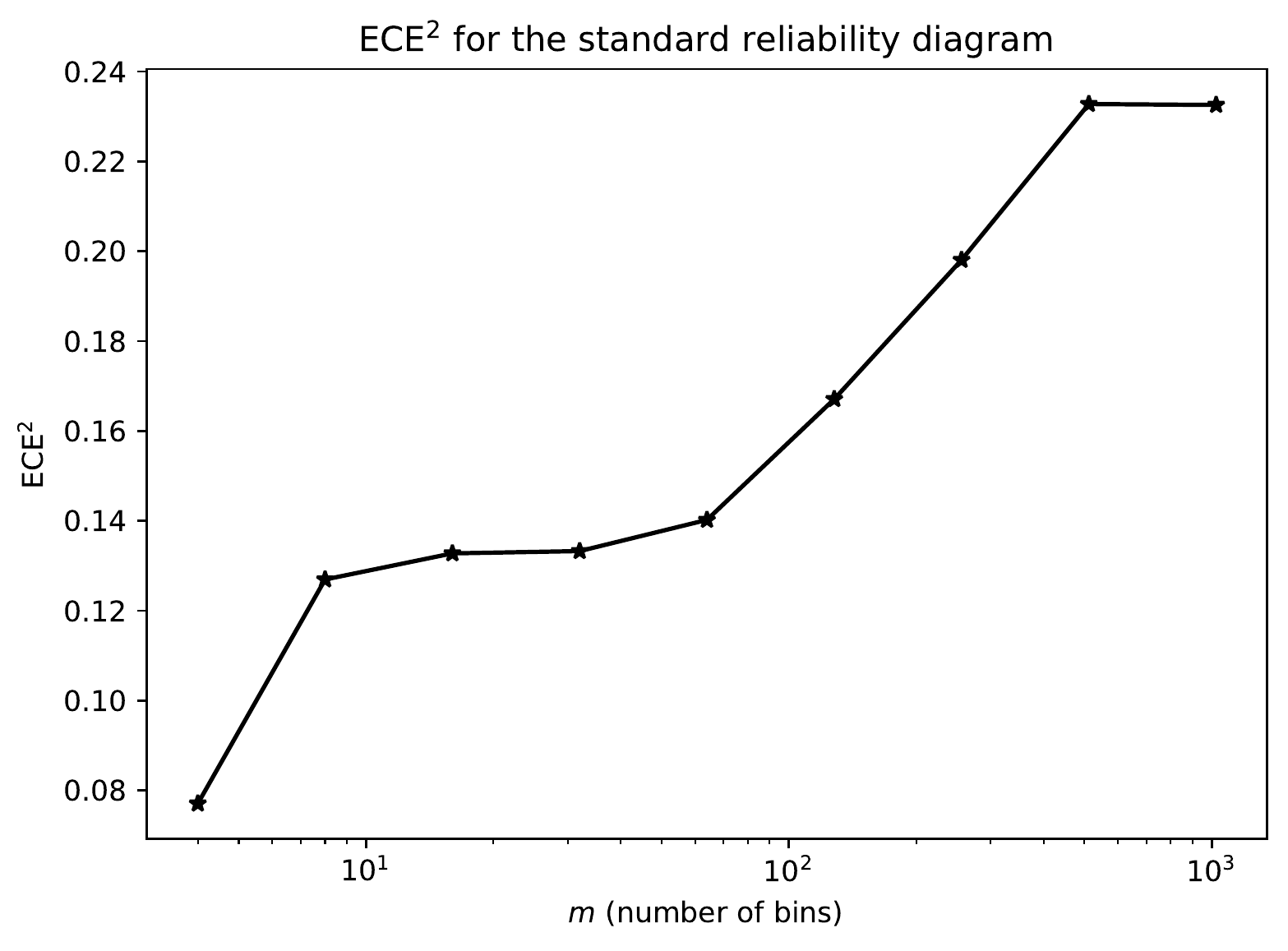}}

\parbox{\imsizes}{\includegraphics[width=\imsizes]
{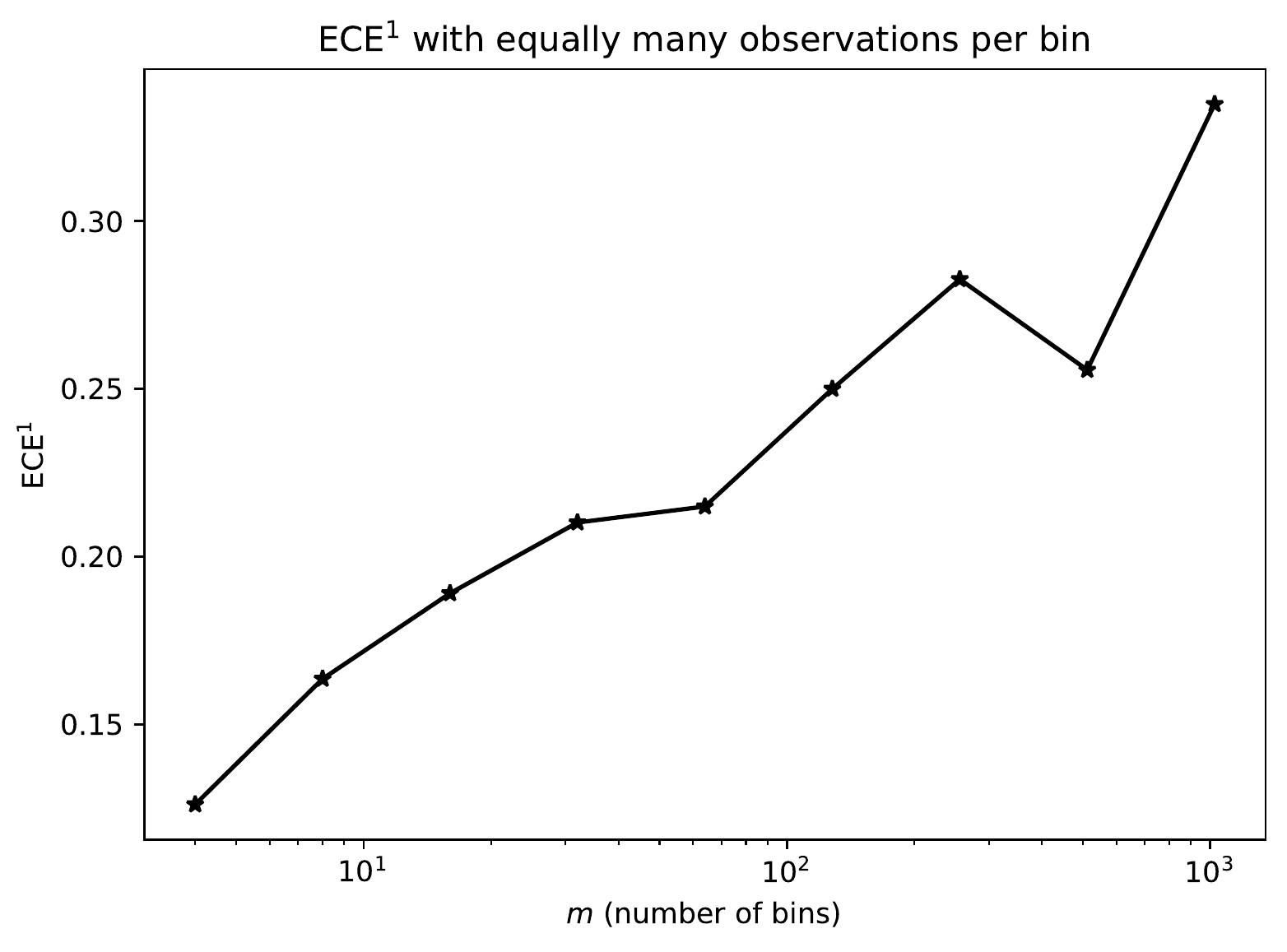}}
\hfil
\parbox{\imsizes}{\includegraphics[width=\imsizes]
{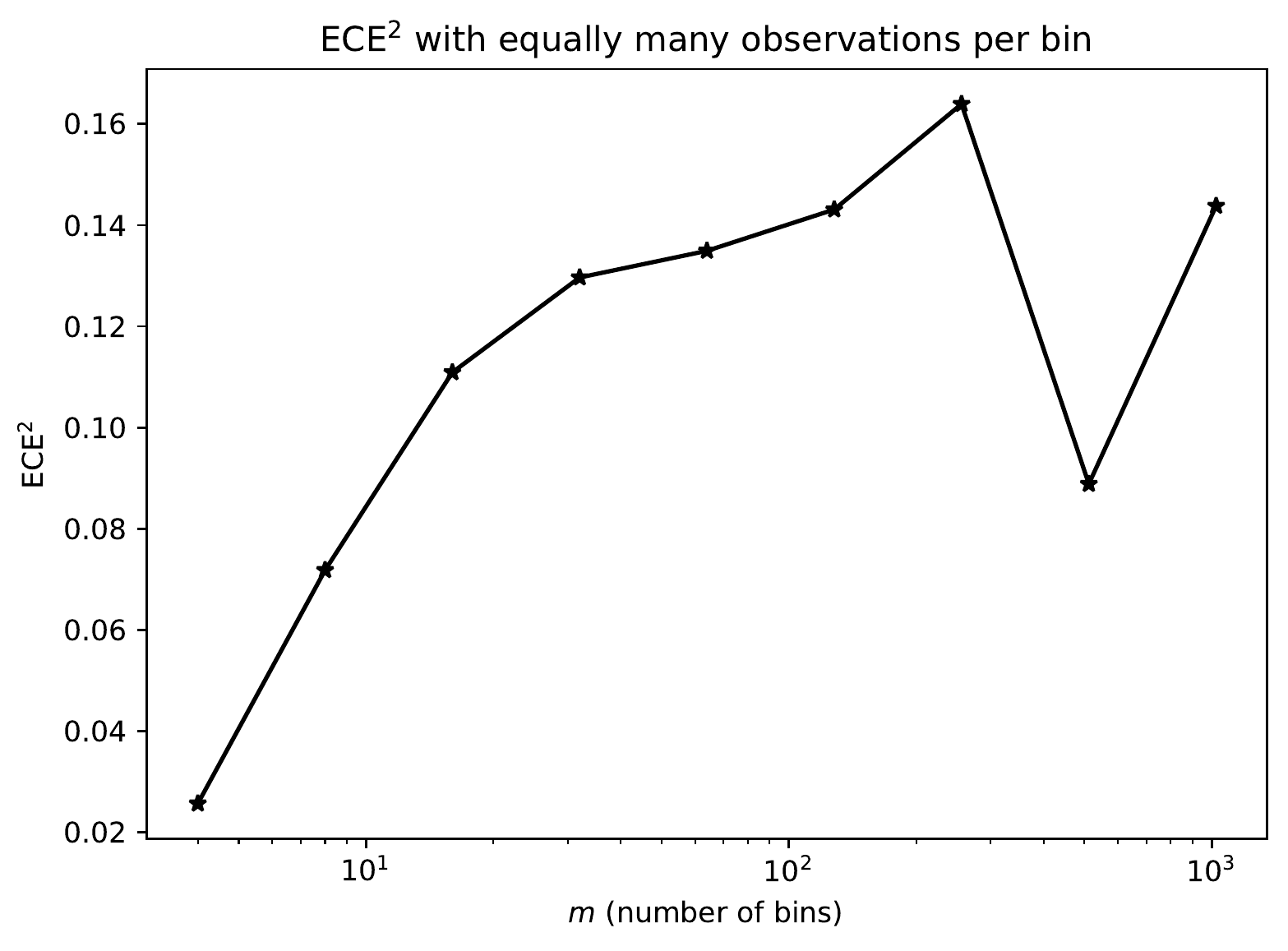}}
\end{center}
\caption{Empirical calibration errors for the Eskimo dog or husky,
         with sample size $n =$ 1,300.}
\label{eskimo-dogece}
\end{figure}

\begin{figure}
\begin{center}
\parbox{\imsize}{\includegraphics[width=\imsize]
{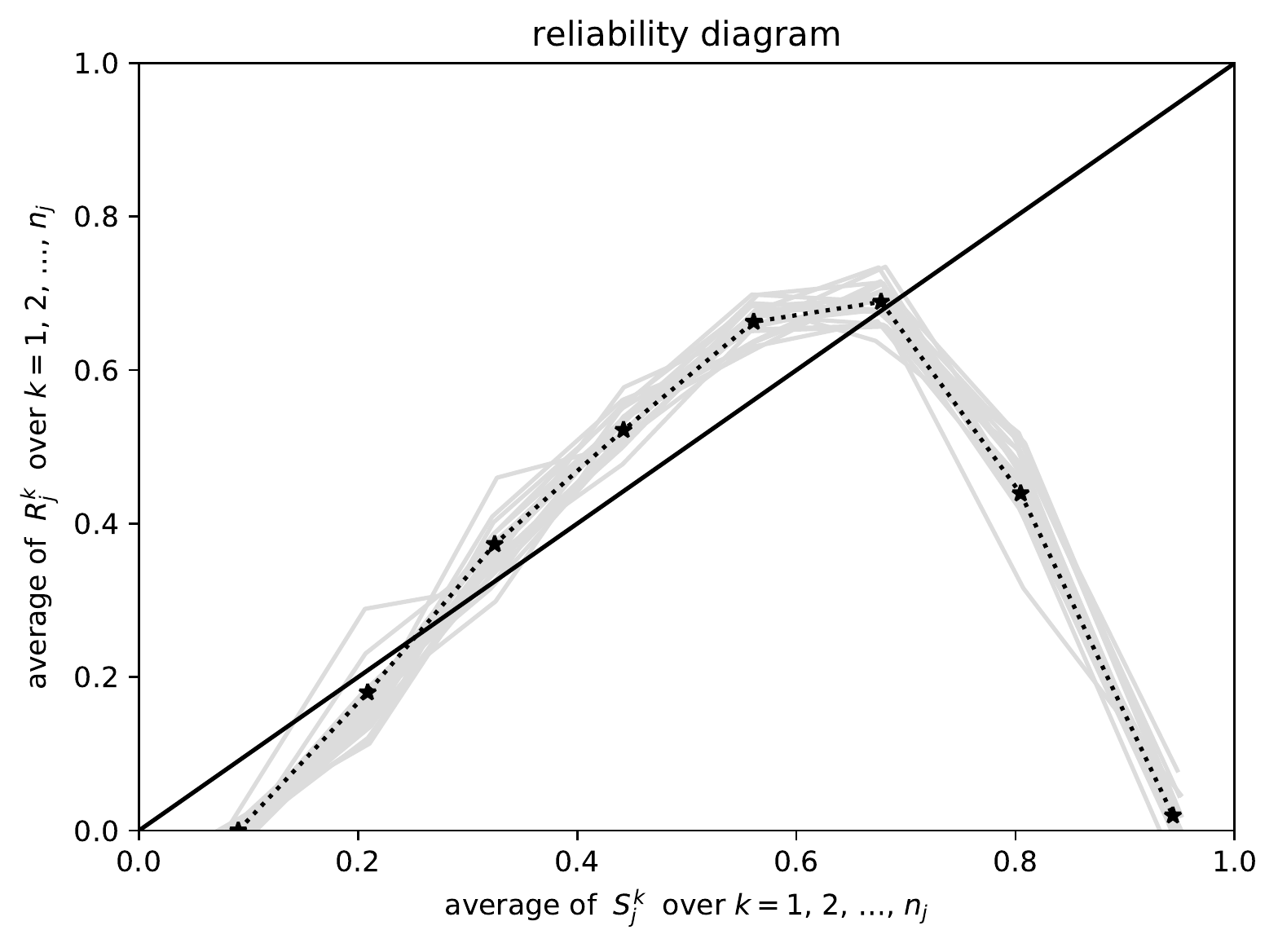}}

\parbox{\imsize}{\includegraphics[width=\imsize]
{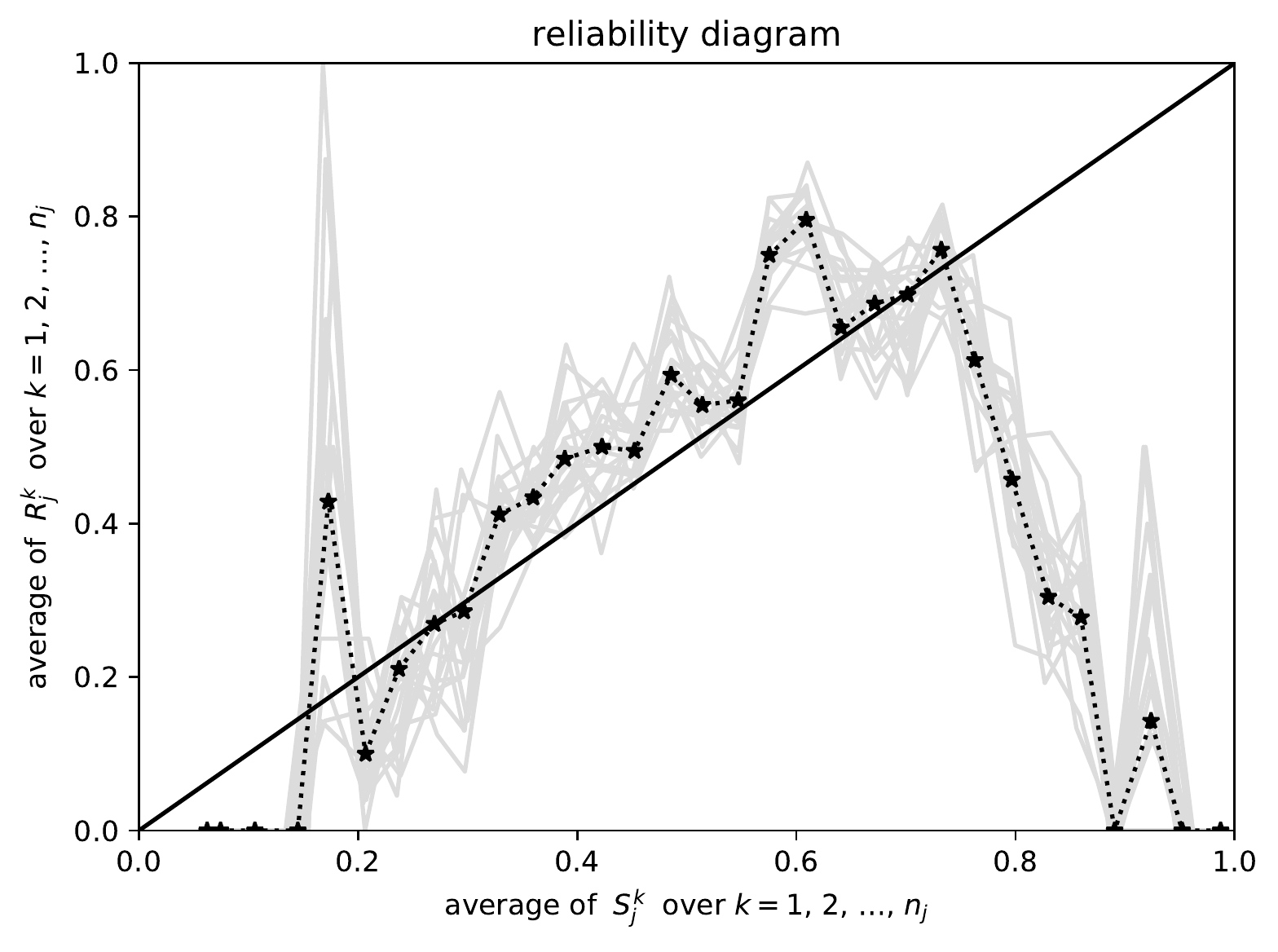}}
\end{center}
\caption{Reliability diagrams for the Eskimo dog or husky,
         with the bins roughly equispaced.
         There are $m = 8$ bins in the upper plot
         and $m = 32$ in the lower plot.}
\label{eskimo-dogprob}
\end{figure}

\begin{figure}
\begin{center}
\parbox{\imsize}{\includegraphics[width=\imsize]
{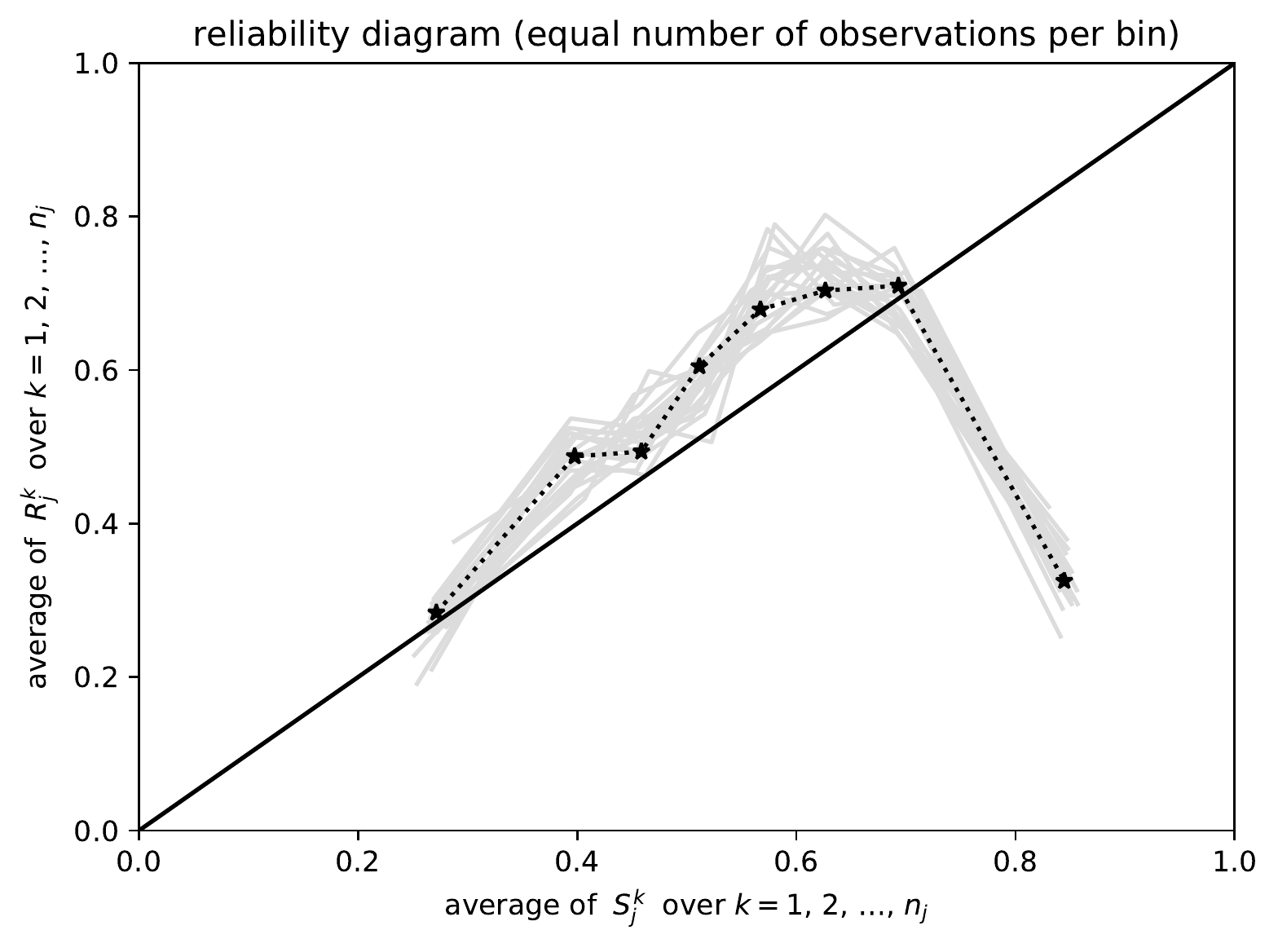}}

\parbox{\imsize}{\includegraphics[width=\imsize]
{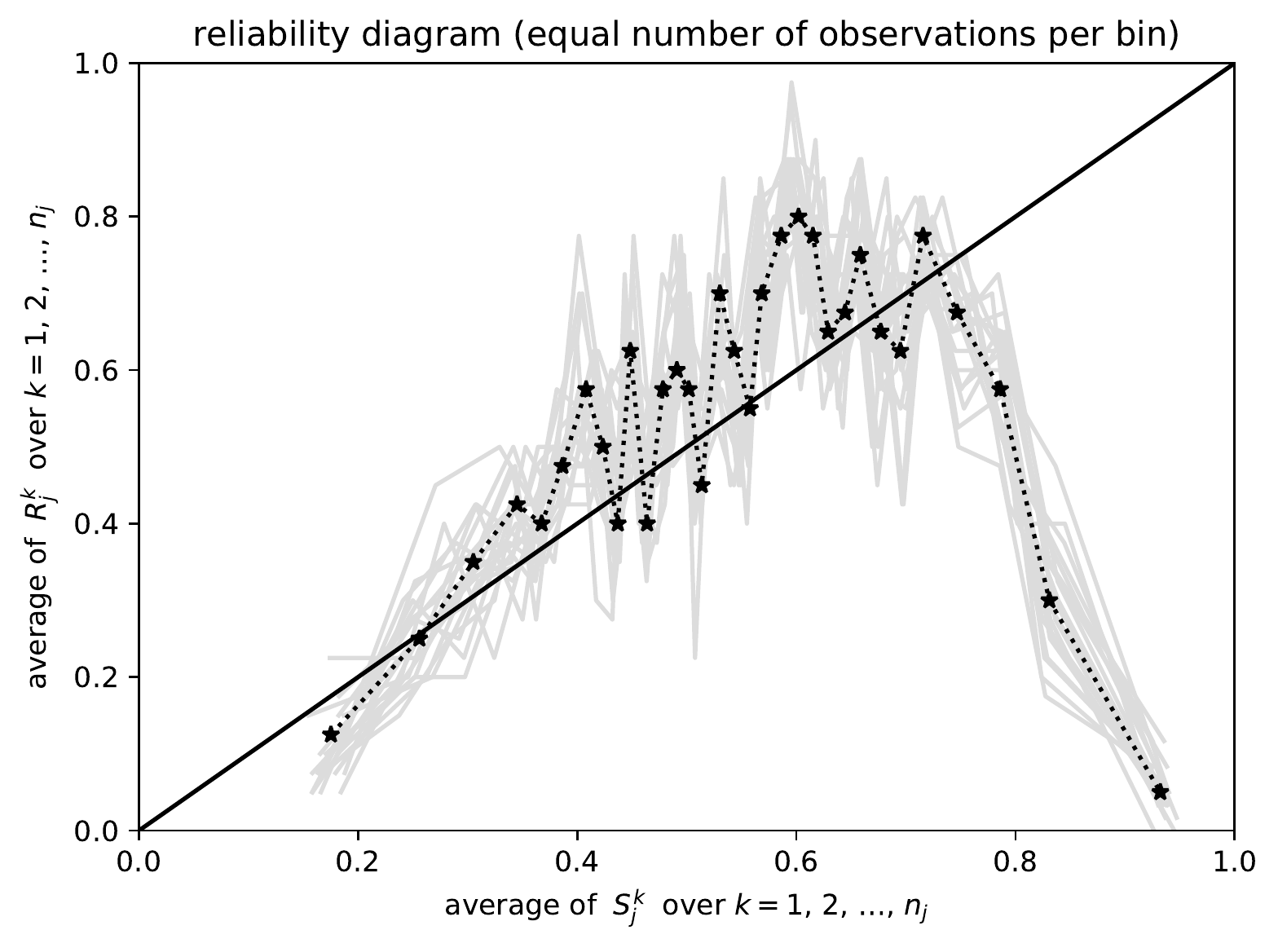}}
\end{center}
\caption{Reliability diagrams for the Eskimo dog or husky,
         with an equal number of observations per bin.
         There are $m = 8$ in the upper plot
         and $m = 32$ in the lower plot.}
\label{eskimo-dogsamp}
\end{figure}

\begin{figure}
\begin{center}
\parbox{\imsize}{\includegraphics[width=\imsize]
{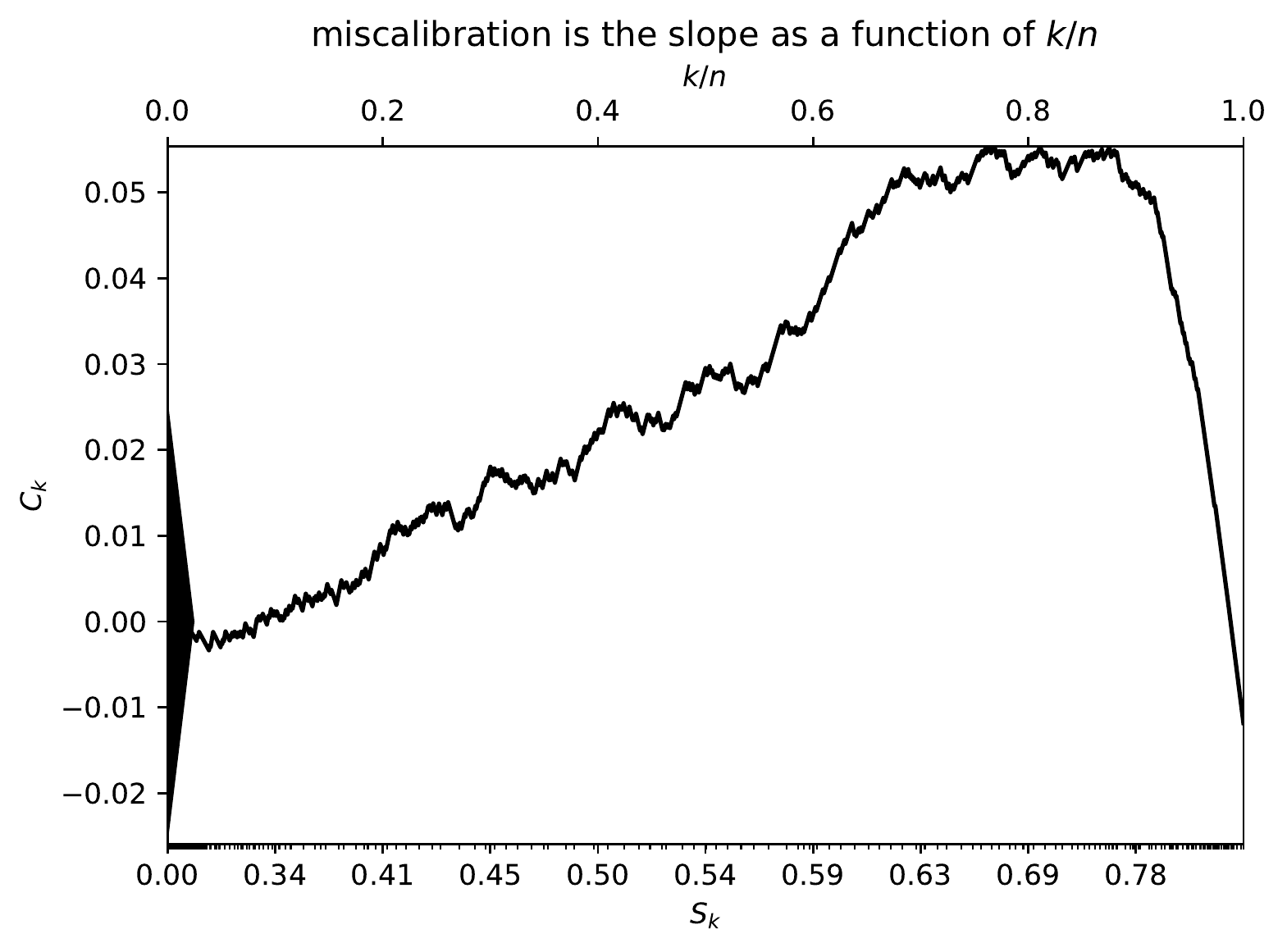}}
\end{center}
\caption{Cumulative plot for the Eskimo dog or husky,
         with sample size $n =$ 1,300.
         The ECCE-MAD is $0.05534 / \sigma_n = 4.274$,
         and the ECCE-R is $0.06715 / \sigma_n = 5.186$;
         the associated asymptotic P-values are 3.8E--5 and 8.6E--7,
         respectively.
}
\label{eskimo-dogcum}
\end{figure}

\begin{figure}
\begin{center}
\parbox{\imsizes}{\includegraphics[width=\imsizes]
{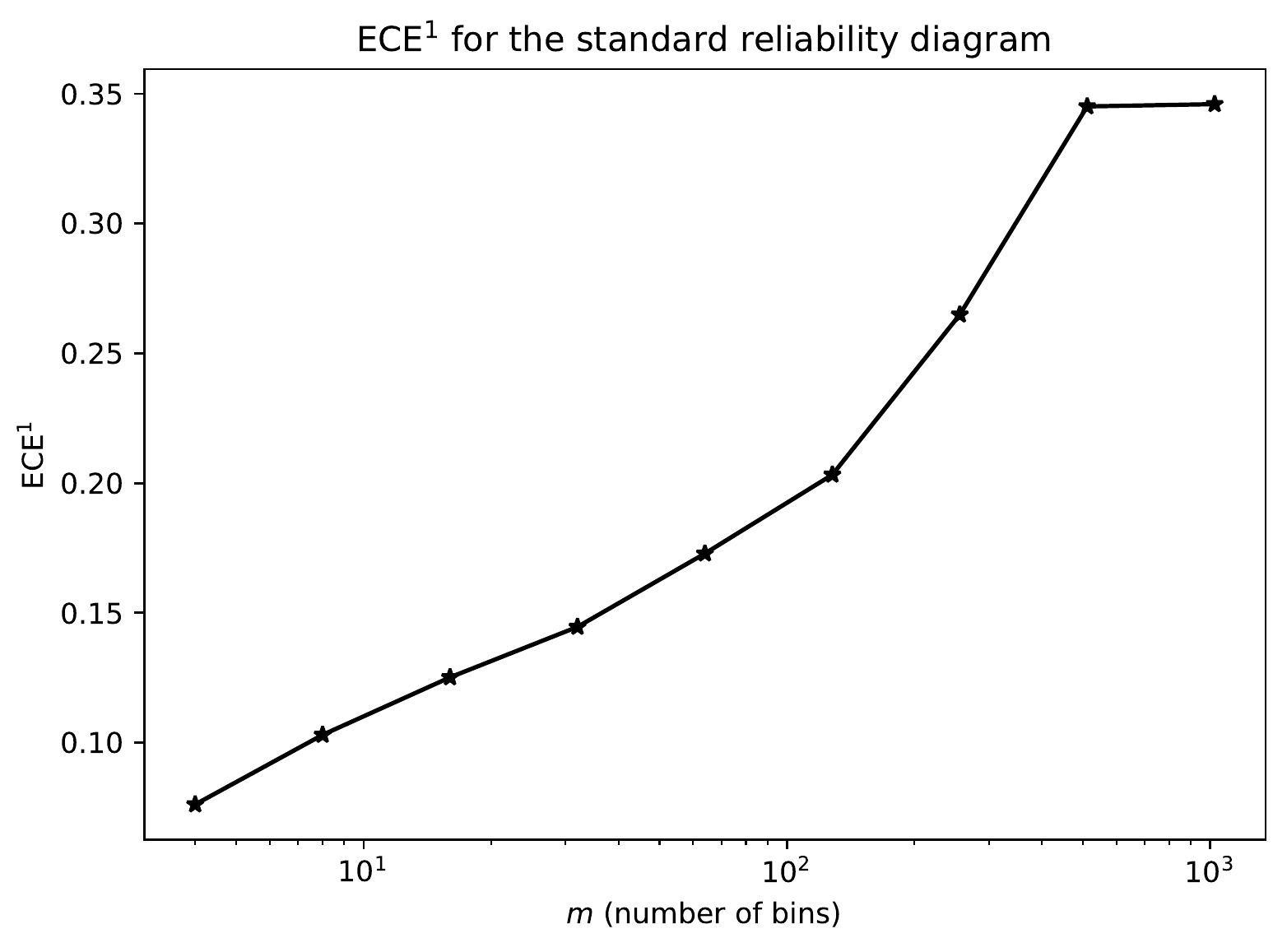}}
\hfil
\parbox{\imsizes}{\includegraphics[width=\imsizes]
{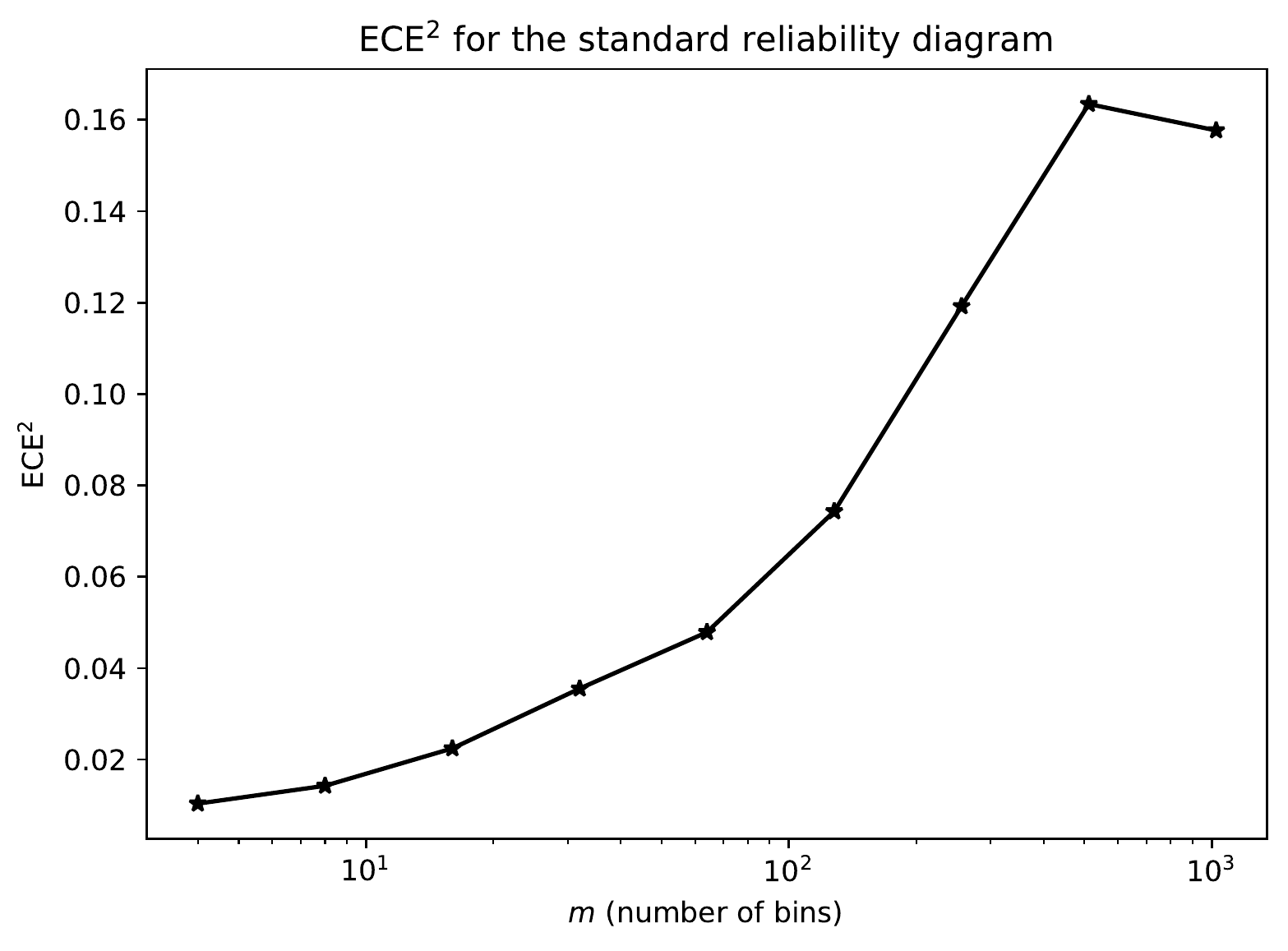}}

\parbox{\imsizes}{\includegraphics[width=\imsizes]
{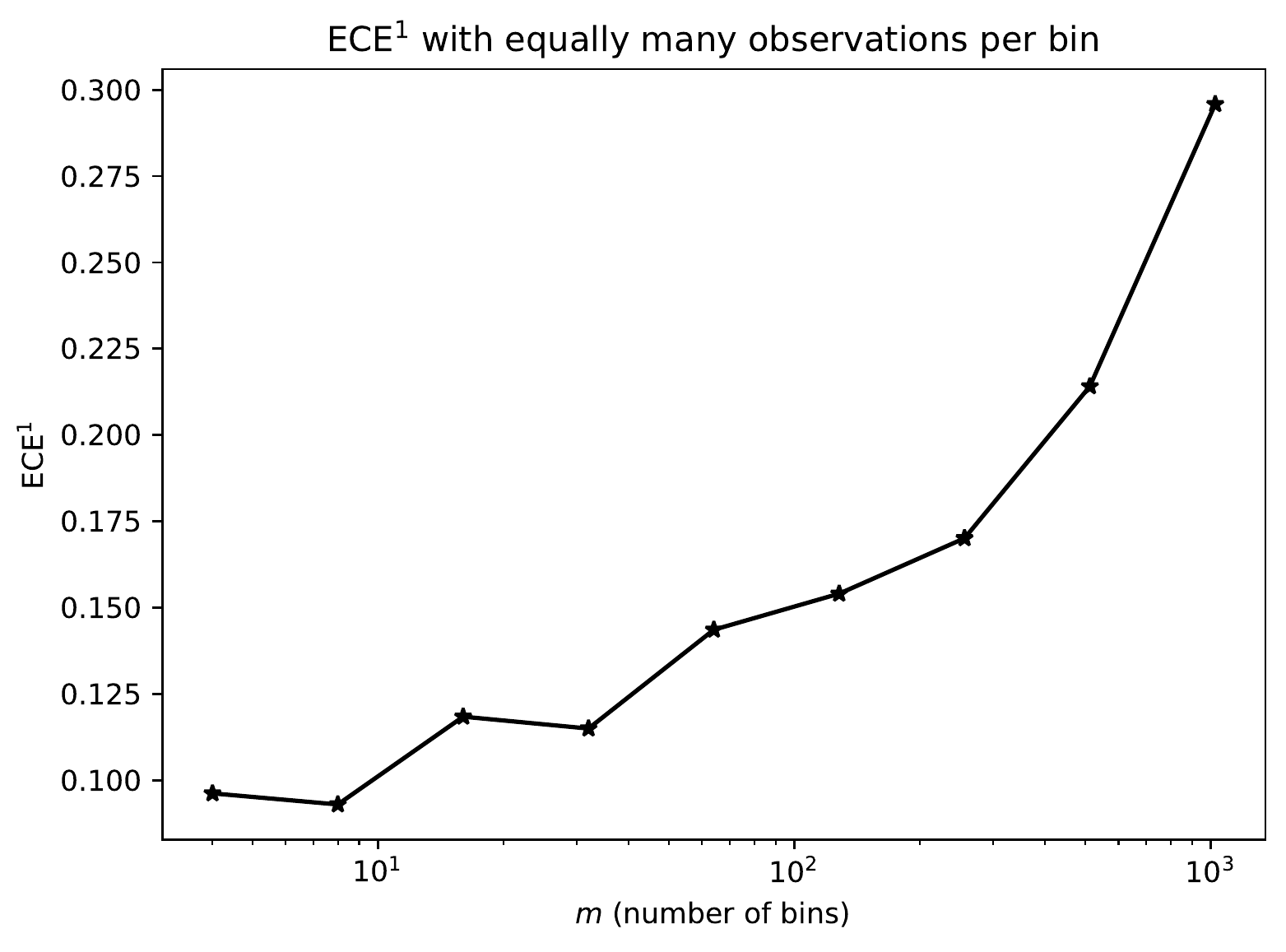}}
\hfil
\parbox{\imsizes}{\includegraphics[width=\imsizes]
{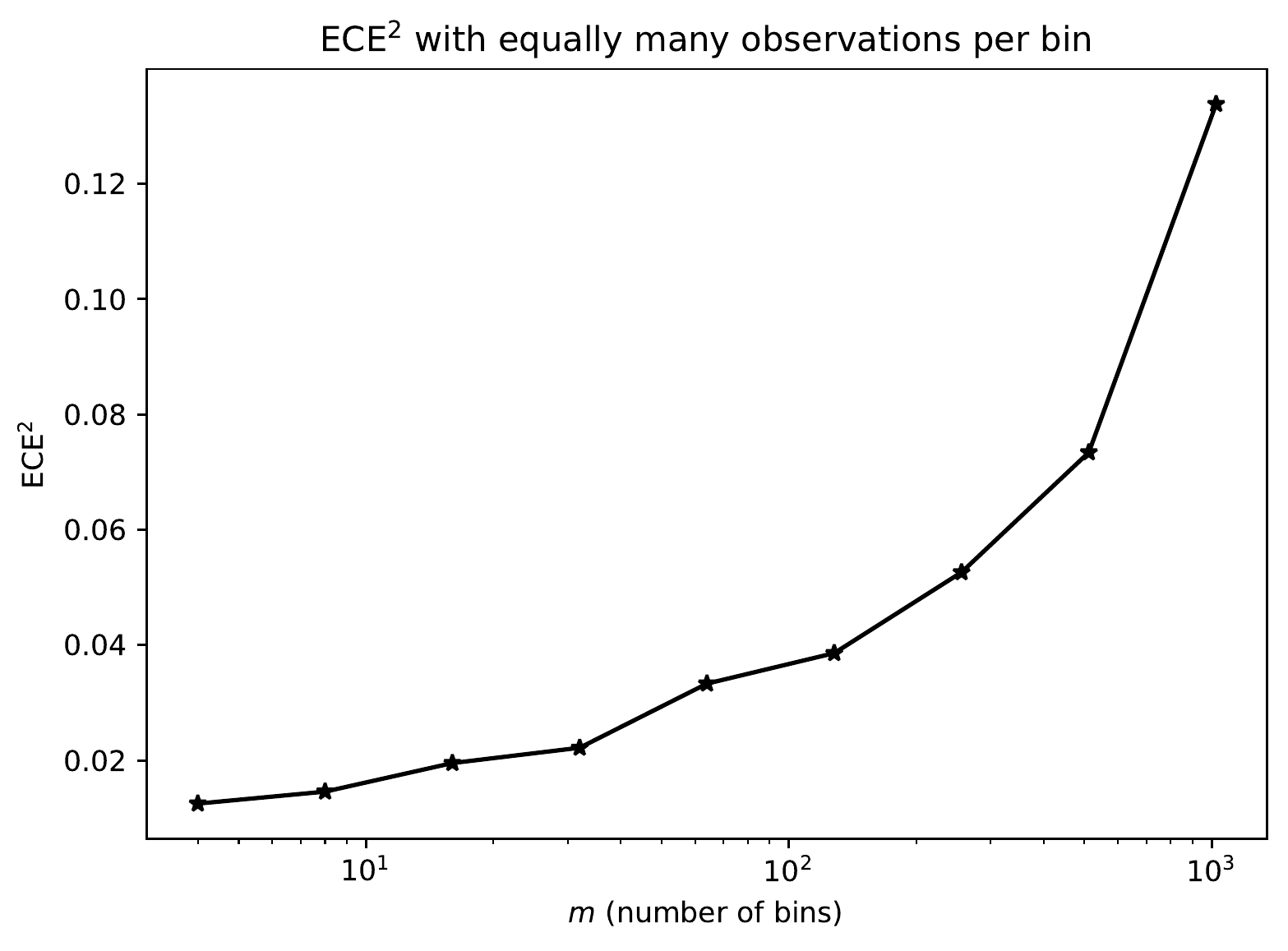}}
\end{center}
\caption{Empirical calibration errors for the wild boar ({\it Sus scrofa}),
         with sample size $n =$ 1,300.}
\label{wild-boarece}
\end{figure}

\begin{figure}
\begin{center}
\parbox{\imsize}{\includegraphics[width=\imsize]
{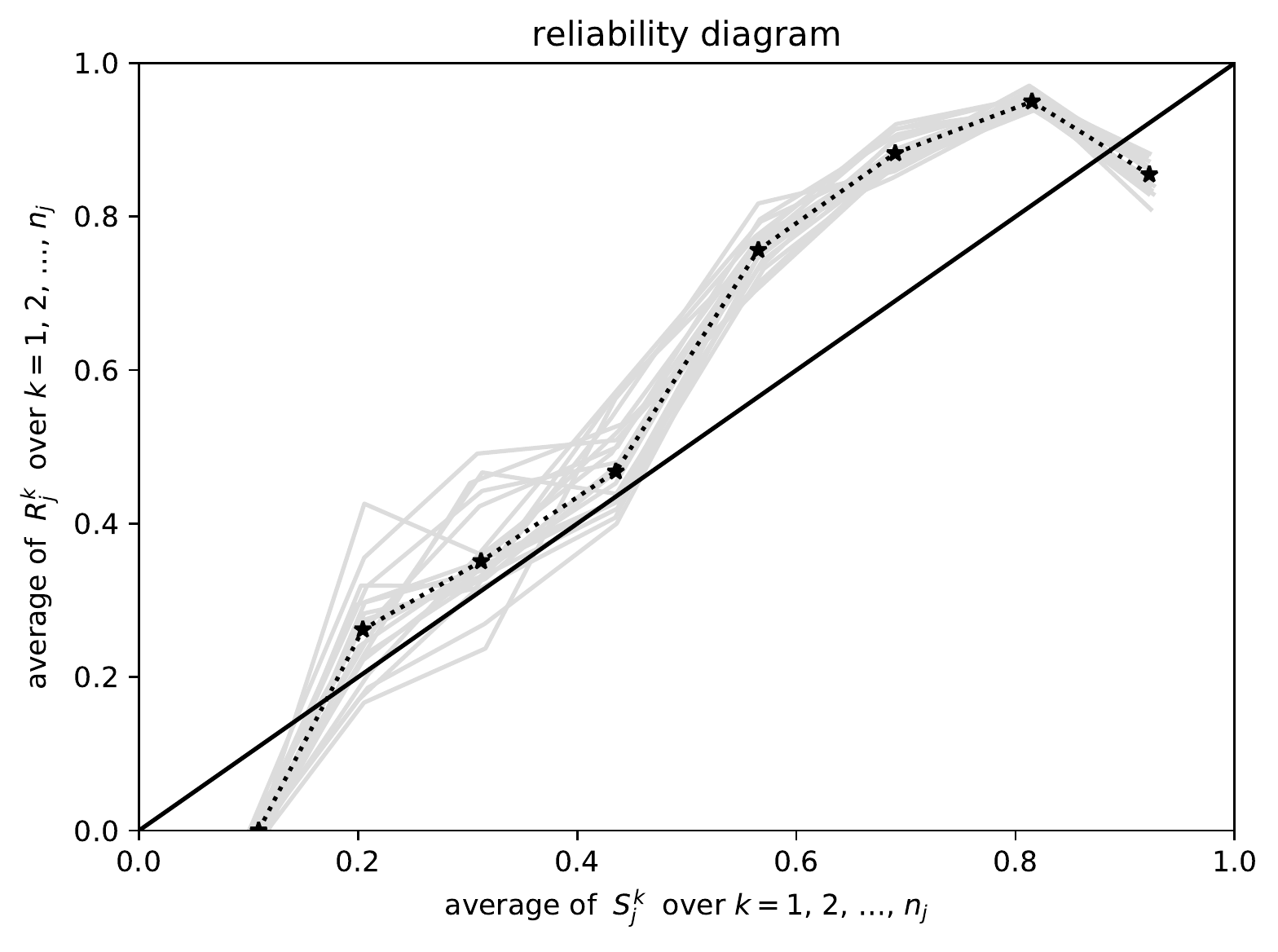}}

\parbox{\imsize}{\includegraphics[width=\imsize]
{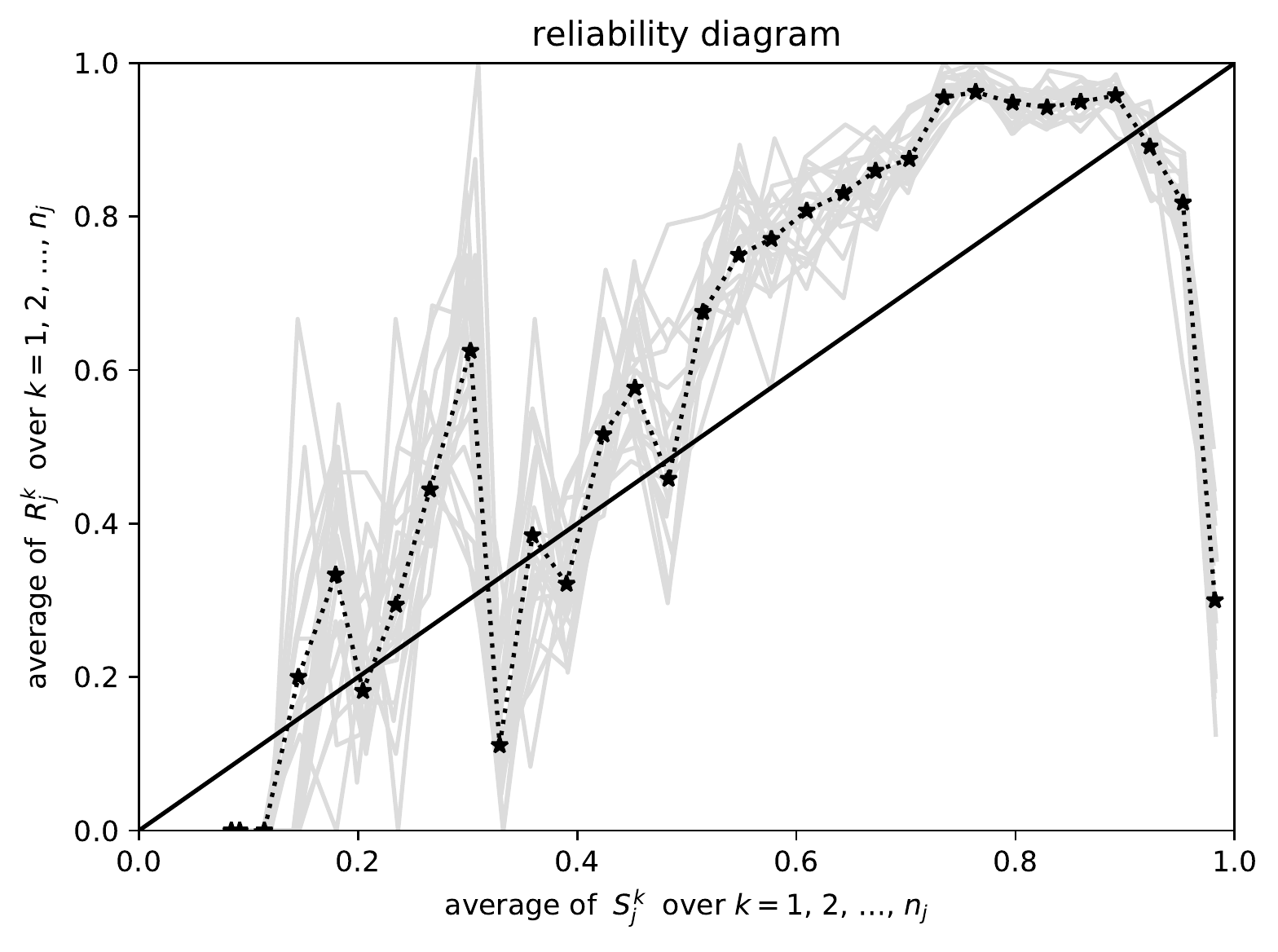}}
\end{center}
\caption{Reliability diagrams for the wild boar ({\it Sus scrofa}),
         with the bins roughly equispaced.
         There are $m = 8$ bins in the upper plot
         and $m = 32$ in the lower plot.}
\label{wild-boarprob}
\end{figure}

\begin{figure}
\begin{center}
\parbox{\imsize}{\includegraphics[width=\imsize]
{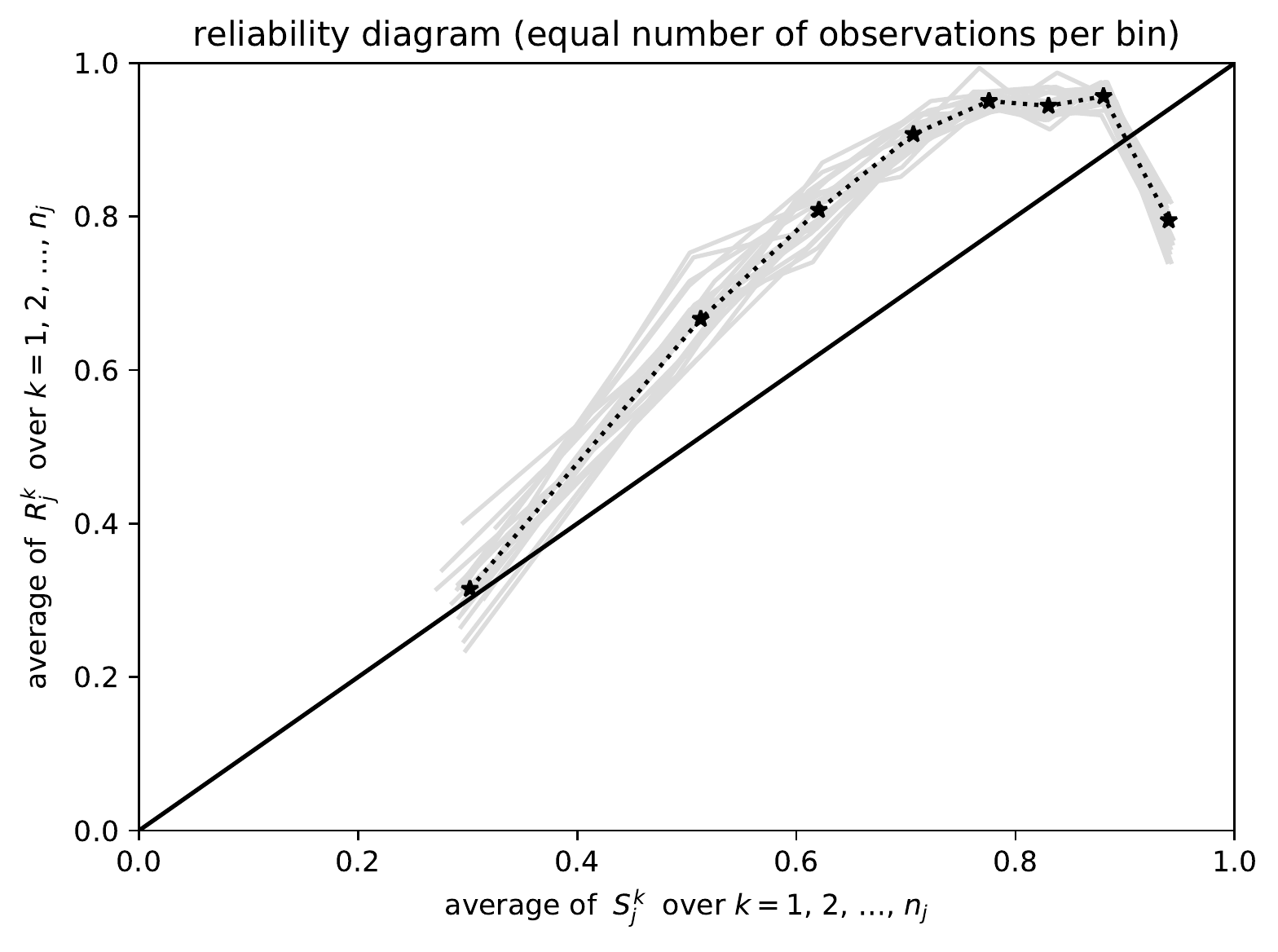}}

\parbox{\imsize}{\includegraphics[width=\imsize]
{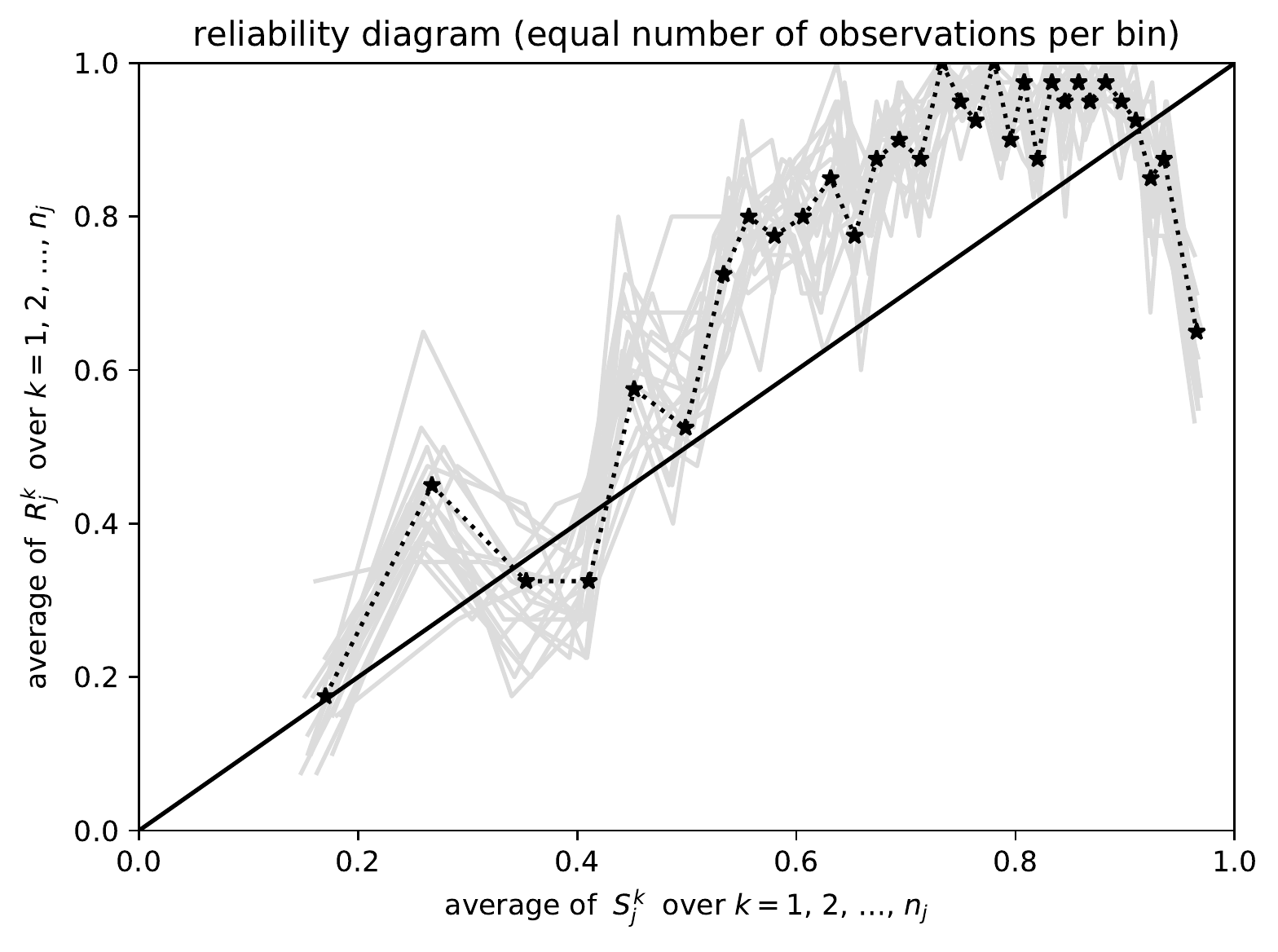}}
\end{center}
\caption{Reliability diagrams for the wild boar ({\it Sus scrofa}),
         with an equal number of observations per bin.
         There are $m = 8$ bins in the upper plot
         and $m = 32$ in the lower plot.}
\label{wild-boarsamp}
\end{figure}

\begin{figure}
\begin{center}
\parbox{\imsize}{\includegraphics[width=\imsize]
{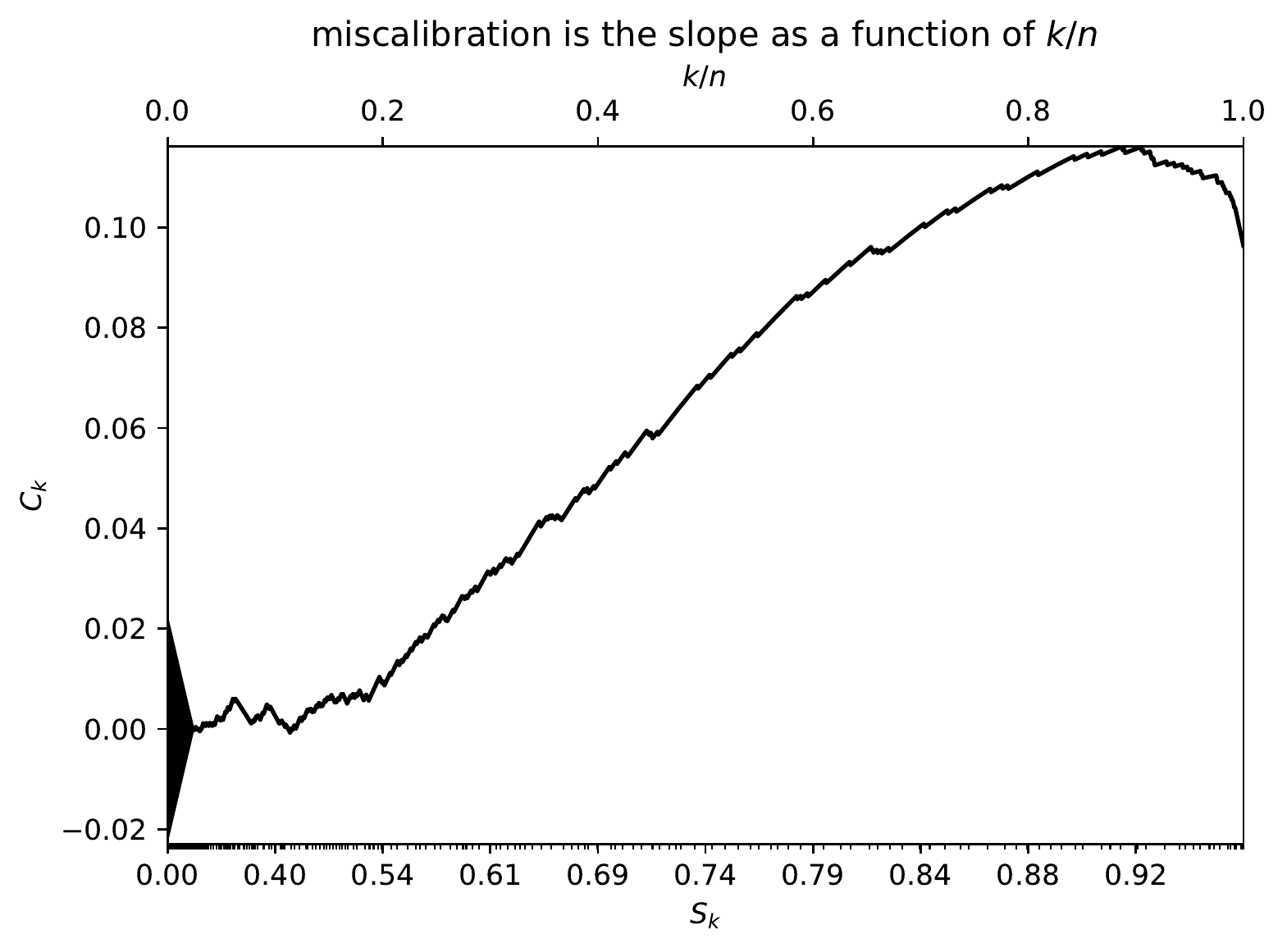}}
\end{center}
\caption{Cumulative plot for the wild boar ({\it Sus scrofa}),
         with sample size $n =$ 1,300.
         The ECCE-MAD is $0.1161 / \sigma_n = 10.14$,
         and the ECCE-R is $0.1172 / \sigma_n = 10.23$;
         both associated asymptotic P-values are zero
         to double-precision accuracy.
}
\label{wild-boarcum}
\end{figure}

\begin{figure}
\begin{center}
\parbox{\imsizes}{\includegraphics[width=\imsizes]
{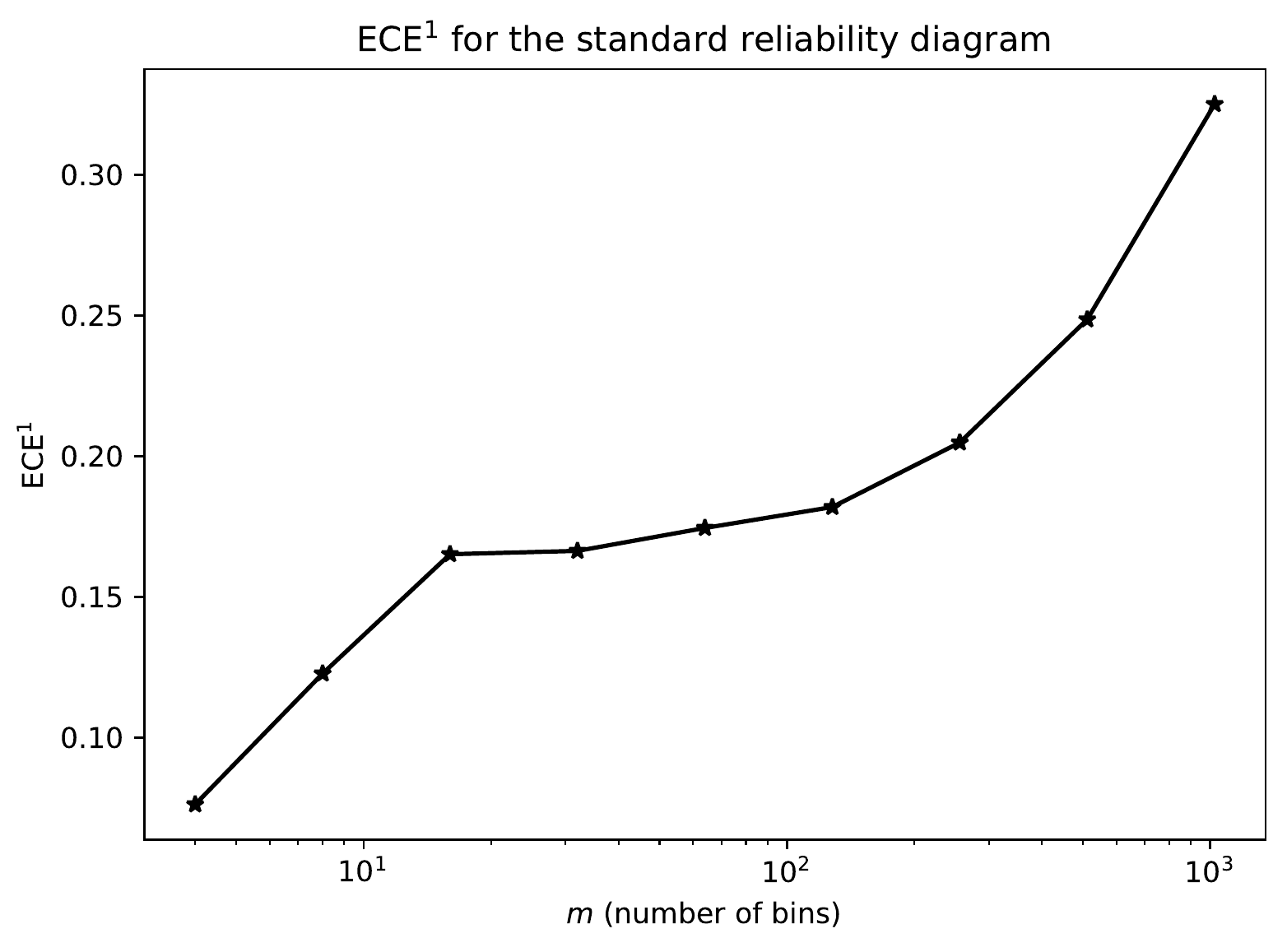}}
\hfil
\parbox{\imsizes}{\includegraphics[width=\imsizes]
{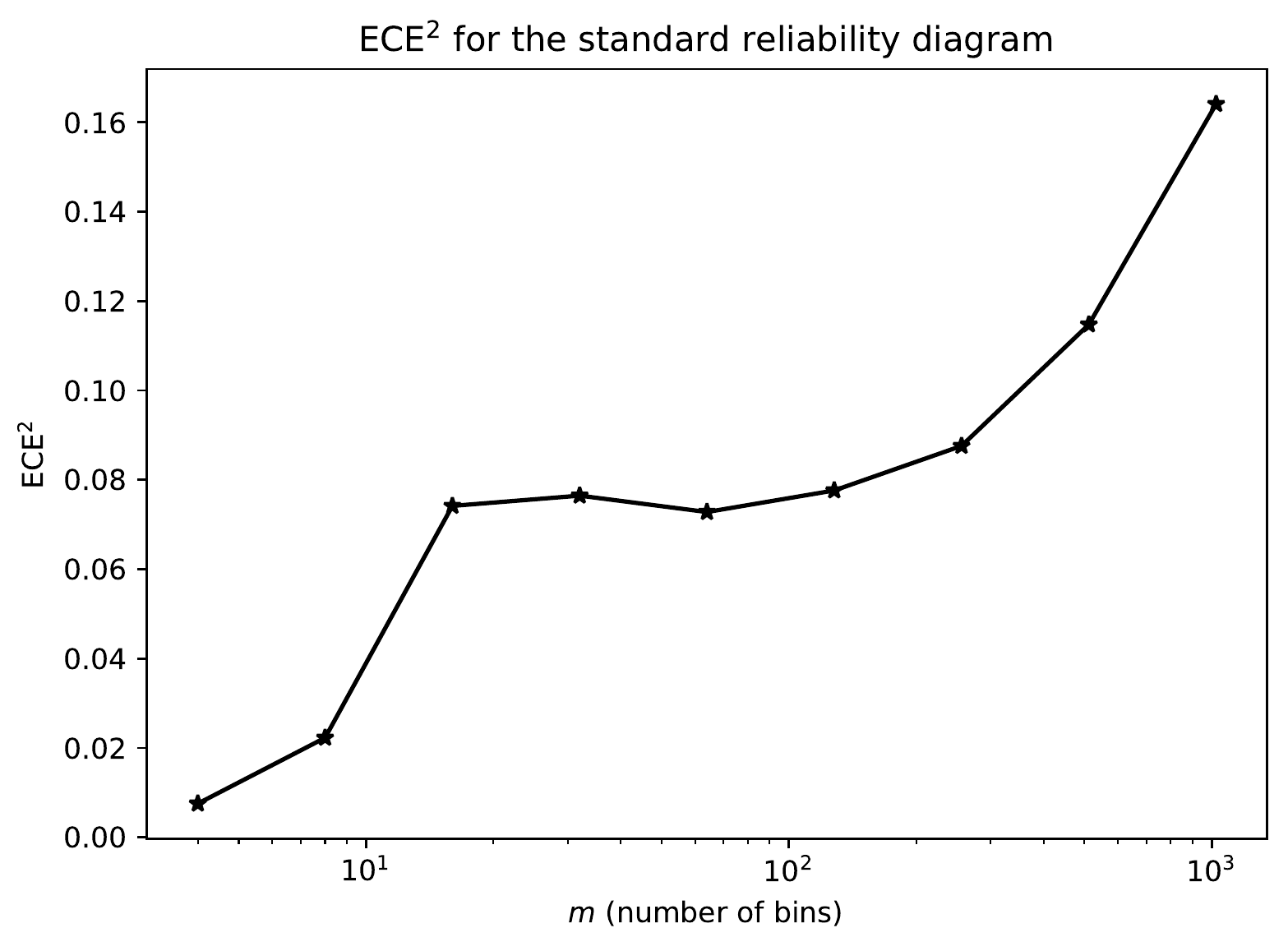}}

\parbox{\imsizes}{\includegraphics[width=\imsizes]
{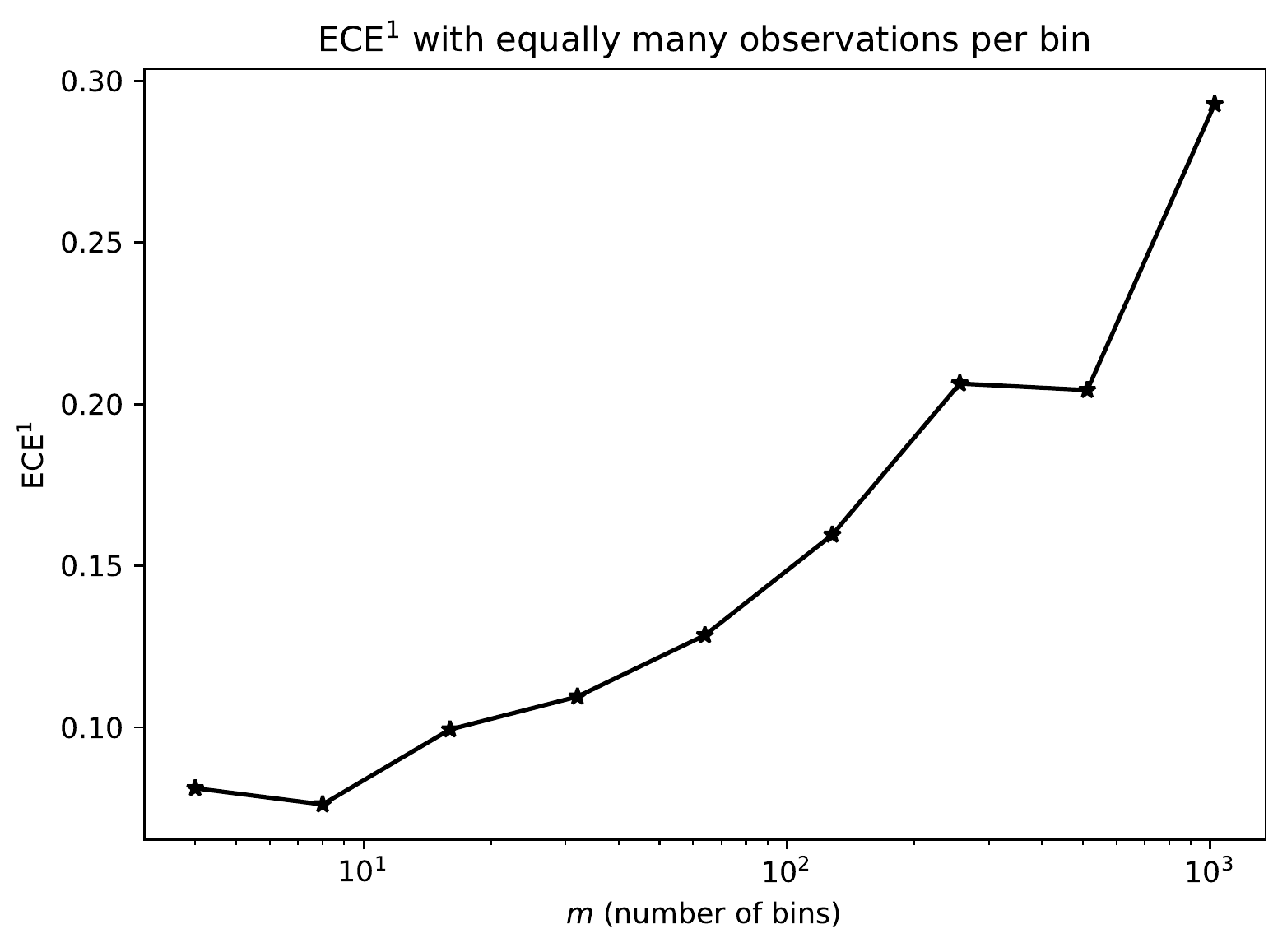}}
\hfil
\parbox{\imsizes}{\includegraphics[width=\imsizes]
{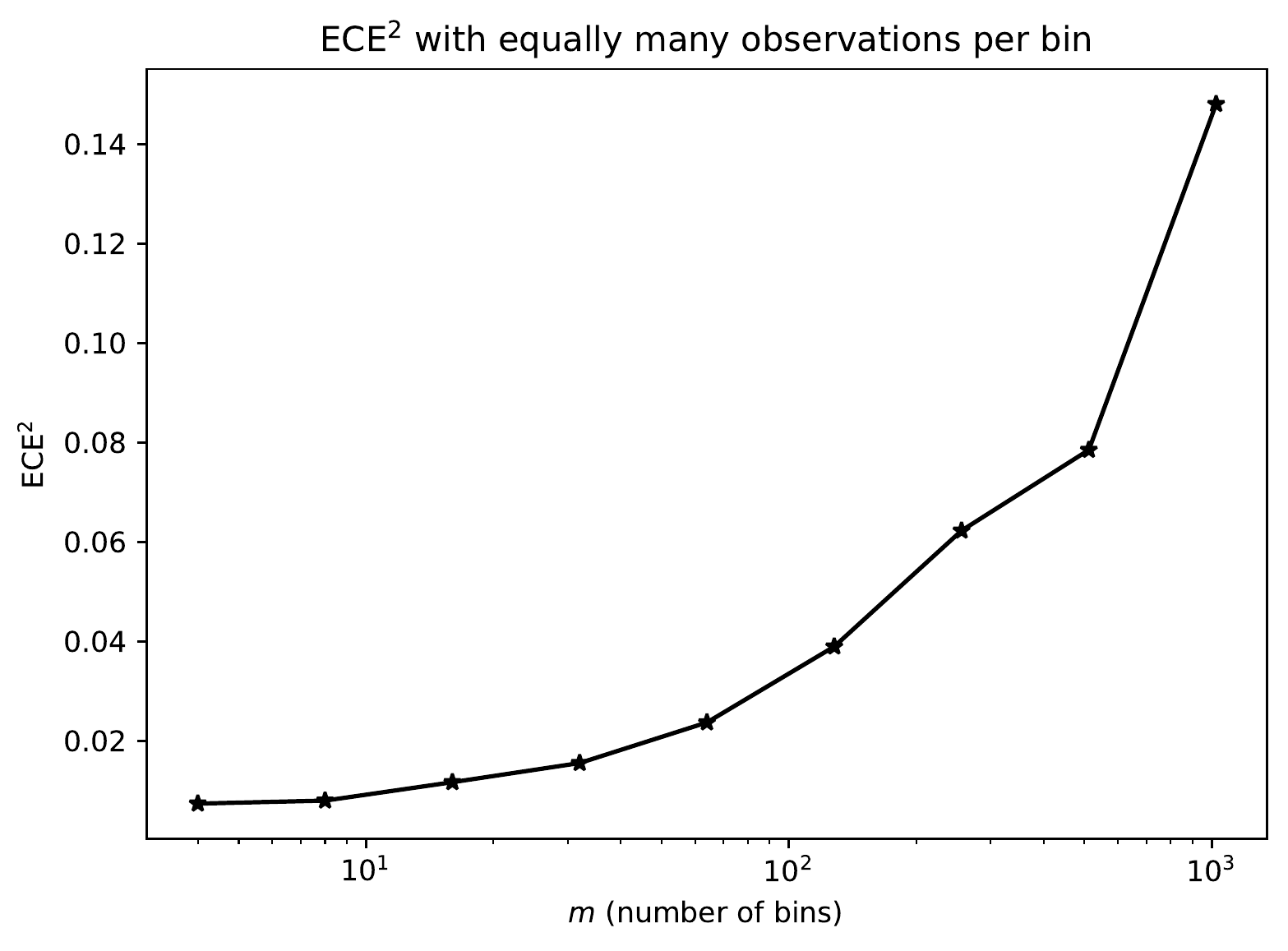}}
\end{center}
\caption{Empirical calibration errors for sunglasses,
         with sample size $n =$ 1,300.}
\label{sunglassesece}
\end{figure}

\begin{figure}
\begin{center}
\parbox{\imsize}{\includegraphics[width=\imsize]
{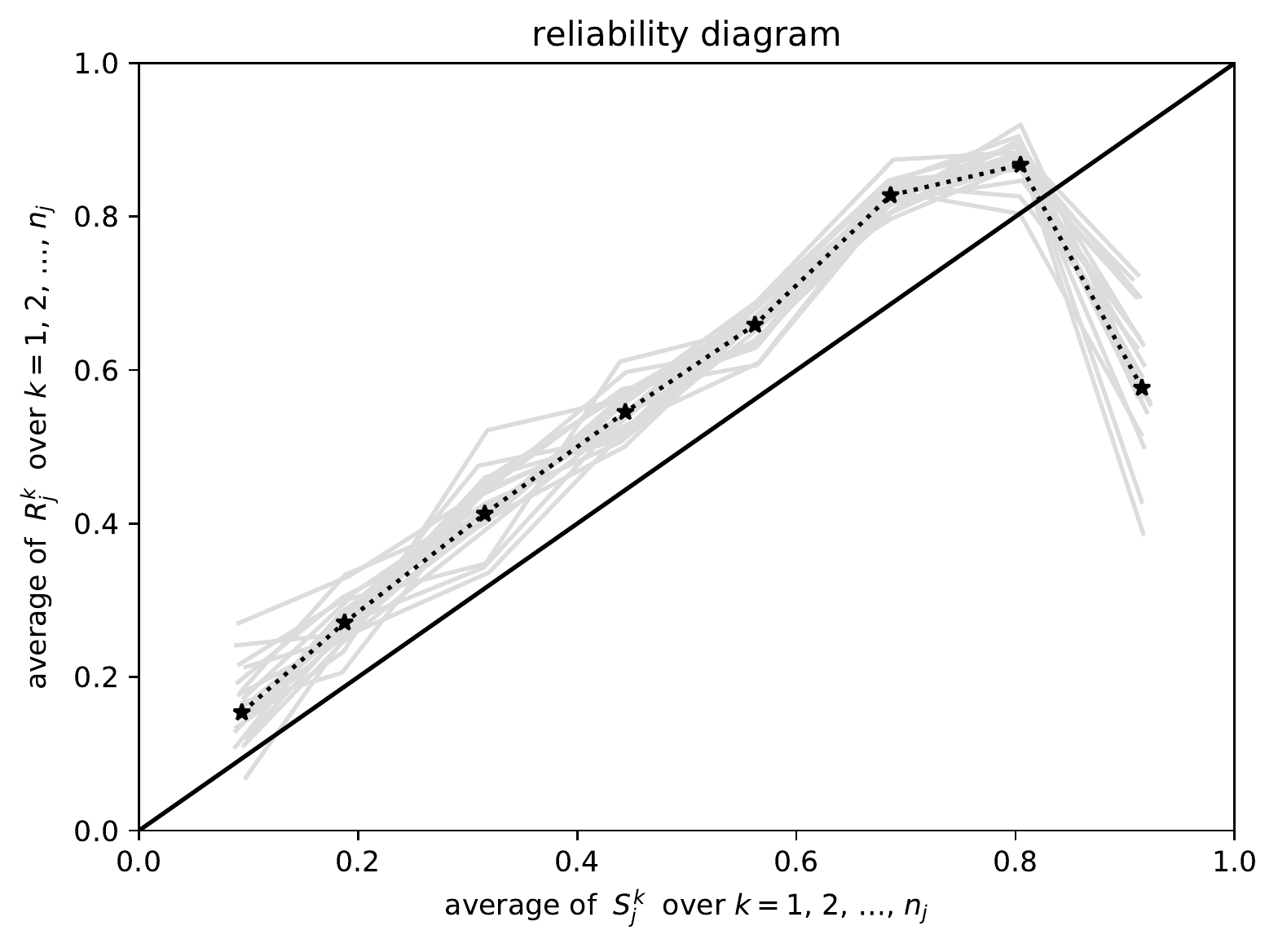}}

\parbox{\imsize}{\includegraphics[width=\imsize]
{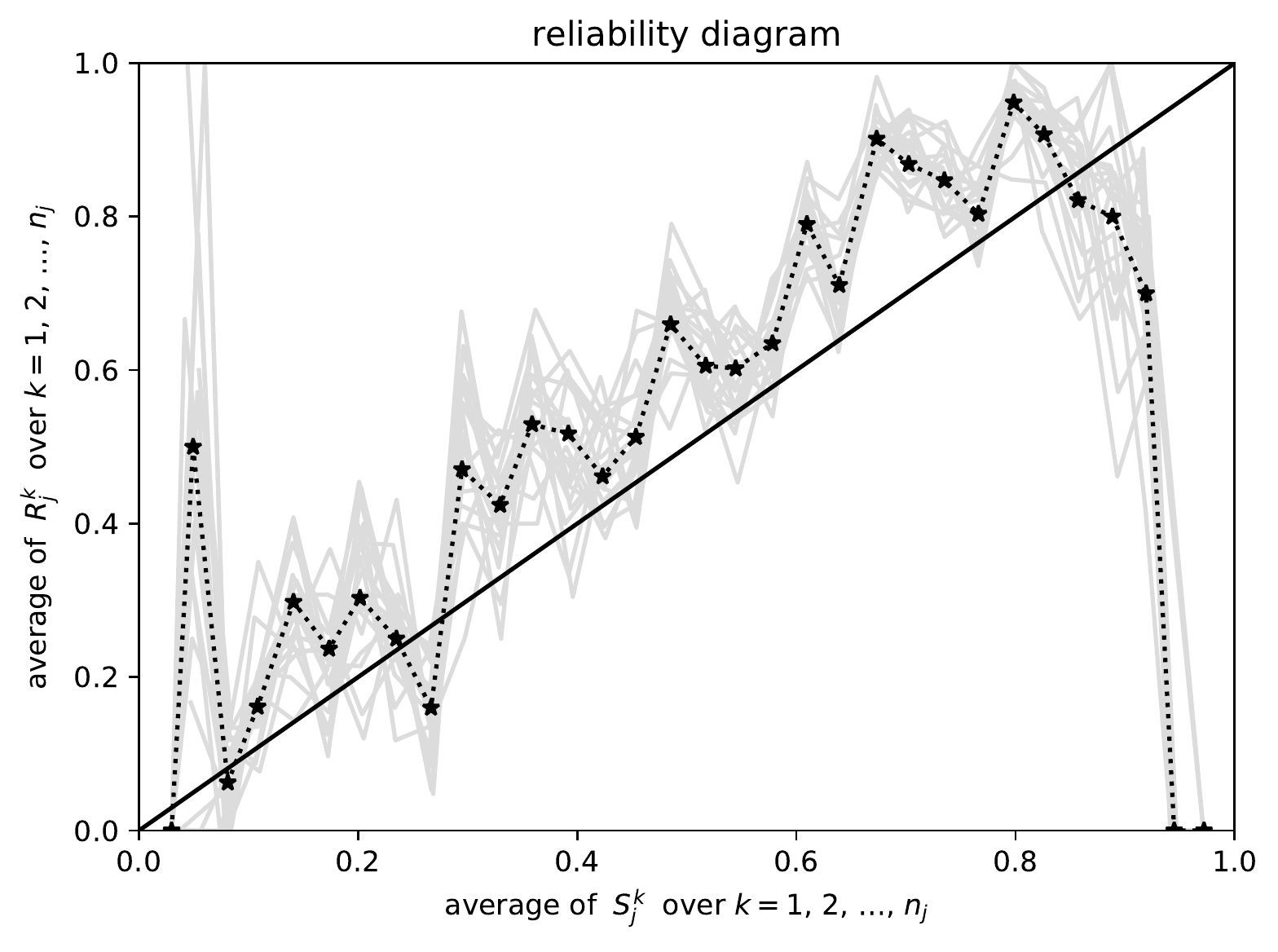}}
\end{center}
\caption{Reliability diagrams for sunglasses,
         with the bins roughly equispaced.
         There are $m = 8$ bins in the upper plot
         and $m = 32$ in the lower plot.}
\label{sunglassesprob}
\end{figure}

\begin{figure}
\begin{center}
\parbox{\imsize}{\includegraphics[width=\imsize]
{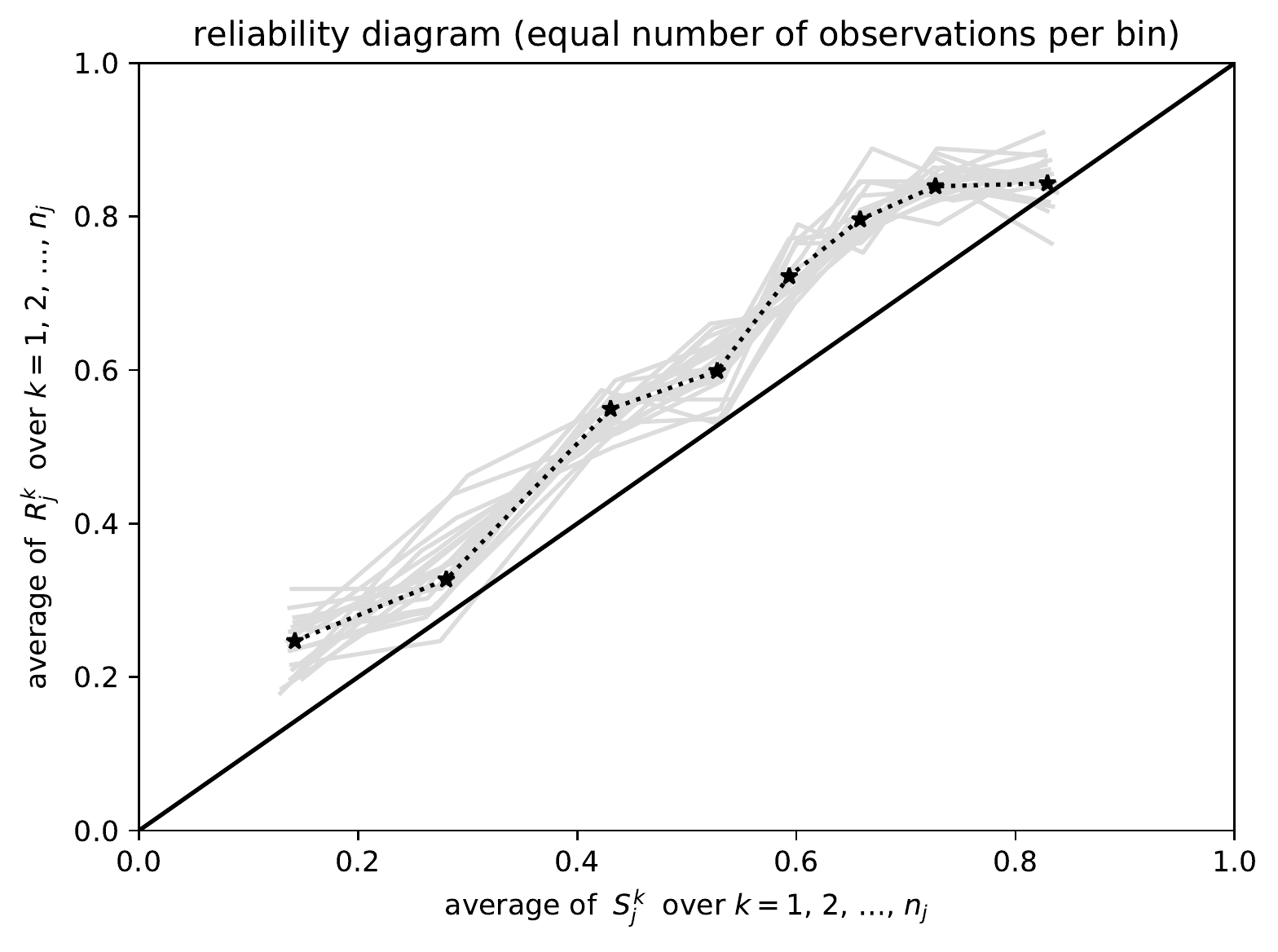}}

\parbox{\imsize}{\includegraphics[width=\imsize]
{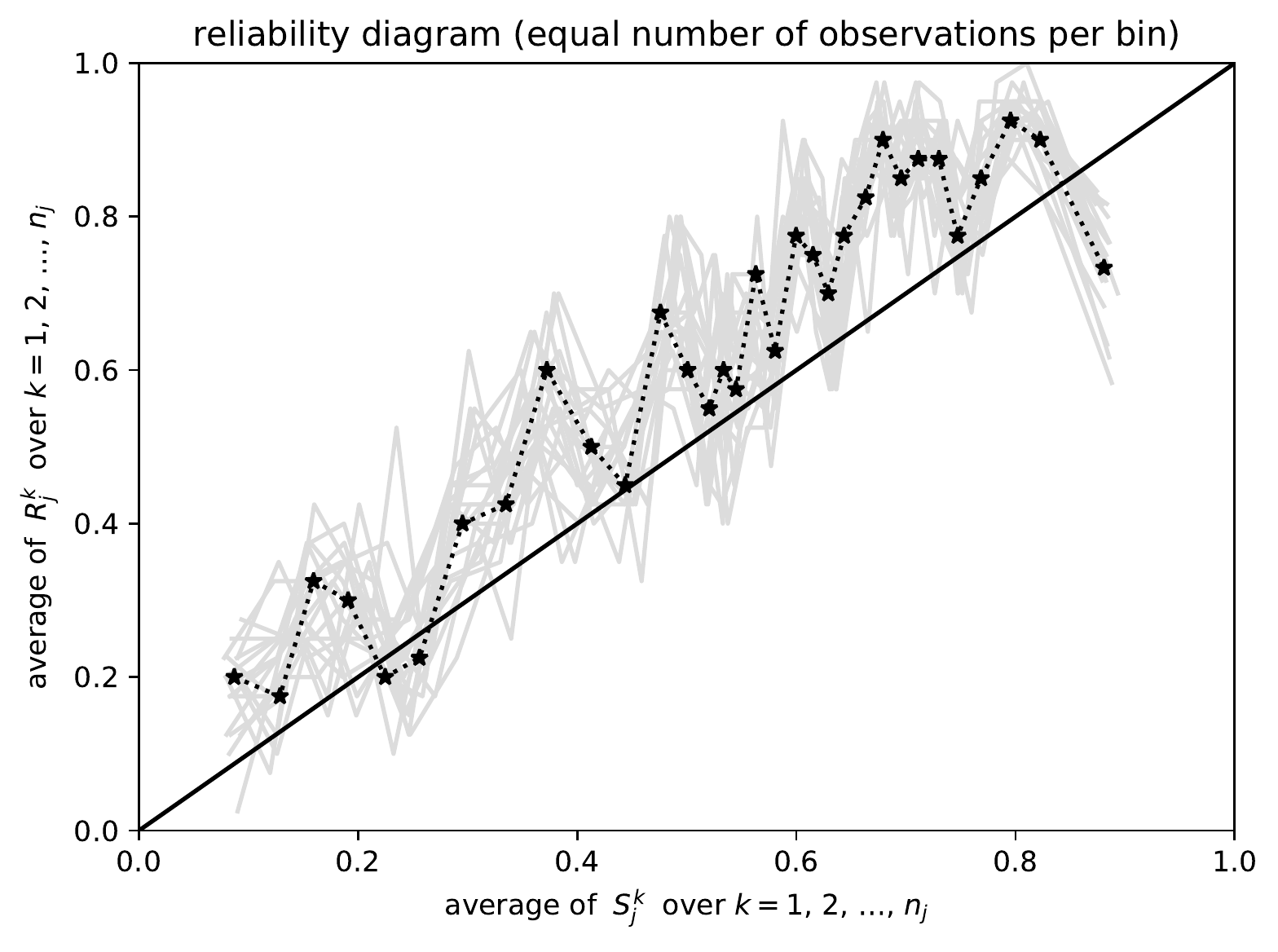}}
\end{center}
\caption{Reliability diagrams for sunglasses,
         with an equal number of observations per bin.
         There are $m = 8$ bins in the upper plot
         and $m = 32$ in the lower plot.}
\label{sunglassessamp}
\end{figure}

\begin{figure}
\begin{center}
\parbox{\imsize}{\includegraphics[width=\imsize]
{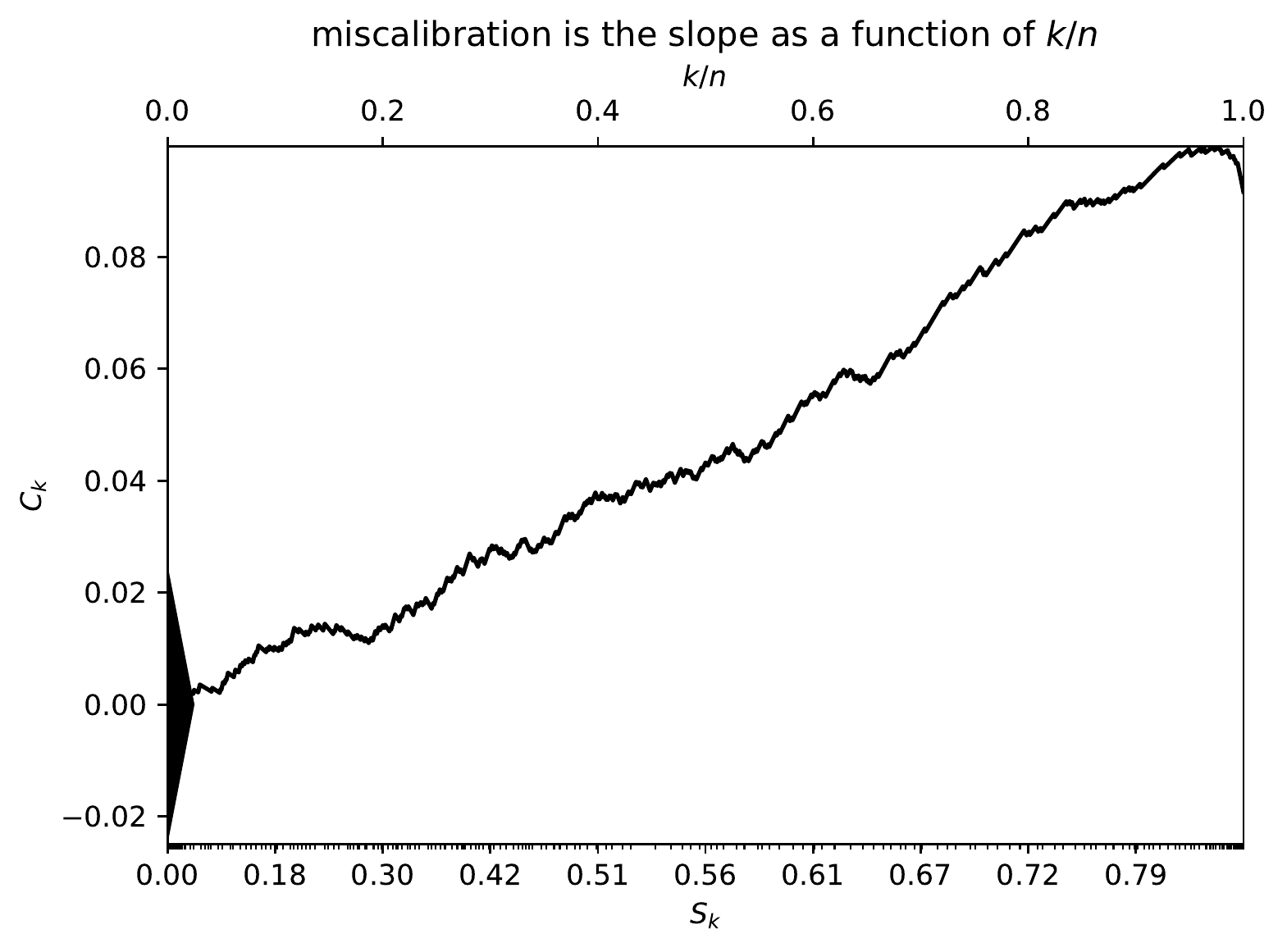}}
\end{center}
\caption{Cumulative plot for sunglasses, with sample size $n =$ 1,300.
         The ECCE-MAD is $0.09972 / \sigma_n = 8.004$,
         and the ECCE-R is $0.09977 / \sigma_n = 8.008$;
         both associated asymptotic P-values are zero
         to double-precision accuracy.
}
\label{sunglassescum}
\end{figure}

\begin{figure}
\begin{center}
\parbox{\imsizes}{\includegraphics[width=\imsizes]
       {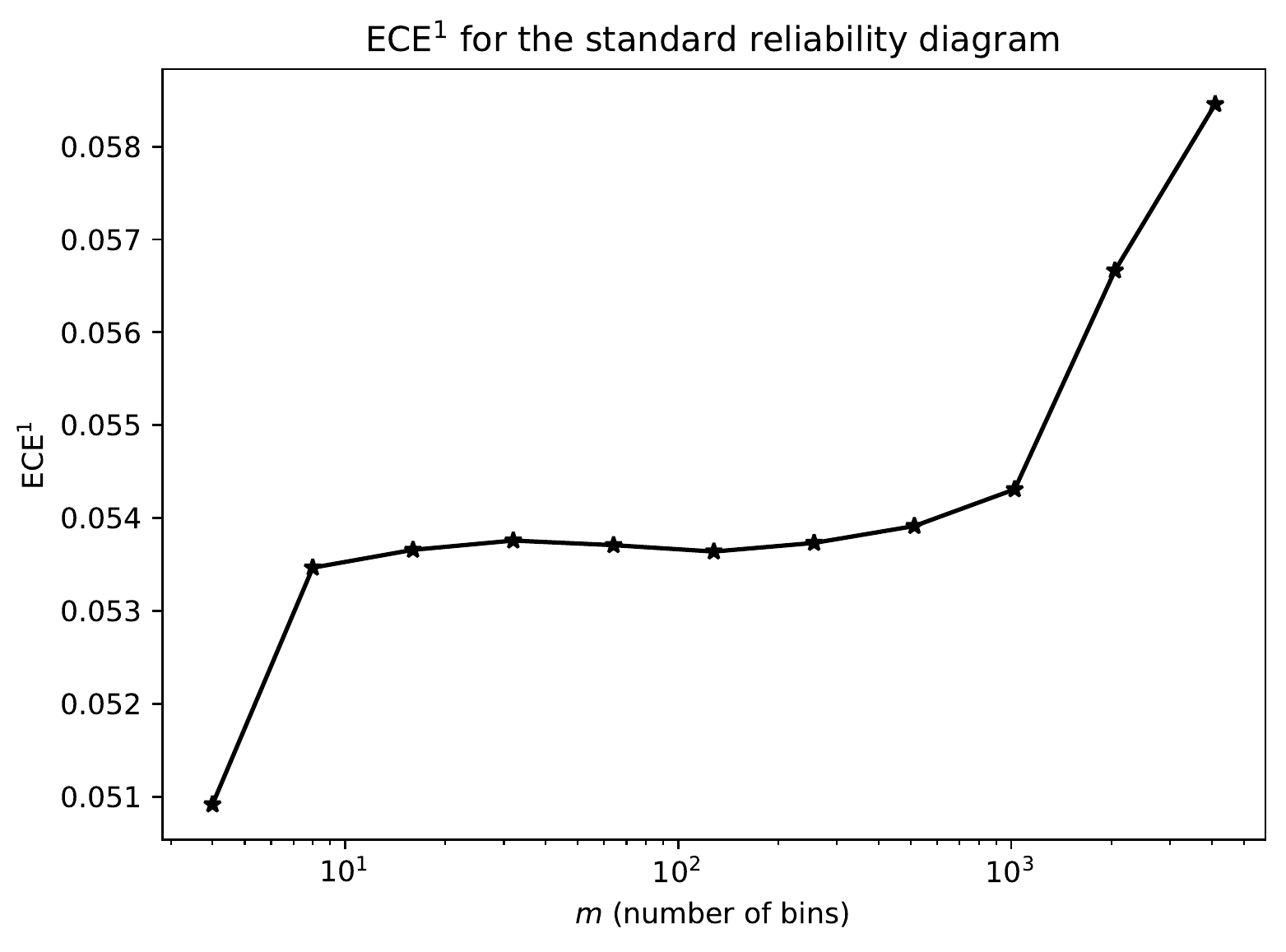}}
\hfil
\parbox{\imsizes}{\includegraphics[width=\imsizes]
       {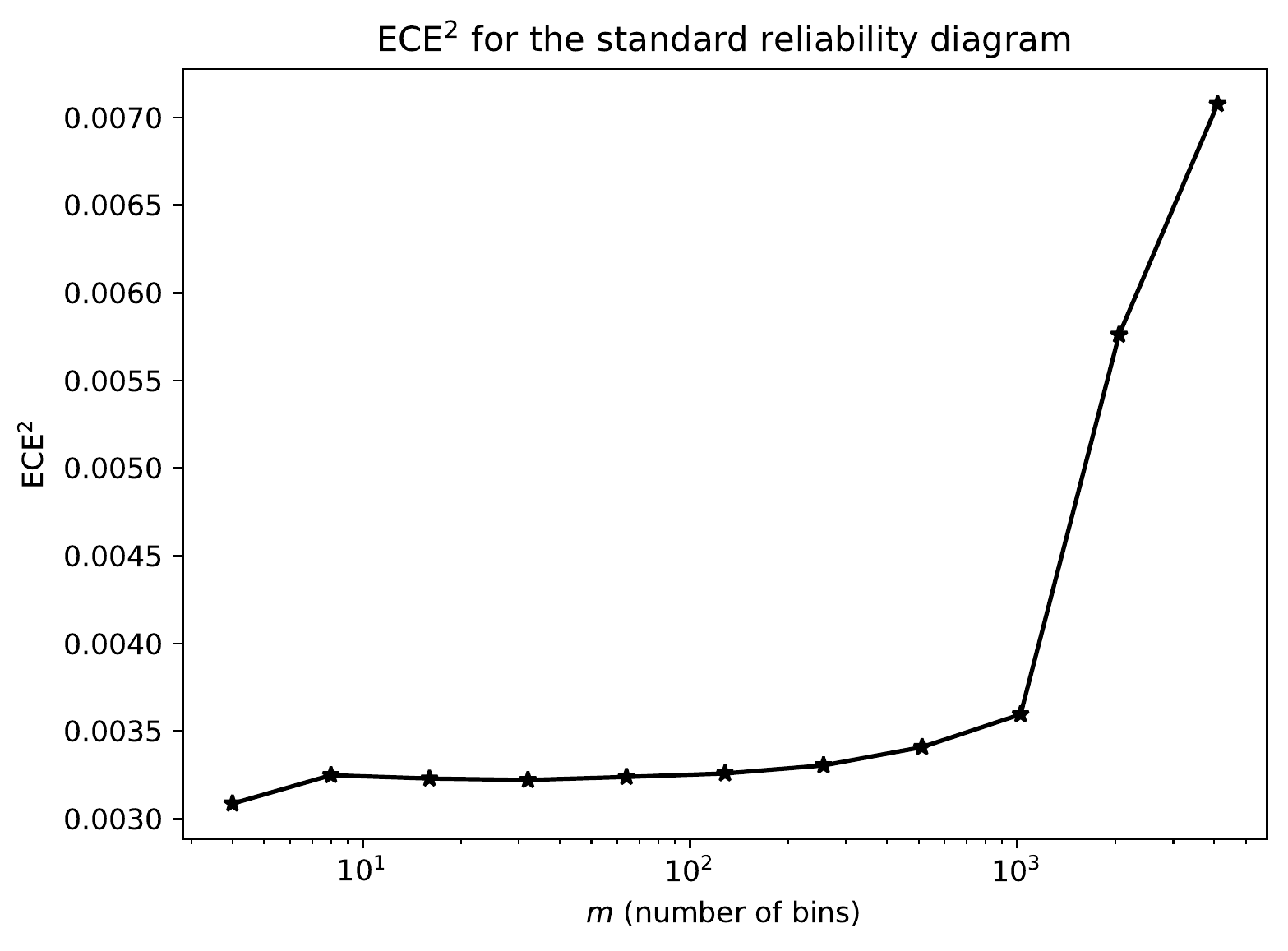}}

\parbox{\imsizes}{\includegraphics[width=\imsizes]
       {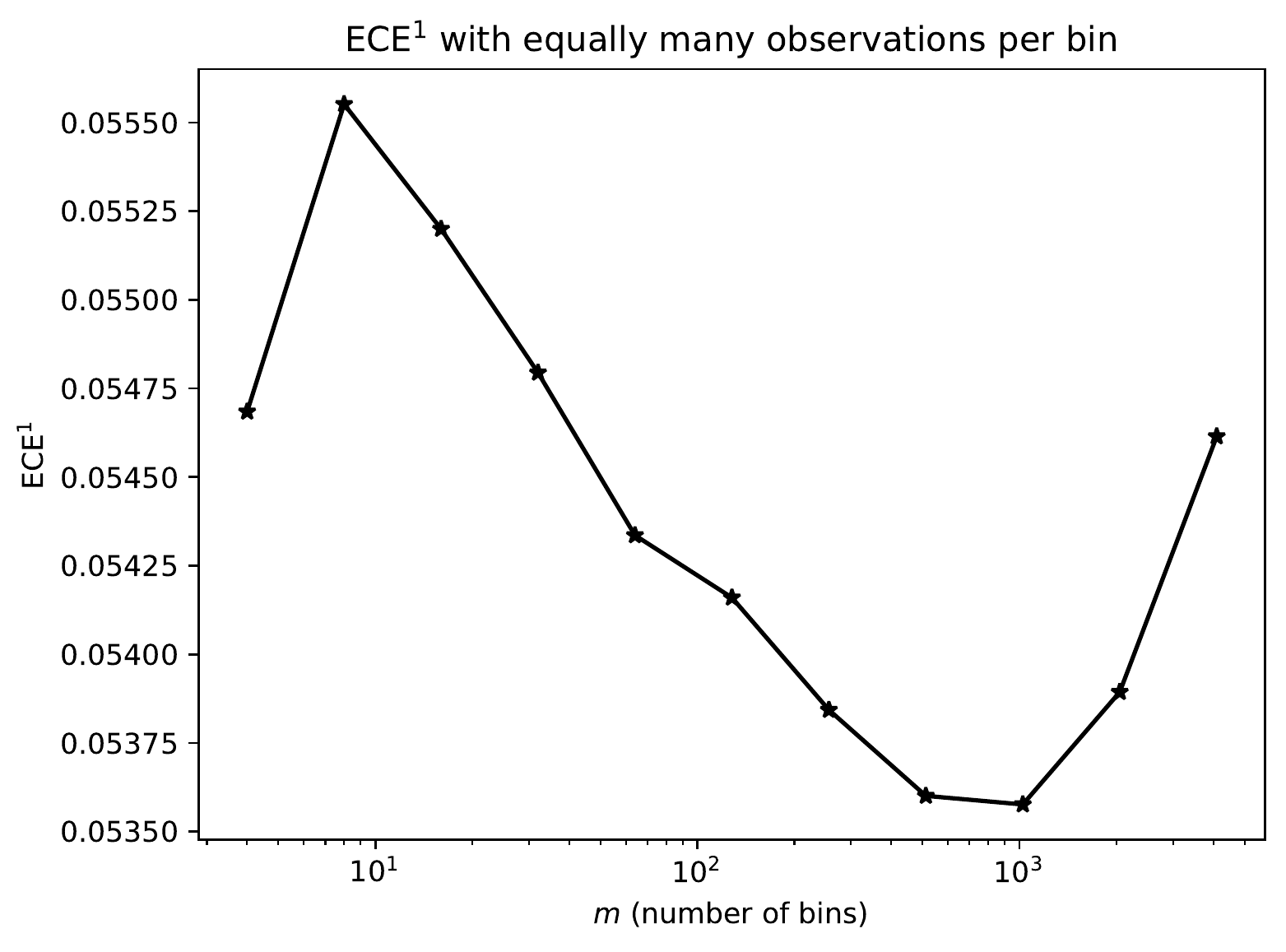}}
\hfil
\parbox{\imsizes}{\includegraphics[width=\imsizes]
       {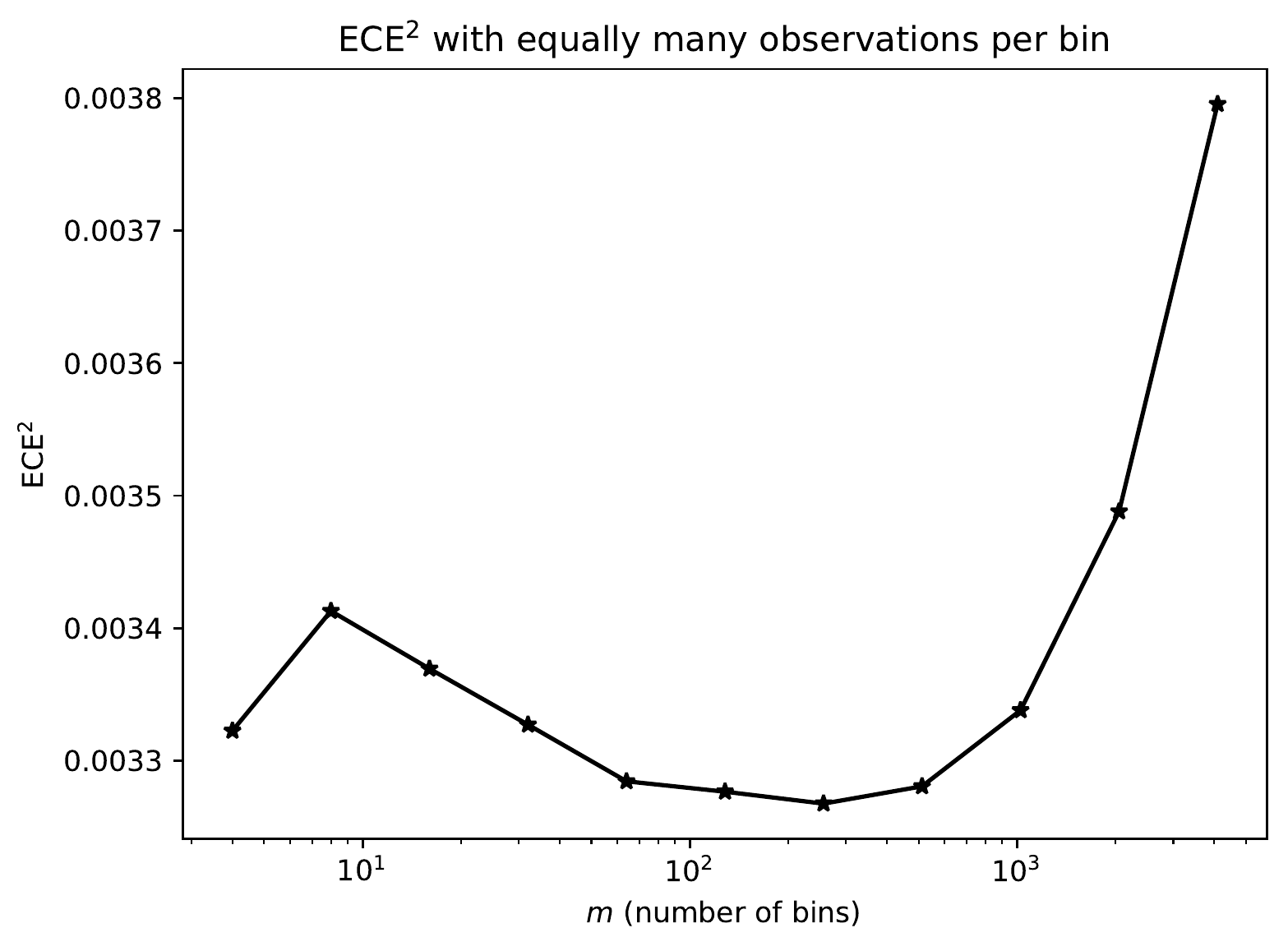}}
\end{center}
\caption{Empirical calibration errors
         for the full ImageNet-1000 training data set,
         with sample size $n =$ 1,281,167.}
\label{imagenetece}
\end{figure}

\begin{figure}
\begin{center}
\parbox{\imsize}{\includegraphics[width=\imsize]
       {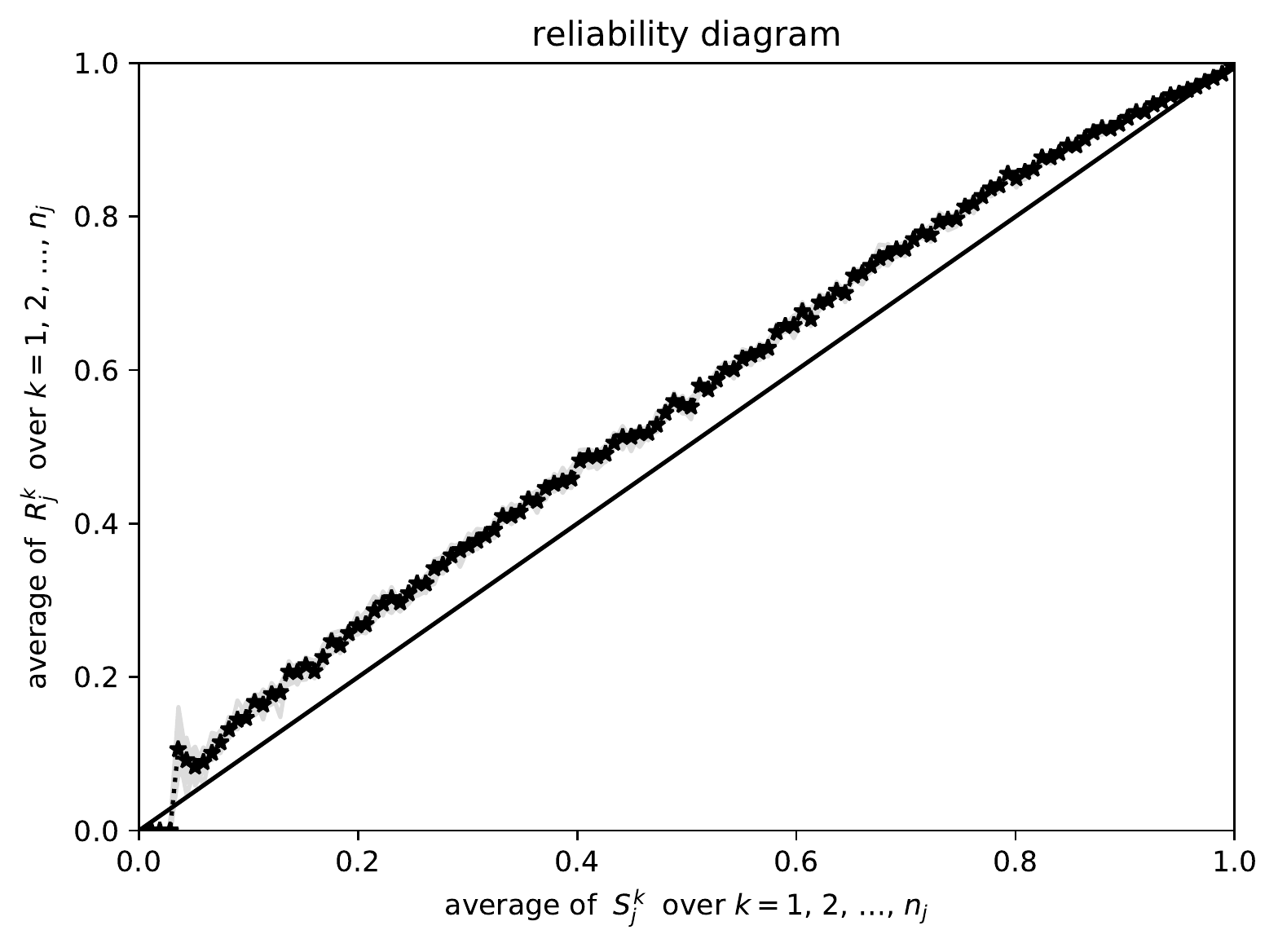}}

\parbox{\imsize}{\includegraphics[width=\imsize]
       {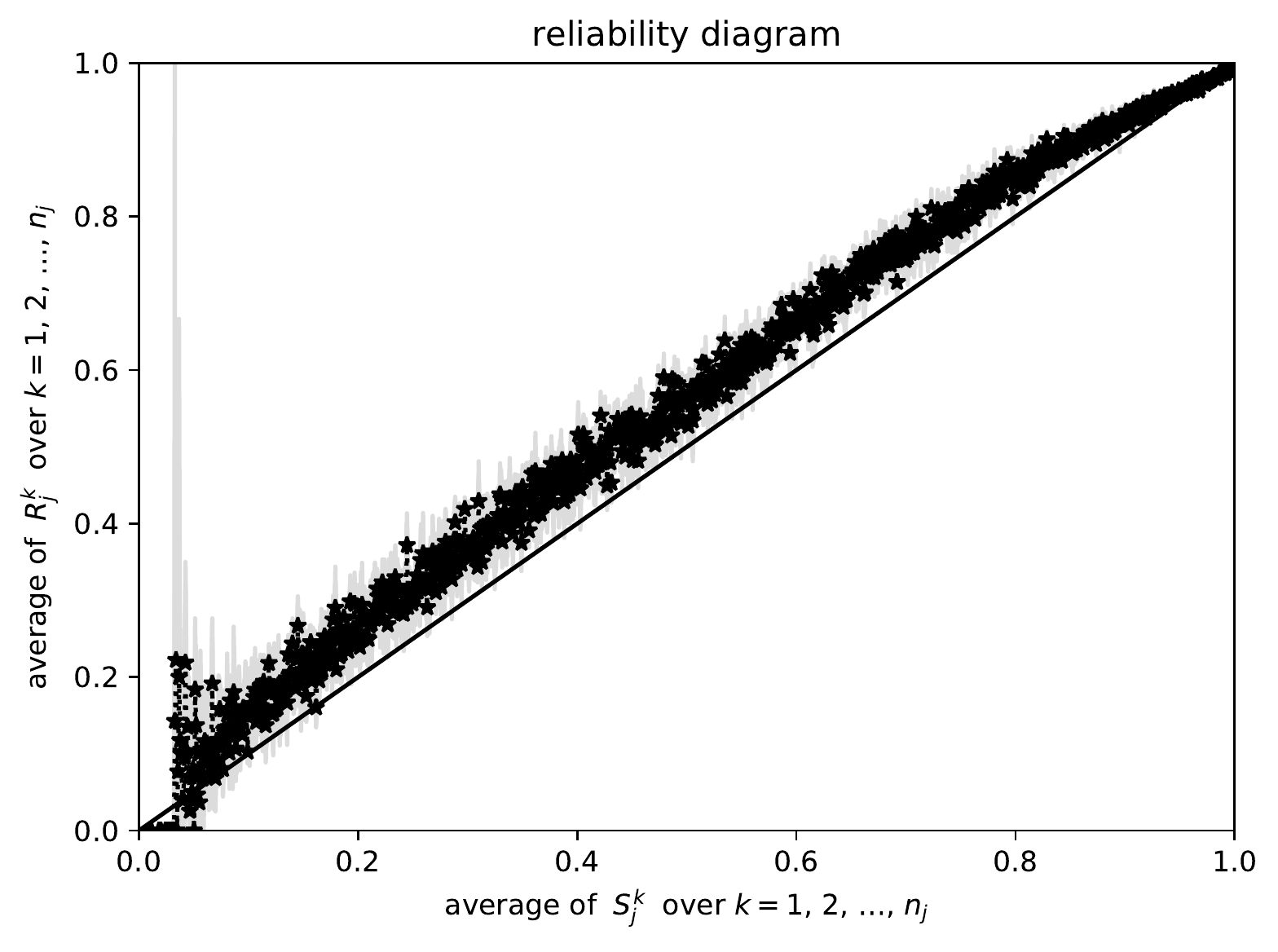}}
\end{center}
\caption{Reliability diagrams for the full ImageNet-1000 training data set,
         with the bins roughly equispaced.
         There are $m = 128$ bins in the upper plot
         and $m =$ 1,024 in the lower plot.}
\label{imagenetprob}
\end{figure}

\begin{figure}
\begin{center}
\parbox{\imsize}{\includegraphics[width=\imsize]
       {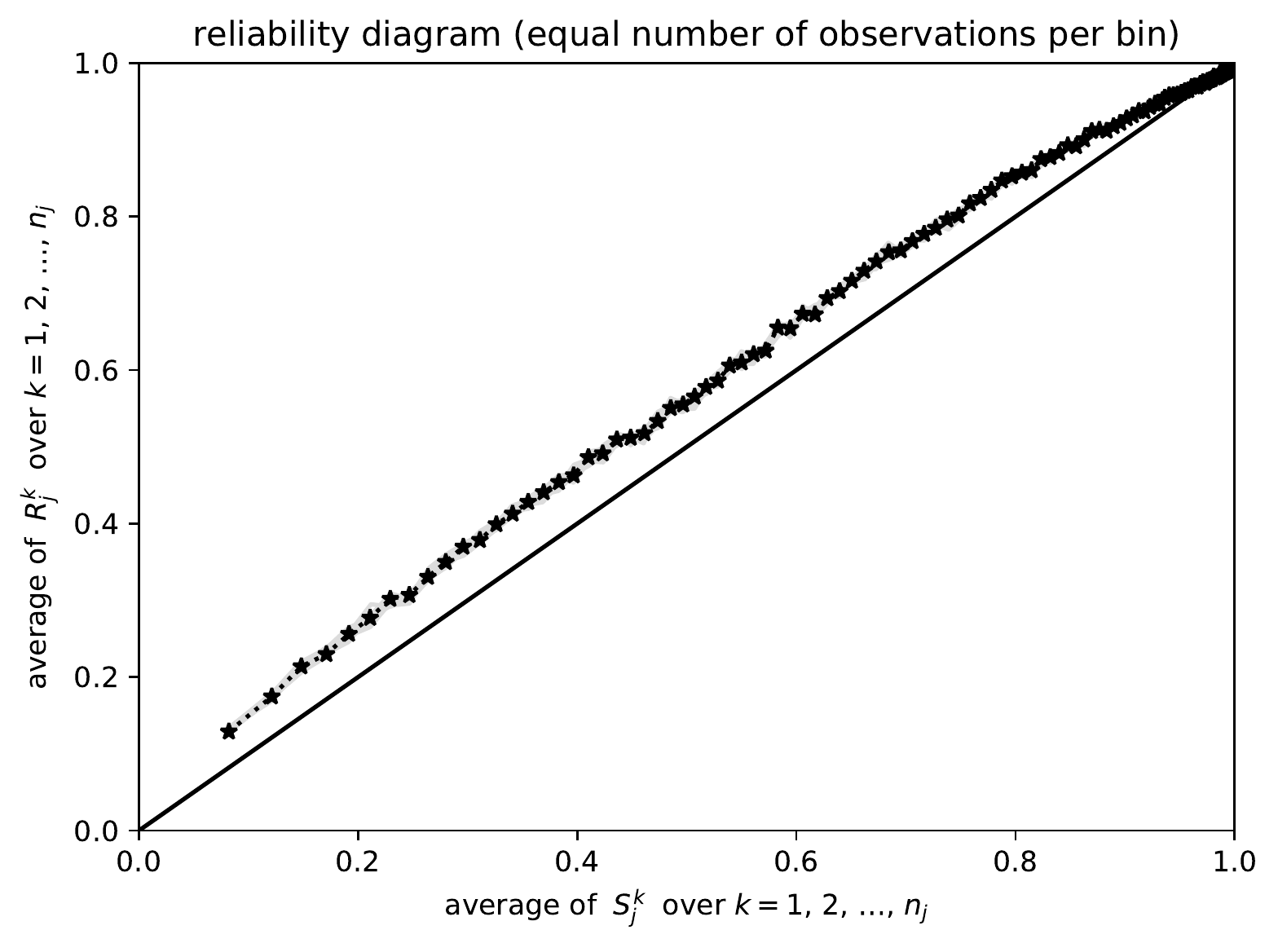}}

\parbox{\imsize}{\includegraphics[width=\imsize]
       {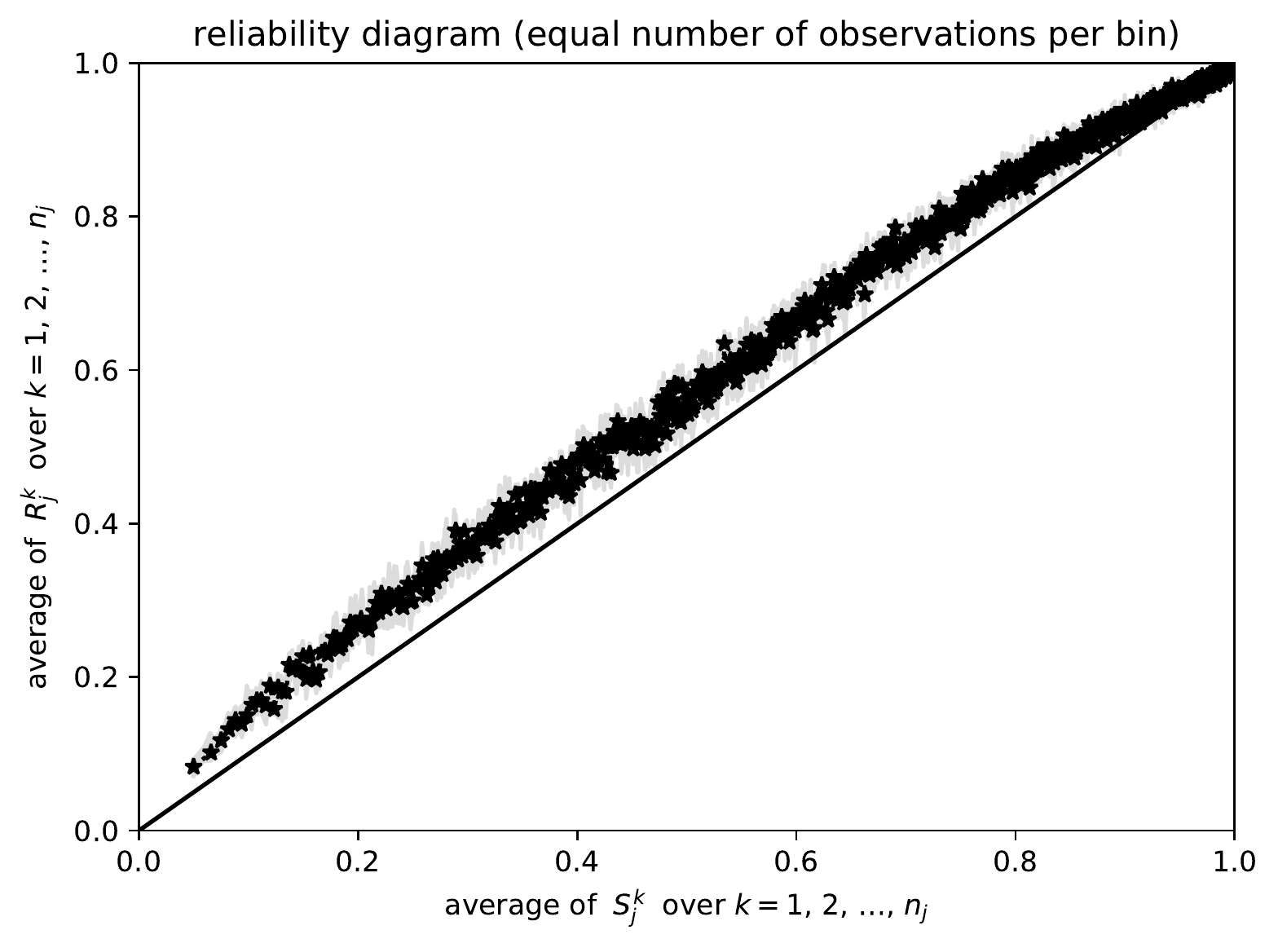}}
\end{center}
\caption{Reliability diagrams for the full ImageNet-1000 training data set,
         with an equal number of observations per bin.
         There are $m = 128$ bins in the upper plot
         and $m =$ 1,024 in the lower plot.}
\label{imagenetsamp}
\end{figure}

\begin{figure}
\begin{center}
\parbox{\imsize}{\includegraphics[width=\imsize]{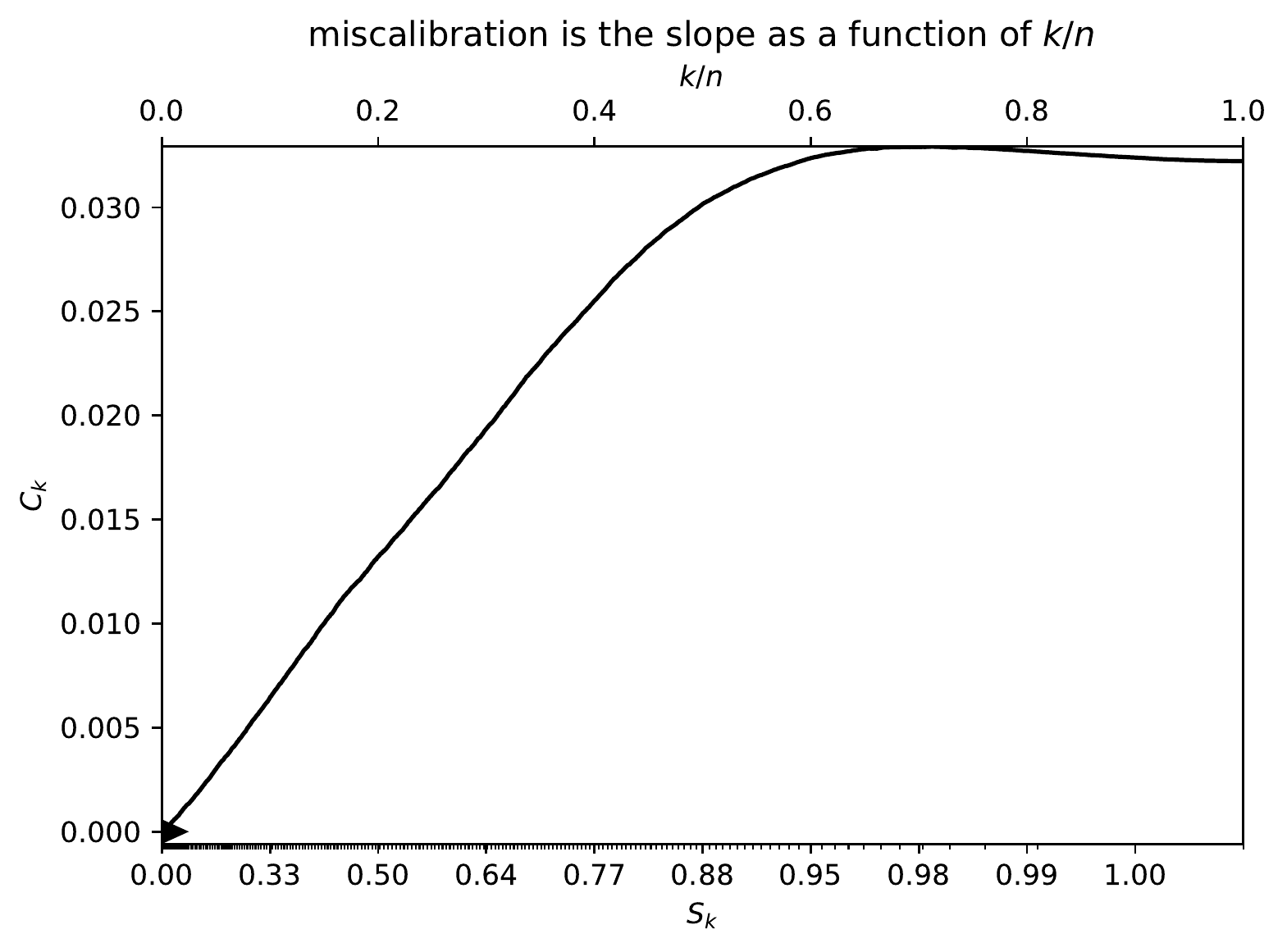}}
\end{center}
\caption{Cumulative plot for the full ImageNet-1000 training data set,
         with sample size $n =$ 1,281,167.
         The ECCE-MAD is $0.03306 / \sigma_n = 111.7$,
         and the ECCE-R is also $0.03306 / \sigma_n = 111.7$;
         these indicate profoundly statistically significant miscalibration,
         courtesy of the large number of observations
         (the actual effect size is more modest, as seen by the values
         without dividing by $\sigma_n$). Both associated asymptotic P-values
         are zero to double-precision accuracy.
}
\label{imagenetcum}
\end{figure}

\section{Conclusion}
\label{conclusion}

A trade-off between statistical confidence and power to resolve variations
as a function of score is inherent to the empirical calibration errors (ECEs)
based on binning, while the empirical {\it cumulative} calibration errors
(ECCEs) have no such explicit trade-off. The theory of Section~\ref{methods}
proves that this trade-off results in the ECEs requiring
infinitely denser observations than what the ECCEs require to attain
the same consistency and power against any fixed alternative distribution,
where ``infinitely denser'' means asymptotically, in the limit of large samples
(or in the limit of high statistical confidence).
Consonant with the asymptotic theory, the examples of Section~\ref{results}
illustrate at the finite sample sizes of greatest interest in practice
the drastically higher density required by the ECEs,
together with the ECEs' trade-off between confidence and resolving power.
The ECEs also exhibit an extreme dependence on the choice of bins,
with different choices of bins yielding significantly different values
for the ECE metrics; choosing among the possible binnings can be confusing,
yet makes all the difference. In contrast, the ECCEs yield trustworthy results
without needing such large numbers of observations
and without needing to set any parameters. Furthermore, P-values
(also known as ``attained statistical significance levels'')
that are asymptotically perfectly-calibrated in the limit of large sample sizes
are available for the ECCEs via the simple, convenient, efficient methods
of~\cite{tygert_pvals}. All in all, the ECEs are unreliable
and largely unusable, while the ECCEs are reliable and easy to use.

\section*{Acknowledgements}

We would like to thank Kamalika Chaudhuri and Isabel Kloumann.

\appendix
\section{Additional figures}

This appendix provides analogues for the samples sizes $n =$ 8,192
and $n =$ 131,072 of Figure~\ref{32768ece} from Subsection~\ref{synthetic}
(Figure~\ref{32768ece} corresponds to the sample size $n =$ 32,768).
Figure~\ref{8192ece} is for $n =$ 8,192;
Figure~\ref{131072ece} is for $n =$ 131,072.

\begin{figure}[b!]
\begin{center}
\parbox{\imsizes}{\includegraphics[width=\imsizes]
       {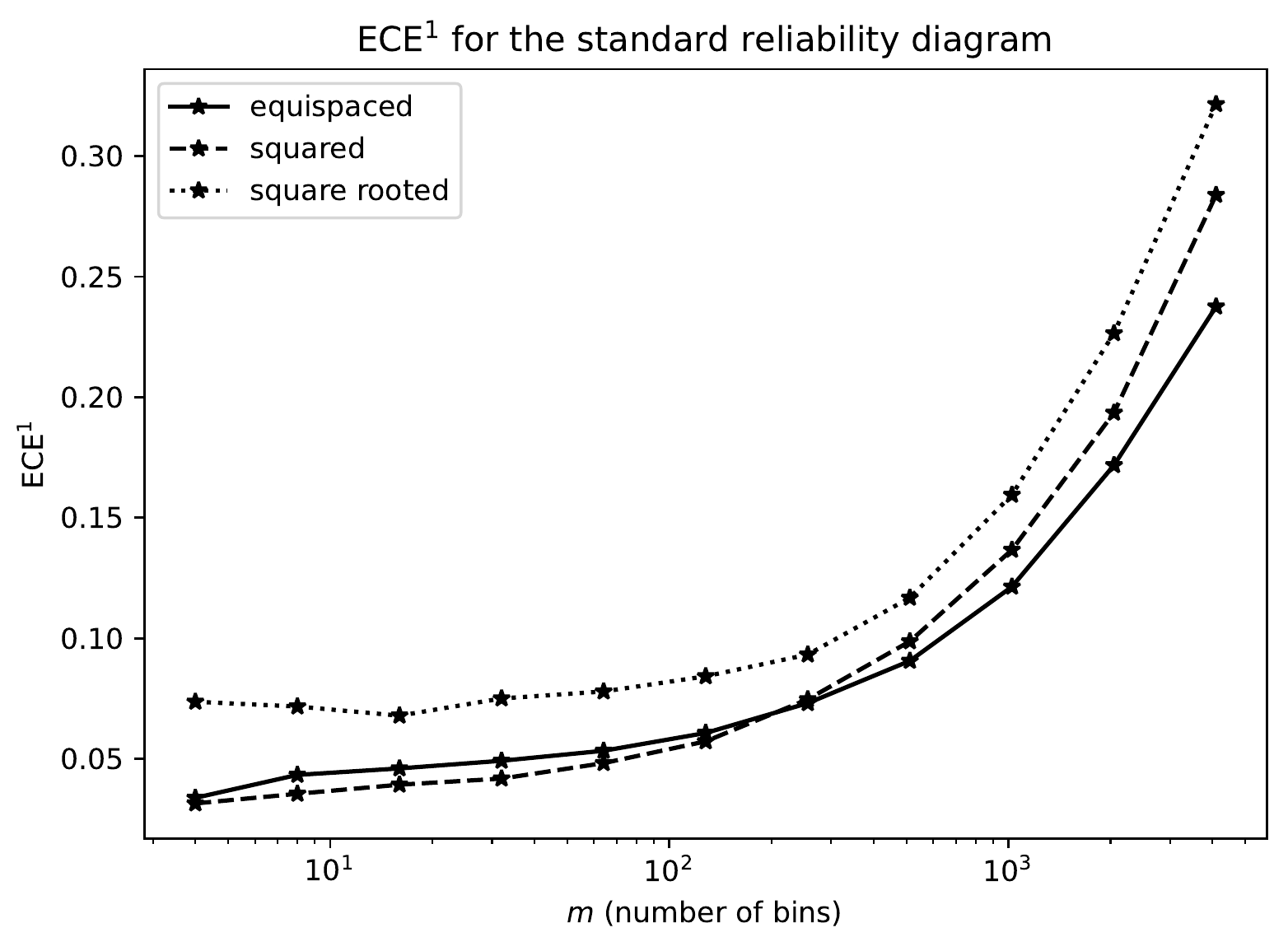}}
\hfil
\parbox{\imsizes}{\includegraphics[width=\imsizes]
       {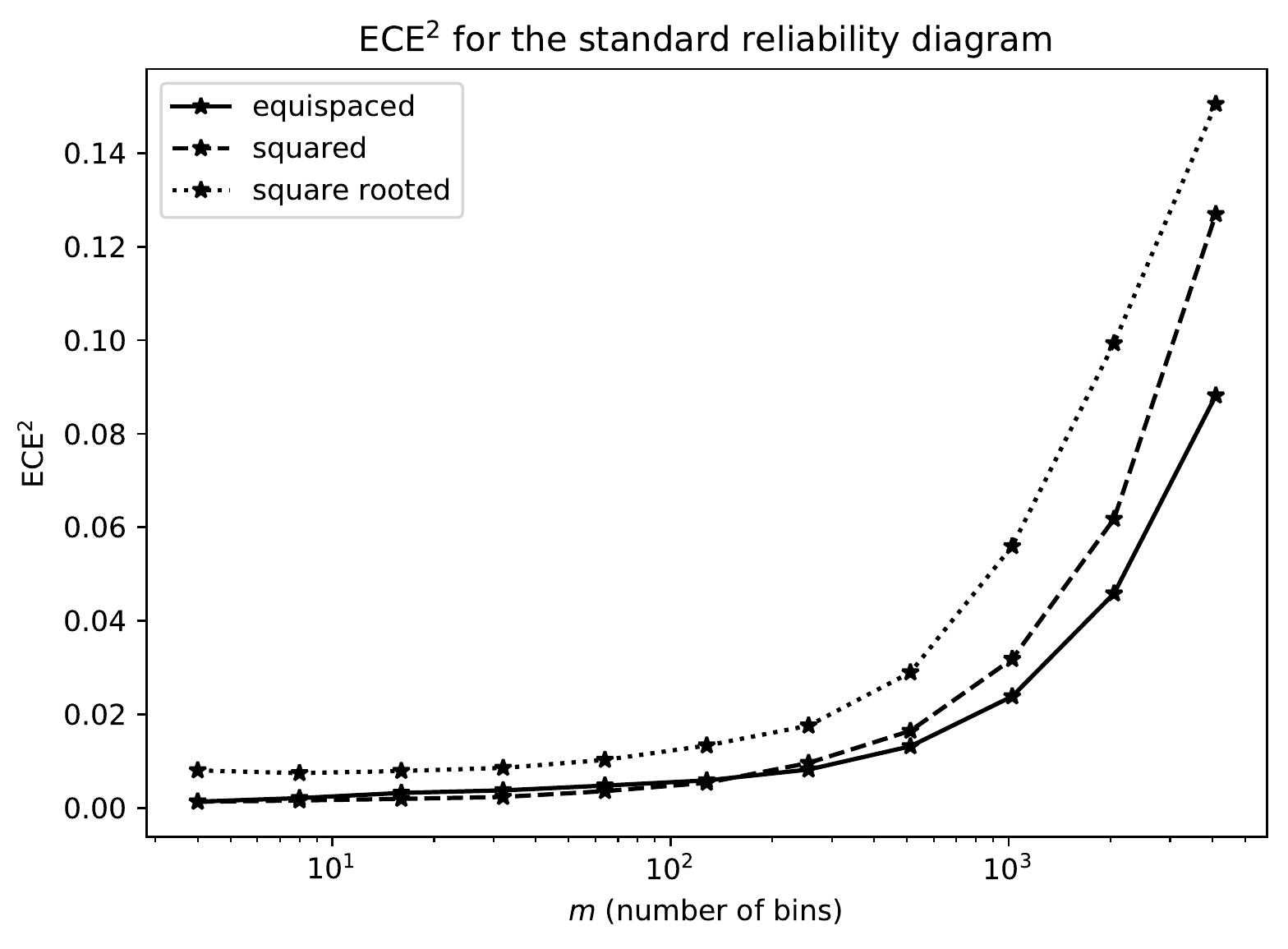}}

\parbox{\imsizes}{\includegraphics[width=\imsizes]
       {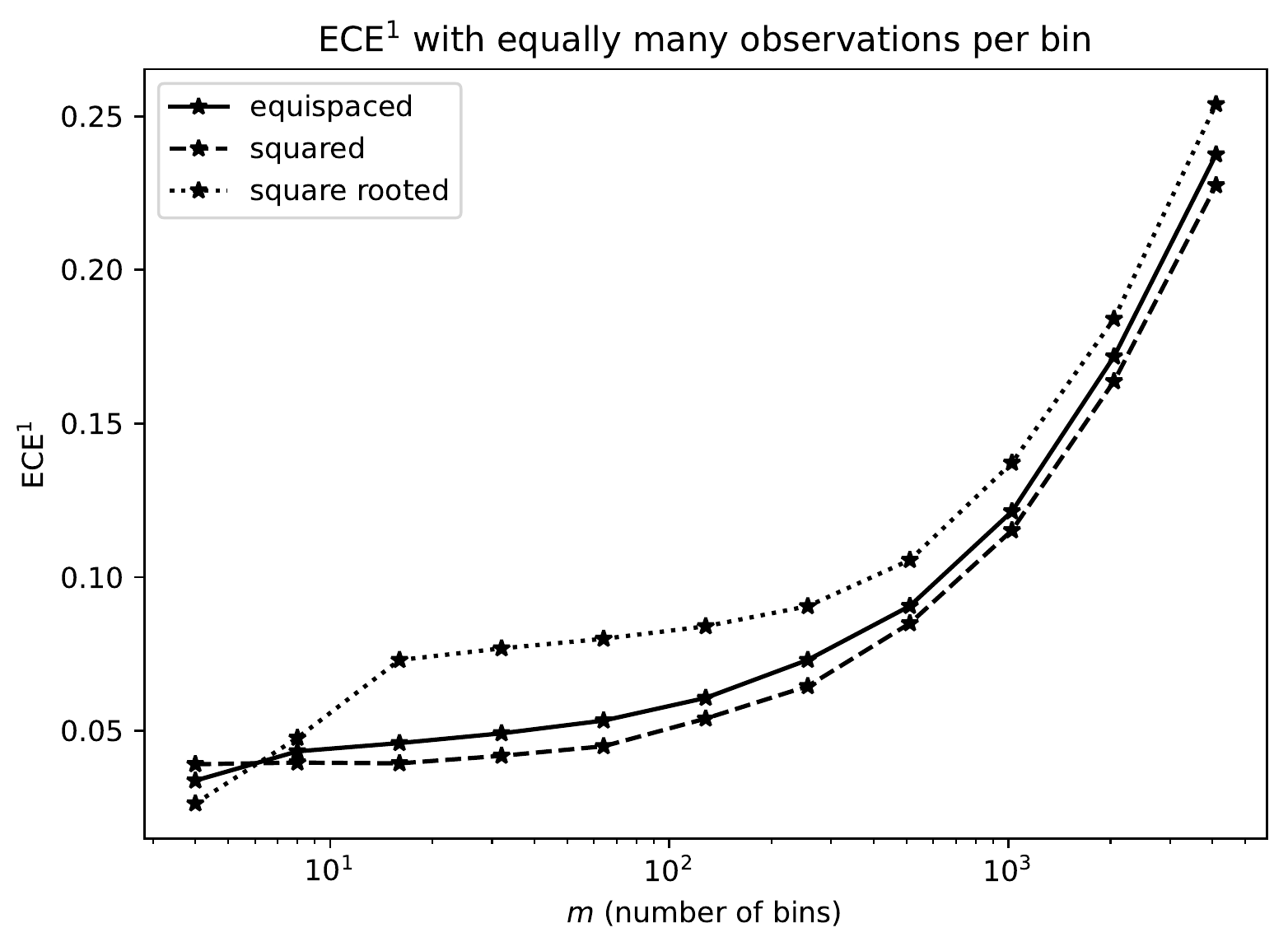}}
\hfil
\parbox{\imsizes}{\includegraphics[width=\imsizes]
       {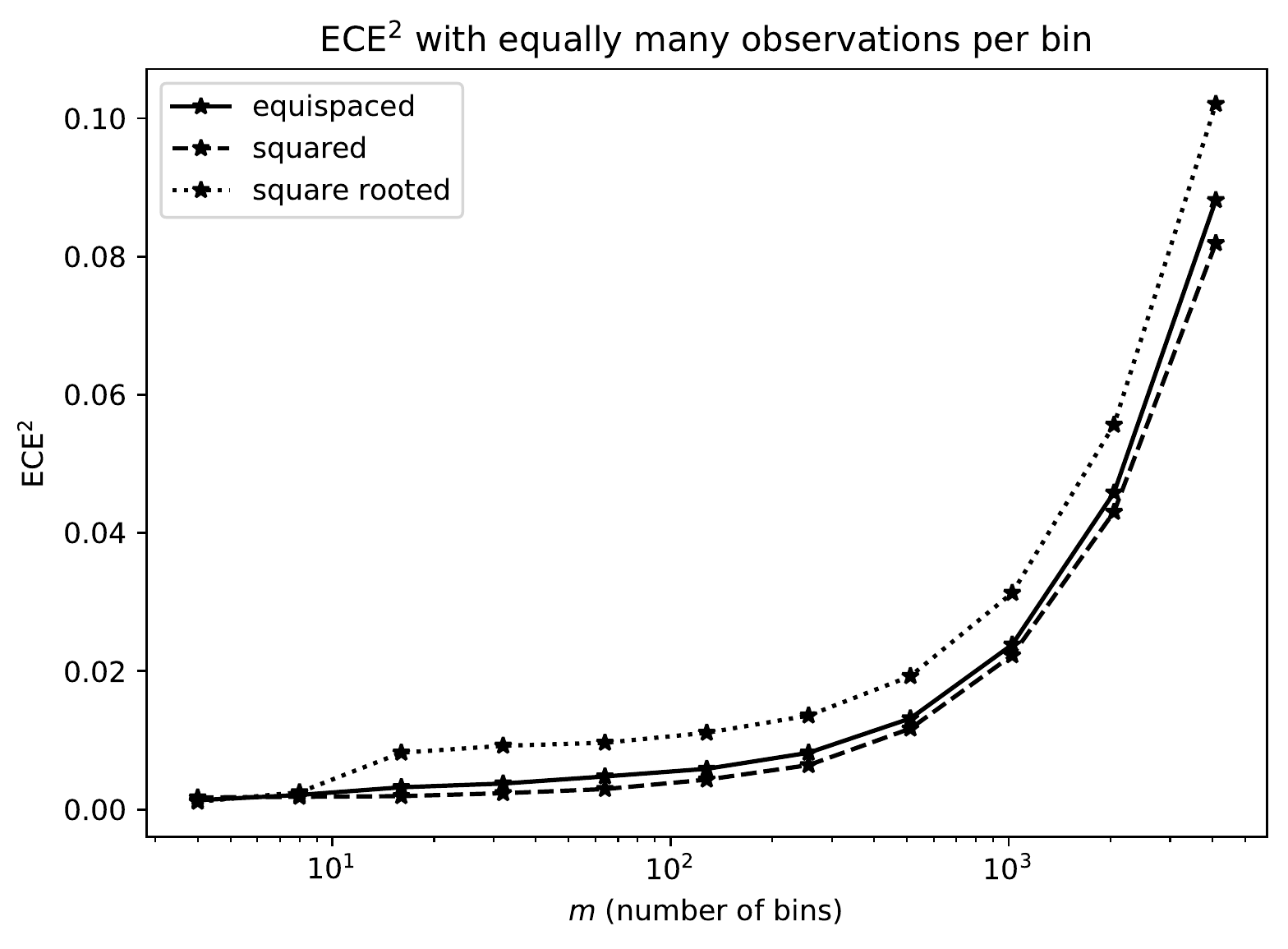}}
\end{center}
\caption{Empirical calibration errors for the synthetic data set
         with the sample size $n =$ 8,192; the scores are equispaced, squared,
         or square rooted, as indicated in the legends.}
\label{8192ece}
\end{figure}

\begin{figure}
\begin{center}
\parbox{\imsizes}{\includegraphics[width=\imsizes]
       {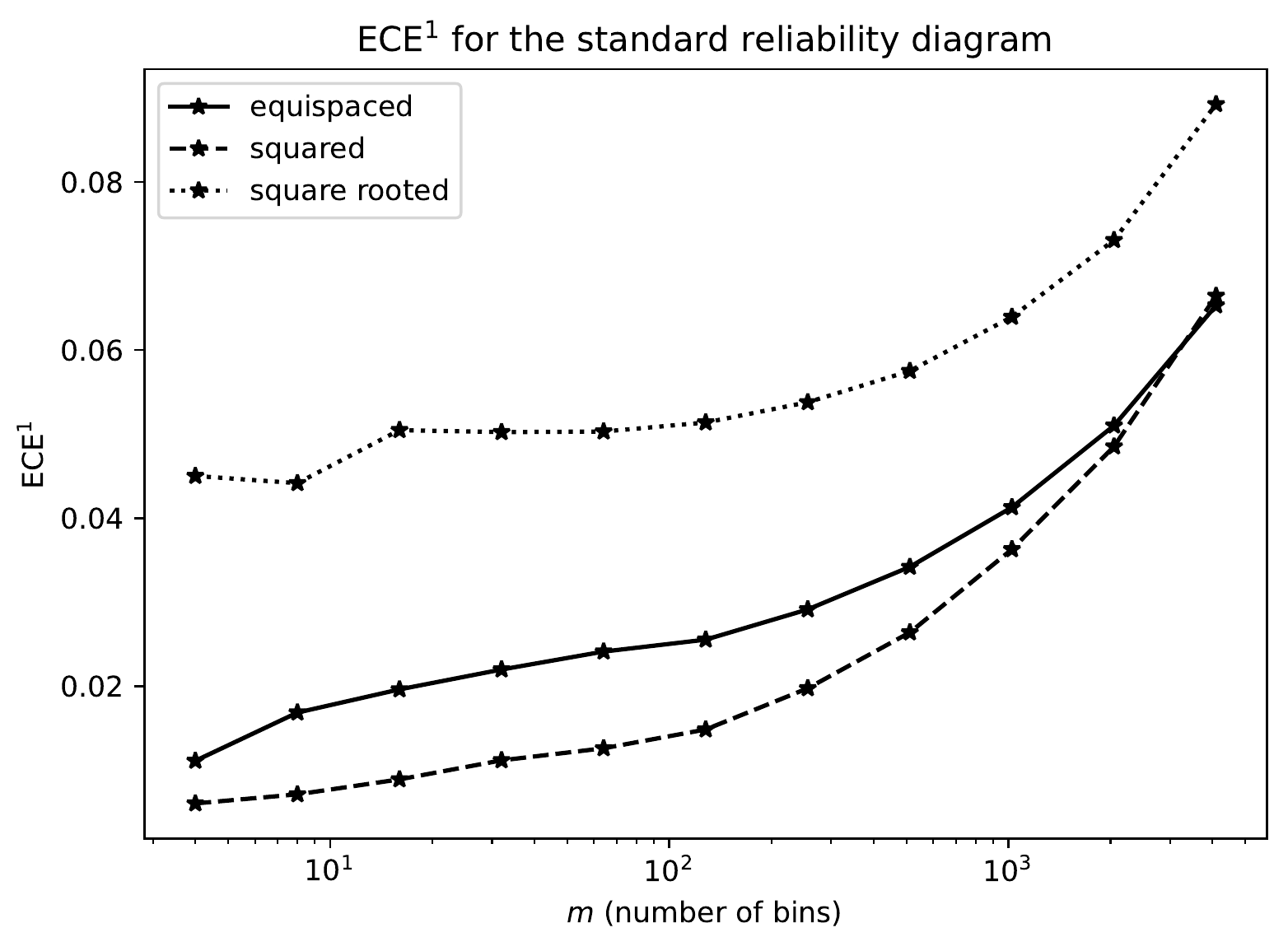}}
\hfil
\parbox{\imsizes}{\includegraphics[width=\imsizes]
       {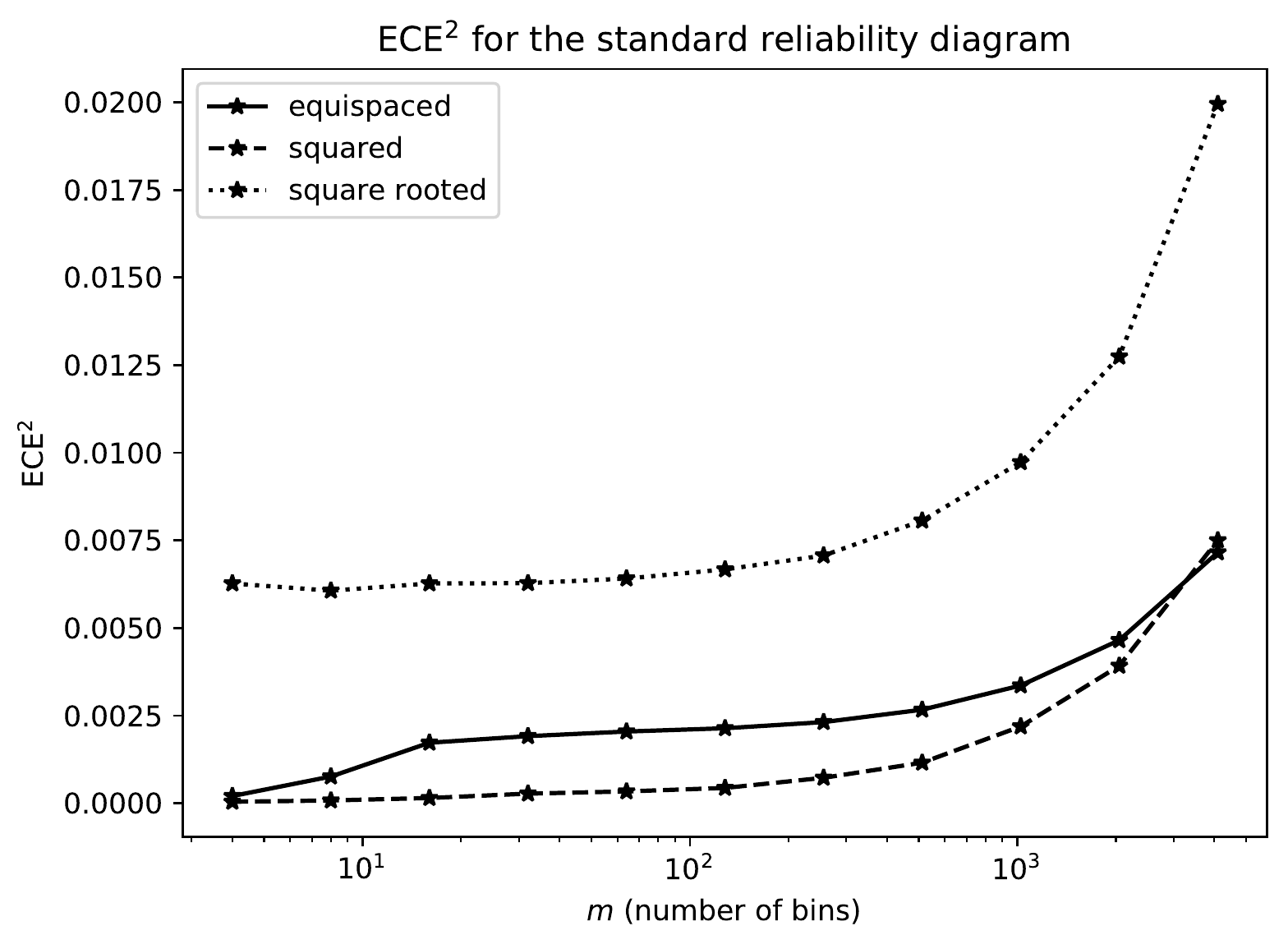}}

\parbox{\imsizes}{\includegraphics[width=\imsizes]
       {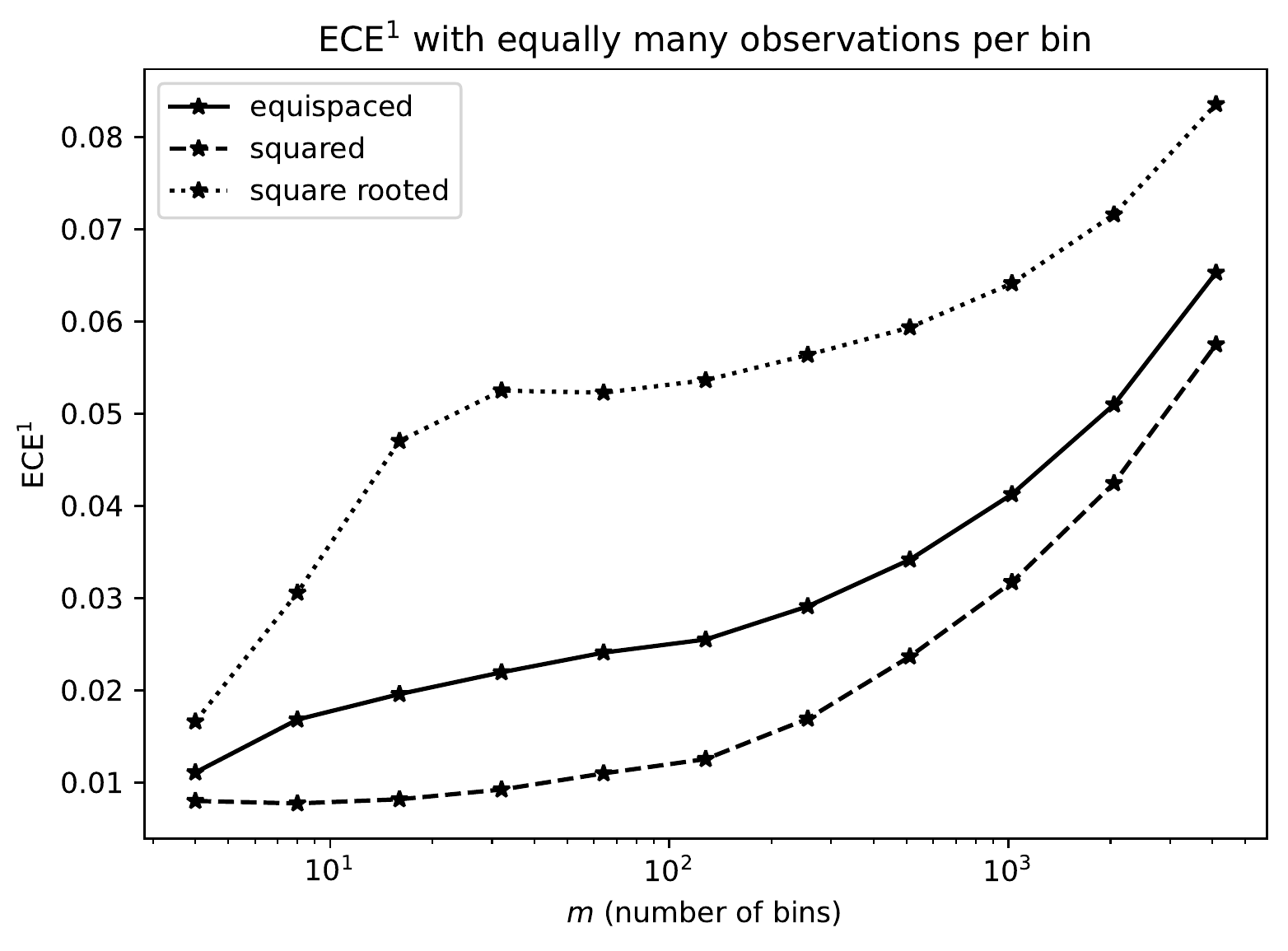}}
\hfil
\parbox{\imsizes}{\includegraphics[width=\imsizes]
       {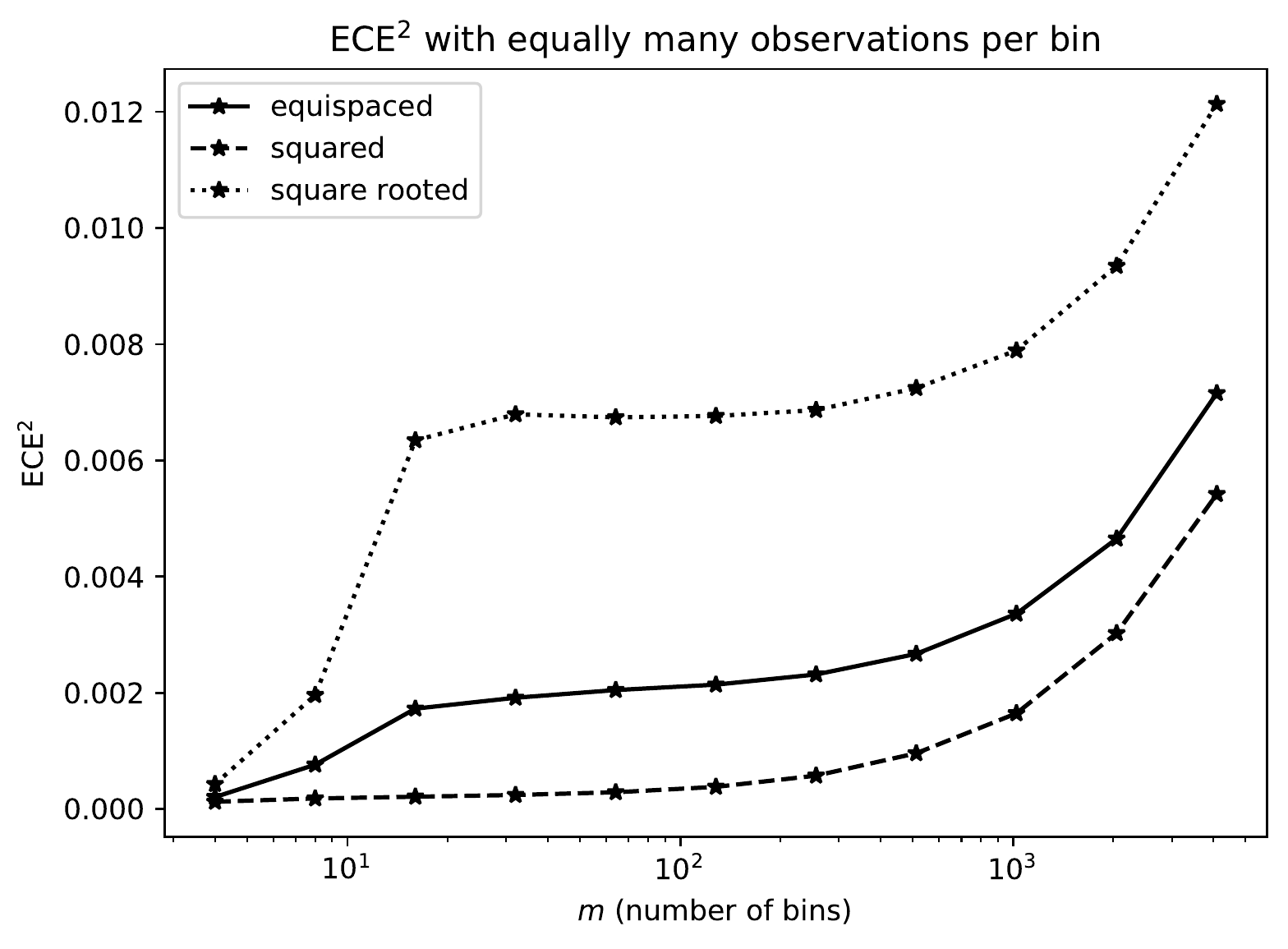}}
\end{center}
\caption{Empirical calibration errors for the synthetic data set
         with the sample size $n =$ 131,072; the scores are equispaced, squared,
         or square rooted, as indicated in the legends.}
\label{131072ece}
\end{figure}

\clearpage

\bibliography{ecevecce}
\bibliographystyle{siam}

\end{document}